\documentclass[letterpaper,11pt,oneside,reqno]{amsart}
\usepackage[sorting=nyt,style=alphabetic,backend=bibtex,hyperref=true,doi=false,url=false,maxbibnames=9,maxcitenames=4,eprint=false]{biblatex}
\makeatletter
\def\blx@maxline{77}
\makeatother
\usepackage{amsmath,amssymb,amsthm,amsfonts}
\usepackage[colorlinks=true,linkcolor=blue,citecolor=red]{hyperref}
\usepackage{graphicx,color,colortbl}
\usepackage{upgreek}
\usepackage[mathscr]{euscript}
\usepackage{hhline}
\DeclareMathAlphabet{\mathpzc}{OT1}{pzc}{m}{it}
\usepackage{mathbbol} 
\allowdisplaybreaks
\numberwithin{equation}{section}

\usepackage{tikz}
\usetikzlibrary{
	shapes,
	arrows,
	positioning,
	decorations.markings,
	decorations.pathmorphing,
	circuits.logic.US,
	circuits.logic.IEC,
	fit,
	calc,
	plotmarks,
	matrix
}

\tikzset{
>=stealth',
help lines/.style={dashed, thick},
axis/.style={<->},
important line/.style={thick},
connection/.style={thick, dotted},
punkt/.style={
rectangle,
rounded corners,
draw=black, thick,
text width=4.5em,
minimum height=2em,
text centered,
},
pil/.style={
->,
thick,
gray,
shorten <=2pt,
shorten >=2pt,}
}

\usepackage{array}
\usepackage{adjustbox}
\usepackage{cleveref}
\usepackage{enumitem}

\usepackage[DIV=12]{typearea}

\usepackage[shadow]{todonotes}
\newtheorem{proposition}{Proposition}[section]
\newtheorem{lemma}[proposition]{Lemma}
\newtheorem{corollary}[proposition]{Corollary}
\newtheorem{theorem}[proposition]{Theorem}
\newtheorem*{theorem*}{Theorem}
\theoremstyle{definition}
\newtheorem{definition}[proposition]{Definition}
\newtheorem{remark}[proposition]{Remark}
\newtheorem*{remark*}{Remark}
\newtheorem{example}[proposition]{Example}

\newcommand\addvmargin[1]{
\node[fit=(current bounding box),inner ysep=#1,inner xsep=0]{};
}

\newcommand\addhmargin[1]{
\node[fit=(current bounding box),inner ysep=0,inner xsep=#1]{};
}

\begin{document}
\title{Yang-Baxter random fields and stochastic vertex models}

\author[A. Bufetov]{Alexey Bufetov}\address{A. Bufetov, 
Hausdorff Center for Mathematics, University of Bonn, Bonn, D-53115 Germany}\email{alexey.bufetov@gmail.com}

\author[M. Mucciconi]{Matteo Mucciconi}\address{M. Mucciconi, 
Department of Physics,
Tokyo Institute of Technology, Tokyo, 152-8551 Japan}\email{matteomucciconi@gmail.com}

\author[L. Petrov]{Leonid Petrov}\address{L. Petrov, Department of Mathematics, University of Virginia, Charlottesville,
VA, 22904 USA,
	and
	Institute for Information Transmission
	Problems, Moscow, 117279 Russia}\email{lenia.petrov@gmail.com}

\date{}

\begin{abstract}
	Bijectivization 
	refines the Yang-Baxter equation into
	a pair of 
	local Markov moves 
	which randomly update the configuration of the vertex model.
	Employing this approach,
	we introduce new Yang-Baxter random fields of Young diagrams based on
	spin $q$-Whittaker and spin Hall-Littlewood symmetric functions.
	We match certain scalar Markovian marginals of these fields
	with 
	(1) the stochastic six vertex model;
	(2) the stochastic higher spin six vertex model;
	and (3) a new vertex model with pushing
	which generalizes the $q$-Hahn PushTASEP introduced recently in
	\cite{CMP_qHahn_Push}.
	Our matchings include models with 
	two-sided stationary initial data,
	and we obtain Fredholm determinantal expressions for 
	the $q$-Laplace transforms of the height functions of all these models.
	Moreover, we also discover
	difference operators
	acting diagonally on spin $q$-Whittaker or (stable) spin Hall-Littlewood symmetric functions.
\end{abstract}

\maketitle

\setcounter{tocdepth}{1}
\tableofcontents
\setcounter{tocdepth}{3}

\section{Introduction} 

\subsection{Overview}

The interplay between symmetric functions and probability blossomed in
the last twenty years. In particular, the frameworks of Schur
processes 
\cite{okounkov2001infinite},
\cite{okounkov2003correlation}
and Macdonald processes
\cite{BorodinCorwin2011Macdonald}
has
lead to a significant progress in understanding a
number of interesting stochastic models from the so-called 
Kardar-Parisi-Zhang 
universality class. 
More recently much attention was directed at
the role of quantum integrability (in the form of the Yang-Baxter equation /
Bethe ansatz) in the theory of symmetric functions, with further
applications to probability. 
It was discovered that combinatorial
properties (most prominently, the Cauchy identity and symmetrization
formulas) of many interesting families of
symmetric functions can be traced back to 
integrability (e.g., see \cite{Borodin2014vertex},
\cite{wheeler2015refined}).
Employing this point of view
and starting with more general solutions to Yang-Baxter equation,
\cite{Borodin2014vertex} and \cite{BorodinWheelerSpinq}
defined two families of symmetric 
functions: the spin Hall-Littlewood (sHL) 
rational symmetric functions and 
the spin $q$-Whittaker (sqW) symmetric polynomials,
which are one-parameter generalizations, respectively, of the classical Hall-Littlewood and
$q$-Whittaker symmetric functions, and obey similar
combinatorial relations. 
See
\Cref{fig:symm_functions_scheme}
for the scheme of various symmetric functions and degenerations between them.

The goal of the present paper 
is to further
study structural properties of the sHL and sqW functions 
and connect them to known and new stochastic models.
Here is a
summary of our results.
\begin{itemize}
	\item 
		Up to now, it was not clear whether new symmetric functions
		coming from integrability are
		eigenfunctions of some difference operators acting on their 
		variables.\footnote{Note, however, that these functions
			(usually taking the form
			$F_\lambda(z_1,\ldots,z_N)$)
			are eigenfunctions
			of vertex models' transfer matrices acting on their \emph{labels} $\lambda$
			(which are tuples of integers $\lambda_1\ge \ldots\ge \lambda_N$
			encoding an arrow configuration). 
			The \emph{variables} $(z_1,\ldots,z_N)$ are tuples of generic
			complex numbers, and the functions are symmetric in the $z_i$'s
			thanks to the Yang-Baxter equation.}
		The presence of such
		operators is both a key structural feature of the theory of
		Macdonald polynomials, and an extremely useful tool for
		applications in probability. 
		We present difference operators acting
		diagonally on the sHL functions
		and on the sqW functions
		which can be used to extract observables 
		($q$-moments of the first row / column)
		of the corresponding measures.
	\item 
		Based on Cauchy identities for sHL / sqW functions, we construct
		\textit{Yang-Baxter} fields of random Young diagrams associated with these
		functions. This allows to relate known stochastic vertex models (stochastic
		six vertex model \cite{BCG6V}, stochastic higher spin vertex model
		\cite{CorwinPetrov2015}, \cite{BorodinPetrov2016inhom}) 
		to sHL
		and sqW functions. 
		In more detail, we match the
		(joint) distribution of the height function in each of these vertex models
		and (joint) distribution 
		of the 
		lengths of the first row / column of Young diagrams
		from the corresponding random field.
		The (joint) distribution of the full diagrams
		is expressed through the (skew) sHL / sqW functions
		in the same manner as in a Schur / Macdonald process.
	\item
		A novel feature of this matching is that 
		we cover a more general class of \emph{two-sided stationary}
		initial conditions in stochastic vertex models.
		These 
		initial conditions depend on two extra parameters 
		(one can think that they encode
		the particle densities on the left and on the right),
		and include the step as well as the stationary translation invariant ones
		(the latter form a one-parameter subfamily).
	\item
		We define a new
		integrable stochastic vertex model with vertex weights 
		expressed through the terminating $q$-hypergeometric series
		$_4\phi_3$.
		These weights come from the R matrix 
		entering 
		the
		Yang-Baxter equation for the sqW functions. 
		The $_4\phi_3$ model generalizes 
		the $q$-Hahn PushTASEP 
		recently introduced in \cite{CMP_qHahn_Push}.
	\item 
		For the three stochastic vertex models mentioned above,
		with the general two-sided stationary initial data, we produce Fredholm
		determinantal expressions for the $q$-Laplace transform
		of the height function at a single point. 
\end{itemize}

Let us now describe our results in more detail.

\begin{figure}[htpb]
	\centering
	\includegraphics{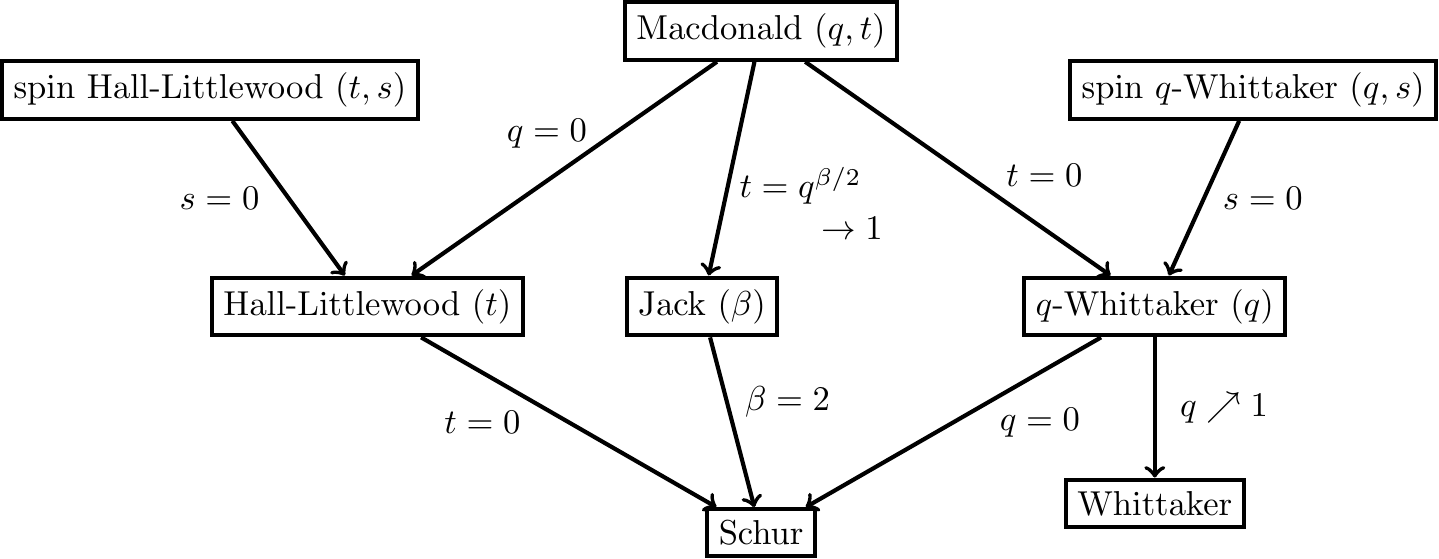}
	\caption{An hierarchy of symmetric functions 
	satisfying Cauchy type summation identities which can be utilized to define 
	random fields of Young diagrams. Arrows mean degenerations. Throughout
	the introduction and most of the text it is convenient
	to replace the parameter
	$t$ by $q$ in spin Hall-Littlewood functions.}
	\label{fig:symm_functions_scheme}
\end{figure}

\subsection{Difference operators}

The sHL functions $\mathsf{F}_\lambda(u_1, \dots, u_n)$ are rational
functions of $n$ variables parametrized by Young diagrams
$\lambda = \lambda_1 \ge \lambda_2 \ge \dots \ge \lambda_{\ell (\lambda)}
>0$, $\lambda_i\in \mathbb{Z}$. 
They can be defined by the following formula:
\begin{equation*}
\mathsf{F}_\lambda(u_1, \dots, u_n)
		:=
		\frac{(1-q)^n}{(q;q)_{n - \ell(\lambda)}}
		\sum_{\sigma \in \mathfrak{S}_n}
		\sigma \biggl\{
			\prod_{1\le i < j\le n} \frac{u_i - q u_j}{u_i - u_j}
			\prod_{i=1}^n \left( \frac{u_i - s}{1 - s u_i} \right)^{\lambda_i}
			\,
			\prod_{i=1}^{\ell (\lambda)} \frac{u_i}{u_i - s}
		\biggr\},
\end{equation*}
where 
$(q;q)_{n - \ell(\lambda)}$ is the $q$-Pochhammer symbol
(cf. \Cref{sub:notation}), 
$\mathfrak{S}_n$ is the permutation group of $n$ elements, and
$\sigma$ acts on the indices of the variables $u_i$, but not
$\lambda_i$ (if $i>\ell(\lambda)$ we have $\lambda_i=0$, 
by agreement). These functions depend on two parameters $q$ and $s$.
The functions $\mathsf{F}_\lambda(u_1, \dots, u_n)$, up
to a certain modification, were introduced in \cite{Borodin2014vertex}; the
modification first appeared in \cite{deGierWheeler2016}. 
In case $s=0$
these functions become standard Hall-Littlewood functions 
\cite[Chapter III]{Macdonald1995}, and for general $s$ many of their
properties are very similar to the ones of the standard
Hall-Littlewood functions (in particular, Cauchy identity,
symmetrization formula, interpretation as a partition function of suitably 
weighted semistandard Young tableaux). 

However, some important properties were
missing; perhaps, the most important one is the presence of difference
operators acting diagonally on $\mathsf{F}_\lambda(u_1, \dots, u_n)$.
We prove that such operators exist.
Define the (Hall-Littlewood versions of) the Macdonald operators by
\begin{equation*}
	\mathop{\mathfrak{D}_r}
	:=
	\sum_{\substack{I\subset\{1,\dots,n \}\\ |I|=r }}
	\biggl(
		\prod_{\substack{i\in I \\ j\in \{1,\dots,n \} \setminus I}}
		\frac{q u_i - u_j}{u_i - u_j}
	\biggr)\,
	T_{0,I},\qquad r=1,2,\ldots,n,
\end{equation*}
where $T_{0,I}$ is the operator setting all $u_i$, $i\in I$, to zero.
Note that the operators $\mathop{\mathfrak{D}_r}$ do not depend on
$s$ and coincide with the standard Macdonald operators. 
We prove the following result.
\begin{theorem}[\Cref{thm:sHL_eigen} in the text]
	\label{thm:diff_op_intro_theorem}
For all Young diagrams $\lambda$ and $n\in \mathbb{Z}_{\ge1}$
we have
	\begin{equation*} 
		\mathop{\mathfrak{D}_r} \mathsf{F}_\lambda(u_1,\ldots,u_n ) =
			e_r(1,q,\dots, q^{n-\ell(\lambda) -1})\,
			\mathsf{F}_\lambda(u_1,\ldots,u_n ),
	\end{equation*}
	where $e_r(x_1,\ldots,x_k )=\sum_{1\le i_1<\ldots<i_r\le k }x_{i_1}\ldots x_{i_r} $
	is the $r$-th elementary symmetric polynomial.
\end{theorem}

Let us now turn to the sqW functions 
$\mathbb{F}^*_{\lambda}(\xi_1,\dots,\xi_m)$.
The shortest way to define them
is via the Cauchy identity
\begin{equation}
\label{eq:Cauchy-into1}
	\sum_{\lambda} \mathbb{F}^*_{\lambda}(\xi_1,\dots,\xi_m)\,\mathsf{F}_{\lambda'} (u_1, \dots, u_n) 
	=
	\prod_{i=1}^m \prod_{j=1}^n \frac{1+\xi_i u_j}{1- u_j s},
\end{equation}
where $\lambda'$ stands for the conjugation of a Young diagram. 
(Note that the left-hand side of \eqref{eq:Cauchy-into1} depends 
on an additional quantization parameter $q$ which enters
both $\mathbb{F}^*_{\lambda}$ and $\mathsf{F}_{\lambda'}$.)
This is indeed a definition of the $\mathbb{F}^*_{\lambda}$'s,
as they can be extracted as coefficients of the expansion thanks to the
orthogonality relation
for the sHL functions
\cite{BCPS2014}
(which we recall in
\Cref{prop:orthog_FF}). 
When $s=0$, $\mathbb{F}^*_\lambda$ becomes the
usual $q$-Whittaker polynomial (i.e., Macdonald polynomial 
\cite[Chapter VI]{Macdonald1995} with $t=0$).

The functions $\mathbb{F}^*_{\lambda}(\xi_1,\dots,\xi_m)$ were in
introduced in \cite{BorodinWheelerSpinq}. They showed that for general $s$ the
family $\{ \mathbb{F}^*_{\lambda}(\xi_1,\dots,\xi_m) \}_{\lambda}$
satisfies natural properties (Cauchy identities and representations
as partition functions). 
The question about the
existence of difference operators acting diagonally on
$\mathbb{F}_\lambda^*$ was open.
We obtain one such
difference operator.
Define the operator acting on rational functions in $(\theta_1,\ldots,\theta_l )$ as follows:
\begin{equation*}
	\mathop{\mathfrak{E}}
	:=
	\sum_{j=1}^l \left( 1+ \frac{s}{\theta_j} \right)^l
	\biggl(\prod_{i \neq j} \frac{\theta_j}{\theta_j - \theta_i} \biggr)
	T_{q^{-1},\theta_j}
	+
	\frac{(-s)^l}{\theta_1 \cdots \theta_l }\, Id.
\end{equation*}
Here $Id$ is the identity operator, and $T_{q^{-1},\theta_j}$ 
acts by multiplying $\theta_j$ by $1/q$.
\begin{theorem}[\Cref{thm:sqW_eigen} in the text]
	We have
	$\mathfrak{E}\, \mathbb{F}^*_{\lambda}(\theta_1,\ldots,\theta_l )
		=
		q^{-\lambda_1} \mathbb{F}^*_\lambda(\theta_1,\ldots,\theta_l )$
		for all Young diagrams $\lambda$ and all $l\in \mathbb{Z}_{\ge1}$.
\end{theorem}

Note that in the classical theory, as well as in the case of sHL
functions, we have many eigenoperators
of all orders, 
rather than just
one. 
The existence of higher order eigenoperators
for the sqW functions remains open.
However, already the presence of one operator 
brings a lot from both
algebraic combinatorial and probabilistic points of view. 
In
particular, first row / column
observables of measures
based on sqW functions 
can be extracted 
and analyzed 
via the 
already standard technique introduced in 
\cite{BorodinCorwin2011Macdonald} for Macdonald measures.

\begin{remark}
	We originally arrived at eigenoperators
	for the sqW and sHL functions through
	$q$-moments of the stochastic higher spin six vertex model computed before
	\cite{CorwinPetrov2015}, \cite{BorodinPetrov2016inhom}.
	Namely, we used the matching via the Yang-Baxter field (see below in \Cref{intro:YB_fields}) 
	to recognize that
	these $q$-moments are at the same time $q$-moments
	of the measures on Young diagrams 
	expressed through the sHL and sqW functions. The difference operators
	arise by reversing the $q$-moment computations starting from known contour
	integrals. 
	However, our proofs of the eigenrelations presented in the paper
	are more direct, and use
	only
	the necessary minimum of the properties
	of the sHL and the sqW functions.
\end{remark}

\subsection{Yang-Baxter fields and matching to stochastic vertex models}
\label{intro:YB_fields}

The usefulness of symmetric functions in probabilistic
questions is greatly emphasized by the frameworks of Schur 
and Macdonald processes. 
This approach
stems from the combination of two general
ideas. 
First, asymptotic behavior of random Young
diagrams with probabilistic weights coming from a symmetric function summation
identity is often accessible via exact computations with symmetric functions.
Second, such Young
diagrams turn out to be related to many natural probabilistic
models. 
In order to quantify this relation, 
one needs to utilize certain
combinatorial structures behind the symmetric functions.

First examples of such usage involved RSK (Robinson-Schensted-Knuth)
correspondence to establish a relation between Schur functions and
models of longest increasing subsequences / last passage percolation / TASEP
\cite{baik1999distribution},
\cite{johansson2000shape}, \cite{OConnell2003}, \cite{OConnell2003Trans}. 
A bit later, a simpler construction
not based on RSK
was suggested
in \cite{BorFerr2008DF}. 
In the present paper we employ
the third type of construction introduced in \cite{BufetovPetrovYB2017}
--- the \emph{Yang-Baxter fields}.
(A more detailed historical overview of 
all these constructions is given in
\Cref{sub:F_G_fields_references}.)
We construct three Yang-Baxter fields 
based on three types of Cauchy identities for 
the sHL and sqW functions.
Let
us formulate a sample result in detail.

Fix $q \in [0,1)$, $s
\in (-1,0)$, and inhomogeneity parameters $\{ \theta_x \}_{x \in
\mathbb{Z}_{\ge 0}}$, $\{ u_y \}_{y \in \mathbb{Z}_{\ge 0}}$,
satisfying $\theta_x \in [-s,-s^{-1}]$, $u_y \in [0,1)$.
Informally, the stochastic higher spin vertex model 
\cite{CorwinPetrov2015}, \cite{BorodinPetrov2016inhom}
is a 
random collection of paths on edges of $\mathbb{Z}_{\geq 0} \times
\mathbb{Z}_{\geq 0}$ such that each vertex $(x,y)$ has one of four
possible types from Figure \ref{fig:weights_HS6VM-copy} and
contributes the weight shown there with $\theta = \theta_x$ and $u =
u_y$. We also need to prescribe (possibly random) boundary conditions
$b_i^{\mathrm{v}}\in \left\{ 0,1 \right\}$, $b_j^{\mathrm{h}}\in
\mathbb{Z}_{\ge0}$, which parametrize the number of arrows coming 
into the quadrant from the left and from below, respectively. 

\begin{figure}[htbp]
    \centering
		\includegraphics[width=.8\textwidth]{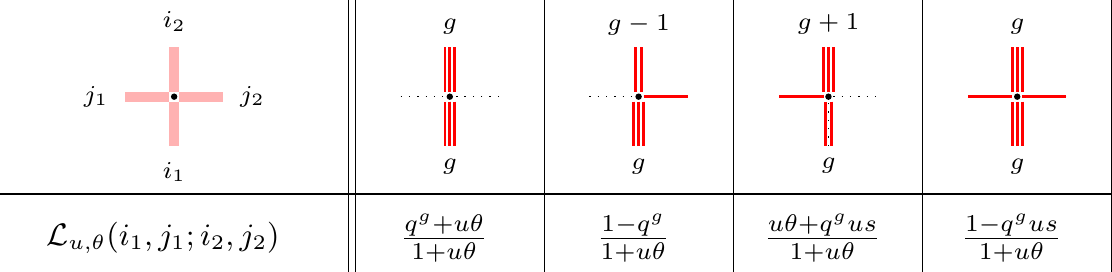}
    \caption{Stochastic vertex weights
		$\mathcal{L}_{u,\theta}(i_1,j_1;i_2,j_2)$ for the higher spin model.}
   \label{fig:weights_HS6VM-copy}
\end{figure}

In more detail, the stochastic higher spin six vertex model
is the (unique) probability
measure on the set of up-right directed paths on
$\mathbb{Z}_{\geq 0} \times
\mathbb{Z}_{\geq 0}$
(with multiple vertical paths allowed per edge, but at most
one horizontal path per edge)
satisfying:
\begin{itemize}
		\item
			Each vertex $(0,y)$ at the vertical boundary $\{ (0,y'):y'\geq 1 \}$
			emanates a path initially pointing to the right if
			$b^\mathrm{v}_y=1$;
		\item Each vertex $(x,0)$ at the horizontal boundary
			$\{ (x',0):x'\geq 1\}$ emanates
			$b^\mathrm{h}_x$
			paths initially pointing upward;
		\item For each $(x,y)$, conditioned to the path configuration at all
			vertices $(x',y')$ such that $x'+y'<x+y$, the probability of a vertex
			configuration $(i_1,j_1;i_2,j_2)$ at $(x,y)$ is given by
			$\mathcal{L}_{u_y,\theta_x}(i_1,j_1;i_2,j_2)$. Moreover, the random
			choices made at diagonally adjacent vertices
			$\ldots,(x-1,y+1),(x,y),(x+1,y-1),\ldots$ are independent under the same
			condition.
\end{itemize}

Take the \textit{step boundary conditions} $b_i^{\mathrm{v}}
\equiv 1$, $b_j^{\mathrm{h}} \equiv 0$. Let
$\mathfrak{h}^{\mathrm{HS}}(x,y)$ be the height function, which is
defined as the number of paths which go through or to the right of the
point $(x,y)$.
We are interested in the distribution of 
$\mathfrak{h}^{\mathrm{HS}}(x,y)$.

On the symmetric function side, 
let us consider a random Young diagram $\lambda^{(x,y)}$ 
with 
\begin{equation}
	\label{eq:intro:dual_measure}
		\mathrm{Prob}( \lambda^{(x,y)} = \nu ) =
		\prod_{\substack{ 1 \leq i \leq x \\ 1\leq j \leq y}} \frac{1 - u_j s}{ 1 + u_j \theta_i } \,
		\mathsf{F}_\nu(u_1, \dots, u_y)\,
		\mathbb{F}^*_{\nu'}(\theta_1, \dots, \theta_x).
	\end{equation}
Cauchy identity \eqref{eq:Cauchy-into1} implies that the sum of the above
probabilities over all $\nu$ is equal to 1, as it should be. 
The next result is a particular case of 
\Cref{thm:sHL_sqW_height_length} in the text:
\begin{theorem}
\label{th-intro-yb:1}
	For any fixed $(x,y)$ in the quadrant,
	the random variables
	$y - \ell(\lambda^{(x,y)})$ and $\mathfrak{h}^{\mathrm{HS}}(x+1,y)$ have the same distribution.
\end{theorem}

Our \Cref{thm:sHL_sqW_height_length} contains a more general
statement. First, it 
provides a matching of the whole 
two-dimensional array $\{\mathfrak{h}^{\mathrm{HS}}(x+1,y)\}$
to an array of scalar observables of a 
Yang-Baxter field of Young diagrams
$\{\lambda^{(x,y)}\}$
which we construct. 
In particular, joint distributions of 
$\mathfrak{h}^{\mathrm{HS}}(x+1,y)$,
when $(x,y)$ follow a down-right path, 
can be accessed through a suitable analogue of 
a Schur or Macdonald process.
Second, 
\Cref{thm:sHL_sqW_height_length} 
includes more general boundary conditions for the 
vertex model, at a cost of suitably modifying the symmetric
functions in the right-hand side of
\eqref{eq:intro:dual_measure}.
Namely, we allow $b_i^{\mathrm{v}}\in\left\{ 0,1 \right\}$ to be independent
Bernoulli random variables, and 
$b_i^{\mathrm{h}}\in \mathbb{Z}_{\ge0}$ to
be independent $q$-negative binomial random variables (cf. \Cref{sub:notation} for the latter).
We call such boundary conditions of the field of Young diagrams
(\emph{two-sided}) \emph{scaled geometric}, they match with
two-sided stationary
boundary conditions in stochastic vertex models.

The matching we just described in \Cref{th-intro-yb:1}
arises in the setting of the 
Cauchy identity
\eqref{eq:Cauchy-into1} 
involving one sHL and one sqW function.
We consider two other Cauchy identities,
one with two sHL functions, and another with two sqW functions.
The vertex models and the corresponding matchings 
are described in \Cref{sub:new_YB_field_sHL_sHL}
and \Cref{sub:new_YB_field_sqW_sqW}.
In all cases we prove analogues of 
\Cref{th-intro-yb:1} (and the more general
\Cref{thm:sHL_sqW_height_length}).
The
matchings between 
stochastic vertex models with two-sided stationary 
boundary conditions and 
symmetric functions have not been known before in any of the cases.

In the sHL/sHL case, on the vertex model side we
obtain the stochastic six vertex model \cite{GwaSpohn1992}, \cite{BCG6V}
and essentially recover (a new degeneration of) the matching of 
\cite{BufetovPetrovYB2017}. 
We observe a curious property that the 
stochastic six vertex model is independent of the 
parameter $s$, while this parameter 
enters the sHL/sHL Yang-Baxter field. 
This independence of $s$
might be explained by 
\Cref{thm:diff_op_intro_theorem}:
the eigenoperators for the spin Hall-Littlewood
polynomials do not depend on $s$ either.

The extension of the sHL/sHL
field matching to the stochastic six vertex model
to the two-sided stationary boundary conditions
is new.
In the sqW/sqW situation the Yang-Baxter
field produces a new integrable stochastic vertex model
with vertex weights expressed through the
terminating $q$-hypergeometric series $_4\phi_3$.
This model generalises the 
$q$-Hahn PushTASEP 
\cite{CMP_qHahn_Push}.
We match the height function of this model 
to a field of random Young diagrams whose distributions
are expressed through a product of two sqW functions.

\subsection{Fredholm determinants for observables}

The difference operators 
$\mathfrak{D}_1$ and $\mathfrak{E}$
diagonal in the sHL or sqW functions, respectively,
allow to express (in a nested contour integral form) 
the $q$-moments of the height
function in each of the three vertex models
with step boundary conditions.
It is known
(e.g., see \cite{BorodinCorwinSasamoto2012})
that such $q$-moment formulas 
can be organized into generating series
leading to Fredholm determinantal formulas
for the $q$-Laplace
transform 
$\mathbb{E}\left( 1/(\zeta q^{\mathcal{H}(x,y)};q)_\infty \right)$.
where $\mathcal{H}(x,y)$ 
is the height function in either of the
three 
models.
This approach works well both for the stochastic six vertex and stochastic
higher spin six vertex models with step boundary conditions.

However, for the $_4\phi_3$ vertex model only finitely many of the $q$-moments exist,
and thus the generating series cannot be used.
Moreover, 
for the more general two-sided stationary
boundary conditions, explicit $q$-moments
are not known and also may not be finite. 
We overcome both these issues 
at the same time
by 
considering an analytic continuation
based on the fusion procedure for vertex models 
\cite{KulishReshSkl1981yang}
(see \cite{CorwinPetrov2015} for a stochastic interpretation of fusion).
We start with the Fredholm determinant
for the (inhomogeneous) stochastic six vertex model
with parameters $(v_x,u_y)$, where $(x,y)\in \mathbb{Z}_{\ge0}^2$.
Then we replace each $u_i$ and $v_j$ by a finite geometric sequence
$u_i,qu_i,\ldots,q^{J_i-1}u_i $
and 
$v_j,qv_j,\ldots, q^{I_j-1}v_j$.
It turns out that the resulting measure
depends on the parameters 
$(v_x,q^{I_x},u_y,q^{J_y})$ in an analytic way.
Then, taking certain specializations
of these parameters, we can get to both the sqW functions
and the two-sided stationary boundary conditions in the vertex models.
The fusion and analytic continuation from 
sHL functions to the sqW ones was first performed in 
\cite{BorodinWheelerSpinq}.

The Fredholm determinantal formula
we obtain in the sqW/sqW setting
in particular
establishes the Fredholm determinant for the
$q$-Hahn PushTASEP which was conjectured in \cite{CMP_qHahn_Push}.

Analytic continuations leading to 
Fredholm determinants for stationary 
stochastic particle systems 
were performed in 
\cite{BorodinCorwinFerrariVeto2013}
($q$-Whittaker measures and random polymers)
and 
\cite{Amol2016Stationary}
(stochastic six vertex model).
In the first reference, the continuation
significantly used the structure of the algebra of symmetric functions.
Our analytic continuation based on fusion is more similar to the approach
taken in the second reference, but due to 
connections with sHL and sqW symmetric functions, the 
argument is more straightforward.

\subsection{Notation}
\label{sub:notation}

Throughout the paper we use the 
$q$-Pochhammer
symbols
\begin{equation}
	\label{eq:q_Pochhammer}
	(a; q)_{n} =
	\begin{cases}
		1,&n=0;\\
		\prod_{i=1}^{n}(1 - a q^{i-1}), &n\ge1;\\
		\prod_{i=n}^{-1}(1-aq^{i})^{-1},&n\le -1,
	\end{cases}
	\qquad \text{and} \qquad
	(a; q)_{\infty} = \prod_{i=1}^{\infty}(1 - a q^{i-1}).
\end{equation}
We also use the notation
\begin{equation}
	\label{eq:hypergeom_series}
	\begin{split}
		\setlength\arraycolsep{1pt}{}_{k+1} \overline{ \phi}_k\left(\begin{minipage}{2.7
		cm}
		\center{$q^{-n}; a_1, \dots,  a_k$}\\\center{$b_1, \dots , b_k $}
		\end{minipage}
		\Big\vert\, q,z\right)
		&
		= 
		\setlength\arraycolsep{1pt}
		\prod_{i=1}^{k}(b_i;q)_n
		\cdot
		{}_{k+1} { \phi}_k\left(\begin{minipage}{2.7
		cm}
		\center{$q^{-n}; a_1, \dots,  a_k$}\\\center{$b_1, \dots , b_k $}
		\end{minipage}
		\Big\vert\, q,z\right)
		\\&
		= 
		\sum_{j=0}^n z^j\, \frac{(q^{-n};q)_j}{(q;q)_j} \prod_{i=1}^k (a_i;q)_j (q^j b_i; q)_{n-j}
	\end{split}
\end{equation}
for the regularized terminating $q$-hypergeometric series.

We say that a random variable $X$ has the \emph{$q$-negative binomial} distribution with parameters $(r,p)$, or $X\sim q\textnormal{-}\mathrm{NB}(r,p)$, if 
\begin{equation}
    \mathrm{ Prob }\{ X=k \}=
    p^k\frac{(r;q)_k}{(q;q)_k} \frac{(p;q)_\infty}{(pr;q)_\infty}.
\end{equation}
In case $r=0$ we say that $X$ is a \emph{$q$-Poisson} random variable of parameter $p$, or
$X\sim q\textnormal{-}\mathrm{Poi}(p)$ (sometimes this distribution is also called \emph{$q$-geometric}). 
Finally, the \emph{Bernoulli} random variable $X\sim \mathrm{Ber}(p)$, $X\in\left\{ 0,1 \right\}$, 
has $\mathrm{Prob}\{ X=1 \}=p$ and $\mathrm{Prob}\{ X=0 \}=1-p$.

\subsection{Outline}

In \Cref{sec:abstract_formalism}
we outline a general formalism
for constructing random fields 
from symmetric (rational) functions.
In \Cref{sec:summary_sHL_sqW} we recall 
the spin Hall-Littlewood and spin $q$-Whittaker
symmetric functions introduced in \cite{Borodin2014vertex}
and \cite{BorodinWheelerSpinq}, respectively.
In \Cref{sec:analytic_continuation}
we consider the general form of the skew Cauchy
equation which follows from the fused Yang-Baxter equation,
and in 
\Cref{sec:scaled_geometric} consider yet another family of its specializations
which we refer to as ``scaled geometric''.
In \Cref{sec:YB_fields_through_bijectivisation}
we apply bijectivization to the Yang-Baxter equations obtaining
local stochastic moves of Yang-Baxter type. 
In \Cref{sec:new_three_fields} we discuss the Yang-Baxter fields thus arising
together with their scalar marginals (projections).
In \Cref{sec:diff_op} we define difference operators 
acting diagonally on our symmetric functions, and study their properties.
In \Cref{sub:fredholm_determinant_for_marginal_processes}
we write down Fredholm determinantal observables for 
stochastic particle systems arising from our 
Yang-Baxter fields.
Finally, in
\Cref{app:YBE}
we list all instances of the Yang-Baxter equation
employed in the paper, and discuss the nonnegativity of 
terms entering these equations.

\subsection*{Acknowledgments}

We are grateful to Alexei Borodin for helpful discussions,
and to Ivan Corwin for valuable comments on an earlier version of 
the text.
This work has started when MM and LP participated
in the program 
``Non-equilibrium systems and special functions'' at MATRIX Institute,
and we are grateful to the organizers for hospitality and support.
The work of AB was partially supported by the 
German Research Foundation under Germany’s Excellence Strategy -- EXC 2047 ``Hausdorff Center for Mathematics''.
LP was partially supported by the NSF grant DMS-1664617.

\section{Random fields from skew Cauchy identities}
\label{sec:abstract_formalism}

In this section we describe an abstract formalism of 
random fields which is applied
to several concrete situations in the rest of the paper.

\subsection{Skew Cauchy structures}
\label{sub:F_G_skew_Cauchy_structure}

The fields we consider in this paper 
are collections of random Young diagrams
indexed by points of the two-dimensional quadrant $\mathbb{Z}_{\ge0}^2$.
A \emph{Young diagram} (=~\emph{partition}) is 
a sequence of integers
$\lambda=(\lambda_1\ge \ldots\ge \lambda_{\ell(\lambda)} > 0)$.
The quantity $\ell(\lambda)$ is called the length of the Young diagram $\lambda$.
Denote by $\mathbb{Y}$ the set of all Young diagrams
including the empty one $\lambda=\varnothing$ 
(by agreement, $\ell(\varnothing)=0$).
It is convenient to be able to
add zeros at the end of a Young diagram
$\lambda$, and to not distinguish the sequences
$(\lambda_1,\ldots,\lambda_\ell )$
and
$(\lambda_1,\ldots,\lambda_\ell,0 )$.

Assume that for every pair of Young diagrams $\lambda,\mu$ and any $k\in \mathbb{Z}_{\ge1}$ we are 
given two functions 
$\mathfrak{F}_{\lambda/\mu}(u_1,\ldots,u_k )$
and 
$\mathfrak{G}_{\lambda/\mu}(u_1,\ldots,u_k )$
(which may also depend on 
some external parameters).
This data is called a 
\emph{skew Cauchy structure}
if the functions 
satisfy the following properties:
\begin{enumerate}[label=\bf{\arabic*.}]
	\item The functions are rational in the $u_i$'s and are symmetric with respect to 
		permutations of $u_1,\ldots,u_k $. 
	\item Define relations
		$\prec_k$
		and 
		$\mathop{\dot\prec_k}$
		on $\mathbb{Y}\times\mathbb{Y}$ such that 
		\begin{equation}
			\label{eq:F_G_prec_relations}
			\mathfrak{F}_{\lambda/\mu}(u_1,\ldots,u_k )\ne 0 \quad \textnormal{iff $\mu\prec_{k}\lambda$};
			\qquad 
			\mathfrak{G}_{\lambda/\mu}(u_1,\ldots,u_k )\ne 0 \quad \textnormal{iff $\mu\mathop{\dot\prec_{k}}\lambda$}.
		\end{equation}
		Moreover, for each $\lambda$ the sets $\left\{ \mu\colon \mu\prec_k\lambda \right\}$
		and $\left\{ \rho\colon \rho\mathop{\dot\prec_k} \lambda \right\}$ are finite.
		By agreement, we extend these relations to $k=0$
		and set 
		$\mathfrak{F}_{\lambda/\mu}(\varnothing)=\mathfrak{G}_{\lambda/\mu}(\varnothing)
		=
		\mathbf{1}_{\lambda=\mu}$.\footnote{Throughout the paper $\mathbf{1}_{A}$ denotes the indicator of $A$.}
	\item (Branching rules)
		For each $1\le m\le k-1$ we have
		\begin{equation}
		\label{eq:F_G_branching}
			\mathfrak{F}_{\lambda/\mu}(u_1,\ldots,u_k )=
			\sum_{\varkappa}\mathfrak{F}_{\lambda/\varkappa}(u_1,\ldots,u_m )\mathfrak{F}_{\varkappa/\mu}(u_{m+1},\ldots,u_k ),
		\end{equation}
		and the same branching rule 
		for $\mathfrak{G}_{\lambda/\mu}$
		(obtained by 
		replacing each $\mathfrak{F}$ above by $\mathfrak{G}$) holds, too.
		Note that the sum over $\varkappa$ above is finite. 
	\item (Skew Cauchy identity)
		There exists a rational function $\Pi(u;v)$
		and
		a subset $\mathsf{Adm}\subseteq \mathbb{C}^2$ 
		such that for all $(u,v)\in \mathsf{Adm}$ one has
		(see \Cref{fig:quadruplet} for the illustration)
		\begin{equation}
			\label{eq:F_G_single_skew_Cauchy}
			\sum_{\nu}
			\mathfrak{F}_{\nu/\mu}(u)
			\mathfrak{G}_{\nu/\lambda}(v)
			=
			\Pi(u;v)
			\sum_{\varkappa}\mathfrak{F}_{\lambda/\varkappa}(u)\mathfrak{G}_{\mu/\varkappa}(v).
		\end{equation}
		Note that the sum over $\varkappa$ in the right-hand side is finite 
		while the sum over $\nu$ in the left-hand side might be infinite.
		The set $\mathsf{Adm}$ corresponds to pairs $(u,v)$ for which the infinite sum converges.
	\item (Nonnegativity)
		There exist two sets 
		$\mathsf{P},\dot{\mathsf{P}}\subseteq \mathbb{C}$ such that
		\begin{equation*}
			\mathfrak{F}_{\lambda/\mu}(u_1,\ldots,u_k )\ge0, \quad u_i\in \mathsf{P}\textnormal{ for all $i$};
			\qquad 
			\mathfrak{G}_{\lambda/\mu}(v_1,\ldots,v_k )\ge0, \quad v_j\in \dot{\mathsf{P}}\textnormal{ for all $j$}.
		\end{equation*}
\end{enumerate}

\begin{figure}[htpb]
	\centering
	\includegraphics{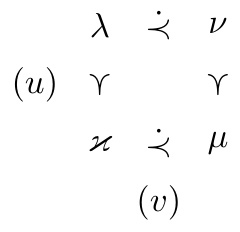}
	\caption{An illustration of relations between the four diagrams
	$\lambda,\mu,\varkappa$, and $\nu$
	in the skew Cauchy identity \eqref{eq:F_G_single_skew_Cauchy}.
	The variable $u$ should be thought of corresponding to the vertical
	direction, and $v$ corresponds to the horizontal one.
	Here we are using the shorthand notation $\mathop{\prec}=\mathop{\prec_1}$ 
	and $\mathop{\dot\prec}=\mathop{\dot\prec_1}$.}
	\label{fig:quadruplet}
\end{figure}

\begin{remark}\label{rmk:generic}
	The functions $\mathfrak{F}_{\lambda/\mu}$
	and 
	$\mathfrak{G}_{\lambda/\mu}$ are rational thus
	might be undefined for special values of the variables $u_i$ or the 
	external parameters. Therefore, all statements in this section should be understood
	in the sense of generic variables and parameters (i.e., outside vanishing sets of some algebraic expressions).
\end{remark}

The branching rules \eqref{eq:F_G_branching}
imply that for any $\mu,\lambda$
the function $\mathfrak{F}_{\lambda/\mu}(u_1,\ldots,u_k)$ vanishes
unless there exists a sequence of Young diagrams $\{\varkappa^{(i)}\}$ with
\begin{equation*}
	\mu\prec_1 \varkappa^{(1)}\prec_1 \varkappa^{(2)}\prec_1 \ldots \prec_1 \varkappa^{(k-1)}\prec_1\lambda.
\end{equation*}
If $\mathfrak{F}_{\lambda/\mu}(u)\ne 0$ for all pairs $\mu\prec_1\lambda$, then
we can replace the relation $\prec_k$
by the existence of a sequence $\varkappa^{(i)}$ as above, 
and \eqref{eq:F_G_prec_relations} will continue to hold.
A similar remark is valid for $\mathop{\dot\prec_k}$, too.

Note also that the skew Cauchy identity for single variables
\eqref{eq:F_G_single_skew_Cauchy} together with \eqref{eq:F_G_branching}
implies the skew Cauchy identity for any number of variables:
\begin{equation}
	\label{eq:F_G_multivar_skew_Cauchy}
		\sum_{\nu}
		\mathfrak{F}_{\nu/\mu}(u_1,\ldots,u_n )
		\mathfrak{G}_{\nu/\lambda}(v_1,\ldots,v_m )
		=
		\prod_{i=1}^{n}\prod_{j=1}^{m}\Pi(u_i;v_j)
		\sum_{\varkappa}\mathfrak{F}_{\lambda/\varkappa}(u_1,\ldots,u_n )\mathfrak{G}_{\mu/\varkappa}(v_1,\ldots,v_m ),
\end{equation}
where $(u_i,v_j)\in \mathsf{Adm}$ for all $i,j$.

\begin{example}
	The prototypical example of a skew Cauchy structure
	is given by the \emph{Schur symmetric polynomials}
	\cite[Chapter I]{Macdonald1995}:
	\begin{equation*}
		\mathfrak{F}_{\lambda/\mu}(u_1,\ldots,u_k )
		=
		\mathfrak{G}_{\lambda/\mu}(u_1,\ldots,u_k )
		=
		s_{\lambda/\mu}(u_1,\ldots,u_k ),
	\end{equation*}
	where $s_{\lambda/\mu}$
	is the 
	skew Schur polynomial.
	The relations
	$\mu\prec_1\lambda$ and $\mu\mathop{\dot \prec_1}\lambda$
	are the same and mean interlacing:
	\begin{equation*}
		\mu\prec\lambda
		\qquad 
		\Leftrightarrow
		\qquad 
		\lambda_1\ge \mu_1\ge \lambda_2\ge \mu_2\ge \ldots. 
	\end{equation*}
	The factor in the right-hand side of the skew Cauchy identity
	is $\Pi(u;v)=\dfrac1{1-uv}$, and the convergence in the left-hand side holds with
	$\mathsf{Adm}=\left\{ (u,v)\colon |uv|<1 \right\}$.
	The nonnegativity sets are $\mathsf{P}=\dot{\mathsf{P}}=\mathbb{R}_{\ge0}$,
	and
	the fact that $s_{\lambda/\mu}(u_1,\ldots,u_k )\ge0$ for $u_i\ge0$ 
	follows from the combinatorial formula for the skew Schur polynomials
	representing them as generating functions of semistandard Young
	tableaux of the skew shape $\lambda/\mu$.
	
	This Schur skew Cauchy structure will serve 
	as a running example throughout this section. 
	In the rest of the paper we consider 
	other skew Cauchy structures 
	associated with spin Hall-Littlewood and 
	spin $q$-Whittaker functions.
\end{example}

\subsection{Gibbs measures}
\label{sub:Gibbs_measures_F_G}

Through the branching rules,
each family of functions $\mathfrak{F}_{\lambda/\mu}$ and $\mathfrak{G}_{\lambda/\mu}$
leads to a version of a Gibbs property.
This property also depends on a choice of parameters
$u_1,u_2,\ldots \in \mathsf{P}$ 
or 
$v_1,v_2,\ldots \in \dot{\mathsf{P}}$,
respectively,
which we assume fixed.

\begin{definition}[Gibbs measures]
	\label{def:F_G_Gibbs}
	A probability measure on a
	(finite or infinite)
	sequence of 
	Young diagrams
	\begin{equation*}
		\lambda^{(0)}\prec_1\lambda^{(1)}\prec_1 \ldots\prec_1\lambda^{(n)}\prec_1\ldots 
	\end{equation*}
	is called \emph{$\mathfrak{F}$-Gibbs}
	(with parameters $u_i$) if for all $m,n$ with $0\le m<n-1$, the 
	conditional distribution of 
	$\lambda^{(m+1)},\ldots,\lambda^{(n-1)} $ given $\tau=\lambda^{(m)}$ and $\rho=\lambda^{(n)}$ has the form
	\begin{equation*}
		\frac{1}{Z}\,\mathfrak{F}_{\lambda^{(m+1)}/\tau}(u_{m+1})
		\mathfrak{F}_{\lambda^{(m+2)}/\lambda^{(m+1)}}(u_{m+2})
		\ldots 
		\mathfrak{F}_{\rho/\lambda^{(n-1)}}(u_n),
	\end{equation*}
	and, in particular, is independent of $\lambda^{(i)}$ with $i<m$ or $i>n$.
	The normalizing constant has the form
	$Z=\mathfrak{F}_{\rho/\tau}(u_{m+1},\ldots,u_n )$ by \eqref{eq:F_G_branching}.
	Note that the set of sequences $\lambda^{(m)}\prec_1\lambda^{(m+1)}\prec_1\ldots \prec_1\lambda^{(n)}$
	with fixed $\lambda^{(m)}$ and $\lambda^{(n)}$ is finite, so there are no convergence issues in defining $Z$.
	
	The \emph{$\mathfrak{G}$-Gibbs} property is defined in a similar way.
\end{definition}

\begin{example}
	In the Schur case with $u_i\equiv u$ for all $i$, the Gibbs property
	reduces to the 
	one with uniform conditional probabilities.
	That is, a measure on 
	an interlacing sequence of diagrams
	$\varnothing\prec \lambda^{(1)}\prec \lambda^{(2)}\prec \ldots $
	is (uniform) Gibbs if,
	conditioned on any $\lambda^{(n)}=\rho$, 
	the distribution of $\lambda^{(1)},\ldots,\lambda^{(n-1)} $ is 
	uniform among all 
	sequences of Young diagrams satisfying the interlacing constraints.
\end{example}

\subsection{Random fields associated to a skew Cauchy structure}
\label{sub:F_G_random_fields}

Fix a skew Cauchy structure $(\mathfrak{F},\mathfrak{G})$
and 
parameters
\begin{equation}\label{eq:F_G_field_parameters_u_v}
	u_1,u_2,\ldots; \ v_1,v_2,\ldots ,
	\quad
	\textnormal{such that}
	\quad
	\textnormal{$(u_x,v_y)\in \mathsf{Adm}$, $u_x\in \mathsf{P}$, $v_y\in \dot{\mathsf{P}}$ for all $x,y$.}
\end{equation}
A random field 
corresponding to this data is a family
of random Young diagrams $\boldsymbol \lambda=\{\lambda^{(x,y)}\}$ 
indexed by points of the quadrant $(x,y)\in \mathbb{Z}_{\ge0}^{2}$
with a certain spatial dependence structure
determined by the functions $\mathfrak{F}_{\nu/\mu}$ and $\mathfrak{G}_{\nu/\mu}$.
We begin by describing the appropriate class of boundary conditions.

\begin{definition}[Gibbs boundary conditions]
	\label{def:F_G_boundary_conditions}
	A random two-sided sequence of Young diagrams
	\begin{equation}
		\label{eq:F_G_field_boundary_conditions}
		\boldsymbol \tau=\bigl(\ldots\succ_1 
		\tau^{(0,3)}\succ_1 
		\tau^{(0,2)} \succ_1 \tau^{(0,1)} \succ_1
		\tau^{(0,0)} \mathop{\dot\prec_1} \tau^{(1,0)}
		\mathop{\dot\prec_1}\tau^{(2,0)}
		\mathop{\dot\prec_1}
		\tau^{(3,0)}
		\mathop{\dot\prec_1}
		\ldots \bigr)
	\end{equation}
	is called an 
	\emph{$(\mathfrak{F},\mathfrak{G})$-Gibbs boundary condition}
	(or a \emph{Gibbs boundary condition}, for short)
	if
	the sequences
	$\{\tau^{(0,y)}\}_{y\ge0}$
	and 
	$\{\tau^{(x,0)}\}_{x\ge0}$
	are $\mathfrak{F}$-Gibbs and $\mathfrak{G}$-Gibbs,
	respectively
	(in the sense of
	\Cref{def:F_G_Gibbs}, with parameters \eqref{eq:F_G_field_parameters_u_v}),
	and, moreover, the sequences
	$\{\tau^{(0,y)}\}_{y\ge1}$
	and 
	$\{\tau^{(x,0)}\}_{x\ge1}$
	are conditionally independent given 
	$\tau^{(0,0)}$.
\end{definition}
For a Gibbs boundary condition $\boldsymbol \tau$ denote
\begin{equation}\label{eq:Z_boundary_values}
	Z^{(x,y)}_{\mathrm{boundary}}:=\sum_{\tau^{(0,0)}}
	\mathfrak{F}_{\tau^{(0,y)}/\tau^{(0,0)}}(u_1,\ldots,u_y )
	\mathfrak{G}_{\tau^{(x,0)}/\tau^{(0,0)}}(v_1,\ldots,v_x ), \qquad 
	(x,y)\in \mathbb{Z}_{\ge0}^{2}.
\end{equation}
This quantity is random and depends on 
$\tau^{(x,0)}$ and
$\tau^{(0,y)}$.

We will mostly deal with the following particular case of Gibbs boundary conditions:
\begin{definition}[Step-type boundary conditions]
	\label{def:F_G_step_type_boundary}
	A Gibbs boundary condition $\boldsymbol \tau$ is called \emph{step-type}
	in the \emph{vertical} (resp., \emph{horizontal}) direction if 
	the $\mathfrak{F}$-Gibbs distribution of the sequence
	$\{\tau^{(0,y)}\}_{y\ge0}$
	(resp., the $\mathfrak{G}$-Gibbs distribution of $\{\tau^{(x,0)}\}_{x\ge0}$)
	is supported on a single sequence.
	That is, the boundary diagrams
	are nonrandom but
	the Gibbs property still holds.

	A \emph{step-type boundary condition} $\boldsymbol \tau$
	is the one which is step-type in both horizontal and 
	vertical directions. For such boundary conditions the 
	quantity 
	$Z^{(x,y)}_{\mathrm{boundary}}$
	\eqref{eq:Z_boundary_values}
	is not random and
	is readily written down (e.g., in some examples
	$\tau^{(0,y)}=\tau^{(x,0)}=\tau^{(0,0)}=\varnothing$).
	See \Cref{sub:F_G_fields_references}
	below for the origin of the term ``step''.
\end{definition}

For $(x,y)\in \mathbb{Z}_{\ge0}^2$
denote the northwest and the southeast quadrants by
\begin{equation*}
	\mathrm{NW}_{(x,y)}:=\{(m,n)\in \mathbb{Z}_{\ge0}^{2}\colon 
	m\le x,\ n\ge y \},\qquad 
	\mathrm{SE}\,_{(x,y)}:=\{(m,n)\in \mathbb{Z}_{\ge0}^{2}\colon 
	m\ge x,\ n\le y \}.
\end{equation*}
We are now in a position to formulate the main definition of the 
section:
\begin{definition}[Random fields]
	\label{def:F_G_field}
	A family of random Young diagrams
	$\boldsymbol \lambda=
	\{\lambda^{(x,y)}\colon (x,y)\in \mathbb{Z}_{\ge0}^{2}\}$
	is called a 
	\emph{random field} 
	associated with the 
	skew Cauchy structure $(\mathfrak{F},\mathfrak{G})$
	and parameters \eqref{eq:F_G_field_parameters_u_v}
	with a Gibbs boundary condition $\boldsymbol \tau$
	if:
	\begin{enumerate}[label=\bf{\arabic*.}]
		\item \label{enum:F_G_field_1} The diagrams satisfy 
			$\lambda^{(x,y)}\prec_1 \lambda^{(x,y+1)}$ and 
			$\lambda^{(x,y)}\mathop{\dot\prec}_1\lambda^{(x+1,y)}$
			for all $(x,y)\in \mathbb{Z}_{\ge0}^{2}$.
		\item The diagrams at the boundary of the quadrant 
			$\mathbb{Z}_{\ge0}^2$ agree with $\boldsymbol \tau$:
			$\lambda^{(x,0)}=\tau^{(x,0)}$, 
			$\lambda^{(0,y)}=\tau^{(0,y)}$ for all $x,y\ge 0$.
		\item \label{enum:F_G_field_3}
			For all $(x,y)\in \mathbb{Z}_{\ge0}^{2}$,
			let us use the shorthand notation
			\begin{equation}
				\label{eq:F_G_quadruplet_notation}
				\varkappa=\lambda^{(x,y)},\qquad 
				\mu=\lambda^{(x+1,y)},\qquad 
				\lambda=\lambda^{(x,y+1)},\qquad 
				\nu=\lambda^{(x+1,y+1)}
			\end{equation}
			(which matches \Cref{fig:quadruplet}).
			We require that 
			\begin{equation}
				\label{eq:F_G_conditioning}
				\begin{split}
					\mathop{\mathrm{Prob}}
					\bigl( \varkappa\mid \lambda^{(m,n)}\colon (m,n)\in 
					\mathrm{NW}_{(x,y+1)}\cup \mathrm{SE}\,_{(x+1,y)} \bigr)
					&=
					\mathop{\mathrm{Prob}}\left( \varkappa\mid \lambda,\mu \right)
					=
					\frac{\mathfrak{F}_{\lambda/\varkappa}(u_{y+1})\mathfrak{G}_{\mu/\varkappa}(v_{x+1})}{Z^{(x,y)}_{\llcorner}}\,,
					\\
					\mathop{\mathrm{Prob}}
					\bigl( \nu\mid \lambda^{(m,n)}\colon (m,n)\in 
					\mathrm{NW}_{(x,y+1)}\cup \mathrm{SE}\,_{(x+1,y)} \bigr)
					&=
					\mathop{\mathrm{Prob}}\left( \nu\mid \lambda,\mu \right)
					=
					\frac{\mathfrak{F}_{\nu/\mu}(u_{y+1})\mathfrak{G}_{\nu/\lambda}(v_{x+1})}{Z^{(x,y)}_{\urcorner}},
				\end{split}
			\end{equation}
			where $Z^{(x,y)}_{\llcorner}$ and $Z^{(x,y)}_{\urcorner}$ are the normalizing constants.
			The skew Cauchy identity \eqref{eq:F_G_single_skew_Cauchy}
			implies that 
			$Z^{(x,y)}_{\urcorner}=\Pi(u_{y+1};v_{x+1})\,Z^{(x,y)}_{\llcorner}$.
	\end{enumerate}
	See \Cref{fig:F_G_field} for an illustration.
	Observe that the restrictions
	on the Young diagrams
	in Condition~\ref{enum:F_G_field_1}
	follow from Condition~\ref{enum:F_G_field_3}.
	Note also that a 
	random field is not determined uniquely by the above conditions.
	We discuss this
	in \Cref{sub:F_G_transitions} below.
\end{definition}

\begin{figure}[htpb]
	\centering
	\includegraphics[width=.5\textwidth]{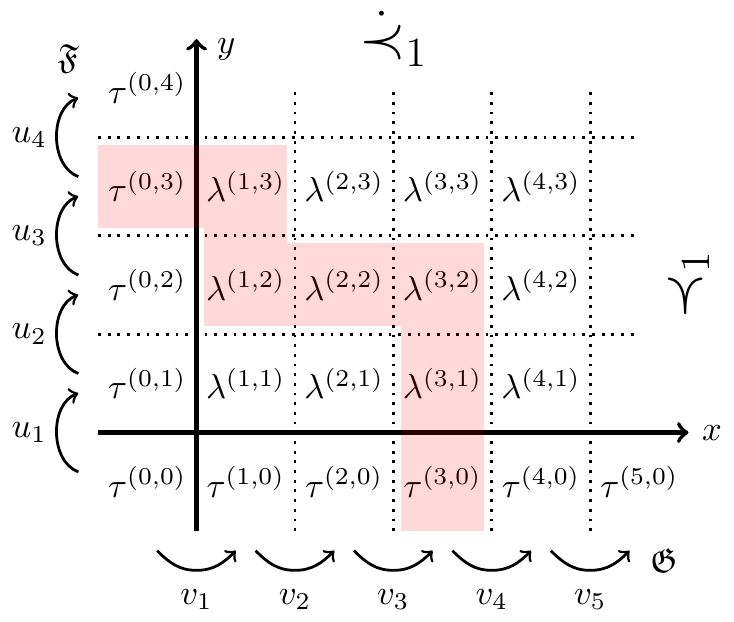}
	\caption{A random field of Young diagrams
		$\{ \lambda^{(x,y)} \}$ 
		with boundary conditions $\boldsymbol \tau$,
		and an example of a down-right path.}
	\label{fig:F_G_field}
\end{figure}

\begin{definition}
	\label{def:F_G_down-right}
	A collection $\left\{ (x_i,y_i) \right\}_{i=1}^{L}\subset\mathbb{Z}_{\ge0}^{2}$, where $L\ge1$, is called a
	\emph{down-right path} if 
	$x_1=0$, $y_L=0$, and 
	the difference between consecutive vertices $(x_{i+1},y_{i+1})-
		(x_i,y_i)$ is either $(0,-1)$ or $(1,0)$ for all $i$.
\end{definition}

\begin{proposition}
	\label{prop:F_G_processes}
	Let $\boldsymbol \lambda$ be a field. Then the joint distribution 
	of the Young diagrams along each down-right path $\left\{ (x_i,y_i) \right\}_{i=1}^{L}$
	conditioned on 
	$\tau^{(0,y_1)}$ and $\tau^{(x_L,0)}$
	has the form
	\begin{equation}
		\label{eq:F_G_processes}
		\frac{1}{Z_{\mathrm{path}}} \,
		\prod_{i\colon y_{i+1}=y_i-1}
		\mathfrak{F}_{\lambda^{(x_{i},y_{i})}/\lambda^{(x_{i+1},y_{i+1})}}(u_{y_i})
		\prod_{i\colon x_{i+1}=x_i+1}
		\mathfrak{G}_{\lambda^{(x_{i+1},y_{i+1})}/\lambda^{(x_{i},y_{i})}}(v_{x_{i+1}}).
	\end{equation}
	The normalizing constant has the form
	\begin{equation*}
		Z_{\mathrm{path}}=Z^{(x_L,y_1)}_{\mathrm{boundary}}
		\prod_{\textnormal{$(x,y)$ below the path}}\Pi(u_y;v_x).
	\end{equation*}
\end{proposition}
\begin{proof}
	This follows by induction on flipping the down-right path using 
	elementary steps $\llcorner\to\urcorner$ (i.e., by replacing the down-right corners
	by the right-down ones). 
	The induction base is the path which first makes only down steps to $(0,0)$
	and then only right steps. For this path the statement follows from
	the Gibbs property of the boundary condition (\Cref{def:F_G_boundary_conditions}).

	The inductive step uses \eqref{eq:F_G_conditioning}.
	Let us fix some $\llcorner$ corner 
	$(x,y)$ in the path and use the 
	notation of \eqref{eq:F_G_conditioning}. 
	Conditioned on $\lambda,\mu$, the Young diagram
	$\varkappa$ is independent of the diagram along the rest of the path.
	Using the induction assumption and \eqref{eq:F_G_conditioning}
	to
	replace the two factors
	corresponding to $(\lambda,\varkappa,\mu)$ in \eqref{eq:F_G_processes}
	by the ones corresponding to $(\lambda,\nu,\mu)$, 
	we obtain the desired joint distribution along the modified down-right path.
\end{proof}
For the special choice 
of the
path which 
first makes only right steps and then only down steps,
we obtain
with the help of the branching
\eqref{eq:F_G_branching}:
\begin{corollary}
	\label{cor:F_G_measure}
	For any $x,y\ge1$ we have
	\begin{equation}
		\label{eq:F_G_measure}
		\mathop{\mathrm{Prob}}(\lambda^{(x,y)}\mid \tau^{(0,y)},\tau^{(x,0)})=
		\frac{
			\mathfrak{F}_{\lambda^{(x,y)}/\tau^{(x,0)}}(u_1,\ldots,u_y )
			\mathfrak{G}_{\lambda^{(x,y)}/\tau^{(0,y)}}(v_1,\ldots,v_x )
		}
		{Z^{(x,y)}}.
	\end{equation}
	The normalizing constant has the form
	\begin{equation*}
		Z^{(x,y)}=Z^{(x,y)}_{\mathrm{boundary}}\,
		\prod_{i=1}^{x}\prod_{j=1}^{y}\Pi(u_y;v_x).
	\end{equation*}
\end{corollary}

Note that for the step-type boundary conditions $\boldsymbol \tau$
there is no need to condition on the boundary values
$\tau^{(0,y)}$ and $\tau^{(x,0)}$ in 
\Cref{prop:F_G_processes} and \Cref{cor:F_G_measure}.
For general Gibbs boundary conditions we 
have the following Gibbs preservation property:

\begin{proposition}
	\label{prop:F_G_preservation_Gibbs_measures}
	For any $(x,y)\in \mathbb{Z}_{\ge0}^2$,
	the two-sided sequence
	\begin{equation*}
		\ldots
		\succ_1
		\lambda^{(x,y+2)}
		\succ_1
		\lambda^{(x,y+1)}
		\succ_1
		\lambda^{(x,y)}
		\mathop{\dot\prec_1} 
		\lambda^{(x+1,y)}
		\mathop{\dot\prec_1} 
		\lambda^{(x+2,y)}
		\mathop{\dot\prec_1} 
		\ldots 
	\end{equation*}
	is an $(\mathfrak{F},\mathfrak{G})$-Gibbs boundary condition
	in the sense of \Cref{def:F_G_boundary_conditions}.
\end{proposition}
\begin{proof}
	Immediately follows from \Cref{prop:F_G_processes}.
\end{proof}

\begin{example}
	In the Schur case the distributions of 
	\Cref{prop:F_G_processes} and \Cref{cor:F_G_measure}
	become the Schur processes and the Schur measures
	introduced in \cite{okounkov2003correlation}
	and \cite{okounkov2001infinite},
	respectively (see also \cite{borodin2005eynard}).
	Early examples of random fields in this case were
	based on Robinson-Schensted-Knuth correspondences.
	Other approaches 
	were suggested more recently
	in, e.g., \cite{BorFerr2008DF},
	\cite{warrenwindridge2009some},
	\cite{BorodinPetrov2013NN}. See
	\Cref{sub:F_G_fields_references} for more historical discussion.
\end{example}

\subsection{Transition probabilities as bijectivizations of the skew Cauchy identity}
\label{sub:F_G_transitions}

Let us fix a skew Cauchy structure $(\mathfrak{F},\mathfrak{G})$,
parameters \eqref{eq:F_G_field_parameters_u_v},
and a Gibbs boundary condition $\boldsymbol \tau$.
\Cref{def:F_G_boundary_conditions}
\emph{does not characterize uniquely} a random field $\boldsymbol \lambda$
corresponding to this data.
Namely, consider any quadruple of neighboring
Young diagrams \eqref{eq:F_G_quadruplet_notation} 
(related as in \Cref{fig:quadruplet})
corresponding to $(x,y)\in \mathbb{Z}_{\ge0}^2$.
Given $\lambda,\mu$, condition \eqref{eq:F_G_conditioning}
characterizes the marginal distributions of 
$\varkappa$ and $\nu$ separately. 
One readily sees that picking 
any joint distribution of $(\varkappa,\nu)$ given $\lambda,\mu$
with required marginals $\varkappa$ and $\nu$
produces a valid random field $\boldsymbol \lambda$
(and this choice can be made independently
at every location $(x,y)$ in the quadrant).
Therefore, one has to employ 
additional considerations to pick 
random fields with interesting properties,
for example, 
possessing scalar Markovian marginals
(see \Cref{sub:F_G_scalar_marginals} below).

It is convenient to encode the choice of the joint distribution of $(\varkappa,\nu)$ given $\lambda$ and $\mu$
in an equivalent form of conditional probabilities. 
This leads to the following definition:
\begin{definition}
	\label{def:F_G_fwd_bwd_transition_probabilities}
	Let $u,v\in \mathbb{C}$ be such that
	$(u,v)\in \mathsf{Adm}$, $u\in \mathsf{P}$, $v\in \dot{\mathsf{P}}$.
	The functions 
	\begin{equation*}
		\mathsf{U}^{\mathrm{fwd}}_{u,v}(\varkappa\to\nu\mid \lambda,\mu),\qquad 
		\mathsf{U}^{\mathrm{bwd}}_{u,v}(\nu\to \varkappa\mid \lambda,\mu)
	\end{equation*}
	on quadruples of diagrams as in \Cref{fig:quadruplet}
	are called, respectively, 
	the \emph{forward} and the \emph{backward} 
	\emph{transition probabilities} if:
	\begin{enumerate}[label=\bf{\arabic*.}]
		\item The functions are nonnegative and 
			sum to one over the second argument:
			\begin{equation}
				\label{eq:F_G_U_sum_to_one}
				\begin{split}
					\sum_{\nu}\mathsf{U}_{u,v}^{\mathrm{fwd}}(\varkappa\to \nu\mid \lambda,\mu)=1
					& \qquad \textnormal{for all triples $\lambda\succ_1 \varkappa\mathop{\dot{\prec}_1}\mu$}
					,\\
					\sum_{\varkappa}\mathsf{U}_{u,v}^{\mathrm{bwd}}(\nu\to\varkappa\mid \lambda,\mu)=1
					&\qquad \textnormal{for all triples
					$\lambda \mathop{\dot{\prec}_1}\nu\succ_1\mu$.}
				\end{split}
			\end{equation}
			We will interpret $\mathsf{U}^{\mathrm{fwd}}(\varkappa\to\nu\mid \lambda,\mu)$ 
			as a conditional distribution of $\nu$ given $\lambda\succ_1\varkappa\mathop{\dot{\prec}_1} \mu$,
			and $\mathsf{U}^{\mathrm{bwd}}$ as the opposite conditional distribution.
		\item The functions satisfy the \emph{reversibility condition}
			\begin{equation}
				\label{eq:F_G_reversibility}
				\mathsf{U}_{u,v}^{\mathrm{fwd}}(\varkappa\to\nu\mid \lambda,\mu)
				\cdot
				\Pi(u;v)\mathfrak{F}_{\lambda/\varkappa}(u)\mathfrak{G}_{\mu/\varkappa}(v)
				=
				\mathsf{U}_{u,v}^{\mathrm{bwd}}(\nu\to\varkappa \mid \lambda,\mu)
				\cdot
				\mathfrak{F}_{\nu/\mu}(u)\mathfrak{G}_{\nu/\lambda}(v).
			\end{equation}
	\end{enumerate}
\end{definition}
Summing both sides 
of \eqref{eq:F_G_reversibility}
over $\varkappa$ and $\nu$ produces the skew
Cauchy identity \eqref{eq:F_G_single_skew_Cauchy}.
Therefore, choosing 
transition probabilities
$\mathsf{U}^{\mathrm{fwd}}_{u,v}$
and $\mathsf{U}^{\mathrm{bwd}}_{u,v}$
corresponds to a refinement (``\emph{bijectivization}'') 
of the skew Cauchy identity (for a general discussion of
bijectivization, see \Cref{sub:bij_summation_citation_from_BP2017} below).
In the following sections we build
bijectivizations of various concrete skew Cauchy identities
out of bijectivizations of the Yang-Baxter equations.

\begin{remark}
	Summing \eqref{eq:F_G_reversibility}
	over $\varkappa$, we get
	\begin{equation}
		\label{eq:F_G_reversibility_nonsymmetric_form}
		\Pi(u;v)
		\sum_{\varkappa}
		\mathsf{U}^{\mathrm{fwd}}_{u,v}(\varkappa\to\nu\mid \lambda,\mu)\cdot
		\mathfrak{F}_{\lambda/\varkappa}(u)\mathfrak{G}_{\mu/\varkappa}(v)
		=\mathfrak{F}_{\nu/\mu}(u)\mathfrak{G}_{\nu/\lambda}(v)
	\end{equation}
	This identity 
	was used in \cite{BorodinPetrov2013NN}
	and \cite{MatveevPetrov2014}
	as a starting point 
	to construct random fields associated with $q$-Whittaker functions. The advantage of 
	\eqref{eq:F_G_reversibility} 
	compared with
	\eqref{eq:F_G_reversibility_nonsymmetric_form} is that the former is more symmetric 
	and does not involve summation. 
\end{remark}

\begin{remark}[Borodin--Ferrari random fields]
	\label{rmk:F_G_Borodin-Ferrari_Fields}
	The
	existence of at least one random
	field corresponding 
	to a skew Cauchy structure
	$(\mathfrak{F},\mathfrak{G})$
	is evident from the above discussion.
	An explicit
	basic construction of a field
	was suggested in \cite{BorFerr2008DF}
	based on an idea of \cite{DiaconisFill1990}.
	Namely, if $\mathsf{U}^{\mathrm{fwd}}(\varkappa\to\nu\mid \lambda,\mu)$
	is \emph{independent} of $\varkappa$, then by \eqref{eq:F_G_reversibility_nonsymmetric_form}
	it must have the form
	\begin{equation*}
		\mathsf{U}_{u,v}^{\mathrm{fwd}}(\varkappa\to\nu\mid \lambda,\mu)=
		\frac{\mathfrak{F}_{\nu/\mu}(u)\mathfrak{G}_{\nu/\lambda}(v)}
		{\Pi(u;v)\sum_{\hat\varkappa}
		\mathfrak{F}_{\lambda/\hat\varkappa}(u)\mathfrak{G}_{\mu/\hat\varkappa}(v)}
	\end{equation*}
	if there exists $\hat \varkappa$ such that 
	$\lambda\succ_1\hat\varkappa\mathop{\dot{\prec}_1}\mu$.
	Though this construction of a random field is rather simple and
	works in full generality for an arbitrary skew Cauchy structure,
	it does not produce all known examples of fields 
	with scalar Markovian marginals. 
	See \Cref{sub:F_G_fields_references}
	below for more discussion.
\end{remark}

Using just the forward transition probabilities,
start with \emph{arbitrary} fixed (not necessarily Gibbs)
boundary values $\lambda^{(x,0)}=\tau^{(x,0)}$ and $\lambda^{(0,y)}=\tau^{(0,y)}$, $x,y\ge0$,
and define a family of random Young diagrams
$\{\lambda^{(x,y)}\}$ indexed by the quadrant as follows.
By induction on $x+y=n$, assume that the Young diagrams
with $x+y\le n-1$ are determined.
Then, independently for each $(x,y)$ with $x+y=n$ and $x,y\ge1$
sample $\lambda^{(x,y)}$
having the distribution 
$\mathsf{U}^{\mathrm{fwd}}_{u_y,v_x}(\lambda^{(x-1,y-1)}\to \lambda^{(x,y)}\mid \lambda^{(x,y+1)},\lambda^{(x+1,y)})$,
where $\lambda^{(x-1,y-1)},\lambda^{(x,y+1)}$ and $\lambda^{(x+1,y)}$
are already determined. 
The next proposition immediately follows from the definitions:

\begin{proposition}
	\label{prop:F_G_from_U_to_fields}
	If the boundary condition $\boldsymbol \tau$
	in the above construction is Gibbs, then
	the resulting collection of random Young diagrams $\{\lambda^{(x,y)}\}$,
	$(x,y)\in \mathbb{Z}_{\ge0}$
	forms a random field in the sense of \Cref{def:F_G_field}.
\end{proposition}

Therefore, random fields associated with a skew Cauchy structure
$(\mathfrak{F}, \mathfrak{G})$
correspond to forward transition probabilities, and vice versa.
Moreover, the probabilities $\mathsf{U}^{\mathrm{fwd}}_{u,v}$
allow to construct a joint distribution on Young diagrams
$\{\lambda^{(x,y)}\}$ indexed by points of the quadrant
$\mathbb{Z}_{\ge0}^{2}$
starting from arbitrary boundary values. However, the Gibbs property on the
boundary is needed for \Cref{prop:F_G_processes} describing joint distributions
of the Young diagrams along down-right paths.
We will not consider non-Gibbs boundary conditions in the present paper.

\subsection{Scalar marginals}
\label{sub:F_G_scalar_marginals}

Let $\boldsymbol \lambda$ be a random field in the sense of \Cref{def:F_G_field}
and $\mathsf{h}\colon \mathbb{Y}\to \mathbb{Z}$ be a function.
When the scalar random variables $\{\mathsf{h}(\lambda^{(x,y)})\}$
indexed by $(x,y)\in \mathbb{Z}_{\ge0}^{2}$
evolve (in the sense of \emph{forward} steps)
independently of the rest of $\boldsymbol \lambda$, 
we call $\mathsf{h}(\boldsymbol \lambda)$ a 
\emph{scalar} (\emph{Markovian}) \emph{marginal} of a field $\boldsymbol \lambda$.

In detail, this independence means the following.
For a finitely supported function $F$ on $\mathbb{Z}$ 
we can write for any field $\boldsymbol \lambda$:
\begin{equation}
	\label{eq:F_G_U_scalar_transition_prob_factorization}
	\sum_{\nu\in \mathbb{Y}}
	F(\mathsf{h}(\nu))\,
	\mathsf{U}^{\mathrm{fwd}}_{u,v}(\varkappa\to\nu\mid \lambda,\mu)
	=
	\sum_{n\in \mathbb{Z}}
	F(n)
	\Bigg(
		\sum_{\nu\colon \mathsf{h}(\nu)=n}
		\mathsf{U}^{\mathrm{fwd}}_{u,v}(\varkappa\to \nu\mid \lambda,\mu)
	\Bigg).
\end{equation}
We say that the random field $\boldsymbol \lambda$ is
\emph{$\mathsf{h}$-adapted} if 
the quantity in the parentheses above
\begin{equation}
	\label{eq:F_G_U_scalar_transition_prob}
	\mathsf{U}^{[\mathsf{h}]}_{u,v}(k\to n\mid \ell,m)
	:=
	\sum_{\nu\colon \mathsf{h}(\nu)=n}
	\mathsf{U}^{\mathrm{fwd}}_{u,v}(\varkappa\to \nu\mid \lambda,\mu)
\end{equation}
depends on $\lambda,\varkappa,\mu$
only through 
$\ell=\mathsf{h}(\lambda)$, 
$k=\mathsf{h}(\varkappa)$, 
and $m=\mathsf{h}(\mu)$. 
The function $\mathsf{U}^{[\mathsf{h}]}_{u,v}$ is nonnegative
and 
$\sum_{n\in \mathbb{Z}}
\mathsf{U}^{[\mathsf{h}]}_{u,v}(k\to n\mid \ell,m)=1$
for all 
$\ell,k,m$ such that there 
exists at least one triple 
$\lambda\succ_1 \varkappa \mathop{\dot{\prec}_1} \mu$.
In words, to sample $\nu$ knowing $\lambda,\varkappa,\mu$
we first look at $\ell,k,m$ and sample $n=\mathsf{h}(\nu)$
independently of any other information about the diagrams $\lambda,\varkappa,\mu$, 
and then sample the rest of the diagram $\nu$.

For a $\mathsf{h}$-adapted field $\boldsymbol \lambda$, the joint distribution of the 
scalar quantities $\mathsf{h}(\lambda^{(x,y)})$, $(x,y)\in \mathbb{Z}_{\ge0}^{2}$
(forming the scalar marginal of $\boldsymbol \lambda$ corresponding to $\mathsf{h}$),
can be described using 
\eqref{eq:F_G_U_scalar_transition_prob}
as forward transition probabilities:
\begin{multline*}
	\mathop{\mathrm{Prob}}
	\left( 
		\mathsf{h}(\lambda^{(x+1,y+1)})=n
		\,\Big\vert \,
		\mathsf{h}(\lambda^{(x,y+1)})=\ell,
		\mathsf{h}(\lambda^{(x,y)})=k,
		\mathsf{h}(\lambda^{(x+1,y)})=m,
	\right)
	\\=
	\mathsf{U}^{[\mathsf{h}]}_{u_{y+1},v_{x+1}}(k\to n\mid \ell,m).
\end{multline*}
Note that while
for a scalar marginal $\mathsf{h}$
the forward transition probabilities 
factorize as in 
\eqref{eq:F_G_U_scalar_transition_prob_factorization}--\eqref{eq:F_G_U_scalar_transition_prob},
the backward ones do not have to factorize in the same way.

\begin{remark}
	One can take 
	an arbitrary set instead of $\mathbb{Z}$ as the target of $\mathsf{h}$
	as this is essentially the index set of equivalence classes of Young diagrams. 
	In the rest of the paper we mostly focus on 
	integer-valued scalar Markovian marginals, but also mention their
	higher-dimensional (multilayer) extensions obtained by 
	refining these equivalence classes.
\end{remark}

Scalar marginals in the Schur case (our running example) are discussed in the
next \Cref{sub:F_G_fields_references}.

\subsection{Existing constructions of random fields}
\label{sub:F_G_fields_references}

This subsection is a brief review of known 
random fields associated with 
skew Cauchy structures 
corresponding to various families of symmetric functions
(see \Cref{fig:symm_functions_scheme} for the hierarchy of symmetric functions
we mention below).

Constructing probability measures on Young diagrams related to the Schur symmetric functions
by means of Markov dynamics on Young tableaux goes back at least to 
\cite{Vershik1986}. The first such mechanism employed in many 
well-known developments in Integrable Probability
starting from \cite{baik1999distribution} and
\cite{johansson2000shape} is the Robinson-Schensted-Knuth (RSK) correspondence.
In particular, the RSK gives rise to a 
random field of Young diagrams associated with Schur functions
whose scalar marginal field is identified with the 
Totally Asymmetric Simple Exclusion Process (TASEP).\footnote{To make a precise identification with the 
standard continuous-time TASEP one has to perform a Poisson-like limit transition which makes
one of the field's discrete coordinates $\mathbb{Z}_{\ge0}$ into continuous $\mathbb{R}_{\ge0}$.
If one makes both coordinates continuous, then the field's scalar marginal can be linked
to the distribution of the length of the longest increasing subsequence of a random permutation.
Besides certain simplification of stochastic mechanisms,
such continuous limits do not introduce any significant changes
into the structure of the fields.
In the present paper we focus only on the fully discrete picture.} 
The distributions in TASEP started from a special initial configuration
called ``step'' (when the particles occupy the negative half-line while the positive half-line is empty)
are then related to the Schur measures and processes introduced in 
\cite{okounkov2001infinite}, 
\cite{okounkov2003correlation}.
The corresponding field of random Young diagrams in this case has step-type 
Gibbs boundary condition in the sense of our \Cref{def:F_G_step_type_boundary}.
Further applications of RSK and its tropical version to particle systems, last passage percolation models,
and random polymers were developed in 
\cite{OConnell2003},
\cite{OConnell2003Trans},
\cite{BBO2004}, \cite{Chhaibi2013},
\cite{Oconnell2009_Toda},
\cite{COSZ2011},
\cite{OSZ2012}, and related works.

Another mechanism of constructing random fields associated with Schur polynomials
was suggested in \cite{BorFerr2008DF}, see also \cite{Borodin2010Schur}. 
(We outlined this construction in \Cref{rmk:F_G_Borodin-Ferrari_Fields}.)
This mechanism was later employed in 
\cite{BorodinCorwin2011Macdonald}
to discover the (continuous-time) $q$-deformation of the TASEP as a scalar marginal in a field
associated with the $q$-Whittaker functions.
The integrable structure of the $q$-TASEP is based on the 
$q$-difference operators diagonal in the $q$-Whittaker polynomials
(these are the $t=0$ 
Macdonald difference operators \cite[Chapter VI.3]{Macdonald1995}).
It soon became apparent, however, that Borodin--Ferrari random fields
cannot produce all known integrable stochastic particle systems
on the line as their Markovian marginals. Early examples of stochastic particle 
systems not coming out of Borodin--Ferrari fields include the 
discrete-time $q$-TASEPs suggested in \cite{BorodinCorwin2013discrete}. 

This issue motivated the search for other constructions of 
random fields, and resulted in discovery of 
$q$-Whittaker and Hall-Littlewood randomizations of the 
RSK correspondence 
\cite{OConnellPei2012}, 
\cite{BorodinPetrov2013NN},
\cite{BufetovPetrov2014}, 
\cite{MatveevPetrov2014},
\cite{BufetovMatveev2017}.
On the $q$-Whittaker side, this 
brought new $q$-TASEPs and $q$-PushTASEPs
whose distributions are expressed through the $q$-Whittaker measures and processes.
The Hall-Littlewood side brought the integrable structure of Hall-Littlewood measures
and processes to the stochastic six vertex model and the ASEP (i.e, TASEP with left and right jumps allowed).

In parallel to these developments a new extension of the 
$q$-TASEP called the $q$-Hahn TASEP was invented
\cite{Povolotsky2013},
\cite{Corwin2014qmunu}. 
Further investigation of this process has 
led to the systematic development of the
spin Hall-Littlewood (sHL) symmetric rational functions
and the associated stochastic vertex models
\cite{BCPS2014},
\cite{Borodin2014vertex},
\cite{CorwinPetrov2015},
\cite{BorodinPetrov2016_Hom_Lectures},
\cite{BorodinPetrov2016inhom}.
In particular, the Yang-Baxter equation for the 
higher spin six vertex model
implies the skew Cauchy identity for the sHL functions. 
Recently, the spin $q$-Whittaker (sqW) symmetric polynomials
were introduced in \cite{BorodinWheelerSpinq}
as the dual complement (which for $s=0$ reduces to the $q\leftrightarrow t$ Macdonald involution) 
of the sHL ones.

These new skew Cauchy structures called for extending the random field constructions which would 
bring interesting scalar marginals.
In 
\cite{BufetovPetrovYB2017}
this was performed in the sHL setting 
based on a new idea of bijectivization of the Yang-Baxter equation 
(we recall it in \Cref{sec:YB_fields_through_bijectivisation} below).
This idea allowed to bypass technical 
difficulties associated with randomizing the RSKs
and, on the other hand, by design has produced a
scalar marginal of the sHL Yang-Baxter field
which is a new dynamical extension of the stochastic six vertex 
model.\footnote{Similar stochastic vertex models 
from Yang-Baxter equations are developed in \cite{ABB2018stochasticization},
but without connecting them to random fields or symmetric functions.}
In this paper we complete the picture by 
constructing Yang-Baxter fields associated
with two other skew Cauchy structures corresponding to the
sqW/sHL and the 
sqW/sqW skew Cauchy identities
(see \Cref{sec:new_three_fields}),
and 
find that their 
scalar marginals are related to 
the stochastic higher spin six vertex model of
\cite{CorwinPetrov2015}, \cite{BorodinPetrov2016inhom}
and to the $q$-Hahn PushTASEP recently 
introduced in \cite{CMP_qHahn_Push}. 
In \Cref{sec:diff_op}
we employ the former connection to discover new
difference operators
acting diagonally on sqW
or stable sHL functions.

\begin{remark}
	One can also define the notion of a random field of Young diagrams associated
	with Macdonald or Jack symmetric functions since they, too, satisfy skew Cauchy identities. 
	However, due to the more complicated ``nonlocal'' structure of the Jack and Macdonald Pieri rules 
	compared to the $q$-Whittaker or Hall-Littlewood 
	ones,\footnote{The Pieri coefficients of the $q$-Whittaker and Hall-Littlewood 
		functions involve products of only nearest neighbor terms (properly understood), while in the 
	Jack and Macdonald cases the products are over all pairs of indices.}
	it seems unlikely that there exist Jack or Macdonald random fields with scalar
	Markovian marginals. In this paper we do not focus on this question.
\end{remark}

\section{\texorpdfstring{Spin Hall-Littlewood and spin $q$-Whittaker functions}{Spin Hall-Littlewood and spin q-Whittaker functions}} 
\label{sec:summary_sHL_sqW}

In this section we review the main properties of the stable spin
Hall-Littlewood and spin $q$-Whittaker symmetric functions
\cite{Borodin2014vertex}, \cite{BorodinWheelerSpinq} which lead to skew
Cauchy structures. These functions are defined as partition functions of
certain ensembles of lattice paths realized through a vertex model formalism.
We fix the main ``\emph{quantization}'' parameter
$q\in(0,1)$. 
In contrast with \Cref{fig:symm_functions_scheme},
throughout the text we use $q$ to denote the quantization parameter
in both spin Hall-Littlewood and spin $q$-Whittaker functions, 
which is convenient when considering 
Yang-Baxter fields based on both families.

\subsection{Young diagrams as arrow configurations}

We represent Young diagrams $\lambda=(\lambda_1\ge \ldots\ge \lambda_{\ell(\lambda)}>0 )$ as
configurations of vertical arrows on $\mathbb{Z}_{\ge0}$.
Let $\lambda$ be written in the multiplicative notation as 
$\lambda=1^{l_1}2^{l_2}\ldots $, where $l_i$ is the number of parts
of $\lambda$ which are equal to $i$. By definition, the arrow configuration
corresponding to $\lambda$, denoted by $|\lambda \rangle$,
contains $l_i$ vertical arrows at location $i$. 
The number of vertical arrows at $0$ is 
assumed infinite which reflects the fact that one can append Young diagrams
by zeros without changing them. 
See
\Cref{fig:arrow_lambda}, left, for an illustration.

\begin{figure}[htpb]
	\centering
	\includegraphics{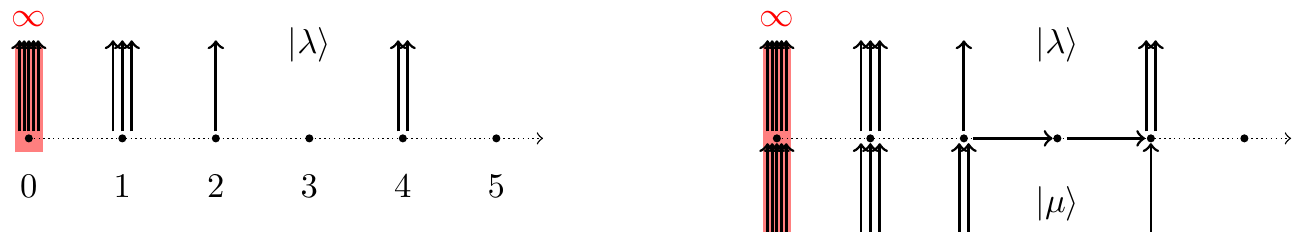}
	\caption{Left: Configuration $|\lambda \rangle$ of vertical arrows corresponding to 
		the Young diagram $\lambda=(4,4,2,1,1,1)$.
		Right: Interlacing of $\lambda$ with $\mu=(4,2,2,1,1,1)$.}
	\label{fig:arrow_lambda}
\end{figure}

\subsection{Stable spin Hall-Littlewood functions} 
\label{sub:sHL}

The first collection of vertex weights we work with is given in 
\Cref{fig:table_w}. Along with $q$,
these weights depend on two quantities $u,s\in \mathbb{C}$, which are called
the \emph{spectral} and the \emph{spin} parameters, respectively.
The weights $w_{u,s}$ satisfy the Yang-Baxter equation, see 
\Cref{app:YBE}.

\begin{figure}[htbp]
  \centering
  \begin{tabular}{c||c|c|c|c}
    \begin{tikzpicture}[baseline=0]
    	\draw[fill] (0,0) circle [radius=0.025];
        \fill[red!30] (-0.04,-0.05) rectangle (0.04,-0.5) node[black, below]{\scriptsize{$i_1$}};
        \fill[red!30] (0.04,0.05) rectangle (-0.04,0.5) node[black,above]{ \scriptsize{$i_2$} };
        \fill[red!30] (0.5,0.04) rectangle (0.05,-0.04) node[black,xshift=0.67cm]{ \scriptsize{$j_2$} };
        \fill[red!30] (-0.05,0.04) rectangle (-0.5,-0.04) node[black,left]{ \scriptsize{$j_1$} };
    	\addvmargin{1mm}
    \end{tikzpicture}
     & \begin{tikzpicture}[baseline=0]
    	\draw[fill] (0,0) circle [radius=0.025];
		\draw [red] (0.04,-0.5) -- (0.04,-0.05);
    	\draw [red] (0,-0.5) -- (0,-0.05);
    	\draw [red] (-0.04,-0.5) -- (-0.04,-0.05);
        \node [below] at (0,-0.5){\scriptsize{$g$}};
        \draw [dotted] (-0.5,0) -- (-0.1,0);
        \draw [dotted] (0.1,0) -- (0.5, 0);
        \draw [red] (0.04,0.05) -- (0.04,0.5);
        \draw [red] (0,0.05) -- (0,0.5);
        \draw [red] (-0.04,0.05) -- (-0.04,0.5);
        \node [above] at (0,0.5){\scriptsize{$g$}};
        \addvmargin{1mm}
  \end{tikzpicture}  &  \begin{tikzpicture}[baseline=0]
        \draw[fill] (0,0) circle [radius=0.025];
        \draw [red] (-0.04,-0.5) -- (-0.04,-0.05);
        \draw [red] (0,-0.5) -- (0,-0.05);
        \draw [red] (0.04,-0.5) -- (0.04,-0.05);
        \node [below] at (0,-0.5){\scriptsize{$g+1$}};
        \draw [dotted] (-0.5,0) -- (-0.1,0);
        \draw [red] (0.05,0) -- (0.5, 0);
        \draw [red] (0.025,0.05) -- (0.025,0.5);
        \draw [red] (-0.025,0.05) -- (-0.025,0.5);
        \node [above] at (0,0.5){\scriptsize{$g$}};
        \addvmargin{1mm}
  \end{tikzpicture} & \begin{tikzpicture}[baseline=0]
        \draw[fill] (0,0) circle [radius=0.025];
        \draw [red] (0,-0.5) -- (0,-0.05);
        \draw [red] (0.04,-0.5) -- (0.04,-0.05);
        \draw [red] (-0.04,-0.5) -- (-0.04,-0.05);
        \node [below] at (0,-0.5){\scriptsize{$g$}};
        \draw [red] (-0.5,0) -- (-0.05,0);
        \draw [red,] (0.05,0) -- (0.5, 0);
        \draw [red] (0,0.05) -- (0,0.5);
        \draw [red] (0.04,0.05) -- (0.04,0.5);
        \draw [red] (-0.04,0.05) -- (-0.04,0.5);
        \node [above] at (0,0.5){\scriptsize{$g$}};
        \addvmargin{1mm}
  \end{tikzpicture} & \begin{tikzpicture}[baseline=0]
        \draw[fill] (0,0) circle [radius=0.025];
        \draw [red] (-0.025,-0.5) -- (-0.025,-0.05);
        \draw [red] (0.025,-0.5) -- (0.025,-0.05);
        \node [below] at (0.1,-0.5){\scriptsize{$g$}};
        \draw [red] (-0.5,0) -- (-0.05,0);
        \draw [dotted] (0.1,0) -- (0.5, 0);
        \draw [red] (0,0.05) -- (0,0.5);
        \draw [red] (0.04,0.05) -- (0.04,0.5);
        \draw [red] (-0.04,0.05) -- (-0.04,0.5);
        \node [above] at (0,0.5){\scriptsize{$g+1$}};
        \addvmargin{1mm}
  \end{tikzpicture}\\
    \hhline{-----}
    \begin{minipage}{3cm}
    \centering
    \vspace{.1cm}
    \small{$w_{u,s}(i_1,j_1; i_2,j_2)$}
    \vspace{.1cm}
    \end{minipage}
    &
    \begin{minipage}{2cm}
    \centering
    \vspace{.1cm}
    $\frac{1- s u q^{g} }{1- s u}$
    \vspace{.1cm}
    \end{minipage}
    &
    \begin{minipage}{2cm}
    \centering
    \vspace{.1cm}
    $\frac{ u (1- s^2 q^g)}{1- s u}$
    \vspace{.1cm}
    \end{minipage}
    &
    \begin{minipage}{2cm}
    \centering
    \vspace{.1cm}
    $\frac{ u - s q^{g}}{1 - s u}$
    \vspace{.1cm}
    \end{minipage}
    & 
    \begin{minipage}{2cm}
    \centering
    \vspace{.1cm}
    $\frac{1- q^{g+1}}{1- s u}$
    \vspace{.1cm}
    \end{minipage}
    \\
  \end{tabular}
  \caption{In the top row we see all acceptable configurations of arrows entering 
	and exiting a vertex; below we reported the corresponding vertex weights $w_{u,s}(i_1, j_1; i_2, j_2)$.} 
	\label{fig:table_w}
\end{figure}
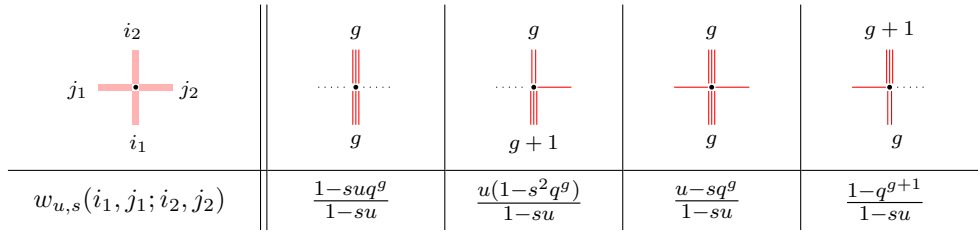

For vertices at the left
boundary we set
\begin{equation} 
	\label{eq:w_boundary}
	w_{u,s} \biggl(\begin{tikzpicture}[baseline=-2.5pt]
    	\draw[fill] (0,0) circle [radius=0.025];
			\node at (0,.3) {$\infty$};
			\node at (0,-.3) {$\infty$};
        \draw [red] (0.1,0) -- (0.5, 0);
        \addvmargin{1mm}
        \addhmargin{1mm}
  \end{tikzpicture} \biggr) = w_{u,s}(\infty,\varnothing; \infty,1) =u, 
	\qquad 
	w_{u,s} \biggl(\begin{tikzpicture}[baseline=-2pt]
    	\draw[fill] (0,0) circle [radius=0.025];
        \draw [dotted] (0.1,0) -- (0.5, 0);
			\node at (0,.3) {$\infty$};
			\node at (0,-.3) {$\infty$};
        \addvmargin{1mm}
        \addhmargin{1mm}
  \end{tikzpicture} \biggr) = w_{u,s}(\infty,\varnothing; \infty,0) =1,
\end{equation}
Both in \Cref{fig:table_w} and in \eqref{eq:w_boundary}, 
we attribute weight zero to all configurations which are not listed.
In particular, the following \emph{arrow conservation property} holds:
\begin{equation}
	\label{eq:sHL_arrow_conservation}
	w_{u,s}(i_1,j_1;i_2,j_2)=0
	\qquad \textnormal{unless $i_1+j_1=i_2+j_2$}.
\end{equation}

\begin{definition}[Interlacing]
	\label{def:interlacing}
	Fix $\mu,\lambda\in \mathbb{Y}$.
	We say that $\mu$ and $\lambda$ \emph{interlace}
	(notation $\mu\prec \lambda$) if there exists
	a configuration of finitely many horizontal arrows between $|\mu \rangle$ and $|\lambda \rangle$
	as in \Cref{fig:arrow_lambda}, right, such that the arrow
	conservation property holds at each vertex.\footnote{
	If such a horizontal arrow configuration 
	exists, then it is unique. The restriction that there are only finitely
	many horizontal arrows ensures that the configuration on the 
	far right is empty.}
	In detail,
	$\mu\prec\lambda$
	if either of the two hold:
	\begin{equation}
		\label{eq:interlace_sHL_def}
		\begin{split}&
			\ell(\lambda)=\ell(\mu)\text{ and }
			\mu_{\ell(\mu)} 
			\le 
			\lambda_{\ell(\lambda)}
			\le \ldots
			\le
			\lambda_2\le \mu_1\le \lambda_1,
			\\&
			\ell(\lambda)=\ell(\mu)+1\text{ and }
			\lambda_{\ell(\lambda)}
			\le \mu_{\ell(\mu)} 
			\le 
			\lambda_{\ell(\lambda)-1}
			\le \ldots
			\le
			\lambda_2\le \mu_1\le \lambda_1.
		\end{split}
	\end{equation}
	Note that for each $\lambda\in \mathbb{Y}$, 
	the number of $\mu$ such that $\mu\prec \lambda$ 
	is finite.
\end{definition}

\begin{definition}
	\label{def:ssHL}
	For $\mu,\lambda\in \mathbb{Y}$ with $\mu\prec \lambda$,
	a \emph{stable spin Hall-Littlewood function} in one variable,
	denoted by $\mathsf{F}_{\lambda/\mu}(u)$,
	is defined as the total weight (=~product of individual vertex weights) of the 
	unique configuration of arrows 
	between $|\mu \rangle$ and $|\lambda \rangle$ as in \Cref{fig:arrow_lambda}, right.
	Here the individual vertex weights are the $w_{u,s}$'s
	from \Cref{fig:table_w}, and the left boundary weights are \eqref{eq:w_boundary}.
	If $\mu\not\prec\lambda$, we set $\mathsf{F}_{\lambda/\mu}(u)=0$.

	In the sequel we will mostly omit the word ``stable'' (cf. \Cref{sub:rmk_non_stable} 
	on connections to the non-stable functions which were originally defined
	in \cite{Borodin2014vertex}), and will also abbreviate the name 
	to simply the \emph{sHL functions}.
\end{definition}

Define the functions with multiple variables inductively via the branching rule
(cf. \eqref{eq:F_G_branching}):
\begin{equation} \label{eq:F_stable_branching_rule}
	\mathsf{F}_{\lambda/\mu}(u_1, \dots, u_k ) = \sum_{\nu}
	\mathsf{F}_{\lambda/ \nu} (u_k)\, \mathsf{F}_{\nu / \mu}(u_1, \dots, u_{k-1}). 
\end{equation}
That is, $\mathsf{F}_{\lambda/\mu}(u_1,\ldots,u_k )$ is a partition function 
of ensembles of up-right paths as in \Cref{fig:paths_up_right}, left, 
with height $k$, spectral parameters $u_1,\ldots,u_k $
corresponding to horizontal slices, 
and boundary conditions $|\mu \rangle$, $0^\infty$, 
$|\lambda \rangle$ and empty at the bottom, left, up, and right, respectively.
The fact that the paths are directed up-right 
corresponds to the arrow conservation property \eqref{eq:sHL_arrow_conservation}.
Note that 
$\mathsf{F}_{\lambda/\mu}(u_1,\ldots,u_k )$
vanishes unless $0\le \ell(\lambda)-\ell(\mu)\le k$, 
but this condition is not sufficient.

The Yang-Baxter equation implies that
$\mathsf{F}_{\lambda/\mu}(u_1,\ldots,u_k )$
is symmetric with respect to permutations of the $u_i$'s, see, e.g.,
\cite[Theorem 3.6]{Borodin2014vertex}.
These functions also satisfy the \emph{stability property}
\begin{equation}\label{eq:sHL_stability}
	\mathsf{F}_{\lambda/\mu}(u_1,\ldots,u_k,0 )=
	\mathsf{F}_{\lambda/\mu}(u_1,\ldots,u_k ).
\end{equation}
For $\mu=\varnothing$, the stable spin Hall-Littlewood functions
admit an explicit 
symmetrization
formula \cite[(45)]{BorodinWheelerSpinq}
which we recall and use in \Cref{sec:diff_op}.
When $s=0$, the stable spin Hall-Littlewood functions
become the usual Hall-Littlewood symmetric polynomials
\cite[Chapter III]{Macdonald1995}.

\begin{figure}[htbp]
\centering
\includegraphics{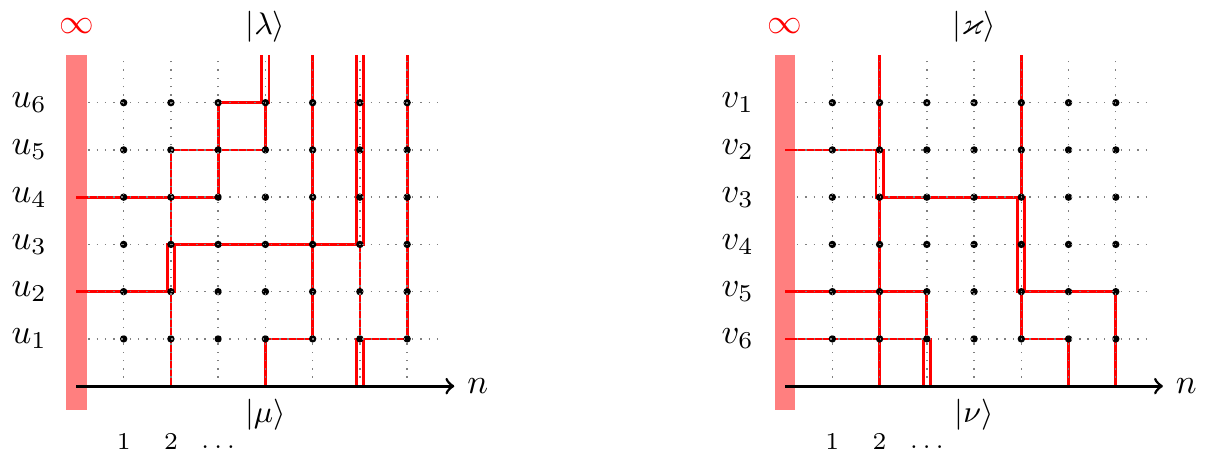}
\caption{Examples of configurations of up-right and down-right paths used in the definitions of $\mathsf{F}_{\lambda/\mu}$ and $\mathsf{F}^*_{\nu/\varkappa}$, respectively.}
\label{fig:paths_up_right}
\end{figure}

\subsection{Remark. Relations to non-stable sHL functions}
\label{sub:rmk_non_stable}

The spin Hall-Littlewood functions
were originally introduced in \cite{Borodin2014vertex} in their non-stable
version which we denote by $\mathsf{F}^{\textnormal{non-st}}_{\lambda/\mu}$.
The stable modification appeared in \cite{deGierWheeler2016}
and \cite{BorodinWheelerSpinq}.
The non-stable sHL functions differ by the boundary condition on the left:
a new horizontal arrow enters at each horizontal slice and each vertical edge on the leftmost column hosts only finitely many arrows.

In detail, the definition of 
$\mathsf{F}^{\textnormal{non-st}}_{\lambda/\mu}$
depends on the number of zero parts in $\lambda=0^{l_0}1^{l_1}2^{l_2}\ldots $
and $\mu=0^{m_0}1^{m_1}2^{m_2}\ldots $, and 
$\mathsf{F}^{\textnormal{non-st}}_{\lambda/\mu}(u)$ vanishes 
unless $l_0+l_1+\ldots=1+m_0+m_1+\ldots$.
When the latter condition holds, we define the single-variable function
$\mathsf{F}^{\textnormal{non-st}}_{\lambda/\mu}(u)$
as the weight of the unique configuration as in \Cref{def:ssHL},
but now 
the horizontal arrow \emph{must} enter at the leftmost boundary,
and the vertex weight at the zeroth column is 
$w_{u,s}(m_0,1;l_0,m_0+1-l_0)$. The multivariable version is 
defined using the branching rule exactly as in \eqref{eq:F_stable_branching_rule}.

There are two possible ways one could specialize the non-stable sHL
functions to obtain our $\mathsf{F}_{\lambda/\mu}$. 
The first is to send both
$l_0$ and $m_0$, the numbers of zeros in $\lambda$ and $\mu$, to infinity. 
By looking at
the weight of the leftmost vertices we see that
\begin{equation*}
	w_{u,s}(m_0,1;l_0,j) \xrightarrow[m_0,l_0\to \infty]{} \frac{u^j}{1 - s u},\qquad j\in \{ 0,1 \},
\end{equation*}
and therefore we obtain
\begin{equation}
	\label{eq:sHL_from_non_stable_1}
	\mathsf{F}_{\lambda/\mu}(u_1,\ldots,u_k ) 
	=
	\prod_{i=1}^{k}(1- s u_i) \times \lim_{m_0,l_0 \to \infty} \mathsf{F}_{\lambda\cup 0^{l_0}/\mu\cup 0^{m_0}}^{\textnormal{non-st}}(u_1,\ldots,u_k ).
\end{equation}
Here $\lambda\cup 0^{l_0}$ means adding $l_0$ zeros to the Young diagram $\lambda$
(which had no zeros originally),
and similarly for $\mu\cup 0^{m_0}$.

Another way is to consider the inhomogeneous vertex model as in \cite{BorodinPetrov2016inhom}
with the spin parameter $s_n$, $n\in \mathbb{Z}_{\ge0}$,
depending on the horizontal coordinate $n$ in \Cref{fig:paths_up_right}.
Taking $\mathsf{F}_{\lambda/\mu}^{\textnormal{non-st}}$
and setting
$s_0=0$ and $s_n=s$, $n>0$, 
from \Cref{fig:table_w} we see that
\begin{equation*}
	w_{u,0}(i_1,1;i_2,0)= 1 - q^{i_2} \qquad \text{and} \qquad w_{u,0}(i_1, 1; i_2, 1) = u.
\end{equation*}
Therefore, 
we obtain 
\begin{equation}
	\label{eq:sHL_from_non_stable_2}
	\mathsf{F}_{\lambda/\mu}(u_1,\dots ,u_k)
	= 
	\frac{1}{(q;q)_{k - \ell(\lambda) +\ell(\mu)}}  \,
	\mathsf{F}^{\textnormal{non-st}}_{\lambda\cup 0^{k-\ell(\lambda)+\ell(\mu)}/\mu}(u_1,\dots ,u_k) \Big\vert_{s_0=0},
\end{equation}
where we assume that $\mu,\lambda$ had no zeros originally.
Equality \eqref{eq:sHL_from_non_stable_2} is particularly useful when 
adapting the results about the non-stable sHL functions 
(like symmetrization formulas or integral
representations \cite{Borodin2014vertex}, \cite{BorodinPetrov2016inhom})
to the stable case.

\subsection{Dual stable spin Hall-Littlewood functions}
\label{sub:dual_sHL}

Let us introduce the dual weights to $w_{u,s}$ from \Cref{fig:table_w} as follows:
\begin{equation} 
	\label{eq:w_w_tilde_relation}
	w^*_{v,s}(i_1,j_1;i_2,j_2) = \frac{(s^2;q)_{i_1}(q;q)_{i_2}}{(q;q)_{i_1}(s^2;q)_{i_2}}\, w_{v,s}(i_2,j_1;i_1,j_2).
\end{equation}
The arrow conservation law \eqref{eq:sHL_arrow_conservation}
implies that $w_{v,s}^*(i_1,j_1;i_2,j_2)$ vanishes unless $i_2+j_1=i_1+j_2$,
and as a result the corresponding
vertex model produces configurations of directed down-right paths 
(see \Cref{fig:paths_up_right}, right).
The explicit form of the weights
$w_{v,s}^*$ is given in \Cref{fig:table_w_tilde}.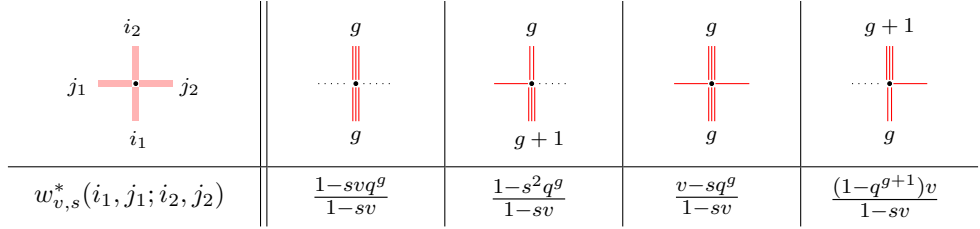
\begin{figure}[htbp]
  \centering
  \begin{tabular}{c||c|c|c|c}
    \begin{tikzpicture}[baseline=0]
    	\draw[fill] (0,0) circle [radius=0.025];
        \fill[red!30] (-0.04,-0.05) rectangle (0.04,-0.5) node[black, below]{\scriptsize{$i_1$}};
        \fill[red!30] (0.04,0.05) rectangle (-0.04,0.5) node[black,above]{ \scriptsize{$i_2$} };
        \fill[red!30] (0.5,0.04) rectangle (0.05,-0.04) node[black,xshift=0.67cm]{ \scriptsize{$j_2$} };
        \fill[red!30] (-0.05,0.04) rectangle (-0.5,-0.04) node[black,left]{ \scriptsize{$j_1$} };
    	\addvmargin{1mm}
    \end{tikzpicture}
     & \begin{tikzpicture}[baseline=0]
    	\draw[fill] (0,0) circle [radius=0.025];
		\draw [red] (0.04,-0.5) -- (0.04,-0.05);
    	\draw [red] (0,-0.5) -- (0,-0.05);
    	\draw [red] (-0.04,-0.5) -- (-0.04,-0.05);
        \node [below] at (0,-0.5){\scriptsize{$g$}};
        \draw [dotted] (-0.5,0) -- (-0.1,0);
        \draw [dotted] (0.1,0) -- (0.5, 0);
        \draw [red] (0.04,0.05) -- (0.04,0.5);
        \draw [red] (0,0.05) -- (0,0.5);
        \draw [red] (-0.04,0.05) -- (-0.04,0.5);
        \node [above] at (0,0.5){\scriptsize{$g$}};
        \addvmargin{1mm}
  \end{tikzpicture}  &  \begin{tikzpicture}[baseline=0]
        \draw[fill] (0,0) circle [radius=0.025];
        \draw [red] (-0.025,0.5) -- (-0.025,0.05);
        \draw [red] (0.025,0.5) -- (0.025,0.05);
        \node [below] at (0.1,-0.5){\scriptsize{$g+1$}};
        \draw [red] (-0.5,0) -- (-0.05,0);
        \draw [dotted] (0.1,0) -- (0.5, 0);
        \draw [red] (0,-0.05) -- (0,-0.5);
        \draw [red] (0.04,-0.05) -- (0.04,-0.5);
        \draw [red] (-0.04,-0.05) -- (-0.04,-0.5);
        \node [above] at (0,0.5){\scriptsize{$g$}};
        \addvmargin{1mm}
  \end{tikzpicture} & \begin{tikzpicture}[baseline=0]
        \draw[fill] (0,0) circle [radius=0.025];
        \draw [red] (0,-0.5) -- (0,-0.05);
        \draw [red] (0.04,-0.5) -- (0.04,-0.05);
        \draw [red] (-0.04,-0.5) -- (-0.04,-0.05);
        \node [below] at (0,-0.5){\scriptsize{$g$}};
        \draw [red] (-0.5,0) -- (-0.05,0);
        \draw [red,] (0.05,0) -- (0.5, 0);
        \draw [red] (0,0.05) -- (0,0.5);
        \draw [red] (0.04,0.05) -- (0.04,0.5);
        \draw [red] (-0.04,0.05) -- (-0.04,0.5);
        \node [above] at (0,0.5){\scriptsize{$g$}};
        \addvmargin{1mm}
  \end{tikzpicture} & \begin{tikzpicture}[baseline=0]
        \draw[fill] (0,0) circle [radius=0.025];
        \draw [red] (-0.04,0.5) -- (-0.04,0.05);
        \draw [red] (0,0.5) -- (0,0.05);
        \draw [red] (0.04,0.5) -- (0.04,0.05);
        \node [below] at (0,-0.5){\scriptsize{$g$}};
        \draw [dotted] (-0.5,0) -- (-0.1,0);
        \draw [red] (0.05,0) -- (0.5, 0);
        \draw [red] (0.025,-0.05) -- (0.025,-0.5);
        \draw [red] (-0.025,-0.05) -- (-0.025,-0.5);
        \node [above] at (0,0.5){\scriptsize{$g+1$}};
        \addvmargin{1mm}
  \end{tikzpicture}\\
    \hhline{-----}
   
    \begin{minipage}{3cm}
    \centering
    \vspace{.1cm}
    \small{$w^*_{v,s}(i_1,j_1; i_2,j_2)$}
    \vspace{.1cm}
    \end{minipage}
    &
    \begin{minipage}{2cm}
    \centering
    \vspace{.1cm}
    $\frac{1- s v q^{g} }{1- s v}$
    \vspace{.1cm}
    \end{minipage}
    &
    \begin{minipage}{2cm}
    \centering
    \vspace{.1cm}
    $\frac{ 1- s^2 q^g }{1- s v}$
    \vspace{.1cm}
    \end{minipage}
    &
    \begin{minipage}{2cm}
    \centering
    \vspace{.1cm}
    $\frac{ v - s q^{g}}{1 - s v}$
    \vspace{.1cm}
    \end{minipage}
    & 
    \begin{minipage}{2cm}
    \centering
    \vspace{.1cm}
    $\frac{(1- q^{g+1})v}{1- s v}$
    \vspace{.1cm}
    \end{minipage}
    \\
  \end{tabular}
  \caption{In the top row we see all acceptable configurations of paths entering and exiting a vertex; 
		below we reported the corresponding vertex weights 
		$w^*_{v,s}(i_1, j_1; i_2, j_2)$.} 
		\label{fig:table_w_tilde}
\end{figure}
The weights $w^*_{v,s}$ at the left boundary
are given by the same formulas as in
\eqref{eq:w_boundary}.

The weights $w^*_{v,s}$ can be obtained from 
$w_{u,s}$ by substituting $u$ with $1/v$, swapping the values
of both horizontal edge indices $j_1$ and $j_2$ (that is if $j_1=0$, then
we change its value 1 and vice versa, and the same for $j_2$), and multiplying
the result by $(v-s)/(1-vs)$. 
This swapping construction of the dual weights was instrumental 
in deriving Cauchy identities for the sHL functions
from the Yang-Baxter equation \cite{Borodin2014vertex}
(a bijectivization of this argument appeared in \cite{BufetovPetrovYB2017}). 
In the present paper we employ a more direct approach with down-right paths which is
better suited for the generalization to spin $q$-Whittaker functions.
The Yang-Baxter equation connecting $w_{u,s}$ and $w^*_{v,s}$
is recorded in \Cref{app:YBE}.

\begin{definition}
	Fix $\varkappa,\nu\in \mathbb{Y}$ with $\varkappa\prec \nu$
	and place the arrow configuration $|\nu \rangle$ 
	\emph{under} $|\varkappa \rangle$. Then there exists a unique
	configuration of horizontal arrows between 
	$|\varkappa \rangle$ and $|\nu \rangle$.
	By definition,
	a \emph{dual stable sHL function} in one variable,
	denoted by $\mathsf{F}_{\nu/\varkappa}^*(v)$,
	is the total weight of this horizontal arrow configuration,
	where the individual vertex weights are the $w^*_{v,s}$'s from \Cref{fig:table_w_tilde}, and 
	the left boundary weights are the same as in \eqref{eq:w_boundary}.
	If $\varkappa\not\prec\nu$, we set $\mathsf{F}^*_{\nu/\varkappa}(v)=0$.
\end{definition}

The multivariable generalization $\mathsf{F}^*_{\nu/\varkappa}(v_1,\ldots,v_k )$ 
is defined via the
branching rule exactly as in \eqref{eq:F_stable_branching_rule}.
It is the partition function of ensembles of down-right paths
as in \Cref{fig:paths_up_right}, right, of height $k$, 
spectral parameters $v_1,\ldots,v_k $ corresponding to horizontal slices,
and boundary conditions $|\varkappa \rangle$, 
$0^\infty$, $|\nu \rangle$, and empty
at the bottom, left, top, and right, respectively.

The relation
\eqref{eq:w_w_tilde_relation} 
between $w^*_{v,s}$ and $w_{u,s}$
implies that 
\begin{equation}
	\label{eq:sHL_sHL_star_relation}
	\frac{\mathsf{c}(\lambda)}{\mathsf{c}(\mu)} 
	\,
	\mathsf{F}_{\lambda/ \mu}(u_1, \dots, u_k) =
	\mathsf{F}^*_{\lambda/ \mu}(u_1, \dots, u_k),
\end{equation}
where the factor $\mathsf{c}$ is
\begin{equation*}
	\mathsf{c}(\mu) =  
	\prod_{i \geq 1} \frac{(s^2 ;q)_{m_i}}{(q;q)_{m_i}}, 
	\qquad \text{for }\mu=1^{m_1} 2^{m_2} \ldots.
\end{equation*}
The symmetry of $\mathsf{F}^*_{\lambda/\mu}(v_1,\ldots,v_k )$
in the $v_i's$ follows from the symmetry of $\mathsf{F}_{\lambda/\mu}$.
The dual sHL function also satisfies the same stability property 
\eqref{eq:sHL_stability} as the non-dual one.

\subsection{The sHL/sHL skew Cauchy structure}
\label{sub:sHL_sHL_structure}

One of the main consequences of the Yang-Baxter equation (either 
\Cref{prop:YBE_rww} or \Cref{prop:sHL_YBE})
is the skew Cauchy identity
for the sHL functions:
\begin{theorem}
	[{\cite{Borodin2014vertex}, \cite{BorodinPetrov2016inhom}, \cite[Section 7.4]{BorodinWheelerSpinq}}]
	\label{thm:skew_Cauchy_sHL_sHL}
	For any two Young diagrams $\lambda,\mu$ and generic parameters $u,v\in \mathbb{C}$
	(cf. \Cref{rmk:generic})
	such that
	$\bigl|(u-s)(v-s)\bigr|<\bigl|(1-su)(1-sv)\bigr|$,
	we have
	\begin{equation}
		\label{eq:skew_Cauchy_sHL_sHL}
		\sum_{\nu} \mathsf{F}^*_{\nu / \lambda}(v) \,\mathsf{F}_{\nu / \mu}(u)
		=
		\frac{ 1 - q u v }{ 1 - u v } 
		\sum_{\varkappa} \mathsf{F}_{\lambda / \varkappa}(u)
		\,\mathsf{F}^*_{\mu / \varkappa}(v).
	\end{equation}
\end{theorem}
We recall a ``bijective'' proof of \Cref{thm:skew_Cauchy_sHL_sHL}
in \Cref{sub:new_YB_field_sHL_sHL} below
which follows the approach
of \cite{BufetovPetrovYB2017}.
This identity
together with the branching rules for the sHL functions 
lead to the first of the skew Cauchy structures we consider in the paper:

\begin{definition}
	\label{def:sHL_sHL_structure}
	The families of functions
	\begin{equation*}
		\mathfrak{F}_{\lambda/\mu}(u_1,\ldots,u_k )=\mathsf{F}_{\lambda/\mu}(u_1,\ldots,u_k ) ,
		\qquad 
		\mathfrak{G}_{\lambda/\mu}(v_1,\ldots,v_k )=\mathsf{F}^*_{\lambda/\mu}(v_1,\ldots,v_k )
	\end{equation*}
	form a skew Cauchy structure in the sense of \Cref{sub:F_G_skew_Cauchy_structure}
	with the following identifications:
	\begin{enumerate}[label=\bf{(\roman*)}]
		\item The relations $\mu\prec_k\lambda$ and $\mu\mathop{\dot{\prec}_k}\lambda$ are the same and
			mean the existence of a sequence of Young diagrams 
			$\mu\prec \varkappa^{(1)}\prec \ldots\prec \varkappa^{(k-1)}\prec \lambda$,
			where $\prec$ is the interlacing relation \eqref{eq:interlace_sHL_def}.
		\item The skew Cauchy identity holds with 
			\begin{equation}
				\label{eq:sHL_Pi}
				\mathsf{Adm}=\left\{ (u,v)\colon \bigl|(u-s)(v-s)\bigr|<\bigl|(1-su)(1-sv)\bigr| \right\},
				\qquad 
				\Pi(u;v)=\dfrac{1-quv}{1-uv}.
			\end{equation}
		\item 
			Let us choose the external parameters $q\in(0,1)$, $s\in(-1,0)$, 
			and take $\mathsf{P}=\dot{\mathsf{P}}=[0,1]$.
			Then the 
			probability weights based on $\mathsf{F}_{\lambda/\mu}(u_1,\ldots,u_k)$ and $\mathsf{F}^*_{\lambda/\mu}(v_1,\ldots,v_k )$
			with $u_i,v_j\in [0,1]$ are nonnegative 
			due to the nonnegativity of all the vertex weights
			in \Cref{fig:table_w,fig:table_w_tilde}.
	\end{enumerate}
	We call this the \emph{sHL/sHL skew Cauchy structure}.
\end{definition}

\begin{remark}
	\label{rmk:nonnegativity_and_Adm_in_sHL_sHL}
	When $u,v\in[0,1)$ and $s\in (-1,0)$, one can 
	check that $(u,v)\in \mathsf{Adm}$.
\end{remark}

\subsection{\texorpdfstring{Spin $q$-Whittaker polynomials}{Spin q-Whittaker polynomials}}
\label{sub:sqW_polynomials}

Along with the sHL functions we will work with the spin $q$-Whittaker (sqW)
polynomials introduced in \cite{BorodinWheelerSpinq} which we recall here.
We start by defining the vertex weights $W_{\xi,s}$ as
\begin{equation} 
	\label{eq:Whit_W}
	W_{\xi,s}(i_1,j_1;i_2, j_2) 
	= 
	\mathbf{1}_{i_1 + j_1 = i_2 + j_2} \, \mathbf{1}_{i_1 \geq j_2}\,  \xi^{j_2}\, 
	\frac{(- s/\xi;q)_{j_2} (- s \xi ; q )_{i_1 - j_2} (q;q)_{i_2} }{(q;q)_{j_2} (q;q)_{i_1 - j_2} (s^2;q)_{i_2} },
\end{equation}
where $i_1,j_1,i_2,j_2\in \mathbb{Z}_{\ge0}$.
Note that in contrast with $w_{u,s}$ and $w^*_{v,s}$ 
used in the definition of 
the sHL functions, here the number of horizontal arrows $j_1,j_2$
can be arbitrary and not just zero or one.

The dual version of the weight $W_{\xi,s}$ is given by 
\begin{equation} 
	\label{eq:W_W_tilde_relation}
	W^*_{\theta,s}(i_1,j_1;i_2,j_2) 
	= 
	\frac{(s^2;q)_{i_1}(q;q)_{i_2}}{(q;q)_{i_1}(s^2;q)_{i_2}}
	\,
	W_{\theta,s}(i_2,j_1;i_1,j_2),
\end{equation}
which is the same relation as between $w$ and $w^*$ \eqref{eq:w_w_tilde_relation}.
The weights $W^*_{\theta,s}$ vanish unless $i_2+j_1=i_1+j_2$, therefore
the dual vertex model will have down-right paths. 
The dependence of both $W_{\xi,s}$ and $W^*_{\theta,s}$ on their respective spectral parameters
$\xi,\theta$ is polynomial.

As explained in \Cref{app:YBE}, there exists a close relation between the weights $W$ and the weights $w$:
the former can be obtained from the latter through a procedure called \emph{fusion}.
The fusion consists in collapsing multiple $w$-weighted rows of vertices
with spectral parameters forming a geometric progression with ratio $q$
Fusion originated in \cite{KulishReshSkl1981yang}
and was employed in
\cite{Borodin2014vertex}, 
\cite{CorwinPetrov2015}, 
\cite{BorodinPetrov2016inhom},
\cite{BorodinWheelerSpinq}
in connection with stochastic vertex models.
In particular, the 
weights $W_{\xi,s}$ and $W^*_{\theta,s}$ satisfy the Yang-Baxter equation
listed in \Cref{app:YBE}.

Define the left boundary weights for $j\in \mathbb{Z}_{\ge0}$
by
\begin{equation}
	\label{eq:W_boundary_weights}
		W_{\xi,s} \biggl(\begin{tikzpicture}[baseline=-2.5pt]
    	\draw[fill] (0,0) circle [radius=0.025];
			\node at (0,.3) {$\infty$};
			\node at (0,-.3) {$\infty$};
			\draw [red] (0.1,0) --++ (0.4, 0) node[right, black] {$j$};
			\draw [red] (0.1,0.05) --++ (0.4, 0);
			\draw [red] (0.1,-0.05) --++ (0.4, 0);
        \addvmargin{1mm}
        \addhmargin{1mm}
	  \end{tikzpicture} \biggr) = 
		W_{\xi,s}^* \biggl(\begin{tikzpicture}[baseline=-2.5pt]
    	\draw[fill] (0,0) circle [radius=0.025];
			\node at (0,.3) {$\infty$};
			\node at (0,-.3) {$\infty$};
			\draw [red] (0.1,0) --++ (0.4, 0) node[right, black] {$j$};
			\draw [red] (0.1,0.05) --++ (0.4, 0);
			\draw [red] (0.1,-0.05) --++ (0.4, 0);
        \addvmargin{1mm}
        \addhmargin{1mm}
	  \end{tikzpicture} \biggr) = 
		\xi^j\,\frac{(-s/\xi;q)_j}{(q;q)_j}.
	\qquad 
\end{equation}

\begin{definition}[Column interlacing]
	\label{def:transposed_interlacing}
	Fix $\mu,\lambda\in \mathbb{Y}$.
	We write
	$\mu\mathop{\prec'}\lambda$
	and say that $\mu$ and $\lambda$ \emph{column-interlace}
	if there exists a configuration of 
	finitely many horizontal arrows between $|\mu \rangle$ and $|\lambda \rangle$
	(located one under another as in \Cref{fig:arrow_lambda})
	such that at each vertex 
	$(i_1,j_1;i_2,j_2)$
	the arrow conservation property 
	$i_1+j_1=i_2+j_2$
	holds, and, moreover, $j_2\le i_1$. Note that now we allow arbitrarily many horizontal arrows per edge.
	(If such a horizontal arrow configuration exists, then it is unique.)
	One can check that
	$\mu\mathop{\prec'}\lambda$
	if and only if $\mu'\prec \lambda'$,
	where $\mu'$ and $\lambda'$ stand for \emph{transposed Young diagrams}:
	\begin{equation*}
		\lambda'_j:=\#\left\{ i\colon \lambda_i\ge j \right\}.
	\end{equation*}
\end{definition}

\begin{definition}
	\label{def:sqW_polynomial}
	For $\mu,\lambda\in \mathbb{Y}$ with $\mu\prec'\lambda$,
	a \emph{spin $q$-Whittaker polynomial} in one variable,
	denoted by $\mathbb{F}_{\lambda'/\mu'}(\xi)$, is defined as the total weight of
	the unique configuration of arrows
	between $|\mu \rangle$ and $|\lambda \rangle$.
	Here the individual vertex weights are $W_{\xi,s}$
	\eqref{eq:Whit_W}, \eqref{eq:W_boundary_weights}.
	If $\mu\not\prec'\lambda$, we set 
	$\mathbb{F}_{\lambda'/\mu'}(\xi)=0$.

	We will abbreviate the name and call $\mathbb{F}_{\lambda'/\mu'}$
	simply the (skew) \emph{sqW polynomial}.
	Note that it is indexed by the transposed Young diagrams
	for consistency with the $s=0$ situation when 
	$\mathbb{F}_{\lambda'/\mu'}$
	turns into the more common skew $q$-Whittaker polynomial
	which is a $t=0$ degeneration of the corresponding Macdonald polynomial
	\cite{Macdonald1995}, \cite{BorodinCorwin2011Macdonald}.
\end{definition}

The \emph{dual sqW polynomials} $\mathbb{F}_{\nu/\varkappa}^*(\theta)$ are defined 
in a similar manner, up to placing 
$|\nu \rangle$~\emph{under}~$|\varkappa \rangle$,
and using the dual vertex weights $W^*_{\theta,s}$
\eqref{eq:W_W_tilde_relation}, \eqref{eq:W_boundary_weights}.
We have (cf. \eqref{eq:sHL_sHL_star_relation})
\begin{equation}
	\label{eq:sqW_sqW_star_relation}
	\frac{\mathsf{c}(\nu)}{\mathsf{c}(\varkappa)}
	\,
	\mathbb{F}_{\nu'/\varkappa'}(\xi_1,\ldots,\xi_k )
	=
	\mathbb{F}_{\nu'/\varkappa'}^*(\xi_1,\ldots,\xi_k ).
\end{equation}

The multivariable polynomials
$\mathbb{F}_{\lambda/\mu}(\xi_1,\ldots,\xi_k)$
and 
$\mathbb{F}^*_{\nu/\varkappa}(\theta_1,\ldots,\theta_k)$
are defined via the branching rules exactly as in \eqref{eq:F_stable_branching_rule}.
One can view them as partition functions
of path ensembles similarly to the ones in 
\Cref{fig:paths_up_right},
but with multiple horizontal arrows allowed per edge.
The Yang-Baxter equation implies that 
$\mathbb{F}_{\lambda/\mu}(\xi_1,\ldots,\xi_k)$
and 
$\mathbb{F}^*_{\nu/\varkappa}(\theta_1,\ldots,\theta_k)$
are symmetric in their respective variables.
They also satisfy the following stability property:
\begin{equation*}
	\mathbb{F}_{\lambda/\mu}(\xi_1,\ldots,\xi_{k-1},-s )=\mathbb{F}_{\lambda/\mu}(\xi_1,\ldots,\xi_{k-1} )
\end{equation*}
(and similarly for $\mathbb{F}_{\nu/\varkappa}^*$),
which follows from the vanishing of the vertex weight
$W_{-s,s}$.
Note that here we are substituting $(-s)$ 
for one of the variables
in contrast with 
the sHL functions where we 
substituted $0$
\eqref{eq:sHL_stability}.

\subsection{The sHL/sqW skew Cauchy structure}
\label{sub:sHL_sqW_structure}

The Yang-Baxter equation for the weights
$(w^*_{v,s},W_{\xi,s})$, see \Cref{prop:sHL_sqW_YBE},
implies the following
``dual'' 
skew Cauchy identity for the sHL and sqW functions:

\begin{theorem}[{\cite[Section 7.3]{BorodinWheelerSpinq}}]
	\label{thm:sHL_sqW_skew_Cauchy}
	For any two Young diagrams $\lambda,\mu$, 
	and
	generic $\xi,u\in \mathbb{C}$ (cf. \Cref{rmk:generic}; in particular, $u\ne s^{-1}$) we have
	\begin{equation}
		\label{eq:sHL_sqW_skew_Cauchy}
		\begin{split}
			\sum_{\nu}\mathbb{F}_{\nu'/\lambda'}(\xi)\,\mathsf{F}^*_{\nu/\mu}(u)
			&=
			\frac{1+u\xi}{1-us}\sum_{\varkappa}
			\mathbb{F}_{\mu'/\varkappa'}(\xi)\,\mathsf{F}^*_{\lambda/\varkappa}(u)
			;\\
			\sum_{\nu}\mathbb{F}^*_{\nu'/\lambda'}(\xi)\,\mathsf{F}_{\nu/\mu}(u)
			&=
			\frac{1+u\xi}{1-us}\sum_{\varkappa}
			\mathbb{F}^*_{\mu'/\varkappa'}(\xi)\,\mathsf{F}_{\lambda/\varkappa}(u)
			.
		\end{split}
	\end{equation}
	Note that the sums over $\nu$ and $\varkappa$ in both sides are 
	actually finite, so there are no convergence issues.
	The above two identities are equivalent: One can multiply the first
	one by $\mathsf{c}(\mu)/\mathsf{c}(\lambda)$ and redistribute the factors
	to get the second one.
\end{theorem}
We give a ``bijective'' proof of \Cref{thm:sHL_sqW_skew_Cauchy}
in \Cref{sub:new_YB_field_sHL_sqW} below. This
leads to the following definition:
\begin{definition}
	\label{def:sHL_sqW_struture}
	The families of functions
	$\mathfrak{F}_{\lambda/\mu}(u_1,\ldots u_k)=\mathsf{F}^*_{\lambda/\mu}(u_1,\ldots,u_k )$
	and $\mathfrak{G}_{\lambda/\mu}(\xi_1,\ldots,\xi_k )=\mathbb{F}_{\lambda'/\mu'}(\xi_1,\ldots,\xi_k )$
	form a skew Cauchy structure in the sense
	of \Cref{sub:F_G_skew_Cauchy_structure}
	with the following identifications:
	\begin{enumerate}[label=\bf{(\roman*)}]
		\item The relations $\prec_k$,
			$\mathop{\dot{\prec_k}}$
			on $\mathbb{Y}\times\mathbb{Y}$
			are 
			\begin{equation*}
				\begin{split}
					\mu\prec_k\lambda
					\qquad \Leftrightarrow \qquad 
					&
					\exists \,\varkappa^{(i)}\in \mathbb{Y}\colon
					\mu\prec\varkappa^{(1)}\prec\ldots \prec \varkappa^{(k-1)}\prec\lambda;
					\\
					\mu\mathop{\dot{\prec}_k}\lambda
					\qquad \Leftrightarrow \qquad 
					&
					\exists\, \rho^{(j)}\in \mathbb{Y}\colon
					\mu\prec' \rho^{(1)}\prec' \ldots\prec' \rho^{(k-1)}\prec'\lambda ,
				\end{split}
			\end{equation*}
			where $\prec$ and $\prec'$ are the usual and the column interlacing
			relations
			(\Cref{def:interlacing,def:transposed_interlacing}).
		\item \label{item:sHL_sqW_Pi} The skew Cauchy identity
			holds with $\mathsf{Adm}=\left\{ (u,\xi)\in \mathbb{C}^2\colon u\ne s^{-1} \right\}$
			and 
			\begin{equation}
				\label{eq:sHL_sqW_Pi}
				\Pi(u;\xi)=\dfrac{1+u\xi}{1-su}.
			\end{equation}
		\item The external parameters of the functions are $q\in(0,1)$ and $s\in(-1,0)$,
			and the nonnegativity sets for the spectral parameters are
			$\mathsf{P}=[0,1]$, $\dot{\mathsf{P}}=[-s,-s^{-1}]$.
			Then the probability weights based on $\mathsf{F}^*_{\lambda/\mu}(u_1,\ldots,u_k )$
			and $\mathbb{F}_{\lambda'/\mu'}(\xi_1,\ldots,\xi_k )$
			are nonnegative for $u_i\in \mathsf{P}$, $\xi_j\in \dot{\mathsf{P}}$
			due to the nonnegativity of the vertex weights in
			\Cref{fig:table_w_tilde} and \eqref{eq:Whit_W}.
	\end{enumerate}
	We call this the \emph{sHL/sqW skew Cauchy structure}. 
\end{definition}
\begin{remark}
	\label{rmk:sHL_sqW_conjugation_does_not_hurt}
	\Cref{def:sHL_sqW_struture} is based on the first
	of the skew Cauchy identities \eqref{eq:sHL_sqW_skew_Cauchy}.
	One readily sees that taking
	the second of these identities
	leads to the same notion of a random field associated
	with the other skew Cauchy structure.
	In other words, one can understand skew Cauchy structures
	up to ``gauge transformations'' of the form
	$(\mathfrak{F}_{\lambda/\mu},\mathfrak{G}_{\nu/\varkappa})\mapsto
	\bigl(
		\frac{c(\lambda)}{c(\mu)}\,\mathfrak{F}_{\lambda/\mu},
		\frac{c(\varkappa)}{c(\nu)}\,\mathfrak{G}_{\nu/\varkappa}
	\bigr)$,
	where $c(\cdot)$ is nowhere vanishing.
	The same remark applies to the two other skew Cauchy structures --- it does not
	matter which of the two families of functions carries the ``$*$''.
\end{remark}

\subsection{The sqW/sqW skew Cauchy structure}
\label{sub:sqW_sqW_structure}
The spin $q$-Whittaker polynomials also satisfy the following
skew Cauchy identity which follows from the Yang-Baxter
equation \eqref{eq:YBE_W_W}:
\begin{theorem}[{\cite[Section 7.1]{BorodinWheelerSpinq}}]
	\label{thm:sqW_sqw_skew_Cauchy_identity}
	For any two Young diagrams $\lambda,\mu$ and parameters $\xi,\theta\in \mathbb{C}$
	with $|\xi \theta|<1$ we have
	\begin{equation}
		\sum_{\nu} \mathbb{F}_{\nu / \lambda}(\xi) \,\mathbb{F}^*_{\nu / \mu} (\theta) 
		=
		\frac{(- s \xi ;q)_\infty (- s \theta ;q)_\infty}{(s^2;q)_\infty ( \xi \theta ; q)_\infty}
		\sum_{\varkappa} \,\mathbb{F}_{\mu / \varkappa}(\xi)
		\mathbb{F}^*_{\lambda / \varkappa} (\theta). 
		\label{eq:sqW_sqw_skew_Cauchy_identity}
	\end{equation}
\end{theorem}

We give a ``bijective'' proof of \Cref{thm:sqW_sqw_skew_Cauchy_identity} 
in \Cref{sub:new_YB_field_sqW_sqW} below. This identity 
motivates the definition of the third 
skew Cauchy structure we consider in the present paper:
\begin{definition}
	\label{def:sqW_sqW_structure}
	The families of functions
	\begin{equation*}
		\mathfrak{F}_{\lambda/\mu}(\theta_1,\ldots,\theta_k )=\mathbb{F}^*_{\lambda/\mu}(\theta_1,\ldots,\theta_k ),\qquad 
		\mathfrak{G}_{\lambda/\mu}(\xi_1,\ldots,\xi_k )=\mathbb{F}_{\lambda/\mu}(\xi_1,\ldots,\xi_k )
	\end{equation*}
	form a skew Cauchy structure in the sense of \Cref{sub:F_G_skew_Cauchy_structure} under the following identifications:
	\begin{enumerate}[label=\bf{(\roman*)}]
		\item The relations $\mu\prec_k\lambda$ and $\mu\mathop{\dot{\prec}_k}\lambda$ are the same and 
			mean the existence of $\varkappa^{(i)}$
			such that $\mu\prec\varkappa^{(1)}\prec\ldots\prec \varkappa^{(k-1)}\prec\lambda $,
			where $\prec$ is the interlacing relation \eqref{eq:interlace_sHL_def}.
		\item The skew Cauchy identity holds with $\mathsf{Adm}=\left\{ (\theta,\xi)\in \mathbb{C}^2\colon |\xi\theta|<1 \right\}$ and
			\begin{equation}
				\label{eq:sqW_Pi}
				\Pi(\theta;\xi)=\dfrac{(-s\xi;q)_{\infty}(-s\theta;q)_{\infty}}{(s^2;q)_{\infty}( \xi \theta;q)_{\infty}}.
			\end{equation}
			Both $\mathsf{Adm}$ and $\Pi$ are symmetric in $\xi$ and $\theta$ so the order is not essential. We
			write $(\theta,\xi)$ to match with the notation of \Cref{sub:F_G_skew_Cauchy_structure}.
		\item The external parameters are $q\in(0,1)$ and $s\in (-1,0)$,
			and
			$\mathsf{P}=\dot{\mathsf{P}}=[-s,-s^{-1}]$. 
			Indeed, 
			$\mathbb{F}^*_{\lambda/\mu}(\theta_1,\ldots,\theta_k )$ and
			$\mathbb{F}_{\lambda/\mu}(\xi_1,\ldots,\xi_k )$
			evaluated 
			at $\xi_i,\theta_j\in [-s,-s^{-1}]$
			are nonnegative due to the nonnegativity of the vertex weights
			\eqref{eq:Whit_W}, \eqref{eq:W_W_tilde_relation}.
	\end{enumerate}
	We call this the \emph{sqW/sqW skew Cauchy structure}. 
\end{definition}

\section{Fusion and analytic continuation}
\label{sec:analytic_continuation}

\subsection{A common generalization of skew Cauchy identities}

The skew Cauchy identities from \Cref{sec:summary_sHL_sqW}
admit a common generalization which can be viewed as an analytic continuation.
In \cite{Borodin2014vertex}, \cite{BorodinPetrov2016inhom}, \emph{principal
specializations} of non-stable spin Hall-Littlewood functions
were considered. They are obtained by 
setting spectral parameters to finite geometric progressions of ratio $q$.
In our context, define 
\begin{align}
    \mathfrak{F}^{(J_1, \dots J_n)}_{\lambda / \mu} (u_1, \dots, u_n) 
		&= 
		\mathsf{F}_{\lambda / \mu}(u_1, q u_1, \dots, q^{J_1-1}u_1, \dots, u_n, q u_n, \dots, q^{J_n-1}u_n) \label{eq:F_principal_spec}
    \\
    \mathfrak{G}^{(I_1, \dots I_m)}_{\lambda / \mu} (v_1, \dots, v_m) 
		&= 
		\mathsf{F}^*_{\lambda / \mu}(v_1, q v_1, \dots, q^{I_1-1}v_1, \dots, v_m, q v_m, \dots, q^{I_m-1}v_m). \label{eq:G_principal_spec}
\end{align}
It is a consequence of the fusion procedure dating back to 
\cite{KulishReshSkl1981yang}
that we can view $\mathfrak{F}^{(J_1, \dots
J_n)}_{\lambda / \mu} (u_1, \dots, u_n)$ as a partition function in a ``smaller'' vertex
model obtained by attaching together $n$ (instead of $J_1+\ldots+J_n$) rows with fused weights
$w^{(J_k)}_{u_k,s}$,
where $k=1,\dots, n$
and
\begin{equation}
	\label{eq:w_fused_J_text}
\begin{split}
	w_{u,s}^{(J)} (i_1,j_1;i_2,j_2) &= \mathbf{1}_{i_1+j_1=i_2+j_2}\,
	\frac{(-1)^{i_1+j_2}q^{\frac{1}{2} i_1(i_1-1+2 j_1)} s^{j_2-i_1} u^{i_1} (u/s;q)_{j_1-i_2} (q;q)_{j_1} }{(q;q)_{i_1} (q;q)_{j_2} (s u;q)_{j_1+i_1}}\\
&
\qquad \qquad \times \setlength\arraycolsep{1pt}
{}_4 \overline{ \phi}_3\left(\begin{minipage}{4cm}
\center{$q^{-i_1}; q^{-i_2},  s u q^J,  qs/u$}\\\center{$s^2,q^{1+j_2-i_1}, q^{1-i_2-j_2+J}$}
\end{minipage}
\Big\vert\, q,q\right),
\end{split}
\end{equation}
where $\setlength\arraycolsep{1pt}{}_4 \overline{ \phi}_3$ is the regularized $q$-hypergeometric series
\eqref{eq:hypergeom_series}.
\begin{remark}
	Note that
	\cite[(31)]{BorodinWheelerSpinq}
	gives a slightly different formula 
	for $w^{(J)}$. However, these two expressions are the 
	same, and the discrepancy in the multiplicative prefactor
	is compensated by the fact that the $_4\overline{\phi}_3$ is 
	not symmetric in the first two arguments.
\end{remark}

Analogously, 
$\mathfrak{G}^{(I_1, \dots I_m)}_{\lambda / \mu} (v_1, \dots, v_m)$ are
partition functions of a vertex model with fused weights $w^{*,(I_k)}_{v_k,s}$,
where
\begin{equation}
	\label{eq:w_fused_I_dual_text}
	w^{*,(I)}_{v,s}(i_1,j_1;i_2,j_2) 
	= 
	\frac{(s^2;q)_{i_1} (q;q)_{i_2} }{ (q;q)_{i_1} (s^2;q)_{i_2} }
	\,
	w^{(I)}_{v,s}(i_2,j_1;i_1;j_2).
\end{equation}
As usual, at the leftmost column of these lattices
we place infinitely many vertical paths. 
More details on the fused weights can be found in
\Cref{app:fusion}.

The weights $w^{(J)}$,
$w^{*,(I)}$ degenerate both to $w, w^*$ and $W, W^*$,
see \Cref{fig:table_positivity_YBE}
below for exact details.
Thus,
\eqref{eq:F_principal_spec},
\eqref{eq:G_principal_spec} interpolate between the spin Hall-Littlewood
functions and the spin $q$-Whittaker functions. 
These functions
satisfy the following
general skew Cauchy identity which we state for an appropriate ``analytic'' range of
parameters:
\begin{theorem} \label{thm:skew_Cauchy_general}
	Fix $m,n\in \mathbb{Z}_{\ge0}$.
	Take $q\in(0,1)$, and let $s \neq 0, u_i, q^{J_i}, v_j, q^{I_j}$ 
	be complex parameters satisfying 
	\begin{equation} \label{eq:condition_parameters_general_skew_Cauchy}
		|s|,|u_k|,|v_l|,|q^{J_k} u_k|,|q^{I_l} v_l|,
		\left| \frac{q^i u_k - s}{1 - q^i s u_k} \right|, 
		\left| \frac{q^i v_l - s}{1 - q^i s v_l} \right|< \delta \qquad \text{for all }k,l,i,
	\end{equation}
	for sufficiently small $\delta >0$ which might depend on $m,n$,
	but not on the other parameters of the 
	functions.\footnote{Here and below in this section
		one can think of 
		$q^{J_k}$ and $q^{I_l}$ as separate symbols independent of $q$,
		because the fused weights $w^{(J)}$ and $w^{*,(I)}$
		depend on $q^{J}$ and $q^{I}$
		in a rational way.
		When $J$ a positive integer, 
		$q^J$ is equal to the $J$-th power of $q$, but we're free to assign
		an arbitrary value to $q^{J}$, for $J$ not necessarily a positive 
		integer (and same for $q^I$).}
	Then we have
	\begin{equation} \label{eq:skew_Cauchy_general}
			\begin{split}
					&\sum_\nu \mathfrak{F}^{(J_1,\dots J_n)}_{\nu / \mu} (u_1, \dots, u_n)
					\mathfrak{G}^{(I_1,\dots I_m)}_{\nu / \lambda} (v_1, \dots, v_m) 
					\\
					&\hspace{10pt} = 
					\prod_{k=1}^{n}\prod_{l=1}^{m}
				\frac{(u_k v_l q^{I_l};q)_\infty (u_k v_l q^{J_k};q)_\infty}
				{(u_k v_l;q)_\infty (u_k v_l q^{I_l + J_k};q)_\infty} \,
				\sum_{\varkappa} \mathfrak{F}_{\lambda/\varkappa}^{(J_1, \dots, J_n)}(u_1, \dots, u_n) 
				\mathfrak{G}^{(I_1, \dots, I_m)}_{\mu / \varkappa} (v_1, \dots, v_m).
			\end{split}
	\end{equation}
\end{theorem}

\begin{remark}
    This identity immediately
		degenerates to 
		the skew Cauchy identities
		\eqref{eq:skew_Cauchy_sHL_sHL}, 
		\eqref{eq:sHL_sqW_skew_Cauchy}, and \eqref{eq:sqW_sqw_skew_Cauchy_identity} 
		after specializing the parameters
		$u_k,v_l$ and $q^{J_k},q^{I_l}$
		as in \Cref{fig:table_positivity_YBE} below.
\end{remark}

The proof of \Cref{thm:skew_Cauchy_general} requires an absolute convergence 
result for spin Hall-Littlewood functions 
with principal specializations:
\begin{proposition} \label{prop:sHL_absolute_integrability}
	Fix $n\in \mathbb{Z}_{\ge1}$.
	Let $q\in(0,1)$.
	Take $s \neq 0, u_i, q^{J_i}$ to be complex parameters satisfying 
	\begin{equation}
			|s|,|u_k|,|q^{J_k} u_k|,\left| \frac{q^i u_k - s}{1 - q^i s u_k} \right|< \delta \qquad \text{for all }k,i,
	\end{equation}
	for some $\delta=\delta_n>0$ (which might depend on $n$). Then
for all Young diagrams $\mu$ we have 
\begin{equation}\label{eq:sHL_absolute_integrability}
    \sum_{\lambda\colon \mu \subseteq \lambda } 
		\left| \mathfrak{F}^{(J_1,\dots J_n)}_{\lambda / \mu} (u_1, \dots, u_n) \right| < \infty.
\end{equation}
\end{proposition}

The proof of \Cref{prop:sHL_absolute_integrability} will be given later in \Cref{sub:absolute_summability}. 
First we use its result to justify the general Cauchy identity \eqref{eq:skew_Cauchy_general}:

\begin{proof}[Proof of \Cref{thm:skew_Cauchy_general} modulo \Cref{prop:sHL_absolute_integrability}]
	By \Cref{thm:skew_Cauchy_sHL_sHL}, identity
	\eqref{eq:skew_Cauchy_general} holds in case $J_1, \dots,
	J_n, I_1, \dots, I_m$ are positive integers. 
	Functions
	$\mathfrak{F}^{(J_1, \dots, J_n)}_{\lambda / \mu},
	\mathfrak{G}^{(I_1, \dots, I_m)}_{\lambda / \mu}$ 
	are finite sums of
	finite products of weights $w^{(J_k)}_{u_k,s}, w^{*,(I_l)}_{v_l,s}$ which
	are
	rational
	functions of $q^{J_k},q^{I_l}$. 
	Therefore,
	$\mathfrak{F}^{(J_1, \dots, J_n)}_{\lambda / \mu},
	\mathfrak{G}^{(I_1, \dots, I_m)}_{\lambda / \mu}$ 
	admit an
	extension to generic complex numbers 
	$q^{J_k},q^{I_l}$.
	This
	implies that the right-hand side of \eqref{eq:skew_Cauchy_general}
	extends to $q^{J_k}, q^{I_l}$ in a complex region, too,
	since the sum over $\varkappa$ is finite (it ranges over $\varkappa
	\subseteq \mu , \lambda$). 
	
	The summation in the left-hand side of
	\eqref{eq:skew_Cauchy_general} has infinitely many terms as the only
	condition on $\nu$ is that $\mu, \lambda \subseteq \nu$. Therefore,
	to show that it can be extended to parameters $q^{J_k},q^{I_l}$ in
	a complex region we need a result of absolute convergence of the sum
	over $\nu$. Under assumptions
	\eqref{eq:condition_parameters_general_skew_Cauchy}, this is a
	consequence of Proposition \ref{prop:sHL_absolute_integrability}.
	Therefore, the left-hand side of \eqref{eq:skew_Cauchy_general} can
	be extended, too.
	
	The equality between the two sides of \eqref{eq:skew_Cauchy_general}
	in a wider region 
	\eqref{eq:condition_parameters_general_skew_Cauchy}
	follows because these expressions agree for infinitely many values of 
	$J_k,I_l$, namely, positive integers: if $|u_k|<\delta$, then $|u_k q^{J_k}|<\delta$ for 
	all $J_k\in \mathbb{Z}_{\ge1}$. This completes the proof.
\end{proof}

Despite the fact that the general skew Cauchy identity 
(\Cref{thm:skew_Cauchy_general})
offers a unified description of all skew
Cauchy structures we study, throughout the text we still consider
possible degenerations separately. 
There are several reasons for this.
First, the spin Hall-Littlewood and the spin $q$-Whittaker functions 
are more basic from an
algebraic standpoint (see, e.g., \Cref{sec:diff_op} where we 
describe difference operators diagonalized by these functions). 
Second, when $u,q^J,v,q^I$ are general
parameters, it is difficult to give a probabilistic interpretation of
the random fields --- the positivity of the measure obtained by 
multiplying $\mathfrak{F}$ and $\mathfrak{G}$ 
is in general not guaranteed. 

\subsection{Absolute summability}
\label{sub:absolute_summability}

We now turn to the proof of the absolute summability result of 
\Cref{prop:sHL_absolute_integrability}.
This proof requires explicit expressions for the fused weights
which can be found in \Cref{app:YBE}.
We begin with a number of preliminary estimates,
and assume that $q\in(0,1)$ throughout the subsection.

\begin{lemma} \label{lemma:bound_w_J_1}
Consider the fused weights $w^{(J)}_{u,s}(i_1,j_1;i_2,j_2)$ defined in \eqref{eq:w_fused_J_text}, with $\gamma=q^J\in\mathbb{C}$
and
$i_1,j_1,i_2,j_2\in \mathbb{Z}_{\ge0}$.
Let $s\neq 0$ and $\delta = \max \{ |s|, |u|, |\gamma u| \} <1$. Then 
\begin{equation} \label{eq:w_J_bound}
	\left| w^{(J)}_{u,s}(i_1,j_1;i_2,j_2) \right| \le (\min\{i_1,j_2\}+1) \, C \delta^{j_2}, 
\end{equation}
where $C$ is a positive constant independent of the vertex configuration $(i_1,j_1;i_2,j_2)$.
\end{lemma}
\begin{proof}
	Expand $w^{(J)}$ combining \eqref{eq:w_fused_J_text} and \eqref{eq:hypergeom_series} as
	\begin{equation} \label{eq:w_J_expansion}
	\begin{split}
			&\frac{(-1)^{i_1+j_2}q^{\frac{1}{2} i_1(i_1-1+2 j_1)} s^{j_2-i_1} u^{i_1} (u/s;q)_{j_1-i_2} (q;q)_{j_1} }{(q;q)_{i_1} (q;q)_{j_2} (s u;q)_{j_1+i_1}}
			\\
			&
			\hspace{10pt}\times
			\sum_{k=0}^{i_1} \frac{q^k}{(q;q)_k}\,
			(q^{-i_1};q )_k( q^{-i_2};q )_k( su \gamma;q )_k( qs/u ;q )_k
			\\&
			\hspace{80pt}\times
			(q^k s^2;q )_{i_1-k}( q^{1+j_2-i_1+k};q )_{i_1-k}( \gamma q^{1-i_2-j_2+k};q )_{i_1-k}.
	\end{split}
	\end{equation}
First, the factor
\begin{equation*}
    \frac{(-1)^{i_1+j_2} (q;q)_{j_1} (su \gamma ;q)_k (q^k s^2;q)_{i_1-k} }{(q;q)_{i_1} (q;q)_{j_2} (s u;q)_{j_1+i_1} (q;q)_k}
\end{equation*}
is bounded in absolute value by a constant independent of $i_1,j_1,i_2,j_2$. The $q$-Pochhammer symbol $(q^{1+j_2-i_1+k};q)_{i_1-k}$ vanishes unless $k \ge i_1 - j_2 $ and its contribution is bounded in absolute value by 1.
The remaining factors are
\begin{equation*}
	q^{\frac{1}{2} i_1(i_1-1+2 j_1)} s^{j_2-i_1} u^{i_1} (u/s;q)_{j_1-i_2}
	q^k
	(q^{-i_1};q )_k( q^{-i_2};q )_k( qs/u ;q )_k
	( \gamma q^{1-i_2-j_2+k};q )_{i_1-k}.
\end{equation*}
Rewrite
\begin{equation} \label{eq:w_J_expansion_1}
    q^{i_1 j_1} (q^{-i_2};q)_k (\gamma q^{1-i_2-j_2+k};q)_{i_1-k} 
		=
		\prod_{l=0}^{k-1} (q^{j_1} - q^{-i_1+j_2 +l}) \prod_{l=0}^{i_1 - k -1}(q^{j_1} - \gamma q^{-l} ),
\end{equation}
where we used the arrow conservation property, and
\begin{equation} \label{eq:w_J_expansion_2}
    q^{i_1(i_1 - 1)/2 + k} (q^{-i_1};q)_k 
		\prod_{l=0}^{i_1 - k -1}(q^{j_1} - \gamma q^{-l} ) 
		=
		(-1)^{i_1} \gamma^{i_1-k} (q^{i_1-k+1};q)_k (q^{j_1}/\gamma;q)_{i_1 -k}. 
\end{equation}
The factors 
$(-1)^{i_1} (q^{i_1-k+1};q)_k$ are bounded in the absolute value.
By substituting \eqref{eq:w_J_expansion_1}, \eqref{eq:w_J_expansion_2} into \eqref{eq:w_J_expansion},
we see that it remains to 
address the term
\begin{equation} \label{eq:w_J_expansion_3}
	s^{j_2-i_1} u^{i_1} (u/s;q)_{j_2-i_1} (qs/u;q)_{k}  \gamma^{i_1-k} (q^{j_1}/\gamma ; q)_{i_1-k} 
	\prod_{l=0}^{k-1}(q^{j_1} - q^{j_2 - i_1 +l}).
\end{equation}
We consider two cases based on the sign of $j_2-i_1$.

\medskip\noindent
\textbf{Case $j_2 \ge i_1$.} 
The factor $\prod_{l=0}^{k-1}(q^{j_1} - q^{j_2 - i_1 +l})$ in \eqref{eq:w_J_expansion_3}
is bounded by a constant independent of the vertex configuration. The remaining terms are
\begin{equation*}
	s^{j_2-i_1}(u/s;q)_{j_2-i_1}\cdot
	u^{i_1} (qs/u;q)_{k}\cdot \gamma^{i_1-k} (q^{j_1}/\gamma ; q)_{i_1-k}.
\end{equation*}
Distributing the factors $s$, $u$, and $\gamma$ into the $q$-Pochhammer
symbols, we can bound the above expression by $\mathrm{const}\cdot \delta^{j_2}$,
where $j_2$ is the total number of terms in the product.
Note that the estimate works uniformly for small $s,u,\gamma$, too.

\medskip\noindent
\textbf{Case $i_1>j_2$.}
Rewriting the $q$-Pochhammer symbol with the negative 
index (cf. \eqref{eq:q_Pochhammer}) and using the fact that $i_1-j_2 \le k \le i_1$, 
we have
\begin{equation*}
	\eqref{eq:w_J_expansion_3}
	=
	(-1)^{i_1-j_2} u^{j_2} (q^{i_1 - j_2 + 1}s/u ;q)_{k-i_1+j_2} \gamma^{i_1 - k}(q^{j_1}/\gamma;q)_{i_1-k} \, 
	q^{\textstyle\binom{i_1 - j_2 +1}{2}} 
	\prod_{l=0}^{k-1}(q^{j_1} - q^{j_2-i_1+l}).
\end{equation*}
The term $(-1)^{i_1-j_2} q^{\binom{i_1 - j_2 +1}{2}} \prod_{l=0}^{k-1}(q^{j_1} - q^{j_2-i_1+l})$ 
is bounded.
The contribution of the term
\begin{equation*}
    u^{j_2} (q^{i_1 - j_2 + 1}s/u ;q)_{k-i_1+j_2} \gamma^{i_1 - k}(q^{j_1}/\gamma;q)_{i_1-k}
\end{equation*}
is bounded by $\mathrm{const} \cdot \delta^{j_2}$
similarly to the previous case.

\medskip 

We see that \eqref{eq:w_J_expansion} can be written as a sum of terms bounded by $\mathrm{const} \cdot \delta^{j_2}$.
Because the number of terms is 
\begin{equation*}
	\le \min\{i_1,i_2\}+1-\max\{0,i_1-j_2\}\le \min\{i_1,j_2\}+1,
\end{equation*}
we get the desired bound.
\end{proof}

\begin{lemma} \label{lemma:bound_w_J_2}
Let 
\begin{equation}
 \sup_{n\in \mathbb{Z}_{\ge 0}} \left| \frac{q^n u - s}{1 - q^n s u} \right| < \delta < 1.
\end{equation}
Then 
\begin{equation}
    |w^{(J)}_{u,s}(0,j_1;i_2,j_2)| \le C(i_2) \, \delta^{j_2},
\end{equation}
where $C(0)=1$ and $C=\sup_k\{C(k)\}<\infty$.
\end{lemma}
\begin{proof}
    This bound follows from \eqref{eq:w_fused_J_text}: 
		after setting $i_1=0$, the $q$-hypergeometric 
		series disappears, and we 
		use the definition of $\delta$ to estimate the prefactor.
\end{proof}

For a Young diagram $\lambda$, let 
$m_i(\lambda)$ be the number of parts in $\lambda$ which are equal to $i$.

\begin{lemma} \label{lemma:bound_sHL}
Let $s \neq 0, u,q^J$ be complex parameters such that
\begin{equation}
	|s|,|u|,|q^{J} u|,\left| \frac{q^i u - s}{1 - q^i s u} \right|< \delta, \qquad \text{for all }i,
\end{equation}
for some $\delta \in (0,1)$. Then there exists $C\ge 1$ such that
for any two Young diagrams $\mu,\lambda$ we have
\begin{equation}
    \left| \mathfrak{F}^{(J)}_{\lambda / \mu}(u) \right| \leq C^{M(\lambda,\mu)} \prod_{i \ge 1} (m_i(\mu) + 1) \delta^{|\lambda| - |\mu|},
\end{equation}
where
\begin{equation}
	M(\lambda,\mu) = 1 + \#\{ i\colon m_i(\mu) \neq 0 \text{ or } m_i(\lambda) \neq 0 \}.
\end{equation}
\end{lemma}
\begin{proof}
	It suffices to assume that $\mu\subseteq \lambda$ (i.e., $\mu_i\le \lambda_i$ for all $i$),
	otherwise the skew function vanishes.
	We have
	\begin{equation}
			\mathfrak{F}^{(J)}_{\lambda / \mu}(u) = \sum_{j_0,j_1, \dots \ge 0} 
			w^{(J)}_{u,s} \biggl(\begin{tikzpicture}[baseline=-2.5pt]
		\draw[fill] (0,0) circle [radius=0.025];
		\node at (0,.3) {$\infty$};
		\node at (0,-.3) {$\infty$};
		\draw [red] (0.1,0) --++ (0.4, 0) node[right, black] {$j_0$};
		\draw [red] (0.1,0.05) --++ (0.4, 0);
		\draw [red] (0.1,-0.05) --++ (0.4, 0);
			\addvmargin{1mm}
			\addhmargin{1mm}
	\end{tikzpicture} \biggr) \prod_{l \ge 1} w^{(J)}_{u,s} (m_l(\mu),j_{l-1} ; m_l(\lambda), j_l),
	\end{equation}
	where
	the infinite sum has only one nonzero term due to arrow preservation.
	From \eqref{eq:w_fused_infinity} we have for the leftmost vertex
	\begin{equation}
			\left| w^{(J)}_{u,s} \biggl(\begin{tikzpicture}[baseline=-2.5pt]
		\draw[fill] (0,0) circle [radius=0.025];
		\node at (0,.3) {$\infty$};
		\node at (0,-.3) {$\infty$};
		\draw [red] (0.1,0) --++ (0.4, 0) node[right, black] {$j$};
		\draw [red] (0.1,0.05) --++ (0.4, 0);
		\draw [red] (0.1,-0.05) --++ (0.4, 0);
			\addvmargin{1mm}
			\addhmargin{1mm}
	\end{tikzpicture} \biggr) \right| \le C \delta^j.
	\end{equation}
	\Cref{lemma:bound_w_J_1,lemma:bound_w_J_2} 
	provide estimates for the remaining 
	vertex weights: they are all bounded by 
	$C \delta^{j_l}$, except if both $m_l(\mu)=m_l(\lambda)=0$ when the bound is given by $\delta^{j_l}$. 
\end{proof}

\begin{lemma} \label{lemma:sum_for_bound}
	Let $0<\delta<1$,
	$C\ge 1$, $1\le A< \delta^{-1}$, and $M(\lambda, \mu)$ be as in \Cref{lemma:bound_sHL}. Then for any 
Young diagram $\mu$ we have
\begin{equation} \label{eq:sum_for_bound}
    \sum_{\nu\colon \mu \subseteq \nu} C^{M(\nu,\mu)} 
		\delta^{|\nu| - |\mu|} A^{\ell(\nu)}
		\prod_{i\ge 1} (m_i(\nu) + 1 )  \le C_1 A_1^{\ell(\mu)},
\end{equation}
where $C_1,A_1 \ge 1$ are constants.
\end{lemma}
\begin{proof}
    The sum over $\nu$ can be visualized as a sum over path configurations in a row of vertices 
		as in \Cref{fig:black_blue_paths}.\begin{figure}
        \centering
        \includegraphics{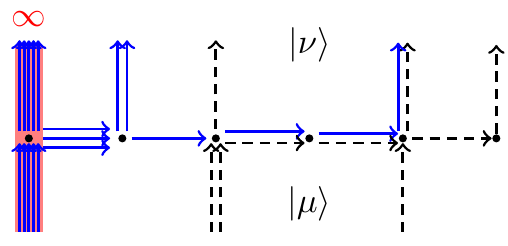}
				\caption{The decomposition of the Young diagram $\nu$ as a free superposition of $\eta$ (black dashed paths) and $\varkappa$ (blue solid paths) used in the proof of \Cref{lemma:sum_for_bound}.}
        \label{fig:black_blue_paths}
    \end{figure}
		We distinguish the paths coming from the
		configuration $\mu$ (black dashed in \Cref{fig:black_blue_paths})
		and when they originate at the leftmost vertex (blue solid in 
		\Cref{fig:black_blue_paths}). Calling $\varkappa$ the Young diagram
		generated by the blue paths and $\eta$ the Young diagram generated by the
		black paths we can write $\nu = \varkappa \cup \eta$ (this
		decomposition is not unique). The sum in the left-hand side of
		\eqref{eq:sum_for_bound} is dominated by a sum where $\varkappa$
		and $\eta$ vary independently, and therefore we have
    \begin{equation} \label{eq:sum_for_bound_1}
    \begin{split}
			\mathrm{lhs}\,\eqref{eq:sum_for_bound} \le & \biggl(
				\sum_{\varkappa} C^{M(\varkappa,\varnothing)}\delta^{|\varkappa|} 
				A^{\ell(\varkappa)}
				\prod_{i\ge 1}
			(m_i(\varkappa) + 1 )
		\biggr) 
        \\
        & \qquad \times
				\biggl( \sum_{\substack{\eta\colon \mu \subseteq \eta \\ \ell(\eta)
				= \ell(\mu) } } C^{M(\eta,\mu)}  
				\delta^{|\eta| - |\mu|}
				A^{\ell(\mu)}
				\prod_{i\ge 1} (m_i(\eta) + 1) 
			 \biggr).
    \end{split}
    \end{equation}
		We estimate separately the two factors in
		\eqref{eq:sum_for_bound_1}, starting with the first one. 
		Since
		$\ell(\varkappa)=\sum_i m_i(\varkappa)$ and
		$|\varkappa|=\sum_{i\ge 1} i\, m_i(\varkappa)$, 
		the summation over
		$\varkappa$ can be rewritten as follows.
		Separate the term $\varkappa=\varnothing$.
		In the remaining sum, first select $\varkappa_1\ge1$ and its multiplicity $r\ge1$;
		then for each $i=1,\ldots,\varkappa_1-1 $, select a multiplicity $m_i\ge0$.
		Summing over all these possibilities, we have 
    \begin{equation*}
        C + 
				\sum_{\varkappa_1\ge 1 } C 
				\sum_{r\ge1}\delta^{r \varkappa_1}C(r+1)A^{r}
				\prod_{i=1}^{\varkappa_1 - 1} 
				\left( \sum_{m_i \ge 0} C^{\mathbf{1}_{m_i>0}} (m_i + 1) \delta^{im_i}A^{m_i}  \right).
    \end{equation*}
		Simplifying the geometric summations and using the fact that $A\delta<1$, we 
		can reduce the above sum to
		\begin{equation*}
			C+C \sum_{\varkappa_1\ge1}\frac{AC \delta^{\varkappa_1}(2-A \delta^{\varkappa_1})}
			{(1-A \delta^{\varkappa_1})^2}
			\prod_{i=1}^{\varkappa_1-1}\left( 1-C\left( 1-\frac{1}{(1-A \delta^i)^2} \right) \right)
		\end{equation*}
		For all $i\ge i_0$, where $i_0$ depends on $C,A,\delta$
		the $i$-th term in the product is less than $\delta^{-1}$ (because $\delta<1$ and the term
		goes to $1$ as $i\to\infty$).
		This implies that the above sum is convergent and thus is estimated from above by a constant.

		We now address the second factor in the right-hand side of
		\eqref{eq:sum_for_bound_1}. We can again bound the sum
		over $\eta$ by a superposition of $\ell(\mu)$
		noninteracting paths starting at $\mu_i$.
		This
		implies that the sum over $\eta$ in \eqref{eq:sum_for_bound_1} is
		dominated by 
    \begin{equation*}
        \prod_{i =1}^{\ell (\mu)} \sum_{k_i \ge \mu_i} C^{M(k_i, \mu_i)} 
				2 A \delta^{k_i - \mu_i} = \left[ 2 C^2 A \left( 1 + C \frac{\delta}{1-\delta} \right) \right]^{\ell (\mu) }.
    \end{equation*}
    This completes the proof.
\end{proof}
\begin{proof}[Proof of \Cref{prop:sHL_absolute_integrability}]
		We first expand $\mathfrak{F}^{(J_1,\dots
		J_n)}_{\lambda / \mu}(u_1,\dots, u_n)$ in
		\eqref{eq:sHL_absolute_integrability}. Utilizing the branching
		rule and the triangle inequality, we can estimate
    \begin{equation} \label{eq:sHL_abs_int_tr_ineq}
			\mathrm{lhs}\,\eqref{eq:sHL_absolute_integrability} 
			\le
        \sum_{\substack{\lambda^1,\dots\lambda^n\colon\\
        \mu \subseteq \lambda^1 \subseteq\cdots \subseteq \lambda^n}} \prod_{i=1}^n \left| \mathfrak{F}^{(J_i)}_{\lambda^i / \lambda^{i-1}} (u_i) \right|.
    \end{equation}
    In order to evaluate the previous nested summation we start from the most external term. 
		For fixed $\lambda^{n-1}$ we have
    \begin{equation*}
        \sum_{\lambda^n\colon\lambda^{n-1} \subseteq \lambda^n} \left| \mathfrak{F}^{(J_n)}_{\lambda^n / \lambda^{n-1}} (u_n) \right| \le \prod_{i \ge 1} (m_i(\lambda^{n-1}) + 1) \sum_{\lambda^n\colon\lambda^{n-1} \subseteq \lambda^n} C^{M(\lambda^n,\lambda^{n-1})} \delta_1^{|\lambda^n| - |\lambda^{n-1}|},
    \end{equation*}
		where we used bound of Lemma \ref{lemma:bound_sHL}
		for some $\delta_1\in(0,1)$.
		We can further
		estimate the sum over $\lambda^n$ 
		with the help of \Cref{lemma:sum_for_bound}, and obtain the bound
		$\prod_{i \ge 1} (m_i(\lambda^{n-1}) + 1) C_1 A_1^{\ell
		(\lambda^{n-1})}$.
		Replacing $\delta_1$ by a smaller value $0<\delta_2<A_1^{-1}$ if needed
		and 
		multiplying this bound by the next term $\mathfrak{F}^{(J_{n-1})}_{\lambda^{n-1}/\lambda^{n-2}}(u_{n-1})$
		in \eqref{eq:sHL_abs_int_tr_ineq},
		we can apply \Cref{lemma:bound_sHL} and then sum over $\lambda^{n-1}$
		with the help of \Cref{lemma:sum_for_bound}.
		Iterating this procedure finitely many times, 
		we get the desired statement with a sufficiently small $\delta_n>0$.
\end{proof}

\section{Scaled geometric specializations}
\label{sec:scaled_geometric}

In this section we introduce a third specialization --- the scaled geometric one ---
of the general fused functions from the previous section. 
This specialization
allows to include into our analysis
stochastic particle systems with more general initial
configurations.

\subsection{Definition of scaled geometric specializations}
\label{sub:scaled_geometric_spec}

In Definitions \ref{def:sHL_sHL_structure}, \ref{def:sHL_sqW_struture},
\ref{def:sqW_sqW_structure} we provided examples of skew Cauchy structures
where the positivity of the measure 
obtained by 
multiplying $\mathfrak{F}$ and $\mathfrak{G}$ 
can be established (under certain restrictions on parameters). 
We now introduce yet another specialization of
\eqref{eq:F_principal_spec}, \eqref{eq:G_principal_spec} which admits a
meaningful probabilistic interpretation --- it corresponds to 
two-sided stationary
initial conditions for stochastic particle systems on the line
arising as marginals of our Yang-Baxter fields.

\begin{definition}[\cite{BorodinPetrov2016inhom}]
The \emph{scaled geometric} specialization with parameter $\alpha$ of the spin Hall-Littlewood function is given by 
\begin{equation}
    \widetilde{\mathsf{F}}_{\lambda / \mu}(\alpha) : = \lim_{\epsilon \to 0} \mathsf{F}_{\lambda / \mu}(-\alpha \epsilon, -\alpha \epsilon q, \dots, -\alpha \epsilon q^{J-1} ) \Big |_{q^J=1/\epsilon}.
\end{equation}
The dual analog of $\widetilde{\mathsf{F}}$ is given by the conjugation $\widetilde{\mathsf{F}}^*_{\lambda / \mu}(\beta) = \frac{\mathsf{c}(\lambda)}{\mathsf{c}(\mu)}\, \widetilde{\mathsf{F}}_{\lambda/\mu}(\beta)$ as in \eqref{eq:sHL_sHL_star_relation}.
The skew functions in multiple variables $\widetilde{\mathsf{F}}_{\lambda/\mu}(\alpha_1,\ldots,\alpha_N )$
are defined in a standard way using the branching rules as in \eqref{eq:F_G_branching}, and similarly for $\widetilde{\mathsf{F}}^*_{\lambda/\mu}$. 
\end{definition}

The functions $\widetilde{\mathsf{F}}_{\lambda/\mu}, \widetilde{\mathsf{F}}^*_{\lambda/\mu}$ 
also admit a lattice construction with the vertex weights
\begin{equation*}
    \widetilde{w}_{\alpha,s} : = \lim_{\epsilon \to 0} w^{(J)}_{-\alpha \epsilon,s} \Big |_{q^J = 1/\epsilon}, 
		\qquad 
		\widetilde{w}^*_{\beta,s} : = \lim_{\epsilon \to 0} w^{*,(I)}_{-\beta \epsilon,s} \Big |_{q^I = 1/\epsilon}.
\end{equation*}
The expressions for these weights are given in \Cref{app:YBE_scaled_geometric}.
The functions $\widetilde{\mathsf{F}}_{\lambda/\mu}, \widetilde{\mathsf{F}}^*_{\lambda/\mu}$ vanish
unless $\mu\subseteq\lambda$ (which means that $\mu_i\le \lambda_i$ for all $i$).

By adding the scaled geometric specialization
to our symmetric functions, we can define \emph{mixed specializations}
$\mathsf{F}_{\lambda/\mu}(u_1,\ldots,u_n;\widetilde \alpha_{1},\ldots,\widetilde{\alpha}_N)$
and 
$\mathbb{F}_{\lambda'/\mu'}(\xi_1,\ldots,\xi_n; \widetilde\alpha_1,\ldots,\widetilde\alpha_N )$. 
They are obtained through the branching rules as
\begin{align*}
	\mathsf{F}_{\lambda/\mu}(u_1,\ldots,u_n;\widetilde \alpha_{1},\ldots,\widetilde{\alpha}_N)
	&=
	\sum_{\varkappa}
	\widetilde{\mathsf{F}}_{\lambda/\varkappa}(\alpha_1,\ldots,\alpha_N )
	\,\mathsf{F}_{\varkappa/\mu}(u_1,\ldots,u_n ),
	\\
	\mathbb{F}_{\lambda'/\mu'}(\xi_1,\ldots,\xi_n; \widetilde\alpha_1,\ldots,\widetilde\alpha_N )
	&=
	\sum_{\varkappa}
	\widetilde{\mathsf{F}}_{\lambda/\varkappa}(\alpha_1,\ldots, \alpha_N)
	\,\mathbb{F}_{\varkappa'/\mu'}(\xi_1,\ldots,\xi_n ),
\end{align*}
and similarly for the dual functions.
By the Yang-Baxter equations (\Cref{app:YBE_scaled_geometric}),
each of these functions is separately 
symmetric 
in 
the two sets of variables.
We will also sometimes use the notation $\mathrm{sHL}(u), \mathrm{sqW}(\xi)$, and $\mathrm{sg}(\alpha)$
to denote the three types of specializations of the 
general symmetric functions \eqref{eq:F_principal_spec}--\eqref{eq:G_principal_spec}.

\subsection{Skew Cauchy structures with mixed specializations} \label{sub:mixed_skew_Cauchy_structures}

The scaled geometric specializations allow to generalize the skew Cauchy identities
considered in \Cref{sec:summary_sHL_sqW}. 
Let us first define the corresponding 
parameter sets $\mathsf{Adm}$ for which the sums in the Cauchy identities
converge.
\begin{definition}[Admissible parameters]
	\label{def:adm_rho}
	Let $\rho$
	be one of the specializations 
	$\mathrm{sHL}(u),\mathrm{sqW}(\xi),\mathrm{sg}(\alpha)$
	and 
	$\rho^*$ be one of 
	$\mathrm{sHL}(v),\mathrm{sqW}(\theta),\mathrm{sg}(\beta)$.
	We define $\mathsf{Adm}(\rho,\rho^*)$ to be symmetric in $\rho\leftrightarrow\rho^*$
	(with the corresponding renaming of the parameters),
	and: 
	\begin{enumerate}[label=\bf{\arabic*.}]
		\item If neither of $\rho$ and $\rho^*$ is scaled geometric, then
			$\mathsf{Adm}(\rho,\rho^*)$ is given in 
			\Cref{def:sHL_sHL_structure,def:sHL_sqW_struture,def:sqW_sqW_structure}
			in the sHL/sHL, sHL/sqW, and sqW/sqW cases, respectively.
		\item In the remaining cases we have
			\begin{equation}
					\mathsf{Adm}(\mathrm{sg}(\alpha);\rho^*) = \begin{cases}
						\{ (\alpha,v)\in \mathbb{C}^2 : |s(s-v)|<|1-sv| \},
						\qquad & \text{if } \rho^*= \mathrm{sHL}(v);\\
						\{ (\alpha, \theta)\in \mathbb{C}^2 : |\alpha \theta|<1 \}, 
						\qquad & \text{if } \rho^*= \mathrm{sqW}(\theta);\\
						\{ (\alpha, \beta)\in \mathbb{C}^2 : |\alpha \beta|<1 \}, 
						\qquad & \text{if } \rho^*=\mathrm{sg}(\beta).
					\end{cases}
			\end{equation}
	\end{enumerate}
\end{definition}

We call a specialization $\rho$ \emph{compatible} with 
sHL functions if $\rho=\mathrm{sHL}(u)$ or $\mathrm{sg}(\alpha)$,
and similarly $\rho$ is compatible with sqW functions if
$\rho=\mathrm{sqW}(\xi)$ or $\mathrm{sg}(\alpha)$.

\begin{theorem} \label{thm:skew_Cauchy_mixed_spec}
		Let the $\mathfrak{F}_{\lambda / \mu }$ be either
		$\mathsf{F}_{\lambda / \mu}$ or $\mathbb{F}_{\lambda ' / \mu '}$ and let
		$\rho$ be a specialization compatible
		with $\mathfrak{F}$. Analogously, 
		let the function $\mathfrak{G}_{\lambda / \mu}$ be either $\mathsf{F}^*_{\lambda / \mu}$ or
		$\mathbb{F}^*_{\lambda ' / \mu '}$, and let $\rho^*$ be 
		compatible with $\mathfrak{G}$. 
		Then for the parameters belonging to $\mathsf{Adm}(\rho,\rho^*)$ we have
    \begin{equation} \label{eq:skew_Cauchy_mixed}
        \sum_{\nu} 
				\mathfrak{F}_{\nu / \mu}(\rho)
				\mathfrak{G}_{\nu / \lambda}(\rho^*) 
				=
				\Pi(\rho ; \rho^*) 
				\sum_{\varkappa} 
				\mathfrak{F}_{\lambda / \varkappa}(\rho) 
				\mathfrak{G}_{\mu / \varkappa} (\rho^*).
    \end{equation}
    The right-hand side $\Pi(\rho; \rho^*)=\Pi(\rho^*;\rho)$ 
		in the sHL/sHL, sHL/sqW, and sqW/sqW cases was described above in
		\Cref{def:sHL_sHL_structure,def:sHL_sqW_struture,def:sqW_sqW_structure},
		respectively,
		and in the remaining cases it is given by 
		(observe that \eqref{eq:skew_Cauchy_mixed}
		does not change if we switch $\rho\leftrightarrow\rho^*$):
    \begin{equation}
			\label{eq:sg_Pi}
        \Pi(\mathrm{sg}(\alpha);\rho^*) = 
				\begin{cases}
					1 + \alpha v, \qquad & \text{if } \rho^*= \mathrm{sHL}(v);\\
					(-s \alpha;q)_{\infty} / ( \alpha \theta ;q )_\infty, \qquad & \text{if } \rho^*= \mathrm{sqW}(\theta);\\
					1/(\alpha \beta ; q)_{\infty}, \qquad & \text{if } \rho^*=\mathrm{sg}(\beta).
        \end{cases}
    \end{equation}
\end{theorem}
\begin{proof}
	The skew Cauchy identity
	\eqref{eq:skew_Cauchy_mixed}
	is obtained by suitably specializing \eqref{eq:skew_Cauchy_general}.
	The convergence conditions 
	$\mathsf{Adm}(\rho,\rho^*)$
	for the infinite sum in the left-hand side of \eqref{eq:skew_Cauchy_general}
	(the right-hand side is always finite)
	can be found in the existing literature 
	\cite{Borodin2014vertex}, 
	\cite{BorodinPetrov2016inhom},
	\cite{BorodinWheelerSpinq}.
	Through the bijectivization which we discuss in \Cref{sec:YB_fields_through_bijectivisation}
	below, the convergence of the left-hand side of \eqref{eq:skew_Cauchy_general}
	is equivalent to the fact that
	the transition probabilities
	$\mathsf{U}^{\mathrm{fwd}}$
	do not assign any probability weight 
	to Young diagrams $\nu$ with infinite first row $\nu_1$ or infinite first column $\nu_1'$.
	In \Cref{prop:U_well_defined} we revisit 
	the origin of the conditions $\mathsf{Adm}(\rho,\rho^*)$
	from this perspective.
\end{proof}

This theorem leads to the following additional skew Cauchy structures
which we now describe in a unified way:

\begin{definition}
\label{def:mixed_skew_Cauchy_structure}
	Let $\mathfrak{F}_{\lambda/\mu}$ be either $\mathsf{F}_{\lambda/\mu}$ or $\mathbb{F}_{\lambda'/\mu'}$, 
	and specializations $\rho_1,\rho_2,\ldots $ be compatible with $\mathfrak{F}$. 
	Also let
	$\mathfrak{G}_{\lambda/\mu}$ be either $\mathsf{F}^*_{\lambda/\mu}$ or $\mathbb{F}^*_{\lambda'/\mu'}$,
	and $\rho_1^*,\rho_2^*,\ldots $ be compatible with $\mathfrak{G}$. 
	Then
	$\mathfrak{F}_{\lambda / \mu}(\rho_1, \dots, \rho_k)$, $\mathfrak{G}_{\lambda
	/ \mu}(\rho_1^*, \dots, \rho_k^*)$ form a skew Cauchy structure in the sense
	of Section \ref{sub:F_G_skew_Cauchy_structure} with the following
	identifications:
	\begin{enumerate}[label=\bf{(\roman*)}]
		\item For any specialization $\rho$ set
			\begin{equation}
					\mu \prec_\rho \mu = 
					\begin{cases}
					\mu \prec \lambda \qquad & \text{if } \rho = \mathrm{sHL},\\
					\mu' \prec \lambda' \qquad & \text{if } \rho = \mathrm{sqW},\\
					\mu \subseteq \lambda \qquad & \text{if } \rho = \mathrm{sg}.\\
					\end{cases}
			\end{equation}
			Then, $\mu \prec_k \lambda$ means the existence of a sequence of
			Young diagrams 
			\begin{equation*}
				\mu \prec_{\rho_1} \nu^{(1)} \prec_{\rho_2} \cdots
				\prec_{\rho_{k-1}} \nu^{(k-1)} \prec_{\rho_{k}} \lambda,
			\end{equation*}
			and 
			$\mu \mathop{\dot{\prec}_k} \lambda$ is defined in the same way 
			with replacing each $\rho_i$ by $\rho_i^*$.
			
	\item The skew Cauchy identity \eqref{eq:skew_Cauchy_mixed} holds for 
		each choice of specializations, with the convergence conditions $\mathsf{Adm}(\rho;\rho^*)$ 
		and the function 
		$\Pi(\rho;\rho^*)$ described above in this subsection.

			\item The external parameters are $q \in (0,1)$ and $s \in (-1,0)$. The
				nonnegativity sets are $\mathsf{P}_{\mathrm{sHL}} =[0,1]$,
				$\mathsf{P}_{\mathrm{sqW}} =[-s,-s^{-1}]$, $\mathsf{P}_{\mathrm{sg}}=[0,-s^{-1}]$, respectively,
				which follows from the nonnegativity of the corresponding vertex weights
				(about the scaled geometric weights, see \Cref{prop:w_tilde_positivity}).
	\end{enumerate}
\end{definition}

We employ this general \emph{mixed skew Cauchy structure} in \Cref{sec:new_three_fields}
below to connect symmetric functions with stochastic particle systems (more precisely, 
stochastic vertex models) having a variety of 
initial conditions.

\section{Yang-Baxter fields through bijectivization}
\label{sec:YB_fields_through_bijectivisation}

In this section we recall the notion of \emph{bijectivization} of summation
identities \cite{BufetovPetrovYB2017}
and show
how to use this procedure to build a random field of 
Young diagrams.\footnote{As far as we know, dynamics coming from
certain straightforward bijectivizations 
of the Yang-Baxter equation were 
used by 
\cite{Manolescu_Grimmett_bond}, 
\cite{Sportiello-private}
for simulations, but without connections to Cauchy identities.}
Our main ingredient is the 
Yang-Baxter equation in its general form with four parameters
$u,v,q^J,q^I$.
In
\Cref{sec:new_three_fields} below we
examine the most interesting degenerations corresponding to particular 
skew Cauchy structures from \Cref{sec:summary_sHL_sqW}.

\subsection{Bijectivization of summation identities}
\label{sub:bij_summation_citation_from_BP2017}
Consider two nonempty finite or countable sets $A$ and $B$,
and assume that to each one of their
elements it is associated a nonzero
weight\footnote{If for some $a_0\in A$ we 
have $\mathbf{w}(a_0)=0$, by replacing $A$ with $A\setminus\left\{ a_0 \right\}$ we can continue to 
assume that all weights are nonzero, and analogously for $B$.} 
$\mathbf{w}$ in such a way that
\begin{equation} \label{sum identity}
    \sum_{a \in A}\mathbf{w}(a) = \sum_{b \in B} \mathbf{w}(b).
\end{equation}
Here and below in the countable case we assume that all infinite sums converge absolutely.
\begin{definition}[\cite{BufetovPetrovYB2017}]
	\label{def:bijectivization}
	A \emph{bijectivization} of the summation identity \eqref{sum identity} is a pair of
	families of transition weights
	$(\mathbf{p}^{\mathrm{fwd}},\mathbf{p}^{\mathrm{bwd}})$ satisfying the
	properties:
	\begin{enumerate}[label=\bf{\arabic*.}]
			\item The forward transition weights sum to one:
			\begin{equation}\label{p fwd sum-to1}
					\sum_{b \in B} \mathbf{p}^{\mathrm{fwd}}(a,b)=1 \qquad \text{for all } a \in A
			\end{equation}
			\item The backward transition weights sum to one:
			\begin{equation}\label{p bwd sum-to1}
					\sum_{a \in A} \mathbf{p}^{\mathrm{bwd}}(b,a)=1 \qquad \text{for all } b \in B
			\end{equation}
			\item The transition weights satisfy the reversibility condition
			\begin{equation} \label{rev cond}
				\mathbf{w}(a)\, \mathbf{p}^{\mathrm{fwd}}(a,b) = 
				\mathbf{w}(b)\, \mathbf{p}^{\mathrm{bwd}}(b,a) \qquad \text{for all } a\in A,\, b\in B.
			\end{equation}
	\end{enumerate}

	If 
	$\mathbf{w}(a), \mathbf{w}(b)>0$ for all $a\in A$, $b\in B$,
	and the transition weights
	$\mathbf{p^{\mathrm{fwd}}}, \mathbf{p^{\mathrm{bwd}}}$ are nonnegative,
	the bijectivization is called \emph{stochastic}.
\end{definition}

On one hand, 
bijectivizations may be viewed as \emph{refinements} 
of the summation identity \eqref{sum identity}
since \eqref{p fwd sum-to1}--\eqref{rev cond} 
imply 
\begin{equation*}
    \sum_{a \in A}\mathbf{w}(a) 
		=
		\sum_{a \in A, b \in B} \mathbf{w}(a) \mathbf{p}^{\mathrm{fwd}}(a,b)
		= 
		\sum_{a \in A,b \in B} \mathbf{w}(b) \mathbf{p}^{\mathrm{bwd}} (b,a) 
		=
		\sum_{b \in B} \mathbf{w}(b).
\end{equation*}
On the other hand, stochastic bijectivizations exactly
correspond to \emph{couplings} between the
probability distributions
$P_A(a)=\mathbf{w}(a)\left( \sum_{a'\in A}\mathbf{w}(a') \right)^{-1}$
and 
$P_B(b)=\mathbf{w}(b)\left( \sum_{b'\in B}\mathbf{w}(b') \right)^{-1}$.
Recall that a coupling is a probability distribution 
$P(a,b)$ on $A\times B$ whose marginals on $A$ and $B$ are $P_A(a)$ and $P_B(b)$, respectively.
The correspondence is given by 
\begin{equation*}
	P(a,b)=
	\frac{\mathbf{w}(a)\,\mathbf{p}^{\mathrm{fwd}}(a,b)}{\sum_{a'\in A}\mathbf{w}(a')}
	=
	\frac{\mathbf{w}(b)\,\mathbf{p}^{\mathrm{bwd}}(b,a)}{\sum_{b'\in B}\mathbf{w}(b')},
\end{equation*}
where the second equality follows from 
\eqref{rev cond} and \eqref{sum identity}.

\begin{remark}
	Forward and backward transition probabilities of a random field
	of Young diagrams
	(\Cref{def:F_G_fwd_bwd_transition_probabilities}) are a
	particular case of the above \Cref{def:bijectivization}
	as they correspond to bijectivizations of the 
	identity
	\eqref{eq:F_G_single_skew_Cauchy}.
	For the skew Cauchy structures described in
	\Cref{sec:summary_sHL_sqW}, however, 
	Cauchy identities follow from the 
	more elementary Yang-Baxter equations,
	and we use the latter to construct bijectivizations
	as building blocks for
	transition probabilities of random fields of Young diagrams.
\end{remark}

When both $|A|>1$ and $|B|>1$, one can readily see that a bijectivization is not unique.

\begin{example} \label{example:biject:A1}
	Assume that the set $A$, in \eqref{sum identity}, consists of the singleton $\{a\}$.
	Then the bijectivization of the identity
	\begin{equation*}
			\mathbf{w}(a) = \sum_{b \in B} \mathbf{w}(b), 
	\end{equation*}
	is unique and is given by
	\begin{equation} \label{biject A=1}
			\mathbf{p}^{\mathrm{fwd}}(a,b)=\frac{\mathbf{w}(b)}{\mathbf{w}(a)}, \qquad \mathbf{p}^{\mathrm{bwd}}(b,a) = 1.
	\end{equation}
	Moreover, in case all weights are positive, \eqref{biject A=1} constitutes a stochastic bijectivization.
\end{example}
\Cref{example:biject:A1}, despite its simplicity, constitutes the only
explicit stochastic bijectivization we will make use of throughout the rest
of the paper. In any other case we only need the existence of a stochastic bijectivization:

\begin{proposition}
	\label{prop:Bij_exists}
	Assume that in
	\eqref{sum identity} we have
	$\mathbf{w}(a),\mathbf{w}(b)\ge 0$ for all $a\in A$ and $b\in B$, and, moreover, 
	the sums in both sides of \eqref{sum identity} are positive.
	Then a stochastic bijectivization $(\mathbf{p}^{\mathrm{fwd}},\mathbf{p}^{\mathrm{bwd}})$
	exists.
\end{proposition}
Recall that if some weights are zero, we exclude the corresponding
elements from $A$ and~$B$.
\begin{proof}[Proof of \Cref{prop:Bij_exists}]
	As an example of a stochastic bijectivization we can take the one corresponding to the 
	coupling which is the product measure, $P=P_A\otimes P_B$. In other words,
	we can take $\mathbf{p}^{\mathrm{fwd}}(a,b)$ to be independent of $a$,
	and similarly for $\mathbf{p}^{\mathrm{bwd}}(b,a)$. Then \eqref{rev cond} implies
	\begin{equation*}
		\mathbf{p}^{\mathrm{fwd}}(a,b)=\frac{\mathbf{w}(b)}{\sum_{b'\in B}\mathbf{w}(b')},
		\qquad 
		\mathbf{p}^{\mathrm{bwd}}(b,a)=\frac{\mathbf{w}(a)}{\sum_{a'\in A}\mathbf{w}(a')},
	\end{equation*}
	and so a stochastic bijectivization exists.
\end{proof}

\subsection{Bijectivization of the Yang-Baxter equation}

Let us now consider bijectivizations of the general fused Yang-Baxter
equation (reproduced from \Cref{prop:YBE_general_fused} in \Cref{app:YBE})
\begin{equation}
	\label{eq:fused_YBE_text}
	\begin{split}
		&\sum_{k_1,k_2,k_3}R_{uv}^{(I,J)}(i_2, i_1; k_2, k_1) \,
		w^{*,(I)}_{v,s} (i_3, k_1; k_3, j_1) \,
		w^{(J)}_{u,s}(k_3,k_2; j_3,j_2) \\
		&\hspace{50pt} = 
		\sum_{k_1,k_2,k_3} w^{*,(I)}_{v,s} (k_3, i_1; j_3, k_1) \,
		w^{(J)}_{u,s}(i_3,i_2; k_3,k_2)\, R_{uv}^{(I,J)}(k_2, k_1; j_2, j_1),
	\end{split}
\end{equation}
where the weights
$w^{(J)}, w^{*,(I)}$ and $R^{(I,J)}$ are defined in \eqref{eq:w_fused_J}, \eqref{eq:w_fused_I_dual}, \eqref{eq:R_fused_I_J}, respectively.

\begin{figure}[ht]
    \centering
		\includegraphics[width=.97\textwidth]{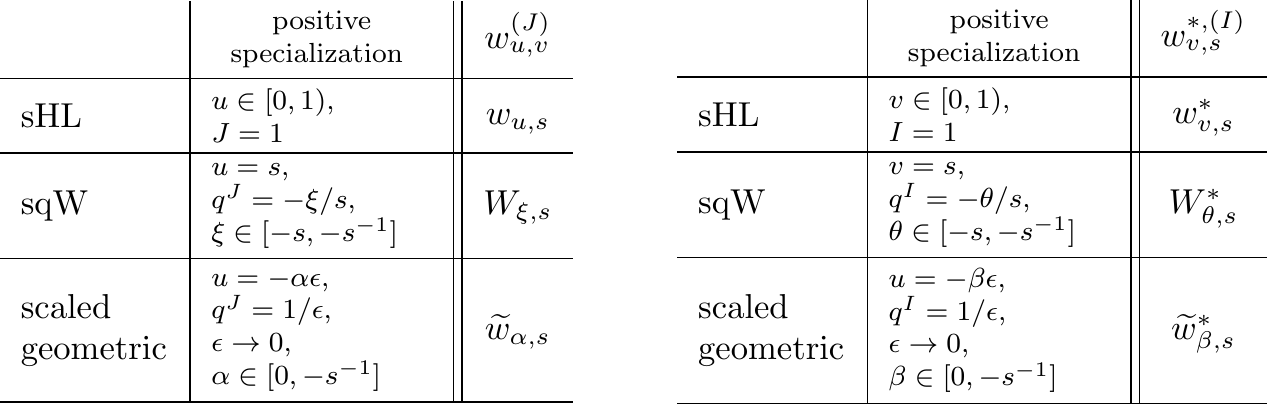}
		\caption{Positive specializations of the Yang-Baxter equation
			\eqref{eq:fused_YBE_text} we consider are obtained by combining a
			specialization from left panel with a specialization from the right panel.
			The other parameters are $q\in(0,1)$ 
			and $s\in(-1,0)$, but when both specializations are sqW,
			we impose the additional restriction $s\ge-\sqrt q$.}
    \label{fig:table_positivity_YBE}
\end{figure}

Equation \eqref{eq:fused_YBE_text}
implies all the other Yang-Baxter equations we use,
by properly specializing the 
parameters $u,q^J,v,q^I$. For certain degenerations of
weights $w^{(J)}_{u,s}$, $w^{(I)}_{v,s}$, $R^{(I,J)}_{uv}$ we can establish
their nonnegativity, and hence construct stochastic bijectivizations of
\eqref{eq:fused_YBE_text} using \Cref{prop:Bij_exists}. 
The list of the
\emph{positive specializations} we employ is summarized in 
\Cref{fig:table_positivity_YBE}, while the proofs of their nonnegativity are
given in \Cref{sub:YBE_nonnegativity}. 
For unified 
notation
here and in \Cref{sub:cross_multiple_dragging,sub:YBF_and_its_marginals}
below 
we use the vertex weights
$R^{(I,J)}_{uv}, w^{(J)}_{u,s},w^{*,(I)}_{v,s}$ 
assuming that they are nonnegative 
(under one of the parameter choices 
in \Cref{fig:table_positivity_YBE}). 

Graphically, we can interpret each summand in the left and right
hand sides of \eqref{eq:fused_YBE_text}
as a weight we attribute to arrangements of paths
across configurations of three vertices with fixed occupation numbers
$i_1,i_2,i_3,j_1,j_2,j_3$ at external edges. 
The global weight of 3-vertex
configurations depends on $R^{(I,J)}_{uv},
w^{(J)}_{u,s},w^{*,(I)}_{v,s}$, 
and
is assigned according to \Cref{fig:IJ_YBE_illustration_S4}. 
In the same figure, $\mathbf{p}^{\mathrm{fwd}}$ and
$\mathbf{p}^{\mathrm{bwd}}$ denote forward and backward transition weights of a
bijectivization of \eqref{eq:fused_YBE_text}.

\begin{figure}[h]
		\centering
		\includegraphics{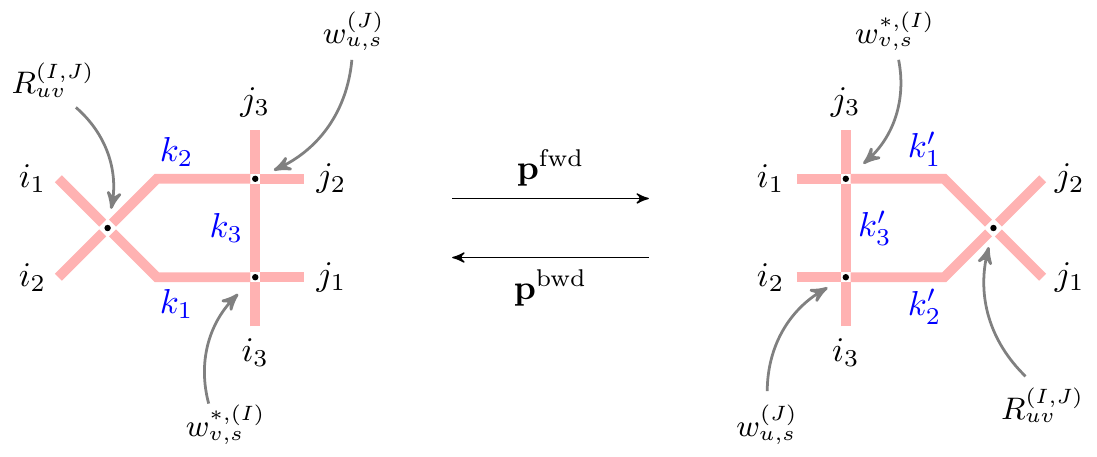}
		\caption{{A graphical representation of the Yang-Baxter Equation \eqref{eq:fused_YBE_text}}
		and its bijectivization.} 
		\label{fig:IJ_YBE_illustration_S4}
\end{figure}

For simplicity we do not include the external occupation numbers 
$i_1,i_2,i_3,j_1,j_2,j_3\in \mathbb{Z}_{\ge0}$
in the notation
$\mathbf{p}^{\mathrm{fwd}}$ and $\mathbf{p}^{\mathrm{bwd}}$.
Let us extend the
definition of $\mathbf{p}^{\mathrm{fwd}},\mathbf{p}^{\mathrm{bwd}}$ by setting
\begin{equation}
	\mathbf{p}^{\mathrm{fwd}} 
	\Bigg( 
		\begin{tikzpicture}[baseline=5,scale=0.5]
			\draw[line width = 1mm,red!30] (-3.5,0) -- (-3.05,0.45); 
			\draw[line width = 1mm,red!30] (-3.5,1) -- (-3.05,0.55);
			\draw[fill] (-3,0.5) circle [radius=0.025];
			\draw[line width = 1mm,red!30] (-2.95,0.55) -- (-2.5,1) -- (-1.55,1);
			\draw[line width = 1mm,red!30] (-2.95,0.45) -- (-2.5,0) -- (-1.55,0);
			\draw[line width = 1mm,red!30] (-1.45,0) -- (-1,0);
			\draw[line width = 1mm,red!30] (-1.5,-0.5) -- (-1.5, -0.05);
			\draw[line width = 1mm,red!30] (-1.5,0.05) -- (-1.5, 0.95);
			\draw[line width = 1mm,red!30] (-1.45,1) -- (-1, 1);
			\draw[line width = 1mm,red!30] (-1.5,1.05) -- (-1.5, 1.5);
			\draw[fill] (-1.5,1) circle [radius=0.025];
			\draw[fill] (-1.5,0) circle [radius=0.025];
			\node[left] at (-3.5,0) {\tiny{$i_2$}};
			\node[left] at (-3.5,1) {\tiny{$i_1$}};
			\node[right] at (-1.5,-0.6) {\tiny{$i_3$}};
			\node[right] at (-1,0) {\tiny{$j_1$}};
			\node[right] at (-1,1) {\tiny{$j_2$}};
			\node[right] at (-1.5,1.6) {\tiny{$j_3$}};
			\node[below] at (-2.3,0) {\tiny{$k_1$}};
			\node[above] at (-2.3,1) {\tiny{$k_2$}};
			\node[left] at (-1.5,0.5) {\tiny{$k_3$}};
		\end{tikzpicture} , 
		\begin{tikzpicture}[baseline=5,scale=0.5]
			\draw[line width = 1mm, red!30] (1,1) -- (1.45,1);
			\draw[line width = 1mm, red!30] (1.55,1) -- (2.5,1) -- (2.95,0.55);
			\draw[line width = 1mm, red!30] (3.05,0.45) -- (3.5,0);
			\draw[line width = 1mm,red!30] (1, 0) -- (1.45,0.0);
			\draw[line width = 1mm,red!30] (1.55,0) -- (2.5,0) -- (2.95,0.45);
			\draw[line width = 1mm,red!30] (3.05, 0.55) -- (3.5,1);
			\draw[line width = 1mm,red!30] (1.5, -0.5) -- (1.5,-0.05);
			\draw[line width = 1mm,red!30] (1.5, 0.05) -- (1.5,0.95);
			\draw[line width = 1mm,red!30] (1.5, 1.05) -- (1.5,1.5);
			\draw[fill] (1.5,1) circle [radius=0.025];
			\draw[fill] (1.5,0) circle [radius=0.025];
			\draw[fill] (3,0.5) circle [radius=0.025];
			\node[left] at (1,1) {\tiny{$i_1'$}};
			\node[left] at (1,0) {\tiny{$i_2'$}};
			\node[left] at (1.5, -0.6) {\tiny{$i_3'$}};
			\node[right] at (3.5, 0) {\tiny{$j_1'$}};
			\node[right] at (3.5, 1) {\tiny{$j_2'$}};
			\node[left] at (1.5, 1.6) {\tiny{$j_3'$}};
			\node[below] at (2.3,0) {\tiny{$k_1'$}};
			\node[above] at (2.3,1) {\tiny{$k_2'$}};
			\node[right] at (1.5,0.5) {\tiny{$k_3'$}};
			\addvmargin{1mm} 
		\end{tikzpicture} 
	\Bigg) 
	= 
	0,
\end{equation}
whenever $(i_1,i_2,i_3,j_1,j_2,j_3) \neq (i_1',i_2',i_3',j_1',j_2',j_3')$, and
analogously for $\mathbf{p}^{\mathrm{bwd}}$. 
Thus, we will view $\mathbf{p}^{\mathrm{fwd}}$ as the probability of a 
Markov transition of
pushing the cross
through a column of two vertices in the right direction, and similarly
$\mathbf{p}^{\mathrm{bwd}}$ corresponds to pushing the cross to the left.
These transitions do not change the external occupation numbers 
$(i_1,i_2,i_3,j_1,j_2,j_3)$, but $\mathbf{p}^{\mathrm{fwd}}$
changes fixed occupation numbers $(k_1,k_2,k_3)$
into \emph{random} $(k_1',k_2',k_3')$, and similarly $\mathbf{p}^{\mathrm{bwd}}$
maps $(k_1',k_2',k_3')$ into random $(k_1,k_2,k_3)$.

\subsection{Dragging a cross through multiple columns. Yang-Baxter fields} 
\label{sub:cross_multiple_dragging}

We now want to bring our discussion a step forward and push
the cross through multiple columns of vertices, from the leftmost one to
the right (and vice versa),
sequentially utilizing the transition probabilities
$\mathbf{p}^{\mathrm{fwd}}$ and $\mathbf{p}^{\mathrm{bwd}}$ 
associated with the
vertex weights 
$w^{(J)}_{u,s}$, $w^{*,(I)}_{v,s}$, and $R^{(I,J)}_{uv}$
which are nonnegative in one of the cases given in Figure
\ref{fig:table_positivity_YBE}.

\begin{figure}[t]
    \centering
		\includegraphics{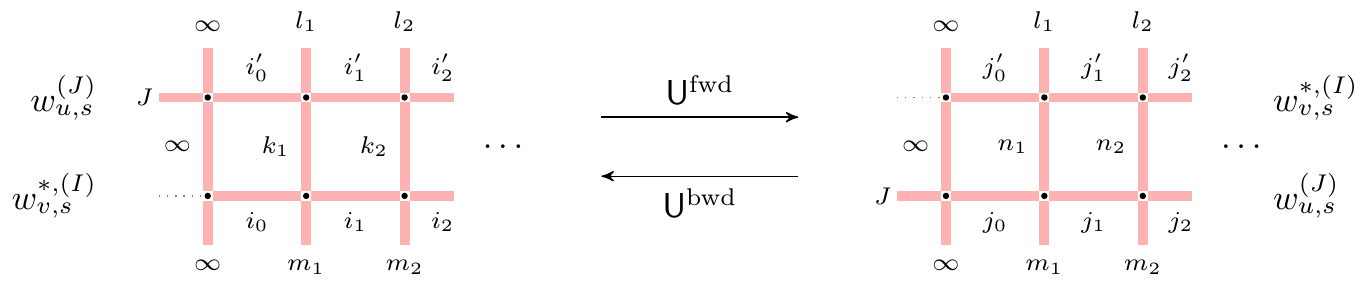}
    \caption{Graphical 
		representation of a transition between two-row path configurations.}
    \label{fig:transition_nu_kappa}
\end{figure}
We consider the 
lattice composed of two infinite rows, that is,
the vertices are indexed by the lattice
$\mathbb{Z}_{\geq 0} \times \{1,2\}$. 
The rows carry vertex weights 
$w^{(J)}_{u,s}$ and $w^{*,(I)}_{v,s}$
(see \Cref{fig:transition_nu_kappa} for an illustration).
As boundary conditions
for the paths
flowing through the lattice
we take:
\begin{equation}
	\label{eq:path_conf_boundary_conditions}
	\begin{minipage}[tb]{.9\textwidth}
		\begin{itemize}
				\item infinitely many paths flow in the vertical direction in the 0-th column;
				\item at the 0-th column no paths enter from the left into the vertex carrying 
					the weight $w^{*,(I)}_{v,s}$, 
					while $J$ paths
					enter from the left into the vertex in the 0-th column
					carrying the weight $w^{(J)}_{u,s}$; 
				\item paths do not stay horizontal forever, that is, at the far right 
					the path configuration must be empty.
		\end{itemize}
	\end{minipage}
\end{equation}
\begin{remark}
	\label{rmk:sqW_spec_careful}
	Under the sqW or scaled geometric
	specializations treating $q^J$ as an independent variable,
	the term ``$J$ paths'' in \eqref{eq:path_conf_boundary_conditions}
	should be understood formally
	and all the vertex weights should undergo these 
	specializations together (see \Cref{rmk:sqW_spec_careful_S6_concrete}
	below for a detailed explanation of this procedure). 
	In the rest of the present section we continue to 
	employ the unified notation for
	all the cases.
\end{remark}

The numbers of vertical arrows in the path configurations 
in \Cref{fig:transition_nu_kappa} 
are encoded by triples 
of 
Young diagrams $\lambda,\varkappa,\mu$ (left)
and
$\lambda,\nu,\mu$ (right), as the 
horizontal edges' occupation numbers are then uniquely determined 
by the arrow preservation. In detail, we have
\begin{equation} \label{eq:Young_diag_mult_notation}
	\lambda= 1^{l_1}2^{l_2} \dots,
	\qquad 
	\mu= 1^{m_1}2^{m_2} \dots, 
	\qquad 
	\varkappa= 1^{k_1}2^{k_2} ,\dots,
	\qquad 
	\nu= 1^{n_1}2^{n_2} \dots.
\end{equation} 
Let us record the corresponding horizontal occupation
numbers by sequences 
$\{ i_h, i_h' \}_{h\ge0}$ 
(for $\lambda,\varkappa,\mu$)
and 
$\{ j_h, j_h' \}_{h\ge0}$
(for $\lambda,\nu,\mu$).

\begin{definition}[Markov operators on Young diagrams]
	\label{def:U_fwd_U_bwd} 
	With the above notation, we define the Markov operators 
	$\mathsf{U}^{\mathrm{fwd}}$
	and 
	$\mathsf{U}^{\mathrm{bwd}}$
	as follows. For 
	$\mathsf{U}^{\mathrm{fwd}}$,
	attach the cross vertex 
	\begin{tikzpicture}[baseline=-3,scale=0.5]
		\draw[fill] (0,0) circle [radius=0.025];
			\draw[dotted, gray] (0.05,-0.05) -- (0.5, -0.5);
			\draw[line width = 1mm, red!30] (0.05, 0.05) -- (0.5, 0.5) node[black, right]{\scriptsize{$J$}};
			\draw[dotted, gray] (-0.05, 0.05) -- (-0.5, 0.5);
			\draw[line width = 1mm, red!30] (-0.05, -0.05) -- (-0.5, -0.5) node[black, left]{\scriptsize{$J$}};
			\addvmargin{1mm}
	\end{tikzpicture}
	to the leftmost column in the configuration
	encoded by $(\lambda,\varkappa,\mu)$, 
	and drag the cross all the way to the right using the 
	transition probabilities $\mathbf{p}^{\mathrm{fwd}}$.
	An intermediate step is displayed in
	\Cref{fig: cross push}. 
	The definition of $\mathsf{U}^{\mathrm{bwd}}$ involves 
	dragging the cross to the left 
	using the 
	transition probabilities $\mathbf{p}^{\mathrm{bwd}}$, and starting from the
	empty cross vertex far to the right.
	In detail,
	\begin{align}
			\label{eq:U_fwd_def}		
			\mathsf{U}^{\mathrm{fwd}}(\varkappa\to\nu\mid \lambda,\mu) = 
			\prod_{h=0}^{\infty}
			\mathbf{p}^{\mathrm{fwd}} \Bigg( \begin{tikzpicture}[baseline=5,scale=0.5]
			\draw[line width = 1mm,red!30] (-3.5,0) -- (-3.05,0.45); 
			\draw[line width = 1mm,red!30] (-3.5,1) -- (-3.05,0.55);
			\draw[fill] (-3,0.5) circle [radius=0.025];
			\draw[line width = 1mm,red!30] (-2.95,0.55) -- (-2.5,1) -- (-1.55,1);
			\draw[line width = 1mm,red!30] (-2.95,0.45) -- (-2.5,0) -- (-1.55,0);
			\draw[line width = 1mm,red!30] (-1.45,0) -- (-1,0);
			\draw[line width = 1mm,red!30] (-1.5,-0.5) -- (-1.5, -0.05);
			\draw[line width = 1mm,red!30] (-1.5,0.05) -- (-1.5, 0.95);
			\draw[line width = 1mm,red!30] (-1.45,1) -- (-1, 1);
			\draw[line width = 1mm,red!30] (-1.5,1.05) -- (-1.5, 1.5);
			\draw[fill] (-1.5,1) circle [radius=0.025];
			\draw[fill] (-1.5,0) circle [radius=0.025];
			\node[left] at (-3.5,0) {\tiny{$j_{h-1}$}};
			\node[left] at (-3.5,1) {\tiny{$j_{h-1}'$}};
			\node[below, yshift=0.1cm] at (-1.5,-0.6) {\tiny{$m_h$}};
			\node[right] at (-1,0) {\tiny{$i_h$}};
			\node[right] at (-1,1) {\tiny{$i_h'$}};
			\node[above, yshift=-0.1cm] at (-1.5,1.6) {\tiny{$l_h$}};
			\node[below, xshift=-0.1cm] at (-2.3,0) {\tiny{$i_{h-1}$}};
			\node[above, xshift=-0.1cm] at (-2.3,1) {\tiny{$i_{h-1}'$}};
			\node[left] at (-1.5,0.5) {\tiny{$k_h$}};
			\end{tikzpicture} , \begin{tikzpicture}[baseline=5,scale=0.5]
			\draw[line width = 1mm, red!30] (1,1) -- (1.45,1);
			\draw[line width = 1mm, red!30] (1.55,1) -- (2.5,1) -- (2.95,0.55);
			\draw[line width = 1mm, red!30] (3.05,0.45) -- (3.5,0);
			\draw[line width = 1mm,red!30] (1, 0) -- (1.45,0.0);
			\draw[line width = 1mm,red!30] (1.55,0) -- (2.5,0) -- (2.95,0.45);
			\draw[line width = 1mm,red!30] (3.05, 0.55) -- (3.5,1);
			\draw[line width = 1mm,red!30] (1.5, -0.5) -- (1.5,-0.05);
			\draw[line width = 1mm,red!30] (1.5, 0.05) -- (1.5,0.95);
			\draw[line width = 1mm,red!30] (1.5, 1.05) -- (1.5,1.5);
			\draw[fill] (1.5,1) circle [radius=0.025];
			\draw[fill] (1.5,0) circle [radius=0.025];
			\draw[fill] (3,0.5) circle [radius=0.025];
			\node[left] at (1,1) {\tiny{$j_{h-1}'$}};
			\node[left] at (1,0) {\tiny{$j_{h-1}$}};
			\node[below, yshift=0.1cm] at (1.5, -0.6) {\tiny{$m_h$}};
			\node[right] at (3.5, 0) {\tiny{$i_h$}};
			\node[right] at (3.5, 1) {\tiny{$i_h'$}};
			\node[above, yshift=-0.1cm] at (1.5, 1.6) {\tiny{$l_h$}};
			\node[below] at (2.3,0) {\tiny{$j_h$}};
			\node[above] at (2.3,1) {\tiny{$j_h'$}};
			\node[right] at (1.5,0.5) {\tiny{$n_h$}};
			\addvmargin{1mm}
			\end{tikzpicture} \Bigg);
			\\[-7pt]
			\label{eq:U_bwd_def}
			\mathsf{U}^{\mathrm{bwd}}(\nu\to\varkappa\mid \lambda,\mu) = 
			\prod_{h=0}^{\infty}
			\mathbf{p}^{\mathrm{bwd}} \Bigg(
			\begin{tikzpicture}[baseline=5,scale=0.5]
			\draw[line width = 1mm, red!30] (1,1) -- (1.45,1);
			\draw[line width = 1mm, red!30] (1.55,1) -- (2.5,1) -- (2.95,0.55);
			\draw[line width = 1mm, red!30] (3.05,0.45) -- (3.5,0);
			\draw[line width = 1mm,red!30] (1, 0) -- (1.45,0.0);
			\draw[line width = 1mm,red!30] (1.55,0) -- (2.5,0) -- (2.95,0.45);
			\draw[line width = 1mm,red!30] (3.05, 0.55) -- (3.5,1);
			\draw[line width = 1mm,red!30] (1.5, -0.5) -- (1.5,-0.05);
			\draw[line width = 1mm,red!30] (1.5, 0.05) -- (1.5,0.95);
			\draw[line width = 1mm,red!30] (1.5, 1.05) -- (1.5,1.5);
			\draw[fill] (1.5,1) circle [radius=0.025];
			\draw[fill] (1.5,0) circle [radius=0.025];
			\draw[fill] (3,0.5) circle [radius=0.025];
			\node[left] at (1,1) {\tiny{$j_{h-1}'$}};
			\node[left] at (1,0) {\tiny{$j_{h-1}$}};
			\node[below, yshift=0.1cm] at (1.5, -0.6) {\tiny{$m_h$}};
			\node[right] at (3.5, 0) {\tiny{$i_h$}};
			\node[right] at (3.5, 1) {\tiny{$i_h'$}};
			\node[above, yshift=-0.1cm] at (1.5, 1.6) {\tiny{$l_h$}};
			\node[below] at (2.3,0) {\tiny{$j_h$}};
			\node[above] at (2.3,1) {\tiny{$j_h'$}};
			\node[right] at (1.5,0.5) {\tiny{$k_h$}};
			\addvmargin{1mm}\end{tikzpicture}
			,
			\begin{tikzpicture}[baseline=5,scale=0.5]
				\draw[line width = 1mm,red!30] (-3.5,0) -- (-3.05,0.45); 
			\draw[line width = 1mm,red!30] (-3.5,1) -- (-3.05,0.55);
			\draw[fill] (-3,0.5) circle [radius=0.025];
			\draw[line width = 1mm,red!30] (-2.95,0.55) -- (-2.5,1) -- (-1.55,1);
			\draw[line width = 1mm,red!30] (-2.95,0.45) -- (-2.5,0) -- (-1.55,0);
			\draw[line width = 1mm,red!30] (-1.45,0) -- (-1,0);
			\draw[line width = 1mm,red!30] (-1.5,-0.5) -- (-1.5, -0.05);
			\draw[line width = 1mm,red!30] (-1.5,0.05) -- (-1.5, 0.95);
			\draw[line width = 1mm,red!30] (-1.45,1) -- (-1, 1);
			\draw[line width = 1mm,red!30] (-1.5,1.05) -- (-1.5, 1.5);
			\draw[fill] (-1.5,1) circle [radius=0.025];
			\draw[fill] (-1.5,0) circle [radius=0.025];
			\node[left] at (-3.5,0) {\tiny{$j_{h-1}$}};
			\node[left] at (-3.5,1) {\tiny{$j_{h-1}'$}};
			\node[below, yshift=0.1cm] at (-1.5,-0.6) {\tiny{$m_h$}};
			\node[right] at (-1,0) {\tiny{$i_h$}};
			\node[right] at (-1,1) {\tiny{$i_h'$}};
			\node[above, yshift=-0.1cm] at (-1.5,1.6) {\tiny{$l_h$}};
			\node[below, xshift=-0.1cm] at (-2.3,0) {\tiny{$i_{h-1}$}};
			\node[above, xshift=-0.1cm] at (-2.3,1) {\tiny{$i_{h-1}'$}};
			\node[left] at (-1.5,0.5) {\tiny{$n_h$}};
		\end{tikzpicture} \Bigg),
	\end{align}
	where $j_{-1}=J$, $j_{-1}'=0$, and $i_h=i_h'=0$ for all sufficiently large $h$.
	All terms
	$\mathbf{p}^{\mathrm{fwd}}$ and $\mathbf{p}^{\mathrm{bwd}}$
	in the infinite products 
	\eqref{eq:U_fwd_def}, \eqref{eq:U_bwd_def}
	belong to $[0,1]$. 
\end{definition}
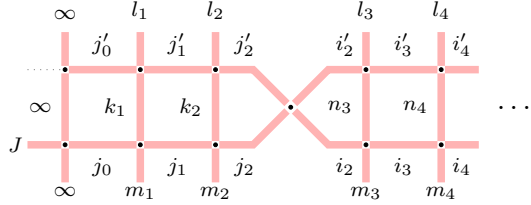
\begin{figure}[htbp]
		\centering
		\begin{tikzpicture}
		
			\foreach \n in {-3,...,-1}{
				\foreach \t in {0,1}{
					\draw[fill] (\n,\t) circle[radius=0.025]; 
				}
			}
		\draw[line width = 1mm, red!30] (-3.5,0) -- (-3.05,0) node[near start, left, black] {\scriptsize{$J$}};
		\draw[dotted, gray] (-3.5,1) -- (-3.05,1);
		\draw[line width = 1mm, red!30] (-3,-0.5) -- (-3,-0.05) node[near start, below, black] {\scriptsize{$\infty$}};
		\draw[line width = 1mm, red!30] (-3,0.05) -- (-3,0.95) node[midway, left, black] {\scriptsize{$\infty$}};
		\draw[line width = 1mm, red!30] (-3,1.05) -- (-3,1.5) node[above, black] {\scriptsize{$\infty$}};
		
		\draw[line width = 1mm, red!30] (-2.95,1) -- (-2.05,1) node[midway,above, black] {\scriptsize{$j_0'$}};
		\draw[line width = 1mm, red!30] (-2.95,0) -- (-2.05,0)
		node[midway,below, black] {\scriptsize{$j_0$}};
		\draw[line width = 1mm, red!30] (-2,-0.5) -- (-2,-0.05) node[near start, below, black] {\scriptsize{$m_1$}};
		\draw[line width = 1mm, red!30] (-2,0.05) -- (-2,0.95) node[midway, left, black] {\scriptsize{$k_1$}};
		\draw[line width = 1mm, red!30] (-2,1.05) -- (-2,1.5) node[above, black] {\scriptsize{$l_1$}};
		
		\draw[line width = 1mm, red!30] (-1.95,1) -- (-1.05,1)
		node[midway,above, black] {\scriptsize{$j_1'$}};
		\draw[line width = 1mm, red!30] (-1.95,0) -- (-1.05,0) node[midway,below, black] {\scriptsize{$j_1$}};
		\draw[line width = 1mm, red!30] (-1,-0.5) -- (-1,-0.05) node[near start, below, black] {\scriptsize{$m_2$}};
		\draw[line width = 1mm, red!30] (-1,0.05) -- (-1,0.95) node[midway, left, black] {\scriptsize{$k_2$}};
		\draw[line width = 1mm, red!30] (-1,1.05) -- (-1,1.5) node[above, black] {\scriptsize{$l_2$}};
		
		\draw[line width = 1mm, red!30] (-0.95,1) -- (-0.5,1) node[near end, above, black] {\scriptsize{$j_2'$}} -- (-0.05,0.55);
		\draw[line width = 1mm, red!30] (-0.95,0) -- (-0.5,0) node[near end, below, black] {\scriptsize{$j_2$}} -- (-0.05,0.45);
		
		\draw[fill] (0,0.5) circle[radius=0.025];
		
		\foreach \n in {1,...,2}{
			\foreach \t in {0,1}{
				\draw[fill] (\n,\t) circle[radius=0.025]; 
			}
		}
		
		\draw[line width = 1mm, red!30] (0.05,0.55) -- (0.5,1) -- (0.95,1) node[midway, above, black] {\scriptsize{$i_2'$}};
		\draw[line width = 1mm, red!30] (0.05,0.45) -- (0.5,0) -- (0.95,0) node[midway, below, black] {\scriptsize{$i_2$}};
		
		\draw[line width = 1mm, red!30] (1,-0.5) -- (1,-0.05) node[near start, below, black] {\scriptsize{$m_3$}};
		\draw[line width = 1mm, red!30] (1,0.05) -- (1,0.95) node[midway, left, black] {\scriptsize{$n_3$}};
		\draw[line width = 1mm, red!30] (1,1.05) -- (1,1.5) node[above, black] {\scriptsize{$l_3$}};
		
		\draw[line width = 1mm, red!30] (1.05,1) -- (1.95,1)
		node[midway,above, black] {\scriptsize{$i_3'$}};
		\draw[line width = 1mm, red!30] (1.05,0) -- (1.95,0) node[midway,below, black] {\scriptsize{$i_3$}};
		
		\draw[line width = 1mm, red!30] (2,-0.5) -- (2,-0.05) node[near start, below, black] {\scriptsize{$m_4$}};
		\draw[line width = 1mm, red!30] (2,0.05) -- (2,0.95) node[midway, left, black] {\scriptsize{$n_4$}};
		\draw[line width = 1mm, red!30] (2,1.05) -- (2,1.5) node[above, black] {\scriptsize{$l_4$}};
		
		\draw[line width = 1mm, red!30] (2.05,1) -- (2.5,1)
		node[midway,above, black] {\scriptsize{$i_4'$}};
		\draw[line width = 1mm, red!30] (2.05,0) -- (2.5,0) node[midway,below, black] {\scriptsize{$i_4$}};
		
		\node at (3,0.5) {$\dots$};
		
		\end{tikzpicture}
			\caption{{As the cross moves to the right,
			it randomly updates values of vertical occupation numbers $k_h$ to $n_h$
			(in $\mathsf{U}^{\mathrm{fwd}}$). 
			Horizontal occupation numbers are updated accordingly.}}
		\label{fig: cross push}
\end{figure}

Because the definition of $\mathsf{U}^{\mathrm{fwd}}$ and 
$\mathsf{U}^{\mathrm{bwd}}$ involves Yang-Baxter equations with 
infinitely many paths, we have to make sure that 
the corresponding infinite sums converge.
Recall the sets $\mathsf{Adm}(\rho,\rho^*)$
from \Cref{def:adm_rho} and the restrictions on 
parameters in \Cref{fig:table_positivity_YBE}
leading to positive specializations.
\begin{proposition}
	For each of the 9 pairs of specializations 
	$(\rho,\rho^*)$
	from Figure \ref{fig:table_positivity_YBE}
	(when $\rho$ and $\rho^*$ correspond to 
	$w_{u,s}^{(J)}$ and $w_{v,s}^{*,(I)}$, respectively)
	when the parameters 
	belong to $\mathsf{Adm}(\rho,\rho^*)$,
	one can choose bijectivizations 
	$\mathbf{p}^{\mathrm{fwd}}$ and $\mathbf{p}^{\mathrm{bwd}}$
	such that the 
	Markov operators
	$\mathsf{U}^{\mathrm{fwd}}$ and 
	$\mathsf{U}^{\mathrm{bwd}}$
	are well-defined by the infinite
	products 
	\eqref{eq:U_fwd_def} and \eqref{eq:U_bwd_def}.
	That is,
	\begin{equation*}
		\sum_{\nu}
		\mathsf{U}^{\mathrm{fwd}}(\varkappa\to\nu\mid \lambda,\mu)=1,
		\qquad 
		\sum_{\varkappa}
		\mathsf{U}^{\mathrm{bwd}}(\nu\to\varkappa\mid \lambda,\mu)=1,
	\end{equation*}
	where the sums are taken over all 
	path configurations 
	as in \Cref{fig:transition_nu_kappa}
	(with $\varkappa$ and $\nu$ encoding the left and right pictures, respectively)
	with the boundary conditions \eqref{eq:path_conf_boundary_conditions}.
	\label{prop:U_well_defined}
\end{proposition}
This implies in particular that the Markov operator
$\mathsf{U}^{\mathrm{fwd}}$
does not produce path configurations with infinitely
long horizontal paths on the right or infinitely many vertical paths in any column 
except the leftmost one.
\begin{proof}[Proof of \Cref{prop:U_well_defined}]
	\textbf{Step 1.}
	The backward transition probabilities
	$\mathsf{U}^{\mathrm{bwd}}(\nu\to\varkappa\mid \lambda,\mu)$
	sum to one over $\varkappa$
	because for fixed $\lambda,\nu,\mu$ the 
	number of possible configurations $\varkappa$
	is finite in all the cases 
	considered in \Cref{sub:sHL_sHL_structure,sub:sHL_sqW_structure,sub:sqW_sqW_structure}.
	Therefore, only finitely many factors in the products \eqref{eq:U_bwd_def}
	differ from $1$. As the individual pieces $\mathbf{p}^{\mathrm{bwd}}$ sum to one
	over all possible outcomes,
	we see that the backward operator $\mathsf{U}^{\mathrm{bwd}}$ is well-defined.

	\medskip\noindent\textbf{Step 2.}
	We will now show that 
	there exists a bijectivization 
	$\mathbf{p}^{\mathrm{fwd}}$ such that
	for all $j\ge 1$ we have
	\begin{equation}
		\label{eq:p_fwd_less_than_1}
		\mathbf{p}^{\mathrm{fwd}} \Bigg( \begin{tikzpicture}[baseline=5,scale=0.5]
		\draw[line width = 1mm,red!30] (-3.5,0) -- (-3.05,0.45); 
		\draw[line width = 1mm,red!30] (-3.5,1) -- (-3.05,0.55);
		\draw[fill] (-3,0.5) circle [radius=0.025];
		\draw[dotted, gray] (-2.95,0.55) -- (-2.5,1) -- (-1.55,1);
		\draw[dotted, gray] (-2.95,0.45) -- (-2.5,0) -- (-1.55,0);
		\draw[dotted, gray] (-1.45,0) -- (-1,0);
		\draw[dotted, gray] (-1.5,-0.5) -- (-1.5, -0.05);
		\draw[dotted, gray] (-1.5,0.05) -- (-1.5, 0.95);
		\draw[dotted, gray] (-1.45,1) -- (-1, 1);
		\draw[dotted, gray] (-1.5,1.05) -- (-1.5, 1.5);
		\draw[fill] (-1.5,1) circle [radius=0.025];
		\draw[fill] (-1.5,0) circle [radius=0.025];
		\node[left] at (-3.5,0) {\tiny{$j$}};
		\node[left] at (-3.5,1) {\tiny{$j$}};
		\node[below, yshift=0.1cm] at (-1.5,-0.6) {};
		\node[right] at (-1,0) {};
		\node[right] at (-1,1) {};
		\node[above, yshift=-0.1cm] at (-1.5,1.6) {};
		\node[below, xshift=-0.1cm] at (-2.3,0) {};
		\node[above, xshift=-0.1cm] at (-2.3,1) {};
		\node[left] at (-1.5,0.5) {};
		\end{tikzpicture} , \begin{tikzpicture}[baseline=5,scale=0.5]
		\draw[line width = 1mm, red!30] (1,1) -- (1.45,1);
		\draw[line width = 1mm, red!30] (1.55,1) -- (2.5,1) -- (2.95,0.55);
		\draw[dotted, gray] (3.05,0.45) -- (3.5,0);
		\draw[line width = 1mm,red!30] (1, 0) -- (1.45,0.0);
		\draw[line width = 1mm,red!30] (1.55,0) -- (2.5,0) -- (2.95,0.45);
		\draw[dotted, gray] (3.05, 0.55) -- (3.5,1);
		\draw[dotted, gray] (1.5, -0.5) -- (1.5,-0.05);
		\draw[dotted, gray] (1.5, 0.05) -- (1.5,0.95);
		\draw[dotted, gray] (1.5, 1.05) -- (1.5,1.5);
		\draw[fill] (1.5,1) circle [radius=0.025];
		\draw[fill] (1.5,0) circle [radius=0.025];
		\draw[fill] (3,0.5) circle [radius=0.025];
		\node[left] at (1,1) {\tiny{$j$}};
		\node[left] at (1,0) {\tiny{$j$}};
		\node[below, yshift=0.1cm] at (1.5, -0.6) {};
		\node[right] at (3.5, 0) {};
		\node[right] at (3.5, 1) {};
		\node[above, yshift=-0.1cm] at (1.5, 1.6) {};
		\node[below] at (2.3,0) {\tiny{$j$}};
		\node[above] at (2.3,1) {\tiny{$j$}};
		\node[right] at (1.5,0.5) {};
	\addvmargin{1mm}\end{tikzpicture} \Bigg) <1.
	\end{equation}
	This condition ensures that all probability mass
	is concentrated on triples 
	$(\lambda,\nu,\mu)$ with boundary conditions \eqref{eq:path_conf_boundary_conditions}, 
	and no positive probability is assigned under $\mathsf{U}^{\mathrm{fwd}}$ to 
	configurations with infinitely long horizontal paths.
	Indeed, 
	if there are $j$ paths escaping
	to the right past $\max(\mu_1,\lambda_1)$, 
	then 
	due to \eqref{eq:p_fwd_less_than_1}
	after a random geometric number of 
	cross draggings to the right
	there will remain $j-1$ paths, and so on
	until the configuration of paths far to the right becomes empty.

	The Yang-Baxter equation with the boundary conditions corresponding to \eqref{eq:p_fwd_less_than_1}
	has the form
	\begin{equation}
		\label{eq:p_fwd_less_than_1_proof}
		\sum_{a=0}^{j}
		\mathrm{weight}
		\Bigg(
			\begin{tikzpicture}[baseline=5,scale=0.5]
				\draw[line width = 1mm,red!30] (-3.5,0) -- (-3.05,0.45); 
				\draw[line width = 1mm,red!30] (-3.5,1) -- (-3.05,0.55);
				\draw[fill] (-3,0.5) circle [radius=0.025];
				\draw[line width = 1mm,red!30] (-2.95,0.55) -- (-2.5,1) -- (-1.55,1);
				\draw[line width = 1mm,red!30] (-2.95,0.45) -- (-2.5,0) -- (-1.55,0);
				\draw[dotted, gray] (-1.45,0) -- (-1,0);
				\draw[dotted, gray] (-1.5,-0.5) -- (-1.5, -0.05);
				\draw[line width = 1mm,red!30] (-1.5,0.05) -- (-1.5, 0.95);
				\draw[dotted, gray] (-1.45,1) -- (-1, 1);
				\draw[dotted, gray] (-1.5,1.05) -- (-1.5, 1.5);
				\draw[fill] (-1.5,1) circle [radius=0.025];
				\draw[fill] (-1.5,0) circle [radius=0.025];
				\node[left] at (-3.5,0) {\tiny{$j$}};
				\node[left] at (-3.5,1) {\tiny{$j$}};
				\node[right] at (-1.5,.5) {\tiny{$a$}};
				\node[above] at (-2,1) {\tiny{$a$}};
				\node[below] at (-2,0) {\tiny{$a$}};
				\node[below, yshift=0.1cm] at (-1.5,-0.6) {};
				\node[right] at (-1,0) {};
				\node[right] at (-1,1) {};
				\node[above, yshift=-0.1cm] at (-1.5,1.6) {};
				\node[below, xshift=-0.1cm] at (-2.3,0) {};
				\node[above, xshift=-0.1cm] at (-2.3,1) {};
				\node[left] at (-1.5,0.5) {};
			\end{tikzpicture}
		\Bigg)
		=
		\sum_{b=0}^{j}
		\mathrm{weight}
		\Bigg(
			\begin{tikzpicture}[baseline=5,scale=0.5]
				\draw[line width = 1mm, red!30] (1,1) -- (1.45,1);
				\draw[line width = 1mm, red!30] (1.55,1) -- (2.5,1) -- (2.95,0.55);
				\draw[dotted, gray] (3.05,0.45) -- (3.5,0);
				\draw[line width = 1mm,red!30] (1, 0) -- (1.45,0.0);
				\draw[line width = 1mm,red!30] (1.55,0) -- (2.5,0) -- (2.95,0.45);
				\draw[dotted, gray] (3.05, 0.55) -- (3.5,1);
				\draw[dotted, gray] (1.5, -0.5) -- (1.5,-0.05);
				\draw[line width = 1mm, red!30] (1.5, 0.05) -- (1.5,0.95);
				\draw[dotted, gray] (1.5, 1.05) -- (1.5,1.5);
				\draw[fill] (1.5,1) circle [radius=0.025];
				\draw[fill] (1.5,0) circle [radius=0.025];
				\draw[fill] (3,0.5) circle [radius=0.025];
				\node[left] at (1,1) {\tiny{$j$}};
				\node[left] at (1,0) {\tiny{$j$}};
				\node [right] at (1.5,.5) {\tiny{$b$}};
				\node[below, yshift=0.1cm] at (1.5, -0.6) {};
				\node[right] at (3.5, 0) {};
				\node[right] at (3.5, 1) {};
				\node[above, yshift=-0.1cm] at (1.5, 1.6) {};
				\node[below] at (2.45,0) {\tiny{$j-b$}};
				\node[above] at (2.45,1) {\tiny{$j-b$}};
				\node[right] at (1.5,0.5) {};
				\addvmargin{1mm}
			\end{tikzpicture}
		\Bigg).
	\end{equation}
	It is possible to choose a bijectivization 
	satisfying \eqref{eq:p_fwd_less_than_1}
	if
	at least one term in the right-hand side of \eqref{eq:p_fwd_less_than_1_proof}
	corresponding to some $b>0$ does not vanish.
	These terms are given by 
	\begin{equation*}
		\mathrm{weight}
		\Bigg(
			\begin{tikzpicture}[baseline=5,scale=0.5]
				\draw[line width = 1mm, red!30] (1,1) -- (1.45,1);
				\draw[line width = 1mm, red!30] (1.55,1) -- (2.5,1) -- (2.95,0.55);
				\draw[dotted, gray] (3.05,0.45) -- (3.5,0);
				\draw[line width = 1mm,red!30] (1, 0) -- (1.45,0.0);
				\draw[line width = 1mm,red!30] (1.55,0) -- (2.5,0) -- (2.95,0.45);
				\draw[dotted, gray] (3.05, 0.55) -- (3.5,1);
				\draw[dotted, gray] (1.5, -0.5) -- (1.5,-0.05);
				\draw[line width = 1mm, red!30] (1.5, 0.05) -- (1.5,0.95);
				\draw[dotted, gray] (1.5, 1.05) -- (1.5,1.5);
				\draw[fill] (1.5,1) circle [radius=0.025];
				\draw[fill] (1.5,0) circle [radius=0.025];
				\draw[fill] (3,0.5) circle [radius=0.025];
				\node[left] at (1,1) {\tiny{$j$}};
				\node[left] at (1,0) {\tiny{$j$}};
				\node [right] at (1.5,.5) {\tiny{$b$}};
				\node[below, yshift=0.1cm] at (1.5, -0.6) {};
				\node[right] at (3.5, 0) {};
				\node[right] at (3.5, 1) {};
				\node[above, yshift=-0.1cm] at (1.5, 1.6) {};
				\node[below] at (2.45,0) {\tiny{$j-b$}};
				\node[above] at (2.45,1) {\tiny{$j-b$}};
				\node[right] at (1.5,0.5) {};
				\addvmargin{1mm}
			\end{tikzpicture}
		\Bigg)
		=w_{u,s}^{(J)}(0,j;b,j-b)
		w_{v,s}^{*,(I)}(b,j;0,j-b)
		R_{uv}^{(I,J)}(j-b,j-b;0,0).
	\end{equation*}
	We now consider the cases of Figure \ref{fig:table_positivity_YBE}
	separately. In the cases involving the sHL specialization, the only
	allowed positive $j$ 
	is $j=1$, and the positivity of the term corresponding to $b=1$ can
	be checked by writing down all possible cases:
	\begin{equation*}
	    \begin{array}{rclrcl}
	    \mathrm{sHL/sHL} &:&  \dfrac{(1-q)(1-s^2)}{(1-su) (1-sv)}, 
	    	\\
		\mathrm{sHL/sqW} &:&  \dfrac{1-q}{1-su }, 
		&
		\qquad \qquad \mathrm{sqW/sHL} &:& \dfrac{1-q}{1-sv },
		\rule{0pt}{25pt}
			\\
		\mathrm{sHL/sg} &:&  \dfrac{(1-q)(1-s^2)}{1-su },
		&
		\qquad \qquad \mathrm{sg/sHL} &:& \dfrac{(1-q)(1-s^2)}{1-sv }.
	    \rule{0pt}{25pt}
	    \end{array}
	\end{equation*}
	All these expressions are positive under the positivity conditions
	from \Cref{fig:table_positivity_YBE}.

	Next, in the sqW/sqW, case, the number of paths $j\ge1$ can be arbitrary, but the product
	$W_{x,s}(0,j;b,j-b)W^*_{y,s}(b,j;0,j-b)$ vanishes unless $b=j$, and
	for $b=j$ it is positive.
	The same is true for the
	sqW/sg and sg/sqW cases. Finally, in the sg/sg case all
	factors of the form
	$\widetilde{w}_{\alpha,s}(0,j;b,j-b)
	\widetilde{w}^*_{\beta,s}(b,j;0,j-b)
	R^{(\mathrm{sg,sg})}_{\alpha,\beta}(j-b,j-b;0,0)$
	are strictly positive.

	\medskip\noindent\textbf{Step 3.}
	Now let us check that
	after dragging the cross through the 
	leftmost column containing infinitely many vertical paths,
	the probability to get infinitely many horizontal 
	paths is zero.
	Clearly, infinitely many horizontal paths might occur only 
	if neither of the specializations $\rho,\rho^*$ is sHL.
	Overall, we need to show that
	\begin{equation}
		\label{eq:p_fwd_sum_to_one_in_leftmost}
		\sum_{j_0,j_0'}
		\mathbf{p}^{\mathrm{fwd}} \Bigg( \begin{tikzpicture}[baseline=5,scale=0.5]
				\draw[line width = 1mm,red!30] (-3.5,0) -- (-3.05,0.45); 
		\draw[dotted, gray] (-3.5,1) -- (-3.05,0.55);
		\draw[fill] (-3,0.5) circle [radius=0.025];
		\draw[line width = 1mm,red!30] (-2.95,0.55) -- (-2.5,1) -- (-1.55,1);
		\draw[dotted, gray] (-2.95,0.45) -- (-2.5,0) -- (-1.55,0);
		\draw[line width = 1mm,red!30] (-1.45,0) -- (-1,0);
		\draw[line width = 1mm,red!30] (-1.5,-0.5) -- (-1.5, -0.05);
		\draw[line width = 1mm,red!30] (-1.5,0.05) -- (-1.5, 0.95);
		\draw[line width = 1mm,red!30] (-1.45,1) -- (-1, 1);
		\draw[line width = 1mm,red!30] (-1.5,1.05) -- (-1.5, 1.5);
		\draw[fill] (-1.5,1) circle [radius=0.025];
		\draw[fill] (-1.5,0) circle [radius=0.025];
		\node[left] at (-3.5,0) {\tiny{$J$}};
		\node[below, yshift=0.1cm] at (-1.5,-0.6) {\tiny{$\infty$}};
		\node[right] at (-1,0) {\tiny{$i_0$}};
		\node[right] at (-1,1) {\tiny{$i_0'$}};
		\node[above, yshift=-0.1cm] at (-1.5,1.6) {\tiny{$\infty$}};
		\node[above, xshift=-0.1cm] at (-2.3,1) {\tiny{$J$}};
		\node[left] at (-1.5,0.5) {\tiny{$\infty$}};
			\end{tikzpicture} , \begin{tikzpicture}[baseline=5,scale=0.5]
				\draw[dotted, gray] (1,1) -- (1.45,1);
		\draw[line width = 1mm, red!30] (1.55,1) -- (2.5,1) -- (2.95,0.55);
		\draw[line width = 1mm, red!30] (3.05,0.45) -- (3.5,0);
		\draw[line width = 1mm,red!30] (1, 0) -- (1.45,0.0);
		\draw[line width = 1mm,red!30] (1.55,0) -- (2.5,0) -- (2.95,0.45);
		\draw[line width = 1mm,red!30] (3.05, 0.55) -- (3.5,1);
		\draw[line width = 1mm,red!30] (1.5, -0.5) -- (1.5,-0.05);
		\draw[line width = 1mm,red!30] (1.5, 0.05) -- (1.5,0.95);
		\draw[line width = 1mm,red!30] (1.5, 1.05) -- (1.5,1.5);
		\draw[fill] (1.5,1) circle [radius=0.025];
		\draw[fill] (1.5,0) circle [radius=0.025];
		\draw[fill] (3,0.5) circle [radius=0.025];
		\node[left] at (1,0) {\tiny{$J$}};
		\node[below, yshift=0.1cm] at (1.5, -0.6) {\tiny{$\infty$}};
		\node[right] at (3.5, 0) {\tiny{$i_0$}};
		\node[right] at (3.5, 1) {\tiny{$i_0'$}};
		\node[above, yshift=-0.1cm] at (1.5, 1.6) {\tiny{$\infty$}};
		\node[below] at (2.3,0) {\tiny{$j_0$}};
		\node[above] at (2.3,1) {\tiny{$j_0'$}};
		\node[right] at (1.5,0.5) {\tiny{$\infty$}};
		\addvmargin{1mm}\end{tikzpicture} \Bigg)=1.
	\end{equation}
	Considering the corresponding Yang-Baxter equation,
	we see that its left-hand side converges thanks to 
	\Cref{prop:emergence_of_the_cross_vertex_weight},
	because the weights of the other two vertices do not depend on the 
	input from the left
	(cf. \eqref{eq:W_boundary_weights}, \eqref{eq:w_tilde_infinity}).
	The right-hand side of this Yang-Baxter equation contains terms of the form
	$w^{*,(I)}_{v,s} (\infty, 0; \infty, j_0')\,
	w^{(J)}_{u,s} (\infty, J; \infty, j_0)\,
	R^{(I,J)}_{uv} (j_0, j_0'; i_0', i_0)$.
	In the sqW/sqW, sqW/sg, and sg/sg cases, the 
	cross vertex weights \eqref{eq:Whittaker_cross_weight}, 
	\eqref{eq:R_sqW_sg}, and \eqref{eq:R_sg_sg}
	are bounded for fixed $i_0, i_0'$. The 
	contribution
	from the other two vertices 
	regulating the convergence of the right-hand side of the Yang-Baxter equation
	amounts to $(\xi\theta)^{j_0}$, $(\xi\beta)^{j_0}$, or $(\alpha\beta)^{j_0}$,
	respectively. 
	The conditions $\mathsf{Adm}(\rho,\rho^*)$ in these cases precisely mean
	that the products of spectral parameters are less than one, so the series converge.
	One then can choose a bijectivization such that \eqref{eq:p_fwd_sum_to_one_in_leftmost}
	holds. This completes the proof.
\end{proof}

\Cref{def:U_fwd_U_bwd} and \Cref{prop:U_well_defined}
thus produce 
``natural'' Markov operators\footnote{These 
operators are not determined uniquely
(except in their action in the 0-th column, cf. \Cref{sub:YBF_and_its_marginals}
below).}
$\mathsf{U}^{\mathrm{fwd}}$ and
$\mathsf{U}^{\mathrm{bwd}}$
associated with each of our skew Cauchy structures.
Denote by $\mathfrak{F}_{\lambda/\varkappa}(\rho)$ and
$\mathfrak{G}_{\mu/\varkappa}(\rho^*)$ the
partition functions of the one-row configurations
in the top and the bottom rows, respectively,
in \Cref{fig:transition_nu_kappa}, left. 
By their very construction through local bijectivizations,
these Markov operators satisfy the reversibility condition
for all $\lambda,\mu,\varkappa,\nu$:
\begin{equation*}
	\mathsf{U}^{\mathrm{fwd}}(\varkappa\to\nu\mid \lambda,\mu)
	\cdot
	\Pi(\rho;\rho^*)\mathfrak{F}_{\lambda/\varkappa}(\rho)\mathfrak{G}_{\mu/\varkappa}(\rho^*)
	=
	\mathsf{U}^{\mathrm{bwd}}(\nu\to\varkappa \mid \lambda,\mu)
	\cdot
	\mathfrak{F}_{\nu/\mu}(\rho)\mathfrak{G}_{\nu/\lambda}(\rho^*).
\end{equation*}
Here $\Pi(\rho;\rho^*)$ is defined in 
\Cref{thm:skew_Cauchy_mixed_spec},
which can be viewed as the properly specialized
term 
$
\dfrac{(uv q^I;q)_{\infty} (uv q^J ; q)_\infty}
{(uv;q)_\infty (uvq^{I+J}; q)_\infty}
$
from the right-hand side of the Cauchy equation
\eqref{eq:skew_Cauchy_mixed}.
Moreover, this quantity is also identified with the 
weight of the cross vertex 
$(J,0;J,0)$
attached to the configuration in 
\Cref{fig:transition_nu_kappa}, left,
before dragging the cross to the right
(see \Cref{prop:emergence_of_the_cross_vertex_weight} for the last equality).

Thus, we have constructed \emph{Yang-Baxter random fields
of Young diagrams}, which are illustrated in \Cref{fig:YBE_field}.
Before discussing concrete details in each of the
different cases in \Cref{sec:new_three_fields} below, in the next \Cref{sub:YBF_and_its_marginals}
we look at 
scalar marginals of our random fields.

\begin{figure}[ht]
    \centering
    \begin{tikzpicture}
    \begin{scope}[shift = {(3.5,0)}]
        \node at (1,0.5) {$\varkappa$};
        \node at (3,0.5) {$\mu$};
        \node at (1,2.5) {$\lambda$};
        \node at (3,2.5) {$\nu$};
        \draw[dotted] (1.2,0.5) -- (2.8,0.5);
        \draw[dotted] (1,0.7) -- (1,2.3);
        \draw[dotted] (1.2,2.5) -- (2.8,2.5);  
        \draw[dotted] (3,0.7) -- (3,2.3);
        \node[above] at (1.65,1.85) {\scriptsize{$\mathsf{U}^{\mathrm{fwd}}$}};
        \node[below] at (2.3,1) {\scriptsize{$\mathsf{U}^{\mathrm{bwd}}$}};
        \node[] (p1) at (3, 2.5) {};
        \node[] (p2) at ( 1 , 0.5 ) {};
        \node[above=-0.3cm of p2] {} edge[pil,bend left = 30] (p1);
        \node[above=-0.3cm of p1] {} edge[pil,bend left = 30] (p2);
    \end{scope}
    \begin{scope}[shift = {(-3.5,0)}]
        \draw[thick, gray, ->] (0,0) -- (3.5,0);
        \draw[thick, gray, ->] (0,0) -- (0,3.5);
        \foreach \n in {0,...,3}{
            \foreach \t in {0,...,3}{
                \draw[fill] (\n,\t) circle[radius=0.025]; 
            }
        }
        \foreach \n in {0,...,3}{
            \draw[dotted] (\n, 0) -- (\n,3);
            \draw[dotted] (0,\n) -- (3,\n);
        }
        \foreach \n in {1,...,3}{
            \node[below] at (\n,0) {\scriptsize{$\varnothing$}};
            \node[left] at (0,\n) {\scriptsize{$\varnothing$}};
        }
        \node[below left] at (0,0) {\scriptsize{$\varnothing$}};
        \node[below] at (1,1) {\scriptsize{$\lambda^{(1,1)}$}};
        \node[below] at (2,1) {\scriptsize{$\lambda^{(2,1)}$}};
        \node[below] at (1,2) {\scriptsize{$\lambda^{(1,2)}$}};
        \node[below] at (2,2) {\scriptsize{$\lambda^{(2,2)}$}};
        \node[below] at (3,1) {\scriptsize{$\lambda^{(3,1)}$}};
        \node[below] at (1,3) {\scriptsize{$\lambda^{(1,3)}$}};
    \end{scope}
    \end{tikzpicture}
		\caption{Yang-Baxter field and forward and backward Markov transition
			operators
			$\mathsf{U}^{\mathrm{fwd}}$ and $\mathsf{U}^{\mathrm{bwd}}$.}
    \label{fig:YBE_field}
\end{figure}
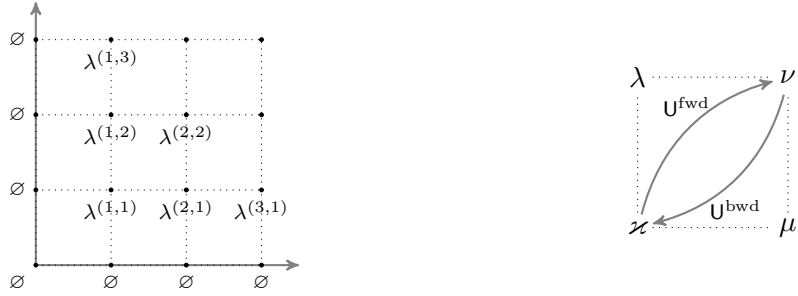

\subsection{Marginals}
\label{sub:YBF_and_its_marginals}

We now apply the discussion of \Cref{sub:F_G_scalar_marginals}
to the Yang-Baxter fields constructed 
above. 
Due to the sequential left-to-right update rule 
in the definition of 
$\mathsf{U}^{\mathrm{fwd}}$,
there is a number of marginals
$\mathsf{h}$ to which our fields $\boldsymbol \lambda$ are adapted to. 

Fix $h\ge2$.
For a Young diagram $\eta=1^{n_1} 2^{n_2} \ldots$ 
introduce the decomposition
\begin{equation} \label{eq:decomp_partition}
	\eta = (\eta^{[<h]}, \eta^{[\geq h]}),
	\qquad 
	\eta^{[<h]} = 1^{n_1} \ldots (h-1)^{n_{h-1}}, \qquad \eta^{[\geq h]} = h^{n_h} (h+1)^{n_{h+1}} \ldots,
\end{equation}
where $\eta^{[<h]}$ and $\eta^{[\ge h]}$ are two new Young 
diagrams.

\begin{proposition}
	\label{prop:exist_marginals}
	Let $\mathsf{h}$ be either of the following functions
	on the set of Young diagrams:
	\begin{itemize}
		\item $\mathsf{h}(\eta)=\ell(\eta)$;
		\item 
			$\mathsf{h}(\eta) 
			=
			\bigl(\eta^{[<h]},\ell(\eta^{[\geq h]})\bigr)$ 
			for some $h\ge2$.
	\end{itemize}
	Then each of the Yang-Baxter fields $\boldsymbol\lambda$ 
	constructed in \Cref{sub:cross_multiple_dragging}
	is adapted to $\mathsf{h}$ 
	in the sense of \Cref{sub:F_G_scalar_marginals}.
\end{proposition}
\begin{proof}
	Let first $h>1$.
	From the definition of
	$\mathsf{U}^{\mathrm{fwd}}$ 
	\eqref{eq:U_fwd_def}
	we see that the random moves of the first $h$
	columns of vertices are independent of those taking place in columns to their
	right.
	Therefore, 
	summing $\mathsf{U}^{\mathrm{fwd}}(\varkappa\to \nu\mid \lambda,\mu)$
	over $\nu$ with fixed $\nu^{[<h]}$ and $\ell(\nu)$, 
	we see that the result is independent
	of $\varkappa^{[\ge h]},\nu^{[\ge h]},\lambda^{[\ge h]},\mu^{[\ge h]}$.
	The quantities 
	$\ell(\varkappa^{[\ge h]}),
	\ell(\nu^{[\ge h]}),
	\ell(\lambda^{[\ge h]}),
	\ell(\mu^{[\ge h]})$
	encode the numbers 
	of paths
	flowing through the horizontal edges between columns $h$ and $h+1$
	(recall that $\lambda$ and $\mu$ are fixed throughout the random update).
	This proves the statement for $\mathsf{h}(\eta)=\bigl(\eta^{[<h]},
	\ell(\eta^{[\geq h]})\bigr)$. 
	
	When $h=0$, that is, when we consider the marginal
	move at the leftmost column, we only record the number of arrows entering the
	lattice (evolving in a marginally Markovian manner), which are
	simply the lengths of $\varkappa,\nu,\lambda,\mu$. This establishes
	the remaining case $\mathsf{h}(\eta)=\ell(\eta)$.
\end{proof}

Let us denote the transition probabilities of the 
marginal processes afforded by \Cref{prop:exist_marginals}
by
$\mathsf{U}^{[0]}$
and
$\mathsf{U}^{[<h]}$,
respectively.

\begin{figure}[htbp]
    \centering
    \includegraphics[scale=1.2]{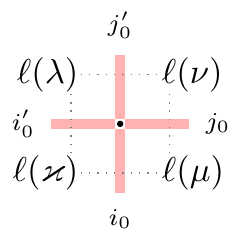}
		\caption{Evolution of the lengths of Young
			diagrams under $\mathsf{U}^{[0]}$. 
			Moving around the vertex along dotted lines 
			changes the 
			length of a Young diagram
			by the 
			occupation number of the edge
			we cross, cf. \eqref{eq:occupation_numbers_lengths}.}
    \label{fig:length_partitions}
\end{figure}

We will mostly be interested in 
the simplest case
$\mathsf{U}^{[0]}$.
One readily sees that 
the action of the Markov operator
$\mathsf{U}^{[0]}$
is encoded as the 
evolution
of the horizontal occupation numbers 
$\{i_0, i_0'\} \to \{j_0,j_0'\}$,
where 
\begin{equation}
	\label{eq:occupation_numbers_lengths}
	i_0=\ell(\mu)-\ell(\varkappa),\quad 
	i_0'=\ell(\lambda)-\ell(\varkappa);
	\qquad 
	j_0=\ell(\nu)-\ell(\mu),\quad 
	j_0'=\ell(\nu)-\ell(\lambda)
\end{equation}
(see \Cref{fig:transition_nu_kappa,fig:length_partitions} for an illustration).
Therefore, let us use the notation 
$\mathsf{U}^{[0]}(i_0, i_0'; j_0, j_0')$
for these transition probabilities. We have
\begin{equation*}
    \mathsf{U}^{[0]}(i_0, i_0'; j_0, j_0') = 
 \mathbf{p}^{\mathrm{fwd}} \Bigg( \begin{tikzpicture}[baseline=5,scale=0.5]
    \draw[line width = 1mm,red!30] (-3.5,0) -- (-3.05,0.45); 
\draw[dotted, gray] (-3.5,1) -- (-3.05,0.55);
\draw[fill] (-3,0.5) circle [radius=0.025];
\draw[line width = 1mm,red!30] (-2.95,0.55) -- (-2.5,1) -- (-1.55,1);
\draw[dotted, gray] (-2.95,0.45) -- (-2.5,0) -- (-1.55,0);
\draw[line width = 1mm,red!30] (-1.45,0) -- (-1,0);
\draw[line width = 1mm,red!30] (-1.5,-0.5) -- (-1.5, -0.05);
\draw[line width = 1mm,red!30] (-1.5,0.05) -- (-1.5, 0.95);
\draw[line width = 1mm,red!30] (-1.45,1) -- (-1, 1);
\draw[line width = 1mm,red!30] (-1.5,1.05) -- (-1.5, 1.5);
\draw[fill] (-1.5,1) circle [radius=0.025];
\draw[fill] (-1.5,0) circle [radius=0.025];
\node[left] at (-3.5,0) {\tiny{$J$}};
\node[below, yshift=0.1cm] at (-1.5,-0.6) {\tiny{$\infty$}};
\node[right] at (-1,0) {\tiny{$i_0$}};
\node[right] at (-1,1) {\tiny{$i_0'$}};
\node[above, yshift=-0.1cm] at (-1.5,1.6) {\tiny{$\infty$}};
\node[above, xshift=-0.1cm] at (-2.3,1) {\tiny{$J$}};
\node[left] at (-1.5,0.5) {\tiny{$\infty$}};
  \end{tikzpicture} , \begin{tikzpicture}[baseline=5,scale=0.5]
    \draw[dotted, gray] (1,1) -- (1.45,1);
\draw[line width = 1mm, red!30] (1.55,1) -- (2.5,1) -- (2.95,0.55);
\draw[line width = 1mm, red!30] (3.05,0.45) -- (3.5,0);
\draw[line width = 1mm,red!30] (1, 0) -- (1.45,0.0);
\draw[line width = 1mm,red!30] (1.55,0) -- (2.5,0) -- (2.95,0.45);
\draw[line width = 1mm,red!30] (3.05, 0.55) -- (3.5,1);
\draw[line width = 1mm,red!30] (1.5, -0.5) -- (1.5,-0.05);
\draw[line width = 1mm,red!30] (1.5, 0.05) -- (1.5,0.95);
\draw[line width = 1mm,red!30] (1.5, 1.05) -- (1.5,1.5);
\draw[fill] (1.5,1) circle [radius=0.025];
\draw[fill] (1.5,0) circle [radius=0.025];
\draw[fill] (3,0.5) circle [radius=0.025];
\node[left] at (1,0) {\tiny{$J$}};
\node[below, yshift=0.1cm] at (1.5, -0.6) {\tiny{$\infty$}};
\node[right] at (3.5, 0) {\tiny{$i_0$}};
\node[right] at (3.5, 1) {\tiny{$i_0'$}};
\node[above, yshift=-0.1cm] at (1.5, 1.6) {\tiny{$\infty$}};
\node[below] at (2.3,0) {\tiny{$j_0$}};
\node[above] at (2.3,1) {\tiny{$j_0'$}};
\node[right] at (1.5,0.5) {\tiny{$\infty$}};
 \addvmargin{1mm}\end{tikzpicture} \Bigg).
\end{equation*}
This probability is determined uniquely 
because the right-hand 
side of the Yang-Baxter 
equation contains a single
summand corresponding to the
state $(J,0;J,0)$ of the cross vertex
(this is enforced by our arrow preservation conventions).
Thus, by 
Example \ref{example:biject:A1},
this probability 
can be written as a ratio of weights of 3-vertex configurations as follows:
\begin{equation} \label{eq:U_00}
		\mathsf{U}^{[0]} (i_0, i_0'; j_0', j_0) 
		= 
		\frac{ w^{*,(I)}_{v,s} (\infty, 0; \infty, j_0')\,
		w^{(J)}_{u,s} (\infty, J; \infty, j_0)\,
		R^{(I,J)}_{uv} (j_0, j_0'; i_0', i_0) }
		{ R^{(I,J)}_{uv} (J, 0; J, 0)\,
		w^{*,(I)}_{v,s} (\infty, 0; \infty, i_0)\,
		w^{(J)}_{u,s} (\infty, J; \infty, i_0') }.
\end{equation}
This expression vanishes unless $i_0+j_0=i_0'+j_0'$.
The quantity $R^{(I,J)}_{uv} (J, 0; J, 0)$ in the denominator has an explicit
form \eqref{eq:sum_cross_R}.

Formula \eqref{eq:U_00} also appeared in the recent work \cite{ABB2018stochasticization}
under the name of ``stochasticization''
of the solution of a Yang-Baxter equation.  
In this paper we 
explicitly link stochasticizations 
to known stochastic vertex models
(including the stochastic six vertex model 
\cite{GwaSpohn1992}, \cite{BCG6V},
the higher spin stochastic six vertex model
\cite{Borodin2014vertex}, \cite{CorwinPetrov2015},
\cite{BorodinPetrov2016inhom},
and a pushing system introduced recently in \cite{CMP_qHahn_Push}), and 
show the existence of the corresponding full Yang-Baxter
fields.
The latter 
further connects observables of stochastic vertex models to 
probability distributions based on
spin Hall-Littlewood and spin $q$-Whittaker functions.

The other marginals $\mathsf{U}^{[<h]}$,
$h\ge2$, lead to multilayer versions of 
stochastic vertex models.
A multilayer version of the stochastic six vertex 
model was introduced recently in \cite{BufetovMatveev2017} (and another such system was 
constructed in \cite{BufetovPetrovYB2017} using Yang-Baxter fields).
Multilayer systems are much less explicit and are not determined uniquely 
due to the non-uniqueness of $\mathsf{U}^{\mathrm{fwd}}$.
They deserve their own study, and 
in the present paper we mostly focus on $\mathsf{U}^{[0]}$.

\section{Three Yang-Baxter fields}
\label{sec:new_three_fields}

\subsection{Preliminaries}
\label{sub:preliminaries}

In this section we present detailed descriptions of the 
Yang-Baxter
fields associated with the 
skew Cauchy structures defined in \Cref{sec:summary_sHL_sqW}
(with the step or scaled geometric boundary conditions).
We also discuss the 
scalar marginals 
$\mathsf{h}(\lambda^{(x,y)})=\ell(\lambda^{(x,y)})$ of these Yang-Baxter fields.
Let us first make two general remarks.

\begin{remark}
	The definitions of the Yang-Baxter fields
	in this section
	involve non-unique local
	bijectivizations, and 
	so these fields are not defined in a unique way.
	However, all our statements hold for any such choice of a bijectivization.
	Moreover, 
	the distributions of the scalar marginals
	$\ell(\lambda^{(x,y)})$
	do not depend on the choice of a bijectivization.
	Thus, for shortness we will use the term
	``the Yang-Baxter field'' to refer to any random field of Young diagrams
	coming from bijectivizations of the Yang-Baxter equations.
\end{remark}

\begin{remark}[On stationary boundary conditions]
	The (two-sided) scaled geometric
	boundary conditions
	for our Yang-Baxter fields 
	match (in scalar marginals viewed as stochastic particle systems
	on the line)
	to initial conditions 
	composed of two half-stationary pieces 
	(of possibly different densities)
	glued together at the 
	origin. For example, for the stochastic six vertex model (as well as for ASEP and TASEP)
	on the line,
	the stationary initial data is the product Bernoulli one, 
	and so the two-sided stationary initial condition 
	is composed of two product Bernoulli configurations of arbitrary 
	densities on the half-lines.
	When the densities match and the systems' parameters are homogeneous
	(i.e., independent of $x,y$), this initial data is indeed stationary 
	under the stochastic evolution on the line.

	However, one can check that for the full Yang-Baxter fields, the
	scaled geometric boundary conditions \emph{do not contain} a subfamily
	of boundary conditions remaining stationary under the evolution
	of the full Young diagrams. Therefore, we distinguish
	the terms ``scaled geometric'' and ``two-sided stationary''
	boundary conditions --- the former refers to full Yang-Baxter fields,
	and the latter --- to stochastic particle systems
	arising as one-dimensional marginals.
\end{remark}

\subsection{The sHL/sHL Yang-Baxter field 
and the stochastic six vertex model} 
\label{sub:new_YB_field_sHL_sHL}

We first discuss the simplest case, the sHL/sHL skew Cauchy structure,
and relate the corresponding field to the stochastic six vertex model of 
\cite{GwaSpohn1992}, \cite{BCG6V}. 
The case of step boundary conditions
essentially parallels \cite{BufetovPetrovYB2017}
(without the dynamic modification of the six vertex model
because here we work with the stable sHL functions instead of the 
non-stable ones).
Formulas for observables and asymptotics
in the six vertex model with two-sided stationary
boundary conditions were studied in \cite{Amol2016Stationary}, 
but its connection to symmetric functions is~new.

\subsubsection{Step boundary conditions}

The sHL/sHL case is obtained by setting $I=J=1$
in \Cref{sec:YB_fields_through_bijectivisation},
and the field depends on the 
parameters
$q\in(0,1)$, $s\in(-1,0)$, and 
$u_y,v_x\in[0,1)$, $x,y\in \mathbb{Z}_{\ge1}$. 
The Yang-Baxter equation corresponding to this skew Cauchy structure is now 
\eqref{eq:sHL_YBE}. 
The reversibility 
property of backward and forward operators takes the following form:

\begin{proposition}\label{prop:reversibility_sHL_sHL}
For any four Young diagrams $\mu, \varkappa, \lambda, \nu$ we have
\begin{multline} \label{eq:reversibility_sHL_sHL}
	\frac{1 - q u v}{1 - u v}
	\,
	\mathsf{U}_{\mathrm{sHL}(u),\mathrm{sHL}(v)}^{\mathrm{fwd}}
	(\varkappa \to \nu \mid \lambda, \mu) 
		\,
		\mathsf{F}_{\lambda / \varkappa}(u) 
		\,
		\mathsf{F}^*_{\mu / \varkappa}(v) 
		\\=
		\mathsf{U}_{\mathrm{sHL}(u),\mathrm{sHL}(v)}^{\mathrm{bwd}}
		(\nu \to \varkappa \mid \lambda, \mu)
		\,
		\mathsf{F}^*_{\nu / \lambda} (v)
		\,
		\mathsf{F}_{\nu / \mu} (u).
\end{multline}
Summing \eqref{eq:reversibility_sHL_sHL} 
over both $\varkappa$ and $\nu$, we obtain
the skew Cauchy identity for the stable 
sHL functions (\Cref{thm:skew_Cauchy_sHL_sHL}).
\end{proposition}
\begin{proof}
		The product $\mathsf{F}_{\lambda / \varkappa}(u) \mathsf{F}^*_{\mu /
		\varkappa}(v)$ is the weight of a configuration $\mu \succ \varkappa \prec
		\lambda$ in a vertex model obtained attaching a $w_{u,s}$-weighted row of
		vertices on top of a $w^*_{v,s}$-weighted row of vertices. The
		leftmost column is occupied by infinitely many paths. The factor
		$(1-quv)/(1-uv)$ is the $R_{uv}$ weight of a cross
		\begin{tikzpicture}[baseline=-5,scale=0.5]
    	\draw[fill] (0,0) circle [radius=0.025];
        \draw[dotted, gray] (0.05,-0.05) -- (0.5, -0.5);
        \draw[ red] (0.05, 0.05) -- (0.5, 0.5);
        \draw[dotted, gray] (-0.05, 0.05) -- (-0.5, 0.5);
        \draw[red] (-0.05, -0.05) -- (-0.5, -0.5);
        \addvmargin{1mm}
    \end{tikzpicture}
    attached at the left of the lattice. 
		Now we employ the definition of $\mathsf{U}^{\mathrm{fwd}}_{\mathrm{sHL}(u),\mathrm{sHL}(v)}$
		and drag the cross all the way to the right, replacing 
		$\varkappa$ by the random $\nu$. 
		This procedure, along with the local
		reversibility condition of the bijectivization 
		leaves
		us with the right-hand side of the desired identity
		\eqref{eq:reversibility_sHL_sHL}.
\end{proof}

The step boundary conditions are 
$\lambda^{(0,y)}=\lambda^{(x,0)}=0^{\infty}=\varnothing$,
and using the forward transition operators
$\mathsf{U}^{\mathrm{fwd}}_{\mathrm{sHL}(u_y),\mathrm{sHL}(v_x)}$
as described in 
\Cref{sub:F_G_transitions}, 
we generate the \emph{the sHL/sHL Yang-Baxter field}
$\boldsymbol \lambda = \{ \lambda^{(x,y)} : x,y \in
\mathbb{Z}_{\geq 0} \}$.
Its distributions are related to the sHL functions:
\begin{proposition}
	The single-point distributions in the sHL/sHL Yang-Baxter
	field with the step boundary conditions have the form
	\begin{equation*}
			\mathrm{Prob}\bigl( \lambda^{(x,y)} = \nu \bigr) 
			= 
			\prod_{\substack{ 1 \leq i \leq x\\1\leq j \leq y}}
			\frac{1 - u_j v_i}{ 1 - q u_j v_i } 
			\,
			\mathsf{F}_\nu(u_1, \dots, u_y)
			\,
			\mathsf{F}^*_\nu(v_1, \dots, v_x),
	\end{equation*}
	where $\nu$ is an arbitrary fixed Young diagram.
	Moreover, joint distributions 
	in this field
	along down-right paths are expressed through 
	products of skew sHL functions as in
	\Cref{prop:F_G_processes}.
\end{proposition}

\subsubsection{Scaled geometric boundary conditions}
\label{ssub:sHL_sHL_two_sided}

Fix additional parameters 
$\alpha,\beta \in [0,-s^{-1}]$,
and 
consider specializations $\rho^{\mathrm{v}}_i$, $\rho^{\mathrm{h}}_i$, $i=-1,0,1,\ldots $:
\begin{equation*}
	\rho^{\mathrm{v}}_{-1}=\mathrm{sg}(\alpha),\quad  
	\rho^{\mathrm{h}}_{-1}=\mathrm{sg}(\beta),
	\qquad 
	\rho^{\mathrm{v}}_y=\mathrm{sHL}(u_y),\quad
	\rho^{\mathrm{h}}_x=\mathrm{sHL}(v_x),\quad 
	x,y\ge0.
\end{equation*}
Let $\boldsymbol \eta$ be the Yang-Baxter field on the lattice $\mathbb{Z}_{\geq -1} \times \mathbb{Z}_{\geq -1}$ generated by the Markov transition operators 
$\mathsf{U}_{\rho^{\mathrm{v}}_y, \rho^{\mathrm{h}}_x}^{\mathrm{fwd}}$. 

\begin{remark}
	\label{rmk:sqW_spec_careful_S6_concrete}
	In defining forward transition operators 
	for scaled geometric (or later
	spin $q$-Whittaker) specializations
	by dragging the 
	cross vertex $(J,0;J,0)$ 
	to the right (as explained in \Cref{sub:cross_multiple_dragging})
	we encounter the issue 
	that the 
	number of paths 
	$J$ should be specialized via $q^J=1/\epsilon$, $\epsilon\to0$,
	and so the vertex $(J,0;J,0)$ no longer makes direct sense.
	However, by 
	\Cref{prop:emergence_of_the_cross_vertex_weight}
	we explicitly know the weight of $(J,0;J,0)$,
	which is equal to
	\begin{equation*}
		R_{uv}^{(I,J)}(J,0;J,0)=
		\frac{(uv q^I;q)_{\infty} (uv q^J ; q)_\infty}
		{(uv;q)_\infty (uvq^{I+J}; q)_\infty}.
	\end{equation*}
	This expression can readily be taken to the 
	scaled geometric or the spin $q$-Whittaker specialization
	(cf. \Cref{fig:table_positivity_YBE}
	for explicit forms of the specializations).
	Therefore, in choosing the bijectivization
	of the Yang-Baxter equation
	in the leftmost column we can still appeal to 
	\Cref{example:biject:A1},
	and conclude that the bijectivization is unique.
\end{remark}

Restricting the field $\boldsymbol \eta$ to the nonnegative quadrant, denote
$\boldsymbol \lambda = \boldsymbol \eta
\vert_{_{\mathbb{Z}_{\geq 0} \times \mathbb{Z}_{\geq 0}}}$.
We call $\boldsymbol\lambda$ the \emph{sHL/sHL Yang-Baxter field
with $(\alpha, \beta)$-scaled geometric boundary
conditions}. 

\begin{proposition} 
	The single-point distributions 
	in the sHL/sHL Yang-Baxter field with the $(\alpha,\beta)$-scaled geometric
	boundary conditions are given by 
	\begin{equation*}
			\mathrm{Prob}\{ \lambda^{(x,y)} = \nu \} = 
			\frac{(\alpha \beta ; q)_\infty}{\prod\limits_{j=1}^y (1+ u_j \beta) 
			\prod\limits_{i=1}^x (1+ v_i \alpha)}\,
			\prod_{\substack{1\le i \le x\\ 1\le j \le y}} 
			\frac{1-u_j v_i}{1-q u_j v_i}\,
			\mathsf{F}_\nu(u_1, \dots, u_y; \widetilde{\alpha}) 
			\mathsf{F}^*_\nu(v_1, \dots, v_x; \widetilde{\beta}).
	\end{equation*}
	Joint distributions 
	in this field
	along down-right paths are expressed through 
	products of skew functions similarly to 
	\Cref{prop:F_G_processes}.
\end{proposition}

\subsubsection{Stochastic six vertex model}
\label{ssub:6v}

We now turn to the scalar Markov marginal 
$\mathsf{h}(\lambda^{(x,y)})=\ell(\lambda^{(x,y)})$
of the sHL/sHL Yang-Baxter field, and match it to the 
stochastic six vertex model.

\begin{figure}[htbp]
    \centering
		\includegraphics[width=.9\textwidth]{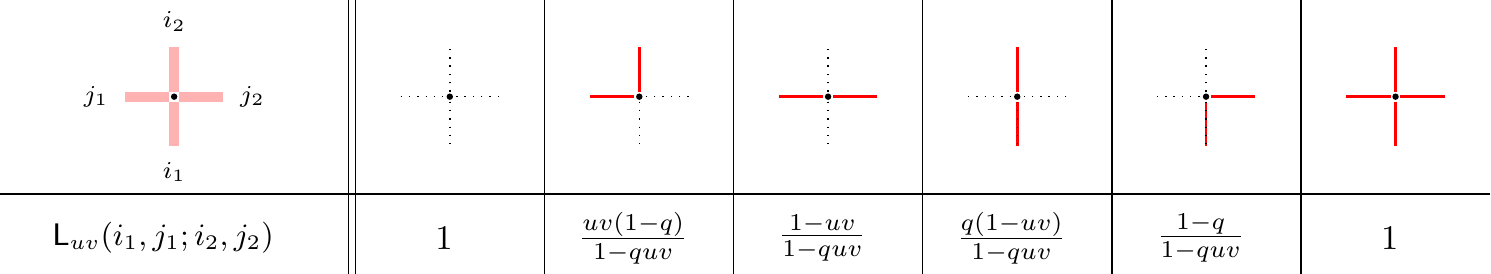}
    \caption{The vertex weights 
		$\mathsf{L}_{uv}(i_1, j_1; i_2, j_2)$
		in the stochastic six vertex model.
		This parametrization of the vertex
		weights follows, e.g.,
		\cite{BufetovMatveev2017}.}
    \label{fig:weights_6VM}
\end{figure}

\begin{definition}[\cite{GwaSpohn1992}, \cite{BCG6V}] \label{def:6vm}
	Fix $q \in (0,1)$ and $u_y,v_x$ such that $0<u_y v_x<1$ for all $x,y \in
	\mathbb{Z}_{\geq 1}$. 
	Consider the stochastic vertex weights
	$\mathsf{L}_{u_yv_x}$ given in \Cref{fig:weights_6VM}. 
	Let us also fix the boundary conditions
	$B^\mathrm{h}=\{b^\mathrm{h}_1,b^\mathrm{h}_2, \dots\}$ and
	$B^\mathrm{v}=\{b^\mathrm{v}_1,b^\mathrm{v}_2, \dots\}$,
	where $b_i^{\mathrm{h}}, b_j^{\mathrm{v}}\in\left\{ 0,1 \right\}$.
	The (inhomogeneous)
	\emph{stochastic six vertex model} with 
	these boundary conditions
	is the (unique) probability measure on the set of up-right
	directed paths on the lattice $\mathbb{Z}_{\geq 0} \times \mathbb{Z}_{\geq 0}$ 
	(with at most one path per vertical or horizontal edge)
	satisfying:
\begin{itemize}
    \item 
			Each vertex $(0,y)$ at the vertical boundary $\{ (0,y'):y'\geq 1 \}$
			emanates a path initially pointing to the right if $b^\mathrm{v}_y=1$;
		\item 
			Each vertex $(x,0)$ at the horizontal boundary $\{ (x',0):x'\geq 1 \}$
			emanates a path initially pointing upward if 
			$b^\mathrm{h}_x=1$;
		\item 
			For each $(x,y)$, conditioned to the path configuration at all vertices
			$(x',y')$ such that $x'+y'<x+y$, the probability of a vertex
			configuration $(i_1,j_1;i_2,j_2)$
			at $(x,y)$
			is given by
			$\mathsf{L}_{u_yv_x}(i_1,j_1;i_2,j_2)$.
			Moreover, the random choices
			made at diagonally adjacent vertices
			$\ldots,(x-1,y+1),(x,y),(x+1,y-1),\ldots$
			are independent under the same condition.
\end{itemize}
In particular, the \emph{step boundary conditions}
correspond to 
\begin{equation} \label{eq:step_bc}
    b^\mathrm{h}_x = 0 \qquad \text{and} \qquad b^\mathrm{v}_y = 1, \qquad \text{for all }x,y\geq 1.
\end{equation}
\end{definition}

Path configurations in the six vertex model with the step boundary conditions are
conveniently encoded by a height function.
Namely, let $\mathfrak{h}^{\mathrm{6V}}(x,y)$
denote the number of paths which pass weakly to the right of the vertex $(x,y)$.
See \Cref{fig:height_6VM_length_sHL_sHL}, right, for an illustration.
The next theorem is a version of 
\cite[Proposition~7.3]{BufetovPetrovYB2017} adapted to our 
boundary conditions in the Yang-Baxter field.

Let $\boldsymbol\lambda=\{\lambda^{(x,y)}\}$
be the sHL/sHL Yang-Baxter field with the step boundary
conditions,
and
$\mathfrak{h}^{\mathrm{6V}}(x,y)$ be the six vertex
height function with the 
step boundary conditions \eqref{eq:step_bc}.

\begin{theorem} \label{thm:length_sHL_sHL_height_6VM}
		The two random fields $\{ y - \ell (\lambda^{(x,y)}):x,y \in
		\mathbb{Z}_{\geq 0} \}$ 
		and $\{ \mathfrak{h}^{\mathrm{6V}}(x+1,y) :x,y \in
		\mathbb{Z}_{\geq 0} \}$ are equal in distribution.
\end{theorem}
\begin{proof}
		Recall that the Markov evolution of this scalar marginal
		$\ell(\lambda^{(x,y)})$
		corresponds to the quantities
		$\mathsf{U}^{[0]}(i_0,i_0';j_0',j_0)$ \eqref{eq:U_00} 
		with $I=J=1$ and $i_0,i_0',j_0,j_0'\in\left\{ 0,1 \right\}$
		which can be readily written down using
		\eqref{eq:w_boundary}, 
		\eqref{eq:w_w_tilde_relation},
		and
		\Cref{fig:table_R}.
		Comparing these quantities
		to the stochastic six vertex weights 
		$\mathsf{L}_{uv}$ in \Cref{fig:weights_6VM},
		while taking into account
		the relation 
		between 
		$i_0,i_0',j_0,j_0'$
		and $\ell(\lambda^{(x,y)})$
		\eqref{eq:occupation_numbers_lengths},
		and all the transformations in the statement, we see that 
		$\mathsf{U}^{[0]}_{\mathrm{sHL}(u),\mathrm{sHL}(v)}(i_0,1-i_0';j_0',1-j_0)
		=
		\mathsf{L}_{uv}(i_0,i_0';j_0',j_0)$ 
		for all
		$i_0,i_0',j_0,j_0'\in\left\{ 0,1 \right\}$.

		Let us consider one 
		case of 
		$(i_0,i_0';j_0,j_0')=(0,1;1,0)$ for illustration.
		From 
		\Cref{fig:weights_6VM}
		we have $\mathsf{L}_{uv}(0,1;1,0)=\frac{uv(1-q)}{1-quv}$.
		Then 
		\begin{equation*}
			\begin{split}
				\mathsf{U}^{[0]}_{\mathrm{sHL}(u),\mathrm{sHL}(v)}
				(0,0;1,1)&=
				\frac{ w^{*}_{v,s} (\infty, 0; \infty, 1)\,
				w_{u,s} (\infty, 1; \infty, 1)\,
			R^{(1,1)}_{uv} (1,1; 0, 0) }
				{ R^{(1,1)}_{uv} (1, 0; 1, 0)\,
				w^{*}_{v,s} (\infty, 0; \infty, 0)\,
				w_{u,s} (\infty, 1; \infty, 0) }
				\\&=
				\frac{ v\cdot
				u\cdot
				\frac{1-q}{1-uv} }
				{\frac{1-q uv}{1-uv}
					\cdot 1\cdot 1
				}=\frac{uv(1-q)}{1-quv},
			\end{split}
		\end{equation*}
		as desired.
		All other cases are analogous.
\end{proof}
\begin{figure}[htbp]
    \centering
		\includegraphics[width=.323\textwidth]{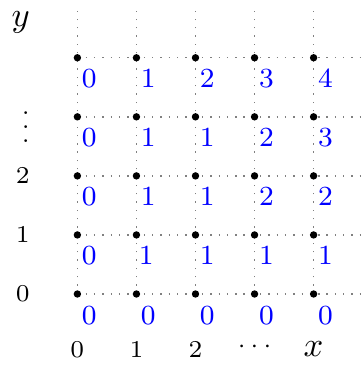}
		\qquad \qquad 
		\includegraphics[width=.36\textwidth]{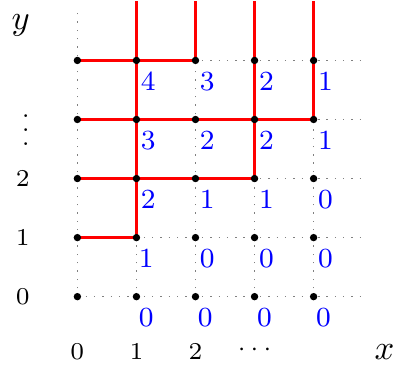}
		\caption{Left: 
		the scalar marginal
		$\ell(\lambda^{(x,y)})$ of the sHL/sHL Yang-Baxter field.
		Right: 
		the corresponding
		realization of the stochastic six vertex model with the step boundary conditions.
		The height function 
		$\mathfrak{h}^{\mathrm{6V}}$
		is indicated, too.}
    \label{fig:height_6VM_length_sHL_sHL}
\end{figure}
\begin{remark}
While \Cref{thm:length_sHL_sHL_height_6VM}
essentially follows from
\cite{BufetovPetrovYB2017}, let us emphasize
what is different here.
\Cref{thm:length_sHL_sHL_height_6VM}
connects the ordinary stochastic six vertex model
with stable spin Hall-Littlewood polynomials.
On the other hand,
previously the same
stochastic six vertex model
was matched to Hall-Littlewood processes and measures
\cite{BorodinBufetovWheeler2016},
\cite{BufetovMatveev2017},
and the 
non-stable spin Hall-Littlewood polynomials
gave rise to a dynamic version of the stochastic six 
vertex model \cite{BufetovPetrovYB2017}.
Therefore, formally 
\Cref{thm:length_sHL_sHL_height_6VM}
is a new statement.

It might seem surprising that the 
stochastic six vertex model $\mathfrak{h}^{\mathrm{6V}}$
is independent of $s$, while 
the field $\lambda^{(x,y)}$ depends on $s$.
A posteriori this might be explained by the fact that the
$s$-dependent stable spin Hall-Littlewood 
functions are eigenfunctions
of the same first $q=0$ Macdonald difference
operator as their $s=0$ versions, the classical 
Hall-Littlewood polynomials.
About difference operators see 
\Cref{sec:diff_op} below.
\end{remark}

Let us now turn to the marginal 
of the sHL/sHL field with the two-sided scaled geometric
boundary conditions described in \Cref{ssub:sHL_sHL_two_sided}. 
First, consider the behavior of $\ell(\lambda^{(x,y)})$ at the boundary:
\begin{proposition}
	\label{prop:sHL_sHL_stat_bc}
Consider the transition probabilities $\mathsf{U}^{[0]}$ given in \eqref{eq:U_00}. Then we have
\begin{align}
	 & \mathsf{U}^{[0]}_{\mathrm{sg}(\alpha),\mathrm{sHL}(v)} (0, i_0'; j_0', j_0)  
	 = 
	 \mathbf{1}_{i_0' + j_0' = j_0} \frac{ (v \alpha)^{j_0'}}{1 + v \alpha}, 
	 \label{eq:U00_bernoulli_horiz}\\
	 & \mathsf{U}^{[0]}_{\mathrm{sHL}(u),\mathrm{sg}(\beta)} (i_0, 0; j_0', j_0)  
	 = 
	 \mathbf{1}_{i_0 + j_0 = j_0'} \frac{ (u \beta)^{j_0}}{1 + u \beta}, 
	 \label{eq:U00_bernoulli_vert} \\
	 & \mathsf{U}^{[0]}_{\mathrm{sg}(\alpha),\mathrm{sg}(\beta)} (0, 0; j_0', j_0) 
	  =
	   \mathbf{1}_{j_0' = j_0} \frac{ (\alpha \beta)^{j_0}}{(q;q)_{j_0}} (\alpha \beta ; q)_\infty.
	    \label{eq:U00_qPoisson}
\end{align}
\end{proposition}
The cases \eqref{eq:U00_bernoulli_horiz}, \eqref{eq:U00_bernoulli_vert}
correspond to the bottom and the left boundaries, respectively, and 
\eqref{eq:U00_qPoisson} arises in the bottom left corner.
Observe that \eqref{eq:U00_qPoisson} defines the $q$-Poisson
distribution (cf. \Cref{sub:notation}).
\begin{proof}[Proof of \Cref{prop:sHL_sHL_stat_bc}]
    This follows by combining 
		\eqref{eq:U_00}
		with the formulas for the boundary weights
		\eqref{eq:w_boundary}, \eqref{eq:w_tilde_infinity}
		and the cross vertex weights (\Cref{fig:table_R_sHL_sg} and \eqref{eq:R_sg_sg}).
		To specialize the
		factor corresponding to $R^{(I,J)}_{u,v}(J,0;J,0)$
		one should use \Cref{rmk:sqW_spec_careful_S6_concrete}.
\end{proof}

Now take the stochastic six vertex model
with independent Bernoulli boundary conditions
(we call these the \emph{two-sided stationary} or \emph{$(\alpha,\beta)$-stationary boundary conditions}):
\begin{equation} \label{eq:6vm_bernoulli_bc}
    b^{\mathrm{h}}_x \sim \mathrm{Ber}\left( \frac{v_x \alpha}{1+ v_x \alpha} \right) 
		\qquad 
		\text{and} \qquad b^{\mathrm{v}}_y \sim \mathrm{Ber}\left( \frac{1}{1+ u_y \beta} \right).
\end{equation}
That is, given a realization of these random variables,
we then consider the stochastic six vertex model with these boundary
conditions according to 
\Cref{def:6vm}.
While the random path configuration in this model is well-defined, 
it cannot be encoded by the height
function $\mathfrak{h}^{\mathrm{6V}}$
in the same way as for the step boundary conditions.
Indeed, if $\alpha>0$, the 
number of paths to the right of any vertex $(x,y)$ is almost surely infinite.
Let us thus introduce the
\emph{centered height function}
\begin{equation}
	\label{eq:centered_height_function}
	\mathcal{H}^{\mathrm{6V}}(x,y)
	=
	\#\{\textnormal{occupied horizontal edges}\}
	-
	\#\{\textnormal{occupied vertical edges}\},
\end{equation}
where we count the edges along a directed up-right 
sequence of cells in the lattice, for example, moving 
$(\frac{1}{2},\frac{1}{2}) \to
(x+\frac{1}{2},\frac{1}{2}) \to (x+\frac{1}{2}, y+\frac{1}{2})$
along straight lines. In other words,
$\mathcal{H}^{\mathrm{6V}}$
has the same gradient as 
$\mathfrak{h}^{\mathrm{6V}}$,
but the constant is defined by 
$\mathcal{H}^{\mathrm{6V}}(0,0)=0$.
The next lemma is a straightforward observation:

\begin{lemma}
	The centered height function 
	$\mathcal{H}^{\mathrm{6V}}(x,y)$ well-defined
	and almost sure finite for all $(x,y)\in \mathbb{Z}_{\ge0}\times \mathbb{Z}_{\ge0}$.
	For $\alpha=\beta=0$ the boundary conditions \eqref{eq:6vm_bernoulli_bc}
	reduce to the step boundary conditions \eqref{eq:step_bc},
	and in this case we have
	$\mathcal{H}^{\mathrm{6V}}(x,y) = \mathfrak{h}^{\mathrm{6V}}(x+1,y)$
	for all $x,y$.
\end{lemma}

The centered height function with the two-sided Bernoulli boundary conditions
\eqref{eq:6vm_bernoulli_bc}
can be identified with a marginal of the sHL/sHL Yang-Baxter
field $\boldsymbol\lambda$
with scaled geometric boundary conditions.

\begin{theorem} \label{thm:sHL_sHL_mixed_height_6VM}
	Let $\mathcal{M}$ be the $q$-Poisson random variable with parameter $\alpha\beta$ independent of the stochastic six vertex model
	with $(\alpha,\beta)$-stationary boundary conditions.
	The two random fields $\{ y - \ell (\lambda^{(x,y)})
	: x,y \in \mathbb{Z}_{\ge 0} \}$
	and $\{\mathcal{H}^{\mathrm{6V}}(x,y) - \mathcal{M} : x,y \in \mathbb{Z}_{\ge 0} \}$
	are
	equal in distribution.
\end{theorem}

\Cref{thm:sHL_sHL_mixed_height_6VM} 
follows in essentially the same way as
\Cref{thm:length_sHL_sHL_height_6VM} by matching the value of
vertex weights $\mathsf{L}_{u_y,v_x}$ and probability laws of entries
$b^{\mathrm{h}}_x, b^{\mathrm{v}}_x$ with those given $\mathsf{U}^{[0]}$
(on the boundary this follows from \Cref{prop:sHL_sHL_stat_bc};
in fact, the 
structure of the concrete formulas
\eqref{eq:U00_bernoulli_horiz},
\eqref{eq:U00_bernoulli_vert}
is essential for the independent boundary conditions).
Let us present a slightly different argument that uses analytic
continuation. This alternative approach is useful 
in other situations (\Cref{sub:new_YB_field_sHL_sqW,sub:new_YB_field_sqW_sqW})
and also in 
\Cref{sub:fredholm_determinant_for_marginal_processes}
for computation of observables of models with two-sided stationary boundary conditions.

\begin{proof}[Proof of \Cref{thm:sHL_sHL_mixed_height_6VM}]
    Consider the sHL/sHL Yang-Baxter field 
		with the usual step boundary conditions, 
		and with shifted indices:
    \begin{equation*}
			\boldsymbol \mu = \{ \mu^{(x,y)} : (x,y) \in \mathbb{Z}_{\ge -I_0 +1} \times \mathbb{Z}_{\ge -J_0 +1} \}.
    \end{equation*}
		Here 
		$I_0,J_0$ are positive integers.
		For $\boldsymbol\mu$ we take the following specializations:
    \begin{equation} \label{eq:parameters_u0_v0_principal}
        u_0, q u_0, \dots q^{J_0-1} u_0, u_1, u_2, \dots \qquad \text{and} \qquad v_0, q v_0, \dots q^{I_0-1}v_0, v_1,v_2, \dots.
    \end{equation}
		
    Call $\boldsymbol \nu = \boldsymbol \mu\vert_{_{\mathbb{Z}_{\geq 0} \times \mathbb{Z}_{\geq 0}}}$ 
		the restriction of $\boldsymbol \mu$ to the nonnegative quadrant.
		Then $\boldsymbol \nu$ is a field
		of random Young diagrams associated to the sHL/sHL skew Cauchy structure with Gibbs
		boundary conditions (\Cref{def:F_G_boundary_conditions}).
		That is, for all $x,y$, the boundary Young
		diagrams $\nu^{(0,y)}, \dots, \nu^{(0,0)}, \dots,
		\nu^{(x,0)}$ are distributed with law
    \begin{equation}
        \frac{1}{Z_{\mathrm{boundary}}^{(x,y)}} 
				\prod_{j=1}^y \mathsf{F}_{\nu^{(0,j)} / \nu^{(0,j-1)} }(u_j) \mathfrak{G}^{(I_0)}_{\nu^{(0,y)}}(v_0)
        \prod_{i=1}^x \mathsf{F}^*_{\nu^{(i,0)} / \nu^{(i-1,0)} }(v_i) \mathfrak{F}^{(J_0)}_{\nu^{(x,0)}}(u_0),
    \end{equation}
    recalling the notation introduced in Section \ref{sec:analytic_continuation}.
	The vertex weights in the definition of the principal specialization of the sHL functions $ \mathfrak{G}_{\nu^{(0,y)}}$ and $\mathfrak{F}_{\nu^{(x,0)}}$ depend on the parameters $u_0,q^{J_0},v_0,q^{I_0}$ in a rational way. 
	Therefore, for any bounded complex-valued cylindric function $f : \boldsymbol \nu \mapsto f(\boldsymbol \nu)$,\footnote{Here ``cylindric'' means that the function depends on $\boldsymbol\nu$ only through the diagrams $\nu^{(x_i,y_i)}$, where $(x_i,y_i)$ run over a finite set (and the set may depend on $f$).} the expected value $\mathbb{E}_{\boldsymbol \nu}(f)$ is a holomorphic function of $u_0, q^{J_0}, v_0, q^{I_0}$ when 
	these parameters are in a small neighborhood of zero. 

    Let $\boldsymbol\lambda$ be the sHL/sHL field with $(\alpha,\beta)$-scaled geometric boundary conditions.
	The above argument shows that the probability of any event depending on a finite region in the field $\boldsymbol\lambda$ is equal to the scaled geometric degeneration 
		\begin{equation*}
			u_0=-\epsilon \alpha, \quad q^{J_0}=1/\epsilon,\qquad 
			v_0=-\epsilon \beta, \quad q^{I_0}=1/\epsilon,\qquad 
			\epsilon\to0,
		\end{equation*}
		of the probability of the same event
		in which the field $\boldsymbol\lambda$ is replaced by $\boldsymbol\nu$.

		Consider now the stochastic six vertex model with the step boundary conditions on
		the shifted lattice $\mathbb{Z}_{\ge -I_0+1} \times \mathbb{Z}_{\ge -J_0+1}$,
		which corresponds to the field $\boldsymbol\mu$ (with parameters \eqref{eq:parameters_u0_v0_principal}).
		Refer to its height function by
		$\mathfrak{h}^{\mathrm{6V}(I_0,J_0)}$. 
		By
		\Cref{thm:length_sHL_sHL_height_6VM}, 
		we have equality in distribution
    \begin{equation} \label{eq:h_y_J0_l}
			\mathfrak{h}^{\mathrm{6V}(I_0,J_0)}(x,y) \stackrel{d}{=} y + J_0 - \ell (\nu^{(x,y)}) \qquad \text{for all } x,y \geq 0
    \end{equation}
		(here $y+J_0$ is simply the shifted vertical coordinate, and $\nu^{(x,y)}=\mu^{(x,y)}$ for $x,y\ge0$).
		Next, let
		$\mathcal{H}^{\mathrm{6V}(I_0,J_0)}(x,y)$ be the centered height function of the restriction of the above vertex model
		to the nonnegative quadrant $\mathbb{Z}_{\ge 0}^2$.
		Then 
    \begin{equation} \label{eq:h_J0_H_M}
			\mathcal{H}^{6\mathrm{V}(I_0,J_0)}(x,y)=
			\mathfrak{h}^{\mathrm{6V}(I_0,J_0)}(x,y) - J_0 
			+ 
			\mathcal{M}_{I_0, J_0} 
			\qquad \text{for all } x,y \geq 0,
    \end{equation}
		where $\mathcal{M}_{I_0,J_0}$ is the random variable counting the number of
		paths originating from the segment $\{-I_0+1\} \times [-J_0 +1, 0]$ and
		vertically crossing the segment $[-I_0 +1, 0] \times \{0 \}$. Combining
		\eqref{eq:h_y_J0_l} and \eqref{eq:h_J0_H_M}, we find that 
		\begin{equation}\label{eq:sHL_sHL_stable_proof_1}
        y - \ell (\nu^{(x,y)}) \stackrel{d}{=} \mathcal{H}^{6\mathrm{V}(I_0,J_0)}(x,y) - \mathcal{M}_{I_0, J_0} \qquad \text{for all }x,y \ge 0.
    \end{equation}

    The probability law of $\mathcal{M}_{I_0, J_0}$ 
		is found from the
		sHL/sHL
		field $\boldsymbol\mu$ which has step boundary conditions
		(hence we can use \Cref{thm:length_sHL_sHL_height_6VM}).
		On the other hand, the update in the initial
		$(I_0,J_0)$ part of $\boldsymbol\mu$
		is restated as a single forward transition
		in the $(I_0,J_0)$-fused field (considered in \Cref{sec:YB_fields_through_bijectivisation} above).
		Therefore, the law of 
		$\mathcal{M}_{I_0, J_0}$ 
		is given by 
		\eqref{eq:U_00}
		with parameters $u_0,q^{J_0},v_0,q^{I_0}$:
    \begin{equation}\label{eq:lenght_sHL_sHL_centered_H}
			\mathrm{Prob} \{ \mathcal{M}_{I_0, J_0} = k \} = \mathsf{U}^{[0]}_{u_0,v_0}(0,0;k,k) \qquad \text{for } k=0 \dots,I_0.
    \end{equation}
		Under the scaled geometric specializations to both $\alpha$ and $\beta$, this distribution becomes
		$q\textnormal{-}\mathrm{Poi}(\alpha\beta)$, cf. \eqref{eq:U00_qPoisson}.
		Taking the scaled geometric specializations in 
		\eqref{eq:sHL_sHL_stable_proof_1},
		we obtain the desired matching between the centered height function 
		$\mathcal{H}^{6\mathrm{V}}$ of the stochastic six vertex model 
		with the $(\alpha,\beta)$-stationary boundary conditions
		and the marginal of the field $\boldsymbol\lambda$.
\end{proof}

\subsection{The sHL/sqW Yang-Baxter field} \label{sub:new_YB_field_sHL_sqW}

Here we consider the Yang-Baxter field 
associated with the dual Cauchy identity between
the sHL and the sqW functions. 
The marginal of the field is the stochastic higher spin six vertex model.
We consider both step and two-sided stationary boundary conditions in the vertex model.
The model with the step boundary conditions was extensively studied starting from
\cite{CorwinPetrov2015}, 
\cite{BorodinPetrov2016inhom}.
Different formulas for observables in the two-sided stationary case
leading to asymptotic results
were obtained recently in \cite{imamura2019stationary}
by a different method.

\subsubsection{Step boundary conditions}

The sHL/sqW field corresponds to setting $v=s$ and $q^I=-\theta/s$ in the notation of \Cref{sec:YB_fields_through_bijectivisation}.
The parameters are $q\in(0,1)$, $s\in(-1,0)$, 
$u\in [0,1)$, $\theta\in[-s,-s^{-1}]$.
The Yang-Baxter equation governing the vertex weights is \eqref{eq:YBE_W_w}.\footnote{Equivalently, 
one could consider $u=s$, $q^J=-\xi/s$, and take $(\xi,v)$ as the parameters.
This leads to a straightforward rewriting of some of the formulas, 
but produces the same marginal process (
cf. \Cref{rmk:sHL_sqW_conjugation_does_not_hurt}).
Therefore, we only consider one of the two dual cases.}
The reversibility condition of the forward and backward transition operators is 
proven in the same way as \Cref{prop:reversibility_sHL_sHL}, and is
given as follows:
\begin{proposition}
For any four Young diagrams $\mu, \varkappa, \lambda, \nu$ we have
\begin{multline} \label{eq:reversibility_sHL_sqW}
	\frac{1 + u \theta}{1 - u s}\,
	\mathsf{U}^{\mathrm{fwd}}_{\mathrm{sHL}(u),\mathrm{sqW}(\theta)} (\varkappa \to \nu \mid \lambda,\mu) \,
	\,
	\mathsf{F}_{\lambda / \varkappa}(u) \mathbb{F}^*_{\mu' / \varkappa'} (\theta) 
	\\=
	\mathsf{U}^{\mathrm{bwd}}_{\mathrm{sHL}(u),\mathrm{sqW}(\theta)} ( \nu \to \varkappa \mid \lambda,\mu) \,
	\,
	\mathbb{F}^*_{\nu'  /\lambda'}(\theta) \mathsf{F}_{\nu/\mu}(u)
\end{multline}
Summing \eqref{eq:reversibility_sHL_sqW} 
over both $\varkappa$ and $\nu$, we obtain
the skew Cauchy identity of
\Cref{thm:sHL_sqW_skew_Cauchy}.
\end{proposition}

The \emph{sHL/sqW Yang-Baxter field} 
$\boldsymbol\lambda=\{ \lambda^{(x,y)} \}$
depends on the parameters
$u_y\in[0,1)$, $\theta_x\in[-s,-s^{-1}]$, $x,y\in \mathbb{Z}_{\ge1}$,
and is generated from the step boundary conditions 
$\lambda^{(x,0)}=\lambda^{(0,y)}=0^{\infty}=\varnothing$ 
by applying the forward transition operators 
$\mathsf{U}^{\mathrm{fwd}}_{\mathrm{sHL}(u_y),\mathrm{sqW}(\theta_x)}$.

\begin{proposition} \label{prop:sHL_sqW_YBF}
	The single-point point distributions in the 
	sHL/sqW field with the step boundary conditions have the form
	\begin{equation*}
		\mathrm{Prob}( \lambda^{(x,y)} = \nu ) =
		\prod_{\substack{ 1 \leq i \leq x \\ 1\leq j \leq y}} \frac{1 - u_j s}{ 1 + u_j \theta_i } \,
		\mathsf{F}_\nu(u_1, \dots, u_y)\,
		\mathbb{F}^*_{\nu'}(\theta_1, \dots, \theta_x).
	\end{equation*}
The joint distributions along down-right paths are expressed through the 
skew functions as in \Cref{prop:F_G_processes}.
\end{proposition}

\subsubsection{Scaled geometric boundary conditions}
\label{ssub:sHL_sqW_two_sided}

Take additional parameters $\alpha,\beta\in[0,-s^{-1}]$,
and consider specializations
\begin{equation*}
	\rho^{\mathrm{v}}_{-1}=\mathrm{sg}(\alpha),\qquad 
	\rho^{\mathrm{h}}_{-1}=\mathrm{sg}(\beta),\qquad 
	\rho^{\mathrm{v}}_y=\mathrm{sHL}(u_y),\qquad 
	\rho^{\mathrm{h}}_x=\mathrm{sqW}(\theta_x).
\end{equation*}
Let $\boldsymbol \eta$ be the Yang-Baxter field on the lattice 
$\mathbb{Z}_{\geq -1} \times \mathbb{Z}_{\geq -1}$ 
generated by the forward transition probabilities constructed using the specializations
(dragging the cross vertex through the leftmost column
should be understood as in \Cref{rmk:sqW_spec_careful_S6_concrete}).
Restricting this field to the nonnegative quadrant,
$\boldsymbol \lambda = \boldsymbol \eta
\vert_{_{\mathbb{Z}_{\geq 0} \times \mathbb{Z}_{\geq 0}}}$,
we get the \emph{sHL/sqW field with the two-sided scaled geometric} (or \emph{$(\alpha,\beta)$-scaled geometric}) \emph{boundary conditions}.

\begin{proposition} \label{prop:sHL_sqW_stat_bc}
	For the field $\boldsymbol\lambda$ defined above we have
	\begin{equation*}
			\mathrm{Prob}\{ \lambda^{(x,y)} = \nu \} 
			=
			\frac{(\alpha \beta ; q)_\infty}
			{\prod\limits_{j=1}^y (1+u_j\beta)}
			\prod_{i=1}^{x}
			\frac{(\alpha \theta_i ; q)_\infty}
			{(-s \alpha ;q)_\infty}
			\prod_{\substack{1\le i \le x\\ 1\le j \le y}}
			\frac{1-u_j s}{1+ u_j \theta_i} \,
			\mathsf{F}_\nu(u_1, \dots, u_y; \widetilde{\alpha}) \,
			\mathbb{F}^*_{\nu'}(\theta_1, \dots, \theta_x; \widetilde{\beta}).
	\end{equation*}
	Joint distributions in
	$\boldsymbol\lambda$
	along down-right paths are expressed similarly to 
	\Cref{prop:F_G_processes}.
\end{proposition}

\subsubsection{Stochastic higher spin six vertex model}
\label{ssub:shv}

The Markovian marginal of the sHL/sqW 
field can be mapped to a known stochastic vertex model
which we now recall.
Let the vertex weights $\mathcal{L}_{u_y,\theta_x}(i_1,j_1;i_2,j_2)$,
$i_1,i_2\in \mathbb{Z}_{\ge0}$,
$j_1,j_2\in \left\{ 0,1 \right\}$,
be given in
\Cref{fig:weights_HS6VM}.
They are stochastic for our values of parameters
in the sense that
$\sum_{i_2,j_2}\mathcal{L}_{u_y,\theta_x}(i_1,j_1;i_2,j_2)=1$
for all $i_1,j_1$.

\begin{figure}[htbp]
    \centering
		\includegraphics[width=.8\textwidth]{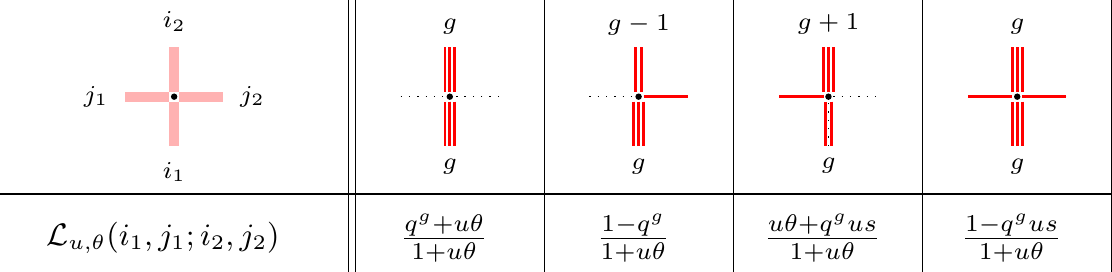}
    \caption{Stochastic vertex weights 
		$\mathcal{L}_{u,\theta}(i_1,j_1;i_2,j_2)$.
		This parametrization of the weights differs
		from the ones employed in \cite{CorwinPetrov2015}
		or \cite{BorodinPetrov2016inhom}, but all these parametrizations
		are related to each other via simple changes of variables.}
    \label{fig:weights_HS6VM}
\end{figure}

\begin{definition}[\cite{CorwinPetrov2015}, \cite{BorodinPetrov2016inhom}] 
	\label{def:hs6vm}
	The (inhomogeneous) \emph{stochastic higher spin six vertex model}
	with the boundary conditions 
	$B^{\mathrm{h}} = \{ b^{\mathrm{h}}_1, b^{\mathrm{h}}_2, \dots \}$ and
	$B^{\mathrm{v}} = \{ b^{\mathrm{v}}_1, b^{\mathrm{v}}_2, \dots \}$,
	$b_i^{\mathrm{v}}\in \left\{ 0,1 \right\}$,
	$b_j^{\mathrm{h}}\in \mathbb{Z}_{\ge0}$,
	is the (unique) probability
	measure on the set of up-right directed paths on 
	$\mathbb{Z}_{\geq 0} \times
	\mathbb{Z}_{\geq 0}$
	(with multiple vertical paths allowed per edge, but at most 
	one horizontal path per edge)
	satisfying:
	\begin{itemize}
			\item 
				Each vertex $(0,y)$ at the vertical boundary $\{ (0,y'):y'\geq 1 \}$
				emanates a path initially pointing to the right if
				$b^\mathrm{v}_y=1$;
			\item Each vertex $(x,0)$ at the horizontal boundary
				$\{ (x',0):x'\geq 1\}$ emanates 
				$b^\mathrm{h}_x$
				paths initially pointing upward;
			\item For each $(x,y)$, conditioned to the path configuration at all
				vertices $(x',y')$ such that $x'+y'<x+y$, the probability of a vertex
				configuration $(i_1,j_1;i_2,j_2)$ at $(x,y)$ is given by
				$\mathcal{L}_{u_y,\theta_x}(i_1,j_1;i_2,j_2)$. Moreover, the random
				choices made at diagonally adjacent vertices
				$\ldots,(x-1,y+1),(x,y),(x+1,y-1),\ldots$ are independent under the same
				condition.
	\end{itemize}
	In particular, the \emph{step boundary conditions} correspond to
	paths entering horizontally at each location and no paths entering through the 
	bottom boundary
	\eqref{eq:step_bc}, exactly as in the stochastic six vertex case
	considered in \Cref{sub:new_YB_field_sHL_sHL}.
\end{definition}

Similarly to the six vertex case, let us encode the 
configuration of paths by the centered height function $\mathcal{H}^{\mathrm{HS}}(x,y)$,
see \eqref{eq:centered_height_function}. That is, 
$\mathcal{H}^{\mathrm{HS}}(0,0)=0$, and the stochastic higher spin six
vertex model paths serve as level 
lines for $\mathcal{H}^{\mathrm{HS}}$. 
For the step boundary conditions we have
$\mathcal{H}^{\mathrm{HS}}(x,y)=\mathfrak{h}^{\mathrm{HS}}(x+1,y)$, 
where $\mathfrak{h}^{\mathrm{HS}}(x,y)$ is the 
number of paths passing weakly to the right of the point $(x,y)\in \mathbb{Z}_{\ge0}^{2}$.

The next proposition suggests the appropriate choice of the 
two-sided stationary boundary conditions for the stochastic higher spin six vertex model:

\begin{proposition}
	\label{prop:two_sided_sHL_sqW}
Consider the transition probabilities $\mathsf{U}^{[0]}$ given in \eqref{eq:U_00}. Then 
\begin{equation}
	\label{eq:U_00_sg_sqW_computation}
	\mathsf{U}^{[0]}_{\mathrm{sg}(\alpha), 
	\mathrm{sqW}(\theta)} (0, i_0'; j_0', j_0)  
	= 
	\mathbf{1}_{i_0' + j_0' = j_0} (\alpha \theta)^{j_0'} \,
	\frac{ (-s/\theta;q)_{j_0'} }{ (q;q)_{j_0'} } \frac{( \alpha \theta ; q )_\infty }{ ( -s \alpha ; q )_\infty }.
\end{equation}
\end{proposition}
\begin{proof}
	This follows by 
	specializing \eqref{eq:U_00} and using the expressions for the boundary weights $W^*,\widetilde{w}$
	(\eqref{eq:W_boundary_weights} and \eqref{eq:w_tilde_infinity}, respectively),
	and the cross weights
	$R^{\mathrm{sqW,sg}}_{\alpha, \theta}$ \eqref{eq:R_sqW_sg}.
	The quantity $R^{(I,J)}_{uv}(J,0;J,0)$ should be specialized
	as described in \Cref{rmk:sqW_spec_careful_S6_concrete}.
\end{proof}

Let us define the 
stochastic higher spin six vertex model
with \emph{two-sided stationary} (or $(\alpha,\beta)$-stationary) boundary conditions
by taking independent random variables on the boundary distributed as
(recall the notation in \Cref{sub:notation})
\begin{equation} \label{eq:hs6vm_NB_Ber_bc}
	b_x^{\mathrm{h}} \sim q\textnormal{-}\mathrm{NB}(-s/\theta_x, \alpha \theta_x) \qquad \text{and} \qquad
	b_y^{\mathrm{v}} \sim \mathrm{Ber}\left( \frac{ 1 }{ 1+ u_x \beta } \right).
\end{equation}
Let $\mathcal{H}^{\mathrm{HS}}(x,y)$ be the corresponding centered height function.

For $\alpha=\beta=0$ the 
boundary conditions \eqref{eq:hs6vm_NB_Ber_bc} reduce to the step one.
When $\beta$ depends on $\alpha$ in a certain way,
the stationarity of the boundary conditions \eqref{eq:hs6vm_NB_Ber_bc}
under the homogeneous stochastic higher spin six vertex
model was checked in a continuous-time degeneration
in \cite[Appendix B.2]{BorodinPetrov2016Exp},
see also \cite{imamura2019stationary} for the full statement and further discussion.

The next result is the analogue of both
\Cref{thm:length_sHL_sHL_height_6VM,thm:sHL_sHL_mixed_height_6VM}
from the stochastic six vertex case.
Let $\boldsymbol\lambda$ be the sHL/sqW Yang-Baxter field with the
$(\alpha,\beta)$-scaled geometric boundary conditions.
\begin{theorem}
	\label{thm:sHL_sqW_height_length}
	Let $\mathcal{M}$ be the $q$-Poisson random variable with parameter $\alpha\beta$
	independent of the stochastic higher spin six vertex model with 
	$(\alpha,\beta)$-stationary boundary conditions.
	Then the two random fields
	$\left\{ y-\ell(\lambda^{(x,y)})\colon x,y\in \mathbb{Z}_{\ge0} \right\}$
	and $\{ \mathcal{H}^{\mathrm{HS}}(x,y)-\mathcal{M}\colon x,y\in \mathbb{Z}_{\ge0} \}$
	are equal in distribution.
\end{theorem}
\begin{proof}
	To obtain the matching in the step case
	(note that when $\alpha$ or $\beta$ is zero, $\mathcal{M}=0$ almost surely),
	it suffices to 
	check that 
	\begin{equation*}
		\mathsf{U}^{[0]}_{\mathrm{sHL}(u),\mathrm{sqW}(\theta)}(i_0,1-i_0';j_0',1-j_0)
		=
		\mathcal{L}_{u,\theta}(i_0,i_0';j_0',j_0)
	\end{equation*}
	for all 
	$i_0,j_0'\in \mathbb{Z}_{\ge0}$ and $i_0',j_0 \in \left\{ 0,1 \right\}$,
	where the left-hand side
	is the specialization 
	of
	\eqref{eq:U_00}.
	This is a straightforward verification. 

	The matching result for the scaled geometric boundary conditions
	is obtained in the same way as in the proof of
	\Cref{thm:sHL_sHL_mixed_height_6VM}.
	Indeed, we can consider the field in $\mathbb{Z}_{\ge -I_0+1}\times \mathbb{Z}_{\ge -J_0+1}$
	with the extra sHL specializations 
	with the 
	parameters $u_0,qu_0,\ldots,q^{J_0-1}u_0 $ and 
	$v_0,qv_0,\ldots,q^{I_0-1} v_0$.
	The desired matching then follows from 
	the expressions
	\eqref{eq:U00_bernoulli_vert},
	\eqref{eq:U00_qPoisson}, \eqref{eq:U_00_sg_sqW_computation}
	for the corresponding specializations of $\mathsf{U}^{[0]}$,
	and analytic continuation.
\end{proof}

\subsection{The sqW/sqW Yang-Baxter field} 
\label{sub:new_YB_field_sqW_sqW}

Let us now turn to the third and final Yang-Baxter field 
associated with the sqW/sqW skew Cauchy structure.
The particle system we obtain as its marginal generalizes the 
$q$-Hahn PushTASEP introduced recently in 
\cite{CMP_qHahn_Push}.

\subsubsection{Step boundary conditions}

The sqW/sqW Yang-Baxter field depends on the parameters
$q\in(0,1)$, $s\in[-\sqrt q,0)$,
$\theta_x,\xi_y\in[-s,-s^{-1}]$.
The reversibility condition associated with the forward and backward transition
probabilities takes the following form:
\begin{proposition}
	For any four Young diagrams $\mu, \varkappa, \lambda, \nu$ we have
	\begin{multline} \label{eq:reversibility_sqW_sqW}
		\frac{(-s \xi;q)_\infty (-s \theta;q)_\infty }{(s^2;q)_\infty (\xi \theta;q)_\infty}\,
		\mathsf{U}^{\mathrm{fwd}}_{\mathrm{sqW}(\xi),\mathrm{sqW}(\theta)}
		(\varkappa \to \nu \mid \lambda,\mu) 
		\,
		\mathbb{F}_{\lambda' / \varkappa'}(\xi) 
		\,
		\mathbb{F}^*_{\mu' / \varkappa'} (\theta) 
		\\= 
		\mathsf{U}^{\mathrm{bwd}}_{\mathrm{sqW}(\xi),\mathrm{sqW}(\theta)}
		( \nu \to \varkappa \mid \lambda,\mu) 
		\,
		\mathbb{F}^*_{\nu'  /\lambda'}(\theta)
		\,
		\mathbb{F}_{\nu'/\mu'}(\xi).
	\end{multline}
	Summing \eqref{eq:reversibility_sqW_sqW} 
	over both $\varkappa$ and $\nu$, we obtain
	the skew Cauchy identity of \Cref{thm:sqW_sqw_skew_Cauchy_identity}.
\end{proposition}


The \emph{sqW/sqW Yang-Baxter field} with the step boundary conditions
$\boldsymbol\lambda$ is, by definition, generated from the boundary conditions 
$\lambda^{(x,0)}=\lambda^{(0,y)}=0^\infty=\varnothing$
by applying the forward transition operators
$\mathsf{U}^{\mathrm{fwd}}_{\mathrm{sqW}(\xi_y),\mathrm{sqW}(\theta_x)}$.

\begin{proposition}
	The single-point distributions in the sqW/sqW field $\boldsymbol\lambda$
	with the step boundary conditions have the form
	\begin{equation*}
		\mathrm{Prob}\bigl( \lambda^{(x,y)} = \nu \bigr) 
		=
		\prod_{\substack{ 1 \leq i \leq x \\ 1\leq j \leq y}} 
		\frac{(s^2;q)_\infty(\xi_i  \theta_j ;q)_\infty}{(-s \xi_i;q)_\infty (-s \theta_j;q)_\infty } 
		\,
		\mathbb{F}_{\nu'}(\xi_1, \dots, \xi_y)
		\,
		\mathbb{F}^*_{\nu'}(\theta_1, \dots, \theta_x).
	\end{equation*}
	The joint distributions in $\boldsymbol\lambda$ along
	down-right paths are expressed through the 
	skew sqW functions as in \Cref{prop:F_G_processes}.
\end{proposition}

\subsubsection{Scaled geometric boundary conditions}
\label{ssub:sqW_sqW_two_sided}

Let $\alpha,\beta\in[0,-s^{-1}]$ be additional parameters.
The \emph{sqW/sqW Yang-Baxter field with two-sided scaled geometric}
(or \emph{$(\alpha,\beta)$-scaled geometric})
\emph{boundary conditions}
is constructed exactly as in 
\Cref{ssub:sHL_sHL_two_sided,ssub:sHL_sqW_two_sided}
by adding scaled geometric specializations
to both boundaries of the quadrant $\mathbb{Z}_{\ge0}^{2}$.

\begin{proposition}
	The single-point distributions in the sqW/sqW Yang-Baxter field $\boldsymbol\lambda$
	with $(\alpha,\beta)$-scaled geometric boundary conditions are given by 
	\begin{multline*}
		\mathrm{Prob}\bigl( \lambda^{(x,y)} = \nu \bigr) 
		=
		(\alpha\beta;q)_\infty
		\prod_{i=1}^{x}
		\frac{(\alpha \theta_i;q)_\infty}{(-s\alpha;q)_\infty}
		\prod_{j=1}^{y}
		\frac{(\beta \xi_j;q)_\infty}{(-s\beta;q)_\infty}
		\\\times
		\prod_{\substack{ 1 \leq i \leq x \\ 1\leq j \leq y}} 
		\frac{(s^2;q)_\infty(\xi_i  \theta_j ;q)_\infty}{(-s \xi_i;q)_\infty (-s \theta_j;q)_\infty } 
		\,
		\mathbb{F}_{\nu'}(\xi_1, \dots, \xi_y;\widetilde\alpha)
		\,
		\mathbb{F}^*_{\nu'}(\theta_1, \dots, \theta_x;\widetilde \beta).
	\end{multline*}
	The joint distributions 
	along down-right paths
	are expressed as in 
	\Cref{prop:F_G_processes}.
\end{proposition}

\subsubsection{Stochastic vertex model with $_4\phi_3$ weights}
\label{ssub:phi_vertex}

The scalar marginal
$\{\ell(\lambda^{(x,y)})\}$
of the sqW/sqW Yang-Baxter field gives rise to a 
new vertex model
which is related to
the $q$-Hahn PushTASEP
from \cite{CMP_qHahn_Push} (we discuss this connection in \Cref{ssub:qHahn_Push}
below). To formulate the vertex model, let us first write down the 
quantities 
\eqref{eq:U_00} under the two sqW specializations:
\begin{equation} \label{eq:PushTASEP_rate}
\begin{split}
	&
	\mathbb{L}_{\xi,\theta}(i_1,j_1;i_2,j_2)
	:=
	\mathsf{U}_{\mathrm{sqW}(\xi),\mathrm{sqW}(\theta)}^{[0]}
	(i_1,j_1;i_2,j_2) 
	\\&
	\hspace{40pt} = 
	\mathbf{1}_{i_1 + j_2 = i_2 + j_1 }\,  
	\frac{ \xi^{i_2} s^{i_1} \theta^{i_2-i_1} \, q^{ j_1 j_2 +\frac{1}{2}i_1(i_1 -1) }
	\, (-s/\theta;q)_{i_2} (-s/\xi;q)_{j_2} }
	{ (-s/\theta;q)_{i_1} (-s/\xi;q)_{j_1} 
	(q;q)_{j_2} (-q/(s\xi);q)_{j_2 -i_2} }
	\\
	&
	\hspace{80pt}
	\times \frac{(s^2 q^{i_1+j_2};q)_\infty 
	(\theta  \xi; q)_\infty}{(-s \xi;q)_\infty (-s \theta ; q)_\infty} 
	\,{}_4 \overline{ \phi}_3
	\left(\begin{minipage}{5.2cm}
	\center{$q^{-j_1}; q^{-j_2}, -s\theta,  -q/(s\xi)$}
	\\
	\center{$-s/\xi,q^{1+i_1-j_1}, -\theta q^{1-j_2-i_1}/s$}
	\end{minipage} \Big\vert\, q,q\right),
\end{split}
\end{equation}
where $i_1,j_1,i_2,j_2\in \mathbb{Z}_{\ge0}$.

\begin{lemma}
	Let $q\in(0,1)$, $s\in[-\sqrt q,0)$,
	$\xi,\theta\in[-s,-s^{-1}]$.
	Then $\mathbb{L}_{\xi,\theta}(i_1,j_1;i_2,j_2)\ge0$
	for all 
	$i_1,j_1,i_2,j_2\in \mathbb{Z}_{\ge0}$.
	Moreover,
	$\sum_{i_2,j_2}\mathbb{L}_{\xi,\theta}(i_1,j_1;i_2,j_2)=1$
	for all $i_1,j_1\in \mathbb{Z}_{\ge0}$.
\end{lemma}
\begin{proof}
	The nonnegativity follows from \Cref{sub:YBE_nonnegativity}
	(in particular, from \Cref{prop:mathbbR_nonnegative}).
	The fact that the weights sum to one is the consequence of the
	Yang-Baxter equation \eqref{eq:YBE_W_W} in the leftmost column, where $i_3=j_3=\infty$,
	together with \Cref{prop:emergence_of_the_cross_vertex_weight}.
\end{proof}

The weights $\mathbb{L}_{\xi,\theta}(i_1,j_1;i_2,j_2)$ 
give rise to a 
stochastic vertex model. Because the arrow preservation
property for these weights reads $i_1+j_2=i_2+j_1$, 
the paths in this stochastic vertex model are directed \emph{up-left}.
\begin{definition}
	\label{def:4phi3_vertex_model}
	Let $\xi_y,\theta_x\in[-s,-s^{-1}]$, $x,y\in \mathbb{Z}_{\ge1}$.
	The (inhomogeneous) \emph{$_4\phi_3$ stochastic vertex model} with the 
	boundary conditions 
	$\{b_1^{\mathrm{h}},b_2^{\mathrm{h}},\ldots \}$
	and
	$\{b_1^{\mathrm{v}},b_2^{\mathrm{v}},\ldots \}$,
	$b_i^{\mathrm{h}},b_j^{\mathrm{v}}\in \mathbb{Z}_{\ge0}$, 
	is the (unique) probability distribution 
	on the set of up-left directed paths on
	$\mathbb{Z}_{\ge0}\times \mathbb{Z}_{\ge0}$
	(with arbitrary nonnegative number of paths allowed per edge, 
	see \Cref{fig:4phi3_vertex_model} for an illustration),
	satisfying:
	\begin{itemize}
		\item The number of paths entering 
			at each location $(x,0)$ on the horizontal boundary
			is equal to $b_x^{\mathrm{h}}$, 
			$x\in \mathbb{Z}_{\ge1}$;
		\item The number of paths exiting
			at each location $(0,y)$ on the vertical boundary
			is equal to $b_y^{\mathrm{v}}$, 
			$y\in \mathbb{Z}_{\ge1}$;
		\item For each $(x,y)$, conditioned on
			the path configuration at all vertices $(x',y')$
			such that $x'+y'<x+y$, the probability of the configuration 
			$(i_1,j_1;i_2,j_2)$ at $(x,y)$ is given by
			$\mathbb{L}_{\xi_y,\theta_x}(i_1,j_1;i_2,j_2)$.
			Moreover, the random choices made at diagonally 
			adjacent vertices $\ldots,(x-1,y+1),(x,y),(x+1,y-1),\ldots$
			are independent under the same condition.
	\end{itemize}
	In particular, the \emph{step boundary conditions}
	correspond to taking $b_i^{\mathrm{h}}=b_i^{\mathrm{v}}=0$
	for all $i\in \mathbb{Z}_{\ge1}$.
\end{definition}

\begin{figure}[htpb]
	\centering
	\includegraphics[width=.35\textwidth]{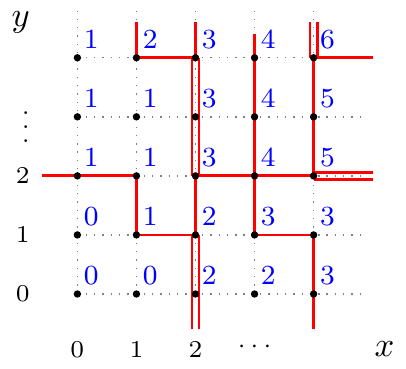}
	\caption{A path configuration in the $_4\phi_3$
	stochastic vertex model, and the corresponding height function.
	(The boundary configuration in the figure are not step.
	Unlike for the two previous vertex models,
	here the step boundary conditions would mean
	no paths crossing the boundary.)}
	\label{fig:4phi3_vertex_model}
\end{figure}

\begin{remark}
	The up-left direction of paths in the $_4\phi_3$ vertex model
	of \Cref{def:4phi3_vertex_model}
	should be contrasted with up-right paths in the stochastic
	vertex model 
	(\Cref{ssub:6v})
	and the stochastic higher spin six vertex model
	(\Cref{ssub:shv}).
	Note however that in the latter two models the number $j$ of 
	paths per horizontal edge is at most one, and 
	so the operation $j\mapsto 1-j$ applied at each horizontal 
	edge turns up-right paths into
	up-left ones. In the sqW/sqW setting the number of paths per horizontal edge 
	can be arbitrary, so the model with the weights $\mathbb{L}_{\xi,\theta}$
	cannot be 
	mapped to a model with up-right directed paths.
\end{remark}

For arbitrary boundary conditions, the configuration of the 
paths is encoded by the height function 
$\mathbb{H}^{\phi}(x,y)$, $x,y\in \mathbb{Z}_{\ge0}$, 
which counts the number of paths which between $(0,0)$ and 
$(x,y)$ (including the paths that pass through $(x,y)$, too).
In other words, paths are the level lines of $\mathbb{H}^{\phi}$.
An example is given in \Cref{fig:4phi3_vertex_model}.
Clearly, $\mathbb{H}^{\phi}(x,y)$ is almost surely finite 
at each $(x,y)$.

\Cref{prop:two_sided_sHL_sqW}
expressing $\mathsf{U}^{[0]}_{\mathrm{sg}(\alpha),\mathrm{sqW}(\theta)}$
on the bottom boundary of $\mathbb{Z}_{\ge0}^{2}$
as the $q$-negative binomial distribution 
$q\textnormal{-}\mathrm{NB}(-s/\theta,\alpha\theta)$
suggests the 
two-sided stationary boundary conditions
for the $_4\phi_3$ stochastic vertex model.
Moreover, on the left boundary, by 
the symmetry \eqref{eq:symmetry_R_fused},
$\mathsf{U}^{[0]}_{\mathrm{sqW}(\xi),\mathrm{sg}(\beta)}$
leads to 
$q\textnormal{-}\mathrm{NB}(-s/\xi,\beta\xi)$.
Therefore, we define the \emph{two-sided}
(or $(\alpha,\beta)$) \emph{stationary} boundary conditions
by taking independent 
\begin{equation*}
	b_x^{\mathrm{h}}
	\sim
	q\textnormal{-}\mathrm{NB}(-s/\theta_x,\alpha\theta_x),
	\qquad 
	b_y^{\mathrm{v}}\sim
	q\textnormal{-}\mathrm{NB}(-s/\xi_y,\beta\xi_y).
\end{equation*}
The step boundary condition
arises when $\alpha=\beta=0$,
and thus $b_x^{\mathrm{h}}=b_y^{\mathrm{v}}=0$ for all $x,y$
(note that this meaning of ``step'' here differs from the 
two previous stochastic vertex models).

Recall the $q$-Poisson distribution (\Cref{sub:notation}).
Let $\boldsymbol\lambda$ be the sqW/sqW Yang-Baxter field
with $(\alpha,\beta)$-scaled geometric boundary conditions.

\begin{theorem}
	\label{thm:sqW_sqW_length}
	Let $\mathcal{M}$ be the $q$-Poisson random variable with parameter
	$\alpha\beta$ which is
	independent of the $_4\phi_3$ stochastic vertex model.
	Then the two random fields
	$\{\ell(\lambda^{(x,y)})\colon x,y\in \mathbb{Z}_{\ge0}\}$ and 
	$\{\mathbb{H}^{\phi}(x,y)+\mathcal{M}\colon x,y\in \mathbb{Z}_{\ge0}\}$ 
	have the same distribution.
\end{theorem}
\begin{proof}
	This is proven similarly to 
	\Cref{thm:sHL_sHL_mixed_height_6VM,thm:sHL_sqW_height_length}
	using analytic continuation. 
	Namely, one starts with the Yang-Baxter field 
	$\boldsymbol\mu$
	in $\mathbb{Z}_{\ge -I_0+1}\times
	\mathbb{Z}_{\ge -J_0+1}$, where the positive
	coordinates $\mathbb{Z}_{\ge1}\times \mathbb{Z}_{\ge1}$ carry the 
	sqW specializations $\{\xi_y\}$ and $\{\theta_x\}$,
	and the extra nonpositive coordinates carry the sHL
	specializations with parameters
	$u_0,qu_0,\ldots,q^{J_0-1}u_0 $ and 
	$v_0,qv_0,\ldots,q^{I_0-1}v_0 $,
	respectively.
	The resulting field 
	depends on 
	$u_0,q^{J_0},v_0,q^{I_0}$
	in an analytic manner, and one can then take 
	$u_0,q^{J_0},v_0,q^{I_0}$
	to the
	scaled geometric specializations.
	
	Before the analytic continuation
	we know that $\ell(\mu^{(x,y)})$ 
	is equal in distribution to the 
	height function $\mathbb{H}^{\phi}_{I_0,J_0}(x,y)$,
	which is defined in the same way as $\mathbb{H}^{\phi}$, but in
	$\mathbb{Z}_{\ge -I_0+1}\times
	\mathbb{Z}_{\ge -J_0+1}$. 
	The 
	number of
	paths originating from the segment $\{-I_0+1\} \times [-J_0 +1, 0]$ and
	vertically crossing the segment $[-I_0 +1, 0] \times \{0 \}$
	becomes, after the scaled
	geometric specializations, the desired $q$-Poisson random variable
	$\mathcal{M}$.
	Therefore, after the specializations
	$\mathbb{H}^{\phi}_{I_0,J_0}(x,y)$
	turns into
	$\mathbb{H}^{\phi}(x,y)+\mathcal{M}$ for all $x,y\in \mathbb{Z}_{\ge0}$.
	On the other hand, 
	$\ell(\mu^{(x,y)})$ becomes $\ell(\lambda^{(x,y)})$ for all
	$x,y\in \mathbb{Z}_{\ge0}$.
	This completes the proof.
\end{proof}

\subsubsection{Connection to PushTASEPs}
\label{ssub:qHahn_Push}

Take arbitrary stochastic
vertex weights $\mathfrak{L}_{(x,y)}(i_1,j_1;i_2,j_2)$,
$x,y\in \mathbb{Z}_{\ge1}$,
$i_1,j_1,i_2,j_2\in \mathbb{Z}_{\ge0}$,
which vanish
unless $i_1+j_2=i_2+j_1$, and 
construct from them a stochastic vertex model 
with up-left paths as in \Cref{def:4phi3_vertex_model}.
We also assume that boundary conditions $b_x^{\mathrm{h}}$
and $b_y^{\mathrm{v}}$, $x,y\in \mathbb{Z}_{\ge1}$, are fixed.\footnote{If these boundary conditions are 
random, then they should be independent of the evolution of the stochastic vertex model.
Therefore, we can first sample the boundary conditions and then proceed with the discussion
conditioned on the values of $b_x^{\mathrm{h}}$, $b_y^{\mathrm{v}}$.}
Path configurations in this vertex model
can be equivalently viewed as trajectories in a 
stochastic particle system on the line with a pushing mechanism.

Indeed, consider the discrete
time dynamics $\mathbf{y}(t)$
on the space of configurations
\begin{equation*}
	\{\mathbf{y}=(y_1>y_2>\ldots )\colon y_i\in \mathbb{Z} \}
\end{equation*}
defined as follows (see \Cref{fig:push_dynamics_length_sqW_sqW}
for an illustration):
\begin{enumerate}[label=\bf{\arabic*.}]
		\item \label{item:PushTASEP_ic} 
			At time $t=0$ we have $y_1(0)=-1$
			and $y_{k}(0)-y_{k+1}(0)-1 = b^{\mathrm{h}}_k$,
			$k\in \mathbb{Z}_{\ge1}$;

		\item \label{item:first_particle_PushTASEP} At each 
			discrete time step $t-1\to t$, $t\in \mathbb{Z}_{\ge1}$, 
			the first particle's location is updated as 
			$y_1(t)=y_1(t-1)-b_t^{\mathrm{v}}$
			(i.e., it jumps by $b_t^{\mathrm{v}}$ to the left);

		\item \label{item:PushTASEP_particle_jump} 
			At each discrete time step $t-1\to t$,
			$t\in \mathbb{Z}_{\ge1}$, 
			the locations of the subsequent particles are 
			updated sequentially. For $i=2,3,\ldots$,
			after the $(i-1)$-st particle
			has moved such that $y_{i-1}(t) = y_{i-1}(t-1) - l$,
			and if the gap was 
			$y_{i-1}(t-1) - y_{i}(t-1) -1 =g$,
			then the 
			$i$-th particle jumps by $L$ to the left with
			probability $\mathfrak{L}_{(i,t)}(g, l; g+L-l, L )$.
\end{enumerate}
The fact that it must be $L\ge l-g$ in the update implies that 
the dynamics preserves the order of the particles. 
Namely, if the jump $l$ of the $(i-1)$-st
particle is longer than the gap, then the $i$-th particle
is pushed to the left. Therefore, the dynamics $\mathbf{y}(t)$
has a built-in pushing mechanism.

\begin{figure}
    \centering
		\includegraphics[width=.3\textwidth]{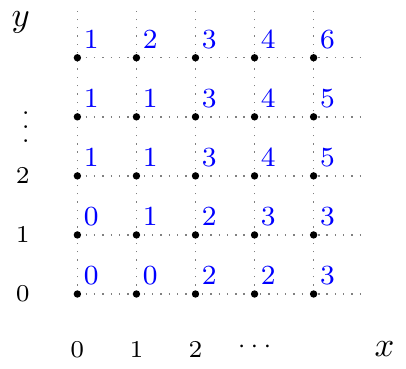}
		\qquad 
		\includegraphics[width=.63\textwidth]{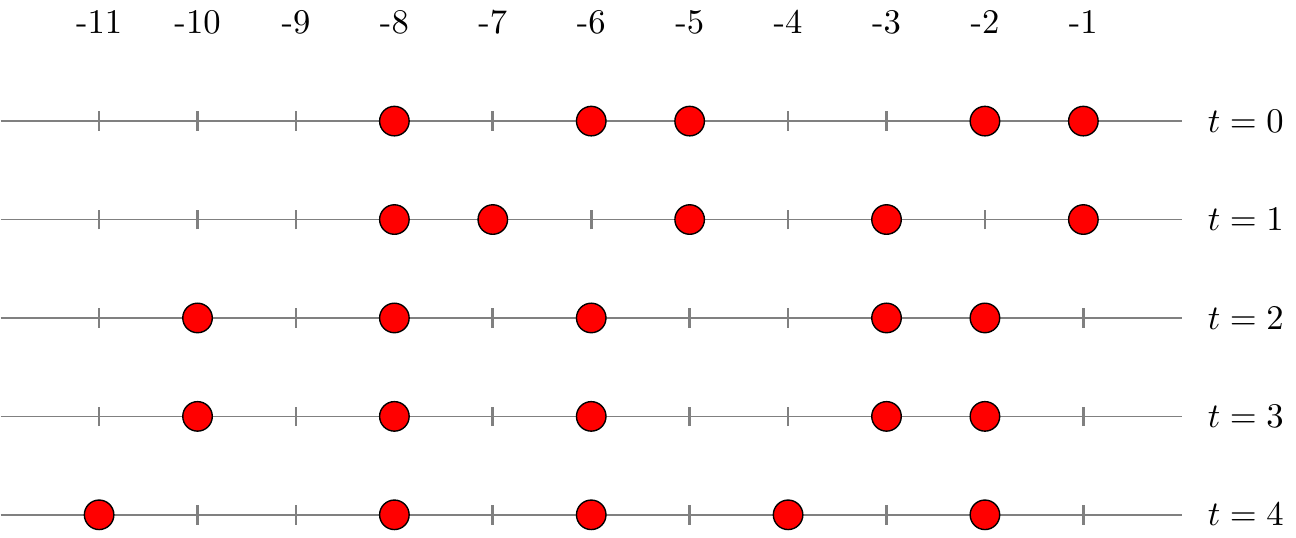}
    \caption{The height function in the 
			vertex model (left) and the corresponding realization of the 
		pushing dynamics (right).}
    \label{fig:push_dynamics_length_sqW_sqW}
\end{figure}

At each discrete time step the dynamics $\mathbf{y}(t)$
might perform an infinite number of jumps. However, due to the sequential
update structure, the evolution of the first $N$ particles $y_1>\ldots>y_N $
is always well-defined, and thus one can define the whole 
dynamics $\mathbf{y}(t)$ via Kolmogorov's extension theorem.

Particle systems with pushing mechanism
have been studied for a long time.
The first example is the PushTASEP 
(also known as the ``long-range TASEP'', or as a degenerate
particular case of the Toom's interface model)
\cite{Spitzer1970}, \cite{derrida1991dynamics}.
The PushTASEP admits many deformations, 
most recent of which is the $q$-Hahn PushTASEP
introduced in \cite{CMP_qHahn_Push}
(see also section 3.2.1 in the latter paper for references to 
known intermediate degenerations).
Recall that the $q$-Hahn PushTASEP 
depends on three parameters
$q\in(0,1)$, $\mu\in(0,1)$, and $\nu\in(-1, \min(\mu,\sqrt q)]$.

\begin{proposition}
	\label{prop:qhahn_push}
	For $\alpha=0$, $\beta=1$, $\xi_y=\mu$, $\theta_x=1$
	for all $x,y\in \mathbb{Z}_{\ge1}$, and 
	$s=-\nu$, the particle system corresponding
	to the $_4\phi_3$ stochastic vertex model
	(i.e., with $\mathfrak{L}_{(i,t)}=\mathbb{L}_{\mu,1}$
	given by \eqref{eq:PushTASEP_rate})
	and $(\alpha,\beta)$-stationary boundary conditions
	coincides with the $q$-Hahn PushTASEP
	from \textnormal{\cite{CMP_qHahn_Push}}
	with the step initial configuration
	$y_i(0)=-i$, $i\in \mathbb{Z}_{\ge1}$.
\end{proposition}
\begin{proof}
	This is obtained in a straightforward way 
	by matching the 
	formulas
	from \cite{CMP_qHahn_Push}
	expressing transition probabilities
	in the $q$-Hahn PushTASEP through the $_4\phi_3$
	$q$-hypergeometric functions with the expression 
	\eqref{eq:PushTASEP_rate}.
	For $\alpha=0$, there are no vertex model paths entering 
	through the bottom boundary.
	Then the boundary conditions on the left are random and independent
	with the distribution 
	$q\textnormal{-}\mathrm{NB}(-s/\xi,\beta\xi)
	=q\textnormal{-}\mathrm{NB}(\nu/\mu,\mu)$,
	which is exactly the jumping 
	distribution of the 
	$q$-Hahn PushTASEP 
	first particle 
	(denoted by 
	$\varphi_{q,\mu,\nu}(\cdot\mid \infty)$
	in \cite{CMP_qHahn_Push}).
\end{proof}

We see that the $_4\phi_3$ stochastic vertex model
from \Cref{ssub:phi_vertex}
in a particular case becomes the $q$-Hahn PushTASEP.
Note also that to match the jumping distribution of the first particle
we needed to employ the independent negative binomial
boundary conditions on the left (vertical) boundary. 
(This effect is also present in the stochastic higher spin six
vertex model, cf. \cite{OrrPetrov2016}.)
The pushing particle system corresponding to the
step boundary conditions in the $_4\phi_3$ stochastic vertex model
is more general than the $q$-Hahn PushTASEP.
Namely, the former can essentially be viewed as the $q$-Hahn PushTASEP 
conditioned on the event that the first particle $y_1$ never
jumps.

\section{Difference operators} 
\label{sec:diff_op}

In this Section we prove that the (stable) spin Hall-Littlewood and
the spin $q$-Whittaker functions are eigenfunctions of certain
($q$-)difference operators acting on symmetric functions.
In this section we denote the quantization parameter
in the sHL functions by $t$ instead of $q$
because the sHL eigenoperators
are the same as in the Macdonald case
(recall that for $s=0$, the sHL functions become the usual
Hall-Littlewood symmetric polynomials, which are the $q=0$
degenerations of the Macdonald symmetric polynomials).

\subsection{Eigenrelations for the spin Hall-Littlewood functions} 
\label{sub:eigenrelation_sHL}

Consider the space of symmetric rational functions
in $u_1,\ldots,u_n $.
Let the operator $T_{q,u_i}$ on this space be
\begin{equation}\label{eq:q_shift_operator}
    T_{q,u_i} f(u_1,\dots, u_n) 
		=
		f(u_1, \dots, u_{i-1},qu_i,u_{i+1}, \dots u_n),
\end{equation}
that is, it acts by multiplying the variable $u_i$ by $q$.
In this subsection we will use the $q=0$ version, $T_{0,u_i}$.
Note that this operator acts only on rational functions whose 
denominators do not contain positive powers of $u_i$.

\begin{definition}[Hall-Littlewood difference operators]
	For $1 \le r \le n$, let the $r$-th Hall-Littlewood difference
	operator be
\begin{equation} \label{eq:Hall_Littlewood_operator}
	\mathop{\mathfrak{D}_r} 
	:=
	\sum_{\substack{I\subset\{1,\dots,n \}\\ |I|=r }} 
	\biggl(
		\prod_{\substack{i\in I \\ j\in \{1,\dots,n \} \setminus I}} 
		\frac{t u_i - u_j}{u_i - u_j}
	\biggr)\,
	T_{0,I},
\end{equation}
with $T_{0,I}=\prod_{i \in I} T_{0,u_i}$.
\end{definition}
The Hall-Littlewood operators are the $q=0$ cases of the Macdonald
difference operators 
\cite[Chapter VI.3]{Macdonald1995}
(the latter are obtained by taking $T_{q,u_i}$
in \eqref{eq:Hall_Littlewood_operator} instead
of $T_{0,u_i}$). 
The operators $\mathop{\mathfrak{D}_r}$ are diagonal 
in the Hall-Littlewood symmetric polynomials 
$\mathsf{F}_\lambda\vert_{_{s=0}}$:\footnote{We have 
	$\mathsf{F}_\lambda(u_1,\ldots,u_n )\vert_{_{s=0}}=Q_\lambda(u_1,\ldots,u_n;t)$ in the standard 
notation of \cite[Chapter III]{Macdonald1995}.}
\begin{equation} \label{eq:Hall_Littlewood_eigen}
	\mathop{\mathfrak{D}_r} \mathsf{F}_\lambda(u_1,\ldots,u_n)
	\vert_{_{s=0}}
	= 
	e_r(1,t,\dots, t^{n-\ell(\lambda) -1})\,
	\mathsf{F}_\lambda(u_1,\ldots,u_n )\vert_{_{s=0}},
\end{equation}
where the eigenvalues
are given in terms of
$e_r (u_1,\dots u_n)$, 
the $r$-th elementary symmetric polynomial:
\begin{equation}
	e_r(z_1,\ldots,z_N )=
	\sum_{1\leq i_1<\cdots <i_r \leq N} z_{i_1} \cdots z_{i_r}.
\end{equation}
In particular, $e_r(z_1,\ldots,z_N )=0$ if $r>N$.

In the following Theorem we extend 
\eqref{eq:Hall_Littlewood_eigen} to the 
spin Hall-Littlewood symmetric functions:
\begin{theorem} \label{thm:sHL_eigen}
	For all Young diagrams $\lambda$ and $n\in \mathbb{Z}_{\ge1}$
	we have
	\begin{equation} \label{eq:sHL_eigen}
		\mathop{\mathfrak{D}_r} \mathsf{F}_\lambda(u_1,\ldots,u_n ) = 
			e_r(1,t,\dots, t^{n-\ell(\lambda) -1})\,
			\mathsf{F}_\lambda(u_1,\ldots,u_n ).
	\end{equation}
\end{theorem}
\begin{remark}
	Certain difference operators acting diagonally on the
	non-stable spin Hall-Littlewood symmetric functions
	were considered in
	\cite{Dimitrov2018GUE}.
\end{remark}

In order to prove \Cref{thm:sHL_eigen} we make use of two preliminary
lemmas. The first one is an explicit expression for the sHL function
$\mathsf{F}_\lambda$ as a sum over the symmetric group
$\mathfrak{S}_n$:
\begin{lemma}
	For any Young diagram $\lambda$ such that $n\ge \ell ( \lambda )$, we have
	\begin{equation} \label{eq:sHL_symmetric_sum}
		\mathsf{F}_\lambda(u_1, \dots, u_n) 
		=
		\frac{(1-t)^n}{(t;t)_{n - \ell(\lambda)}} 
		\sum_{\sigma \in \mathfrak{S}_n} 
		\sigma \biggl\{ 
			\prod_{1\le i < j\le n} \frac{u_i - t u_j}{u_i - u_j} 
			\prod_{i=1}^n \left( \frac{u_i - s}{1 - s u_i} \right)^{\lambda_i} 
			\,
			\prod_{i=1}^{\ell (\lambda)} \frac{u_i}{u_i - s} 
		\biggr\}.
	\end{equation}
	Here the symmetric group acts on the 
	indices of the variables $u_i$, but not 
	on $\lambda_i$.
\end{lemma}
\begin{proof}
	This is a corollary of \cite[Theorem 5.1]{Borodin2014vertex}
	which gives an 
	analogous expression for the
	non-stable spin Hall-Littlewood function. 
	The degeneration to the
	stable case is obtained as 
	in~\eqref{eq:sHL_from_non_stable_2}. 
	Symmetrization formula \eqref{eq:sHL_symmetric_sum}
	for the stable case appeared
	earlier in \cite{deGierWheeler2016} and 
	\cite{BorodinWheelerSpinq}.
\end{proof}

\begin{lemma} \label{lemma:Hall_Littlewood_trivial}
We have
\begin{equation}
    \mathop{\mathfrak{D}_r} 
		\biggl( \sum_{\sigma \in \mathfrak{S}_n} 
		\sigma \biggl\{ \prod_{1\le i < j\le n} \frac{u_i - t u_j}{u_i - u_j} \biggr\} \biggr) 
		=
		e_r(1, \dots, t^{n-1}) \sum_{\sigma \in \mathfrak{S}_n} 
		\sigma \biggl\{ \prod_{1\le i < j\le n} \frac{u_i - t u_j}{u_i - u_j} \biggr\}.
\end{equation}
\end{lemma}
\begin{proof}
    This is the $\lambda=\varnothing$ 
		case of the known Hall-Littlewood relation 
		\eqref{eq:Hall_Littlewood_eigen}. 
		Notice that the symmetrized sum 
		in fact does not depend on the variables
		$u_1,\ldots,u_n $:
		\begin{equation*}
			\sum_{\sigma \in \mathfrak{S}_n} 
			\sigma 
			\biggl\{ \prod_{1\le i < j\le n} 
				\frac{u_i - t u_j}{u_i - u_j}
			\biggr\}
			=
			\frac{(t;t)_n}{(1-t)^n},
		\end{equation*}
		see
		\cite[Chapter III.1, formula (1.4)]{Macdonald1995}.
\end{proof}
\begin{proof}[Proof of Theorem \ref{thm:sHL_eigen}]
	For a fixed Young diagram $\lambda$ we define
	\begin{equation*}
			A=\prod_{1\le i < j\le n} \frac{u_i - t u_j}{u_i - u_j}, \qquad 
			B=\prod_{i=1}^n \left( \frac{u_i - s}{1 - s u_i} \right)^{\lambda_i}, \qquad 
			C= \prod_{i=1}^{\ell (\lambda)} \frac{u_i}{u_i - s}.
	\end{equation*}
	With this notation, using \eqref{eq:sHL_symmetric_sum},
	the left-hand side of \eqref{eq:sHL_eigen} can be written as
	\begin{equation} \label{eq:sHL_eigen_lhs}
			c_\lambda \sum_{\substack{ I \subset \{1,\dots,n \} \\ |I| = r}} 
			\prod_{\substack{i\in I \\ j\in \{1,\dots,n\} \setminus I }} 
			\frac{t u_i - u_j}{u_i - u_j} 
			\,
			T_{0,I}
			\sum_{\sigma \in \mathfrak{S}_n} \sigma \left\{ ABC \right\},
	\end{equation}
	where $c_\lambda = (1-t)^n/(t;t)_{n-\ell (\lambda)}$. We first observe that 
	\begin{equation*}
		T_{0,I}  \sigma \{ C \} = 
			\begin{cases}
			\sigma\{ C \} \qquad & \text{if }I\subseteq \{ \sigma_{\ell (\lambda)+1}, \dots, \sigma_n \},\\
			0 \qquad & \text{otherwise}.
			\end{cases}
	\end{equation*}
	Therefore, we can reduce the sum over the symmetric group in
	\eqref{eq:sHL_eigen_lhs} to permutations $\sigma$ such that
	$I\subseteq \sigma(\{\ell (\lambda)+1 ,\dots, n \})$. 
	Moreover, we see that the claim of \Cref{thm:sHL_eigen} follows
	for 
	$r>n-\ell(\lambda)$ since both sides of \eqref{eq:sHL_eigen} vanish.
	Thus we will now assume that $r \le n - \ell(\lambda)$.
    
	For a given permutation $\sigma$ define the ordered sets
	$V_\sigma,W_\sigma$ as
	\begin{align*}
		V_\sigma &= \sigma (\{ 1,\dots, \ell (\lambda) \})
		=
		\left\{ v_1,\ldots,v_{\ell(\lambda )}  \right\}, 
		\\
		W_\sigma &= \sigma (\{ \ell (\lambda) + 1,\dots, n \})
		=
		\{w_1,\ldots,w_{n-\ell(\lambda)} \}=
		I\cup K,
	\end{align*}
	and rewrite \eqref{eq:sHL_eigen_lhs} as
	\begin{equation}
		\label{eq:sHL_eigen_lhs_1}
			c_\lambda \sum_{\substack{ I \subset \{1,\dots,n \} \\ |I| = r}} 
			\sum_{\substack{ \sigma \in \mathfrak{S}_n \\ I \subseteq W_\sigma}} 
			\sigma \left\{ BC \right\} \prod_{\substack{i\in I \\ j\in \{1,\dots,n\} \setminus I }} 
			\frac{t u_i - u_j}{u_i - u_j}\, T_{0,I} \sigma \left\{ A \right\},
	\end{equation}
	where we used the fact that $\sigma\{ BC \}$ only depends on variables $u_j$ for $j \in V_\sigma$. 
	We now focus on the remaining factors.
	For two disjoint or coinciding 
	ordered sets $S_1,S_2$ denote $P(S_1,S_2):=\prod_{i\in S_1,\,j\in S_2}\frac{u_i-tu_j}{u_i-u_j}$.
	When $S_1=S_2$, the product is only over $i<j$.
	We have
	\begin{align*}
		&\prod_{\substack{i\in I \\ j\in \{1,\dots,n\} \setminus I }} 
		\frac{t u_i - u_j}{u_i - u_j}\, T_{0,I} \sigma \left\{ A \right\}
		=
		P(I^c,I)
		P(V_\sigma,V_\sigma)P(V_\sigma,K)
		\\&\hspace{190pt}\times
		T_{0,I}
		\Bigl( P(I,K)P(I,I)P(V_\sigma,I)P(K,I)P(K,K) \Bigr)
		\\&\hspace{15pt}=
		P(V_\sigma,V_\sigma)
		P(V_\sigma,W_\sigma)
		P(K,I)
		\,T_{0,I}
		\Bigl( P(I,K)P(I,I)P(K,K) \Bigr)
		\\&\hspace{15pt}=
		\prod_{1\le i<j\le \ell(\lambda)}
		\frac{u_{v_i}-tu_{v_j}}{u_{v_i}-u_{v_j}}
		\prod_{i\in V_\sigma,\, j\in W_\sigma}
		\frac{u_i-tu_j}{u_i-u_j}
		\prod_{i\in W_\sigma\setminus I, j\in I}\frac{u_i-tu_j}{u_i-u_j}
		\,
		T_{0,I}
		\prod_{1\le i<j\le n-\ell(\lambda)}\frac{u_{w_i}-tu_{w_j}}{u_{w_i}-u_{w_j}}.
	\end{align*}
	In the above calculation we used the fact that $T_{0,I}$ acts
	on $P(S, I)$, $S\ne I$, by turning it into one. The action $T_{0,I}P(I,I)$
	does not make sense before the symmetrization (i.e., summation over $\sigma$),
	and so we do not apply $T_{0,I}$ to this expression just yet.
	In the last line, 
	the first two products
	are independent of $I$ and of the ordering of
	$W_\sigma$, and the last two products are independent of the ordering
	of $V_\sigma$. Therefore,
	we can rearrange the two summations in
	\eqref{eq:sHL_eigen_lhs_1} as
	\begin{equation} \label{eq:sHL_eigen_lhs_bis}
			\begin{split}
				&c_\lambda \sum_{\substack{V \subseteq \{1,\dots,n\} \\ |V|=\ell (\lambda),\, W=V^c }} 
				\sum_{\tau \in \mathfrak{S}_{\ell(\lambda)} } 
				\prod_{\substack{v\in V\\ w\in W}} \frac{u_v - t u_w}{u_v - u_w} \, 
				\tau \biggl\{ \prod_{1 \le i < j \le \ell(\lambda)}
					\frac{ u_{v_i } - t u_{v_j } }{ u_{v_i} - u_{ v_j } }
					\,BC \biggr\}
					\\
					&
					\hspace{80pt} \times \sum_{\substack{I\subset W \\ |I|=r}} \,
					\prod_{\substack{i \in I \\ w \in W \setminus I}} \frac{t u_i - u_w}{u_i - u_w} \, 
					T_{0,I} 
					\sum_{\sigma \in \mathfrak{S}_{n-\ell(\lambda)}} 
					\sigma \biggl\{
						\prod_{1 \le i <j \le n-\ell(\lambda)} \frac{u_{w_i} - t u_{w_j} }{u_{w_i} - u_{w_j}}
					\biggr\}.
			\end{split}
	\end{equation}
	The permutations 
	$\tau,\sigma$ permute the variables $u_{v_i},u_{w_j}$, $v_i\in V$, $w_j\in W$, acting
	respectively on indices $i$ and $j$. 
	We can now employ \Cref{lemma:Hall_Littlewood_trivial} 
	to transform the second line of \eqref{eq:sHL_eigen_lhs_bis} into
	\begin{equation}
			e_r(1,t,\dots, t^{n - \ell(\lambda) -1}) 
			\sum_{\sigma \in \mathfrak{S}_{n-\ell(\lambda)}} 
			\sigma 
			\biggr\{ 
				\prod_{1 \le i <j \le n-\ell(\lambda)} \frac{u_{w_i} - t u_{w_j} }{u_{w_i} - u_{w_j}}
			\biggr\}.
	\end{equation}
	Therefore, 
	\begin{equation}
			\begin{split}
				\text{lhs \eqref{eq:sHL_eigen} } & = e_r(1,t,\dots, t^{n - \ell(\lambda) -1})\,
				c_\lambda \sum_{\substack{V \subseteq \{1,\dots,n\} \\ |V|=\ell (\lambda),\, W=V^c } }\;
				\sum_{ \substack{ \tau \in \mathfrak{S}_{\ell(\lambda)}
				\\
				\sigma \in \mathfrak{S}_{n-\ell(\lambda)} } } 
				\prod_{\substack{v\in V\\ w\in W}} \frac{u_v - t u_w}{u_v - u_w}
					\\
					& \qquad 
					\times \sigma \biggl\{ 
						\prod_{1 \le i <j \le n-\ell(\lambda)} \frac{u_{w_i} - t u_{w_j} }{u_{w_i} - u_{w_j}}
					\biggr\}\;
					\tau \biggl\{ 
						\prod_{1 \le i < j \le \ell(\lambda)} \frac{ u_{v_i } - t u_{v_j } }{ u_{v_i} - u_{ v_j } }\, BC 
					\biggr\}.
			\end{split}
	\end{equation}
	The summations in the right-hand side of this last expression (along
	with the factor $c_\lambda$) are easily rearranged into the
	symmetrized sum \eqref{eq:sHL_symmetric_sum} producing the spin Hall
	Littlewood function $\mathsf{F}_\lambda$.
\end{proof}

\subsection{Orthogonality}
The spin Hall-Littlewood functions enjoy the following orthogonality:

\begin{proposition}\label{prop:orthog_FF}
	For any Young diagrams $\lambda,\mu$, we have
\begin{equation} \label{eq:orthog_FF}
	\frac{(t;t)_{n - \ell(\lambda)}}
	{(1-t)^nn!}
	\oint_{\gamma} 
	\frac{dz_1}{ 2 \pi \mathrm{i} z_1}
	\cdots 
	\oint_{\gamma} 
	\frac{dz_n}{ 2 \pi \mathrm{i} z_n}
	\prod_{1\le i \neq j\le n} \frac{z_i - z_j}{z_i - t z_j}\,
	\mathsf{F}_\lambda(z_1, \dots,z_n)
	\mathsf{F}^*_\mu(1/z_1, \dots,1/z_n)
	= 
	\mathbf{1}_{\lambda = \mu},  
\end{equation}
where $\gamma$ is a positively oriented contour encircling $0$, $t^{k} s$ for all $k\geq 0$,
and the contour $t\gamma$ (its image under the multiplication by $t$), but not the point $s^{-1}$.
\end{proposition}
\begin{proof}
	This is consequence of \cite[Corollary 7.5]{BorodinPetrov2016inhom},
	where an analogous result is stated for the non-stable spin
	Hall-Littlewood functions with inhomogeneous parameters
	(the corresponding homogeneous result goes back to \cite{BCPS2014}).
	The desired orthogonality
	relation \eqref{eq:orthog_FF} then follows with the help 
	of the limit
	\eqref{eq:sHL_from_non_stable_2}.   
\end{proof}

The orthogonality property
\eqref{eq:orthog_FF} 
resembles the orthogonality of the Hall-Littlewood polynomials
with respect to the Macdonald's torus scalar product 
\cite[Chapter VI, (9.10)]{Macdonald1995}
at $q=0$. The general $(q,t)$ scalar product is
\begin{equation}
	\label{eq:Mac_orthog}
	\langle f,g \rangle_n:=\frac{1}{(2\pi\mathrm{i})^n n!}\int_{\mathbb{T}^n}
	f(z)\overline{g(z)}\prod_{1\le i\ne j\le n}
	\frac{(z_i/z_j;q)_{\infty}}{(tz_i/z_j;q)_{\infty}}\prod_{i=1}^{n}\frac{dz_i}{z_i},
\end{equation}
where the integration is over the $n$-dimensional torus in $\mathbb{C}^n$
(i.e., over the positively oriented unit circles).
In the presence of the additional spin parameter $s$, the Hall-Littlewood orthogonality 
extends to \eqref{eq:orthog_FF}.

The usual $q$-Whittaker polynomials
are orthogonal with respect to \eqref{eq:Mac_orthog} with $t=0$.
At present it is not clear how to extend this orthogonality property
to the spin $q$-Whittaker polynomials. 
This problem could be related to the following observation.
While all the $q=0$ Macdonald difference operators 
are diagonal in the spin Hall-Littlewood functions (as 
well as in the usual Hall-Littlewood polynomials), 
the situation on the spin $q$-Whittaker side is more complicated.
In the next subsection we discuss an $s$-deformation of the 
first $t=0$ Macdonald $q$-difference operator which acts diagonally on the 
spin $q$-Whittaker polynomials.

\subsection{Eigenrelation for the spin $q$-Whittaker polynomials}

Fix $l\in \mathbb{Z}_{\ge1}$.
Define the operator acting on rational functions in $(\theta_1,\ldots,\theta_l )$:
\begin{equation}
	\mathop{\mathfrak{E}}
	:=
	\sum_{j=1}^l \left( 1+ \frac{s}{\theta_j} \right)^l 
	\biggl(\prod_{i \neq j} \frac{\theta_j}{\theta_j - \theta_i} \biggr)
	T_{q^{-1},\theta_j} 
	+
	\frac{(-s)^l}{\theta_1 \cdots \theta_l }\, Id.
\end{equation}
Here $T_{q^{-1},\theta_j}$ are the shifts \eqref{eq:q_shift_operator}, and $Id$ is
the identity operator.

\begin{theorem}\label{thm:sqW_eigen}
For any Young diagram $\lambda$, we have
\begin{equation} \label{eq:sqW_eigen}
	\mathfrak{E}\, \mathbb{F}_{\lambda}(\theta_1,\ldots,\theta_l ) 
	=
	q^{-\lambda_1} \mathbb{F}_\lambda(\theta_1,\ldots,\theta_l ).
\end{equation}
\end{theorem}
\begin{remark}
	When $s=0$, the operator $\mathfrak{E}$ reduces to the first Macdonald operator 
	with $t=0$ and the parameter $q^{-1}$ instead of $q$.
	Then \eqref{eq:sqW_eigen} turns into the known eigenrelation for
	the usual $q$-Whittaker polynomials (cf. the operators
	$\tilde D_n^r$ with $r=1$
	in \cite[Section 2.2.3]{BorodinCorwin2011Macdonald}).
	It is not clear whether there exist appropriate $s$-deformations
	of the higher $t=0$ Macdonald difference operators 
	which would be diagonal in the spin $q$-Whittaker polynomials.
\end{remark}

\Cref{thm:sqW_eigen} follows from a ``duality'' relation for 
$\mathfrak{E}$ in
\Cref{lemma:sqW_duality_eigen} below.
Define
\begin{equation} \label{eq:D_tilde}
    \widetilde{\mathop{\mathfrak{D}}} := q^{-n} \left( Id + (q-1) \mathop{\mathfrak{D}_1} \right),
\end{equation}
where $\mathop{\mathfrak{D}_1}$ is the first Hall-Littlewood operator \eqref{eq:Hall_Littlewood_operator}
with parameter $t$ replaced by $q$.
\begin{lemma}
	\label{lemma:sqW_duality_eigen}
	Consider the function 
	\begin{equation} \label{eq:sHL_sqW_partition_function}
			\Pi(u_1, \dots, u_n; \theta_1, \dots, \theta_l)
			=
			\prod_{i=1}^n \prod_{j=1}^l \frac{ 1 + u_i \theta_j }{ 1 - s u_i }.
	\end{equation}
	Then we have
	\begin{equation} \label{eq:equality_sHL_sqW_operator}
		\mathfrak{E}\, \Pi = \widetilde{\mathfrak{D}} \Pi,
	\end{equation}
	where the operator $\widetilde{\mathop{\mathfrak{D}}}$ acts in
	$u_1,\dots, u_n$, while $\mathop{\mathfrak{E}}$ acts in $\theta_1,
	\dots, \theta_l$.
\end{lemma}

\begin{proof}
	When $\mathop{\mathfrak{E}}$ acts on a factorized function 
	$G(\theta_1, \dots, \theta_l)=g(\theta_1) \cdots g(\theta_l)$, it admits the integral representation
	\begin{equation} \label{eq:sqW_op_integral_rep}
		\mathfrak{E}\, G(\theta_1, \dots, \theta_l) 
			=
			G(\theta_1, \dots, \theta_l) 
			\frac{1}{2\pi\mathrm{i}}\oint_{\gamma_{0,\boldsymbol \theta}} \,
			\prod_{i=1}^l \frac{z + s}{z-\theta_i} \frac{g(q^{-1} z )}{g(z)} \frac{dz}{z}.
	\end{equation}
	Here $\gamma_{0, \boldsymbol \theta}$ is a positively 
	oriented contour (or a union of contours)
	encircling $0$, $\theta_i$ for $i=1,\dots, l$, and no other singularity of the integrand. 
	Here the function $g(z)$ should be such that we can choose a contour 
	in whose neighborhood the expression $g(q^{-1}z)/g(z)$ is holomorphic, and such that 
	no singularities of $g(q^{-1}z)/g(z)$ fall inside $\gamma_{0,\boldsymbol\theta}$.
	
	Analogously, the action of $\widetilde{\mathop{\mathfrak{D}}}$ on 
	$H(u_1, \dots, u_n) = h(u_1) \cdots h(u_n)$, where $h(0)=1$, has the form
	\begin{equation} \label{eq:sHL_op_integral_rep}
			\widetilde{\mathfrak{D}} H (u_1, \dots, u_n) 
			=
			H (u_1, \dots, u_n) \frac{1}{2\pi\mathrm{i}}\oint_{\gamma_{0,\boldsymbol u}} 
			\prod_{j=1}^n \frac{w - q^{-1}u_j }{w - u_j} \frac{1}{h(w)} \frac{dw}{w}.  
	\end{equation}
	The contour $\gamma_{0, \boldsymbol u}$ is 
	positively oriented and contains $0$, $u_i$ for $i=1, \dots ,n$ 
	and no other singularity of the integrand
	(again, under suitable assumptions on $h(w)$).
	Both integral expressions \eqref{eq:sqW_op_integral_rep} and 
	\eqref{eq:sHL_op_integral_rep} follow by straightforward
	residue calculus.
	
	Let now both operators $\mathfrak{E}$ and $\widetilde{\mathfrak{D}}$
	act on the function $\Pi$ \eqref{eq:sHL_sqW_partition_function}
	which has product form in both families of variables $u_i$ and $\theta_j$.
	We assume that $1+\theta_ju_i\ne 0$ for all $i,j$ 
	(otherwise $\Pi$ is identically zero).
	From \eqref{eq:sqW_op_integral_rep} we have
	\begin{equation*}
		\frac{{\mathfrak{E}}\,  \Pi (u_1,\dots,u_n ; \theta_1,\dots,\theta_l)}
		{ \Pi (u_1,\dots,u_n ; \theta_1,\dots,\theta_l)} 
		=
		\frac{1}{2\pi\mathrm{i}}\oint_{\gamma_{0,\boldsymbol \theta}} 
		\prod_{i=1}^l \frac{z + s}{z-\theta_i} \prod_{j=1}^n \frac{1+q^{-1} u_j z}{1+u_j z} \frac{dz}{z}.
	\end{equation*}
	On the other hand, from \eqref{eq:sHL_op_integral_rep} we have
	\begin{equation*}
		\frac{\widetilde{{\mathfrak{D}}}\, \Pi(u_1,\dots,u_n ; \theta_1,\dots,\theta_l)}
		{ \Pi (u_1,\dots,u_n ; \theta_1,\dots,\theta_l)} 
		=
		\frac{1}{2\pi\mathrm{i}}\oint_{\gamma_{0,\boldsymbol u}} 
		\prod_{j=1}^n \frac{w - q^{-1}u_j }{w - u_j} 
		\prod_{i=1}^l \frac{1-sw}{1+\theta_i w} \frac{dw}{w}.
	\end{equation*}
	The previous two expressions
	are identical after a change of variables $w=-1/z$ (note that the
	extra minus sign corresponds to changing the contour's orientation).
\end{proof}

\begin{proof}[Proof of Theorem \ref{thm:sqW_eigen}]
    First recall the Cauchy identity 
		between the sHL and sqW functions (\Cref{sub:sHL_sqW_structure}):
    \begin{equation} \label{eq:Cauchy_id_sHL_sqW}
        \sum_{\lambda} \mathsf{F}_\lambda (u_1, \dots , u_n) \mathbb{F}^*_{\lambda'} (\theta_1, \dots, \theta_l) = \Pi (u_1, \dots, u_n ; \theta_1, \dots, \theta_l),
    \end{equation}
    Combining \eqref{eq:equality_sHL_sqW_operator} with \eqref{eq:Cauchy_id_sHL_sqW}
    and employing the eigenrelation of sHL functions given by \Cref{thm:sHL_eigen}, we obtain
    \begin{equation} 
			\sum_\mu \mathsf{F}_\mu(u_1, \dots, u_n) \,
			{\mathfrak{E}} \, 
			\mathbb{F}^*_{\mu'}( \theta_1, \dots, \theta_l) 
			=
			\sum_\mu q^{-\ell (\mu)} \mathsf{F}_\mu (u_1, \dots, u_n)
			\mathbb{F}^*_{\mu'}(\theta_1, \dots, \theta_l).
    \end{equation}
		This
		implies the equality between single terms of the summations due to 
		orthogonality of the sHL functions
		(\Cref{prop:orthog_FF}). Because $\ell(\mu)=\mu_1'$, we get the desired
		eigenrelation.
\end{proof}

\section{Fredholm determinants for marginal processes} \label{sub:fredholm_determinant_for_marginal_processes}

Here we derive Fredholm determinant expressions for the
$q$-Laplace transform of the random variable $-\ell (
\lambda^{(x,y)} )$, where $\boldsymbol\lambda=\{\lambda^{(x,y)}\}$
is one of the Yang-Baxter fields described in \Cref{sec:new_three_fields}.
In the sHL/sHL case, these Fredholm formulas are known
\cite{BCG6V}, \cite{Amol2016Stationary}.
In the sHL/sqW case, they are present in the literature
for the step and step-stationary boundary conditions
\cite{CorwinPetrov2015}, \cite{BorodinPetrov2016inhom},
\cite{BorodinPetrov2016Exp}.
A Fredholm determinantal formula for the stochastic higher spin six vertex model
appears also in the recent work \cite{imamura2019stationary}, 
though here we establish a different formula for this case.
In the sqW/sqW case, a similar Fredholm formula for the $q$-Hahn PushTASEP
was recently conjectured
in \cite{CMP_qHahn_Push}, and here we prove this conjecture.

In this section we return to calling the main quantization parameter
by $q$ throughout.

\subsection{Six vertex model observables through difference operators}

In this subsection we rederive known results about the $q$-moments of the
six vertex model 
\cite{BCG6V},
\cite{BorodinPetrov2016inhom} 
making use of the difference operators acting on spin
Hall-Littlewood functions.
Consider the inhomogeneous stochastic six vertex model with the step boundary
conditions and height function $\mathfrak{h}^{\mathrm{6V}}$. 
Recall that the model depends on
the parameters $v_x$ and $u_y$, $x,y\in \mathbb{Z}_{\ge1}$,
which we assume positive (cf.~\Cref{ssub:6v}).
Let $u_1, u_2, \dots$ be spaced in
such a way that
\begin{equation} \label{eq:condition_u_q_moments}
		q \sup_i\{u_i\} < \inf_i \{ u_i \}.
\end{equation}

\begin{proposition}\label{prop:6VM_q_moments}
	Under \eqref{eq:condition_u_q_moments} we have
    \begin{equation} \label{eq:6vm_q_moment_nested}
    \begin{split}
			\mathbb{E}^{\mathrm{step}}\bigl( q^{l\, \mathfrak{h}^{\mathrm{6V}} (x+1,y )} \bigr) 
			&= 
			q^{l(l-1)/2} \oint_{\gamma[\mathbf{u}|1] } \cdots \oint_{\gamma[\mathbf{u}|l] } 
			\prod_{1 \le A < B \le l} \frac{z_A - z_B}{z_A - q z_B}
			\\
			& \hspace{30pt} \times
			\prod_{k=1}^l \left\{ \prod_{i=1}^y \frac{q z_k - u_i}{z_k - u_i} \prod_{i=1}^x \frac{1 - z_k v_i}{1- q z_k v_i} \frac{dz_k}{2 \pi \mathrm{i} z_k} \right\},
    \end{split}
    \end{equation}
    where the positively orientated contour 
		$\gamma[\mathbf{u}|j]=\gamma_\mathbf{u} \cup r^{j-1}C_0$ for $z_j$
		is the union of a curve $\gamma_\mathbf{u}$ that encircles $u_1,\dots, u_y$
		and no other pole of the integrand, 
		and the dilation $r^{j-1}C_0$ of an arbitrary small 
		circle $C_0$ around $0$.
		Moreover, 
		$r>q^{-1}$, and the shifted contour $q\gamma_\mathbf{u}$ must
		lie completely to the left of $\gamma_\mathbf{u}$ and completely
		to the right of $r^{l-1}C_0$.
\end{proposition}
\begin{proof}
		This is an application of the eigenrelations from
		\Cref{thm:sHL_eigen}. 
		Recall the operator
		$\widetilde{\mathop{\mathfrak{D}}}$ \eqref{eq:D_tilde}
		acting diagonally on the sHL functions as
    \begin{equation*}
        \widetilde{\mathop{\mathfrak{D}}} \mathsf{F}_\lambda = q^{- \ell (\lambda)} \mathsf{F}_\lambda.
    \end{equation*}
    From \Cref{sub:new_YB_field_sHL_sHL} we 
		have the identification
		$\mathfrak{h}^{\mathrm{6V}}(x+1,y) = y - \ell ( \lambda^{(x,y)} )$,
		where $\lambda^{(x,y)}$ is the sHL/sHL field.
		Therefore, we have
    \begin{equation*}
			\mathbb{E}^{\mathrm{step}}\bigl( q^{l\, \mathfrak{h}^{\mathrm{6V}} (x+1,y )} \bigr)
			=
			q^{ly}\,
			\frac{\widetilde{\mathfrak{D}}^l \Pi(u_1,\dots,u_y ; v_1, \dots, v_x) }
			{ \Pi(u_1,\dots,u_y ; v_1, \dots, v_x) },
    \end{equation*}
    where 
    \begin{equation*}
        \Pi(u_1,\dots,u_y ; v_1, \dots, v_x) 
				= \prod_{i=1}^y \prod_{j=1}^x \frac{1-q u_i v_j}{ 1-u_i v_j }.
    \end{equation*}
		The nested contour formula
		\eqref{eq:6vm_q_moment_nested} follows by recursively applying integral
		expression \eqref{eq:sHL_op_integral_rep} for the action of
		$\widetilde{\mathfrak{D}}$ on factorized functions.
\end{proof}

\Cref{prop:6VM_q_moments} combined with well-known manipulations
of summations of nested contour integrals like 
\eqref{eq:6vm_q_moment_nested} 
(e.g., see \cite[Section 3]{BorodinCorwinSasamoto2012})
give rise to a Fredholm determinant\footnote{On Fredholm determinants in general see, e.g.,
\cite{Bornemann_Fredholm2010}.}
expression for the one-point distribution of
$\mathfrak{h}^{\mathrm{6V}}$.

\begin{theorem} \label{thm:6vm_q_laplace}
Consider the stochastic six vertex model with step boundary conditions. We have
\begin{equation} \label{eq:6vm_q_Laplace}
    \mathbb{E}^{\mathrm{step}}
		\left( \frac{1}{(\zeta q^{\mathfrak{h}^{\mathrm{6V}} (x+1,y) } ;q)_\infty} \right) 
		=
		\det\left(Id + \mathsf{K} \right)_{ L^2(\mathsf{C}) },
		\qquad 
		\zeta\in\mathbb{C} \setminus {\mathbb{R}_{>0}}.
\end{equation}
The expression in the right-hand side of \eqref{eq:6vm_q_Laplace} is the Fredholm determinant of the kernel
\begin{equation}
	\label{eq:6V_Fredholm_kernel}
	\mathsf{K}(w,w') 
	= 
	\frac{1}{2 \mathrm{i}} \int_{d + \mathrm{i}\mathbb{R} } 
	\frac{(-\zeta)^r}{ \sin(\pi r) } 
	\frac{\mathsf{f}(w)/\mathsf{f}(q^r w) }{q^r w - w'}\, dr,
\end{equation}
where $d \in (0,1)$, and 
\begin{equation*}
    \mathsf{f}(w) = \prod_{i=1}^y (w - u_i)^{-1} \prod_{i=1}^x (1- v_i w).
\end{equation*}
The kernel $\mathsf{K}$ is defined on the Hilbert space $L^2
(\mathsf{C})$, where $\mathsf{C}$ is a closed positively oriented curve encircling
$0,u_1,u_2, \dots$ such that, for all $r \in d +
\mathrm{i}\mathbb{R}$, $\mathsf{C}$ contains $q^r \mathsf{C}$ but not
$q^{-r} v_i^{-1}$ for $i=1,2,\dots$.
\end{theorem}
We present the main steps of the proof of
the Fredholm determinantal formula,
and refer to 
\cite{BorodinCorwin2011Macdonald} or 
\cite{BorodinCorwinSasamoto2012} 
for detailed explanations.
\begin{proof}[Idea of proof of \Cref{thm:6vm_q_laplace}]
	Assume first \eqref{eq:condition_u_q_moments} and $|\zeta|<1/q$, and
	consider the nested contour expression \eqref{eq:6vm_q_moment_nested}. 
	We can deform 
	all contours, one by one, to be the same $\mathsf{C}$ around 
	$0,u_1,u_2, \dots$, and such that $\mathsf{C}$
	contains its image under multiplication by
	$q$. This contour shift will cross poles $z_A=q z_B$, $A<B$,
	and one can rewrite \eqref{eq:6vm_q_moment_nested} as
	\begin{equation*}
			(q;q)_l \sum_{\lambda \vdash l} \frac{1}{m_1!m_2! \cdots}
			\int_{\mathsf{C}} \cdots \int_{\mathsf{C}} \,\det_{i,j=1}^{\ell
			(\lambda) }\left( \frac{1}{w_i - q^{\lambda_j}w_j } \right)
			\prod_{k=1}^{\ell (\lambda)} f(w_k) / f(q^{\lambda_k}w_k)
			\frac{dz_k}{2 \pi \mathrm{i}},
	\end{equation*}
	where the sum is taken over all partitions $\lambda$ of $l$,
	and $m_i=m_i(\lambda)$ are the multiplicities of the parts $i$ in $\lambda$.
	Summing over $l$, we have
	\begin{equation*}
		\sum_{l \ge 0} \frac{\zeta^l}{(q;q)_l} \,\mathbb{E}^{\mathrm{step}} 
		\bigl( q^{l\, \mathfrak{h}^{\mathrm{6V}} (x+1,y )} \bigr) 
		= 
		\mathbb{E}^{\mathrm{step}} 
		\left( \frac{1}{(\zeta q^{\mathfrak{h}^{\mathrm{6V}} (x+1,y) } ;q)_\infty} \right),
	\end{equation*}
	where we used the absolute summability of the left-hand side (since
	$0<q^{l\,\mathfrak{h}^{\mathrm{6V}}}<1$ and we assumed $|\zeta| <
	1/q$) to exchange the summation with the expectation sign and the
	$q$-binomial theorem. The result we obtain is the Fredholm determinant
	of the kernel 
    \begin{equation*}
        \sum_{n \geq 0} \frac{\zeta^n}{w'-q^n w} \mathsf{f}(w)/\mathsf{f}(q^n w) = \frac{1}{2 \mathrm{i}} \int_{ d + \mathrm{i} \mathbb{R} } \frac{(-\zeta)^r}{\sin(\pi r)} \frac{\mathsf{f}(w)/\mathsf{f}(q^r w)}{q^r w - w'}.
    \end{equation*}
		Once we reach \eqref{eq:6vm_q_Laplace}, we can relax
	conditions on $u_i$'s and $\zeta$ since both side are analytic
	functions of their parameters. Formula \eqref{eq:6vm_q_Laplace}
	holds for any choice of $u_i,v_j \in (0,1)$ and $\zeta \in
	\mathbb{C} \setminus q^{\mathbb{Z}_{\ge 0}}$
	(in particular, we can always find $d$ in \eqref{eq:6V_Fredholm_kernel}
	such that $\mathsf{C}$ satisfies the required properties).
\end{proof}

In the next theorem we perform fusion of the sHL parameters.
Recall the principal specializations
$\mathfrak{F}^{(J_0,\dots, J_y)}(u_0,\dots, u_y),
\mathfrak{G}^{(I_0,\dots, I_x)}(v_0,\dots, v_x)$ defined in
\eqref{eq:F_principal_spec}, \eqref{eq:G_principal_spec}. Parameters
$u_l,J_l,v_k,I_k$ are complex numbers satysfying 
\eqref{eq:condition_parameters_general_skew_Cauchy} which we reproduce here:
\begin{equation}
	\label{eq:condition_parameters_general_skew_Cauchy_S9}
	|s|,|u_k|,|v_l|,|q^{J_k} u_k|,|q^{I_l} v_l|,
	\left| \frac{q^i u_k - s}{1 - q^i s u_k} \right|, 
	\left| \frac{q^i v_l - s}{1 - q^i s v_l} \right|< \delta \qquad \text{for all }
	0\le k\le y,\ 
	0\le l\le x,\ 
	i\ge0.
\end{equation}
for sufficiently small $\delta>0$ which might depend on $x,y$, but not on the other parameters.
\begin{theorem} 
	\label{thm:q_Laplace_princ_spec}
	With the above notation, we have for all
	$\zeta\in\mathbb{C} \setminus {\mathbb{R}_{>0}}$:
	\begin{equation} 
		\label{eq:q_Laplace_princ_spec}
		\begin{split}&
		\prod_{\substack{0\le k\le y\\0\le j\le x}}
				\frac{(u_k v_l;q)_\infty (u_k v_l q^{I_l + J_k};q)_\infty}{(u_k v_l q^{I_l};q)_\infty (u_k v_l q^{J_k};q)_\infty} 
				\\&\hspace{70pt}\times\sum_{\lambda} 
				\frac{
				\mathfrak{F}_{\lambda}^{(J_0, \dots, J_y)}(u_0, \dots, u_y) 
				\mathfrak{G}^{(I_0, \dots, I_x)}_{\lambda} (v_0, \dots, v_x) 
				}{(\zeta q^{- \ell (\lambda) } ;q )_\infty} 
				= \det\left( Id + K \right)_{L^2( C )}.
			\end{split}
	\end{equation}
	The kernel $K$ is defined as
	\begin{equation*}
			K(w,w') = 
			\frac{1}{2 \mathrm{i}} \int_{d + \mathrm{i}\mathbb{R} } 
			\frac{\bigl(-\zeta \bigr)^r}{ \sin(\pi r) } \frac{f(w)/f(q^r w) }{q^r w - w'} dr,
	\end{equation*}
	where $d \in (0,1)$ and
	\begin{equation*}
			f(w) = \prod_{k=0}^y 
			\frac{(q^{J_k}u_k/w;q)_\infty}{(u_k/w;q)_\infty}
			\prod_{l=0}^x \frac{(v_l w ; q)_\infty}{(q^{I_l} v_l w ; q )_\infty}.
	\end{equation*}
	The contour $C$ is a closed positively oriented curve encircling
	$0,q^k u_i$ for $k,i\ge 0$ and such that, for all $r \in d +
	\mathrm{i}\mathbb{R}$, $C$ contains $q^r C$ and $q^{r+k} q^{J_l}u_l$
	for all $k,l\ge 0$, but leaves outside $1/(q^{r+k}v_l)$ and
	$1/(q^kq^{I_l} v_l)$ for all $k, l \ge 0$. 
\end{theorem}
\begin{proof}
	Considering principal specializations in 
	\Cref{thm:6vm_q_laplace}, 
	we see that \eqref{eq:q_Laplace_princ_spec} holds
	for any $J_0, \dots J_y, I_0, \dots, I_x$ positive integers. 
	Indeed, this follows from the computation 
	for $I,J\in\mathbb{Z}_{\ge1}$:
	\begin{equation}
		\label{eq:why_q_rJ_appears}
		\begin{split}
			&
			\frac{q^r w-u}{w-u}
			\frac{q^r w-qu}{w-qu}
			\ldots 
			\frac{q^r w-u q^{J-1}}{w-u q^{J-1}}
			\frac{1-vw}{1-q^r vw}
			\frac{1-qvw}{1-q^r qvw}
			\ldots
			\frac{1-q^{I-1}vw}{1-q^r q^{I-1}vw}
			\\&\hspace{40pt}=
			q^{rJ}
			\frac{(uq^J/w;q)_\infty}{(u/w;q)_\infty}
			\frac{(q^{-r}u/w;q)_\infty}{(q^{-r}uq^J/w;q)_\infty}
			\frac{(vw;q)_\infty}{(vq^{I}w;q)_\infty}
			\frac{(vq^{I}wq^r;q)_\infty}{(vwq^r;q)_\infty}.
		\end{split}
	\end{equation}
	The factor $q^{rJ}$ (leading to 
	$q^{r(J_0+\ldots+J_y )}$ in the kernel) 
	disappears after replacing $\zeta$ by 
	$\zeta q^{-J_0-\ldots-J_y}$.
	This change of variable accounts for the fact 
	that in the left-hand side of 
	\eqref{eq:q_Laplace_princ_spec}
	we 
	take the $q$-Laplace transform of $q^{-\ell(\lambda)}$ as opposed to the height
	function in \Cref{thm:6vm_q_laplace}.
	
	By the absolute convergence result of 
	\Cref{prop:sHL_absolute_integrability} 
	and the 
	boundedness of $1/(\zeta q^{- \ell (\lambda)};q)_\infty$,
	the 
	left-hand side of \eqref{eq:q_Laplace_princ_spec} is an analytic
	function of $q^{J_k},q^{I_l}$ under the bounds 
	\eqref{eq:condition_parameters_general_skew_Cauchy_S9}.
	In order
	to establish the analyticity of the Fredholm determinant we first
	observe that, due to the compactness of $C$ and of the
	image of $r\to q^r$ for $r \in d + \mathrm{i}\mathbb{R}$, there
	exists a constant $M_1$ independent of $J_k$ or $I_l$ such that 
	\begin{equation*}
		\sup_{\substack{w,w' \in C,\\ r \in d + \mathrm{i}\mathbb{R}}} 
		\left| \frac{ f(w)/f(q^r w) }{q^r w - w'} \right| < M_1.
	\end{equation*}
	This implies that $|K(w,w')|<M_2$ integrating over $r$
	due to the exponential decay of $1/\sin\pi r$ for large $|r|$. 
	We can thus estimate the Fredholm determinant of $K$ with
	\begin{equation*}
		\sum_{l \ge 0} \frac{1}{l!} \int_C \cdots \int_C 
		\left|\det_{i,j=1}^l(K(w_i,w_j)) \right| dw_1 \cdots dw_l  \le \sum_{l \ge 0} \frac{ l^{l/2} M_3^l }{l!},
	\end{equation*}
	where we used the Hadamard inequality to bound the determinant of
	$K(w_i,w_j)$, and $M_3 = M_2 \ell(C)$. This shows that the
	right-hand side of \eqref{eq:q_Laplace_princ_spec} is an
	absolutely convergent sum of analytic functions and hence it is
	analytic. This completes the proof.
\end{proof}

\begin{remark}
	Fredholm determinantal expression \eqref{eq:q_Laplace_princ_spec}
	degenerates to a number of known results.
	In
	particular, considering the specialization
	$u_0=-\alpha\epsilon$, $v_0=-\beta\epsilon$ $q^{J_0}=q^{I_0}=1/\epsilon$, 
	$\epsilon\to0$, and 
	$J_1=\ldots=J_y=I_1=\ldots=I_x=1$,
	we recover the expression for the
	$q$-Laplace transform of the height function of the six vertex
	model with two-sided stationary bound conditions
	from
	\cite[Proposition 4.1]{Amol2016Stationary} (in the latter one has to set $\mu=0$). 
	The latter formula is obtained by a more involved
	analytic continuation in $q^{J_0}$ than in the proof of 
	\Cref{thm:q_Laplace_princ_spec}.
\end{remark}

\subsection{Higher spin six vertex model observables}

The
eigenrelations for the sqW or sHL functions 
give rise to moment formulas for the stochastic higher spin six vertex model.
Consider the model with step-Bernoulli boundary
conditions $\alpha=0$, $\beta\ne 0$ (see \Cref{sub:new_YB_field_sHL_sqW}).
Assume that the parameters
$\beta,\theta_1,\theta_2, \dots$ are spaced in such a way that
\begin{equation} \label{eq:conditions_hs6vm_q_moments}
    q \inf \left\{ (-1/\theta_i)_{i\ge 1} \cup (-1/\beta) \right\} >  \sup \left\{ (-1/\theta_i)_{i\ge 1} \cup (-1/\beta) \right\}.
\end{equation}
Let $\mathfrak{h}^{\mathrm{HS}}(x,y)$ be the height function of this model, i.e.,
the number of paths in the vertex model which are weakly to the right of the point $(x,y)$.

Following the same approach as in the proof of
\Cref{prop:6VM_q_moments} (applying
either $\widetilde{\mathfrak{D}}$ or $\mathfrak{E}$ from
\Cref{sec:diff_op} to the sum of the corresponding Cauchy identity), we obtain a
$q$-moment formula which was first written down in \cite{CorwinPetrov2015}:
\begin{proposition}
We have
\begin{equation} \label{eq:hs6vm_q_moments}
\begin{split}
	\mathbb{E} \left( q^{l \,\mathfrak{h}^{\mathrm{HS}} (x+1,y)  } \right) &= (-1)^l q^{l(l-1)/2} \oint_{\Gamma[\boldsymbol \theta,  \beta | 1 ]} \cdots \oint_{\Gamma[\boldsymbol \theta, \beta | l ]} \prod_{1 \le A < B \le l} \frac{z_A - z_B}{z_A - q z_B}
    \\
    & \hspace{30pt} \times
    \prod_{k=1}^l \left\{ \prod_{i=1}^y \frac{q z_k - u_i}{z_k - u_i } \prod_{i=1}^x \frac{1-z_k s}{1+z_k \theta_i} \frac{1}{1+\beta z_k} \frac{dz_k}{2 \pi \mathrm{i} z_k } \right\},
    \end{split}
\end{equation}
where the positively oriented
contour $\Gamma[\boldsymbol \theta,  \beta | j ]$ 
is around $-1/\beta, -1/\theta_1, \dots -1/\theta_x$, $q \Gamma[\boldsymbol \theta,  \beta | j + 1 ]$,
and no other pole of the integrand.
\end{proposition}

The observable with both $\alpha,\beta$ nonzero 
(i.e., with the two-sided stationary boundary conditions)
admits the following
Fredholm determinantal expression:

\begin{theorem} \label{thm:hs6vm_q_laplace}
Consider the higher spin six vertex model with two-sided stationary 
boundary conditions with parameters $\alpha,\beta$. 
Let $\mathcal{H}^{\mathrm{HS}}$ be the centered height function of this model
(cf. \Cref{sub:new_YB_field_sHL_sqW}), and
let $\mathcal{M}\sim q \textnormal{-}\mathrm{Poi}(\alpha \beta)$ be independent of the vertex model.
Then we have
\begin{equation} \label{eq:hs6vm_q_laplace}
	\mathbb{E} \left( \frac{1}{(\zeta q^{\mathcal{H}^{\mathrm{HS}}(x,y) - \mathcal{M}};q)_\infty} \right) = \det(Id + \mathcal{K})_{L^2\left( \mathcal{C} \right)}.
\end{equation}
The kernel $\mathcal{K}$ is defined by
\begin{equation} \label{eq:kernel_hs6vm}
    \mathcal{K}(w,w') = \frac{1}{2 \mathrm{i}} \int_{d + \mathrm{i} \mathbb{R} } \frac{(- \zeta)^r}{\sin(\pi r)} \frac{\mathpzc{f}(w) / \mathpzc{f}(q^r w) }{q^r w - w'}dr,
\end{equation}
where $d \in (0,1)$ and
\begin{equation}
	\label{eq:f_hsm_formula}
    \mathpzc{f}(w) = \frac{(-\alpha/w ; q)_\infty}{(- \beta w ; q)_\infty} \prod_{l=1}^y 
		\frac{1}{w - u_l}
		\prod_{l=1}^x \frac{(s w ; q)_\infty}{(-\theta_l w ; q)_\infty}.
\end{equation}
Here $\mathcal{C}$ is a closed complex contour encircling $0,u_1,u_2, \dots$ and such that for all $r\in d+\mathrm{i}\mathbb{R}$, $\mathcal{C}$ contains $- q^{r+k} \alpha$ for all $k \ge 0$, but leaves outside $1/(q^{r+k} s)$ and $-1/(q^k \beta), 1/(q^k \theta_l)$ for all $k,l \ge 0$.
\end{theorem}

\begin{proof}
	We use an analytic continuation argument staring from identity
	\eqref{eq:q_Laplace_princ_spec}. Considering specializations
	$\mathrm{sg}(\alpha)$ for $u_0, q^{J_0}$ and
	$\mathrm{sg}(\beta),\mathrm{sqW}(\theta_1),\mathrm{sqW}(\theta_2), \dots$,
	respectively for $v_0, q^{J_0}, v_1, q^{J_1}, v_2, q^{J_2}, \dots$, we
	can prove expression \eqref{eq:hs6vm_q_laplace} for values $\alpha,
	u_l, \beta, \theta_l$ is a small neighborhood of the origin. Once
	\eqref{eq:hs6vm_q_laplace} is established for parameters in
	an open set, we can perform an analytic continuation, always keeping
	them in a region where they define a probability measure. This is possible
	since both sides of \eqref{eq:hs6vm_q_laplace} can
	be written as absolutely convergent series of holomorphic functions in
	$\alpha, u_l, \beta, \theta_l$.
\end{proof}

Using the integral expression for the 
$q$-moments \eqref{eq:hs6vm_q_moments} we can obtain an alternative expression for the Fredholm determinant:

\begin{theorem} \label{thm:hs6vm_q_laplace_bis}
	Assume conditions \eqref{eq:conditions_hs6vm_q_moments}.
	Let $\widetilde{\mathcal{C}}$ be a closed positively oriented contour encircling
	$-1/\beta, -1/\theta_1, - 1/\theta_2 ,\dots$ and which does not
	contain any point of the interior of $q \, \widetilde{\mathcal{C}}$.
	Then the Fredholm determinantal formula 
	\eqref{eq:hs6vm_q_laplace} holds when replacing
	$\mathcal{C}$ with $\widetilde{\mathcal{C}}$.
\end{theorem}

\begin{proof}
		The $\alpha=0$ case of this Theorem can be shown following the
		steps outlined in the proof of \Cref{thm:6vm_q_laplace} (which in
		turn goes along the lines of	\cite{BorodinCorwin2011Macdonald},
		\cite{BorodinCorwinSasamoto2012}). When $\alpha > 0$ a $q$-moment
		expansion of the $q$-Laplace transform is not possible since the
		$l$-th $q$-moment becomes infinite for $l$ large enough. In order to
		include the case where $\alpha > 0$, we first produce a result
		analogous to that of Theorem \ref{thm:q_Laplace_princ_spec} and
		subsequently we use analytic continuation.
    
    We start by restating the result for $\alpha = 0$ as
    \begin{equation} 
		\label{eq:q_Laplace_sHL_sqW}
		\prod_{ k=1 }^y \frac{1}{1+u_k \beta}
		\prod_{\substack{1\le k\le y\\1\le j\le x}}
		\frac{ 1 - u_k s}{1 + u_k \theta_j} \sum_{\lambda} \frac{	\mathsf{F}_{\lambda}(u_1, \dots, u_y) \mathbb{F}^*_{\lambda^\prime} (\theta_1, \dots, \theta_x; \widetilde{\beta}) }{(\zeta q^{y - \ell (\lambda) } ;q )_\infty} 
		=
		\det\left( Id + \mathcal{K}\,\big\vert_{\alpha=0} \right)_{L^2 ( \widetilde{\mathcal{C}} )},
	\end{equation}
	where we used $y-\ell(\lambda^{(x,y)})\stackrel{d}{=}\mathfrak{h}^{\mathrm{HS}}(x+1,y)$,
	and the summation in the
	left hand side of \eqref{eq:q_Laplace_sHL_sqW} makes sense for
	$u_i, \theta_i, \beta, s$ in a complex neighborhood of the origin
	(under
	\eqref{eq:conditions_hs6vm_q_moments}). We can consider principal
	specializations of the sHL function and write the more general
	identity
    \begin{equation*} 
		\label{eq:K_tilde_hsm}
		\begin{split}
		& \prod_{ k=0 }^y \frac{(- u_k q^{J_k} \beta ; q )_\infty }{(- u_k \beta ; q )_\infty }
		\prod_{\substack{0\le k\le y\\1\le j\le x}}
		\frac{ (u_k s;q)_\infty (- \theta_j u_k q^{J_k} ; q )_\infty }
		{ (u_k q^{J_k} s ;q)_\infty (- \theta_j u_k ;q )_\infty }
		\\& \hspace{40pt} \times
		\sum_{\lambda} \frac{	\mathfrak{F}^{(J_0,\dots,J_y)}_{\lambda}(u_0, \dots, u_y) \mathbb{F}^*_{\lambda^\prime} (\theta_1, \dots, \theta_x; \widetilde{\beta}) }{(\zeta q^{J_0+\cdots+J_y - \ell (\lambda) } ;q )_\infty} 
		=
		\det\bigl( Id + \widetilde{\mathcal{K}} \bigr)_{L^2 ( \widetilde{\mathcal{C}} )},
	    \end{split}
	\end{equation*}
	which again holds for $u_i, q^{J_i}u_i, \beta, \theta_i$ close to the origin. 
	Here $\widetilde{\mathcal{K}}$ is given by
	\eqref{eq:kernel_hs6vm} up to replacing $\zeta$ by $\zeta q^{J_0 + \dots
	+J_y - y }$, and $\mathpzc{f}$ by
	\begin{equation*}
		\widetilde{\mathpzc{f}}(w) = \frac{1}{(-\beta w ; q)_\infty} 
		\prod_{l=0}^y \frac{(q^{J_l}u_l/w ;q)_\infty}{(u_l/w ; q)_\infty} 
		\prod_{l=1}^x \frac{(sw ;q )_\infty}{(-\theta_l w ;q)_\infty}.
	\end{equation*}
	(here we used computation \eqref{eq:why_q_rJ_appears}).
	We can now replace $\zeta$
	by $\zeta q^{-J_0}$ in both sides of \eqref{eq:K_tilde_hsm},
	and
	specialize parameters $u_l, q^{J_l} u_l$ as
	$\mathrm{sg}(\alpha), \mathrm{sHL}(u_1), \dots \mathrm{sHL}(u_y)$ to
	deduce the claim of the theorem for $\alpha, u_l, \beta, \theta_l,s$ in a neighborhood of the origin. 
	Indeed, under this specialization 
	we take $J_1=\ldots=J_y=1 $, and so 
	in the left-hand side we obtain the observable
	$(\zeta q^{y-\ell(\lambda)};q)_\infty^{-1}$, 
	and in the right-hand side the extra power $q^{ry}$ is absorbed by going back from
	$\widetilde{\mathpzc{f}}(w)$
	to 
	$\mathpzc{f}(w)$
	\eqref{eq:f_hsm_formula}.
	The analytic restrictions on the parameters $\alpha,\beta,u_l,\theta_k,s$
	can be further
	relaxed since both the $q$-Laplace transform and the Fredholm
	determinant are well defined and analytic
	when the parameters
	correspond to a probability measure and, moreover, satisfy
	\eqref{eq:conditions_hs6vm_q_moments}.
\end{proof}

\begin{remark}
	Another Fredholm determinantal formula 
	for the stochastic higher spin six vertex model
	with two-sided stationary boundary conditions
	was obtained recently in \cite{imamura2019stationary}.
	While this formula differs from ours, 
	one should in principle be able to transform one to the other.
	We do not focus on this in the present work.
\end{remark}

\subsection{${}_4 \phi_3$ stochastic vertex model observables.}

By using the fact that the ${}_4 \phi_3$ vertex model is equivalent in
distribution to a marginal of the sqW/sqW field we can obtain 
contour integral expressions for the $q$-moments
of the height function $\mathbb{H}^\phi$
(described in \Cref{sub:new_YB_field_sqW_sqW}).
Indeed, this is possible by employing the eigenoperator
$\mathfrak{E}$.
However, only finitely many of the $q$-moments 
exist, and this also involves certain bounds
on the parameters.
Consider the model with step-stationary boundary conditions
$\alpha=0,\beta\ne 0$.
Assume that $\beta,\theta_1,\theta_2, \dots$ 
satisfy \eqref{eq:conditions_hs6vm_q_moments}.

\begin{proposition}
If $l$ is such that $q^l > \max_{1\le t \le y} \{ \beta \xi_t \}$, we have
\begin{equation}
\begin{split}
    \mathbb{E}^{\mathrm{step}} 
		\left( q^{-l\, \mathbb{H}^\phi (x,y) } \right) &= (-1)^l q^{l(l-1)/2} \oint_{\Gamma[\boldsymbol \theta, \beta | 1 ]} \cdots \oint_{\Gamma[\boldsymbol \theta, \beta | l ]} \prod_{1 \le A < B \le l} \frac{z_A - z_B}{z_A - q z_B}
    \\
    & \hspace{30pt} \times
    \prod_{k=1}^l \left\{ \prod_{i=1}^t \frac{z_k - s/q}{z_k + \xi_i/q } \prod_{i=1}^x \frac{1-z_k s}{1+z_k \theta_i} \frac{1}{1+\beta z_k} \frac{dz_k}{2 \pi \mathrm{i} z_k } \right\},
    \end{split}
\end{equation}
where $\Gamma[\boldsymbol \theta, \beta | j ]$ is a
positively oriented contour
around $-1/\beta, -1/\theta_1, \dots
-1/\theta_x$, $q \Gamma[\boldsymbol \theta,  \beta | j + 1 ]$, and no
other pole of the integrand. In case $q^l \le \max_{1\le t
\le y} \{ \beta \xi_t \}$ we have $\mathbb{E}^{\mathrm{step}}_{ \phi
\mathrm{VM} } \left( q^{-l\, \mathbb{H}^\phi (x,y) } \right)=\infty$.
\end{proposition}

Despite the fact that the distribution of
$\mathbb{H}^{\phi}$ is not characterized by its $q$-moments since
only finitely many of them exist, we can still write down
Fredholm determinant expressions for the $q$-Laplace transform of 
$\mathbb{H}^{\phi}$.

\begin{theorem} \label{thm:phiVM_q_laplace}
Consider the ${}_4 \phi_3$ stochastic vertex model with two-sided
stationary boundary conditions with parameters
$(\alpha,\beta)$.
Let $\mathcal{M}\sim q \textnormal{-}\mathrm{Poi}(\alpha \beta)$ be independent of the vertex model.
We have 
\begin{equation} \label{eq:phiVM_q_laplace}
    \mathbb{E}_{\phi \mathrm{VM}(\alpha, \beta)} 
		\bigl( \frac{1}{(\zeta q^{-\mathbb{H}^\phi(x,y) - \mathcal{M}};q)_\infty} \bigr) 
		=
		\det(Id + \mathbb{K})_{L^2\left( \mathfrak{C} \right)}.
\end{equation}
The kernel $\mathbb{K}$ is defined by
\begin{equation}
    \mathbb{K}(w,w') = \frac{1}{2 \mathrm{i}} \int_{d + \mathrm{i} \mathbb{R} } \frac{(- \zeta)^r}{\sin(\pi r)} \frac{\mathbb{f}(w) / \mathbb{f}(q^r w) }{q^r w - w'}dr,
\end{equation}
where $d \in (0,1)$ and
\begin{equation}
    \mathbb{f}(w) = \frac{(-\alpha/w ; q)_\infty}{(- \beta w ; q)_\infty} \prod_{l=1}^y \frac{( - \xi_l / w ; q )_\infty}{( s / w ; q)_\infty} \prod_{l=1}^x \frac{( s w ; q)_\infty}{( -\theta_l w ; q )_\infty}.
\end{equation}
Here $\mathfrak{C}$ is a closed complex contour encircling $0,q^k s$
for $k\ge 0$ and such that, for any $r \in d + \mathrm{i}\mathbb{R}$,
$\mathfrak{C}$ contains $q^r \mathfrak{C}$ and $-q^{r+k} \xi_l,
-q^{r+k} \alpha$ for all $k,l \ge 0$, but leaves outside $1/(q^{r+k}
s)$ and $-1/(q^k \theta_l), 1/(q^k \beta)$ for all $k,l \ge 0$.
\end{theorem}

\begin{proof}
Expression \eqref{eq:phiVM_q_laplace} is derived from the general
summation identity \eqref{eq:q_Laplace_princ_spec} 
in the same way as 
\Cref{thm:hs6vm_q_laplace}. First we
establish \eqref{eq:phiVM_q_laplace} for parameters $\alpha, \beta, s,
\xi_l, \theta_l$ is a small neighborhood of the origin by considering
specializations of $u_0, q^{J_0}, u_1, q^{J_1}, \dots, v_0, q^{I_0},
v_1, q^{I_1}$ in \eqref{eq:q_Laplace_princ_spec}. Subsequently we
relax conditions on these parameters moving them away from the origin
but keeping them in real intervals in such a way that they always
define a probability measure. This 
is possible due to the analyticity of both
sides of 
\eqref{eq:phiVM_q_laplace} in the parameters.
\end{proof}

\begin{theorem} \label{thm:phiVM_q_laplace_bis}
Assume \eqref{eq:conditions_hs6vm_q_moments} and let $\widetilde{\mathfrak{C}}$ be a closed complex contour encircling $-1/\beta, -1/\theta_1, -1 / \theta_2, \dots$ and that does not contain any point of the interior of $q \, \widetilde{\mathfrak{C}}$. Then expression \eqref{eq:phiVM_q_laplace} holds  with contour $\mathfrak{C}$ replaced by $\widetilde{\mathfrak{C}}$.
\end{theorem}

\begin{proof}
		This alternative determinantal expression for the $q$-Laplace
		transform follows from
		\Cref{thm:hs6vm_q_laplace_bis}
		using the sqW specializations and 
		subsequent analytic continuation.
\end{proof}

\begin{remark}
	Both \Cref{thm:phiVM_q_laplace,thm:phiVM_q_laplace_bis} 
	degenerate to Fredholm determinantal formulas
	for the $q$-Hahn pushTASEP. In particular, expression given by
	\Cref{thm:phiVM_q_laplace_bis} was conjectured in
	\cite{CMP_qHahn_Push} (Conjecture 3.11)
	for step initial conditions.
	Therefore, we have established this conjecture.
	Moreover, by sending all parameters to $1$,
	one can also get the proof of
	\cite[Conjecture 4.6]{CMP_qHahn_Push}
	on the Laplace transform of the one-point
	observable in the beta polymer like model 
	introduced in \cite{CMP_qHahn_Push}.
\end{remark}

\appendix
\section{Yang-Baxter equations} 
\label{app:YBE}

Here we review the Yang-Baxter equations used
throughout the paper. 

\subsection{Basic cases}
\label{sub:app_basic}

All Yang-Baxter equations we use
can be traced to the following basic
one:
\begin{proposition}
	\label{prop:YBE_rww}
	Consider the vertex weights $w,r$ defined respectively in 
	\Cref{fig:table_w}
	and \Cref{fig:table_r}. Then we have 
\begin{equation} \label{eq:YBE_rww}
\begin{split}
&
\sum_{k_1,k_2,k_3}
r_{u/v}(i_2, i_1; k_2, k_1)\, w_{v,s} (i_3, k_1; k_3, j_1)\, w_{u,s}(k_3,k_2; j_3,j_2) \\
&\hspace{50pt}
= 
\sum_{k_1,k_2,k_3} 
w_{v,s} (k_3, i_1; j_3, k_1)\, w_{u,s}(i_3,i_2; k_3,k_2)\, r_{u/v}(k_2, k_1; j_2, j_1)
\end{split}
\end{equation}
for all $i_1,i_2,j_1,j_2\in\left\{ 0,1 \right\}$
and $i_3,j_3\in \mathbb{Z}_{\ge0}$.
A
visual representation of this equation
is given in Figure \ref{fig:YBE}.
\end{proposition}

\begin{figure}[htbp]
  \centering
	\includegraphics{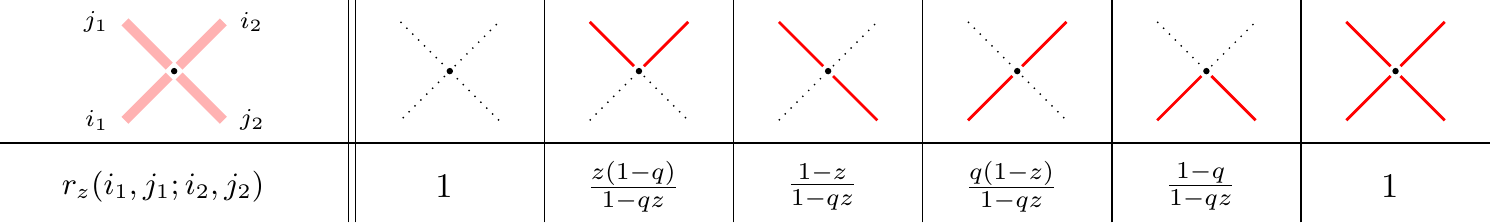}
  \caption{In the top row we see all acceptable configurations of paths entering and exiting a vertex; below we reported the corresponding vertex weights $r_z(i_1, j_1; i_2, j_2)$.} \label{fig:table_r}
\end{figure}

\begin{figure}[htbp]
\centering
\includegraphics[width=.7\textwidth]{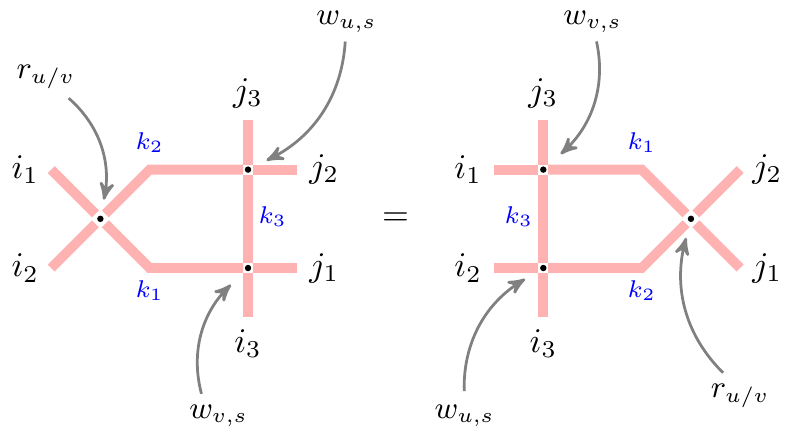}
\caption{A schematic representation of the Yang-Baxter equation \eqref{eq:YBE_rww}.} 
\label{fig:YBE}
\end{figure}

\begin{proof}[Proof of \Cref{prop:YBE_rww}]
	This is established by a straightforward verification.
	Equation \eqref{eq:YBE_rww} appeared in several other
	works, including \cite{Mangazeev2014}, \cite{Borodin2014vertex},
	\cite{BorodinWheelerSpinq}.
\end{proof}

As explained in Section \ref{sub:dual_sHL}, from vertex weights $w_{u,s}$ one
can define the dual weights $w^*_{v,s}$ by changing $u$ to $1/v$, swapping the
value of horizontal occupation numbers $0 \leftrightarrow 1$,
and 
multiplying by $(s-v)/(1 - s v)$ in order
to assign weight $1$ to the empty configuration. 
These manipulations clearly
preserve the structure of the Yang-Baxter equation, provided that the same swapping of the
occupation numbers is applied to the cross weight $r_z$. This leads to the 
definition of the cross weight $R_z$, see \Cref{fig:table_R}, 
also normalized so that the empty configuration has weight $1$.

\begin{figure}[htbp]
  \centering
	\includegraphics{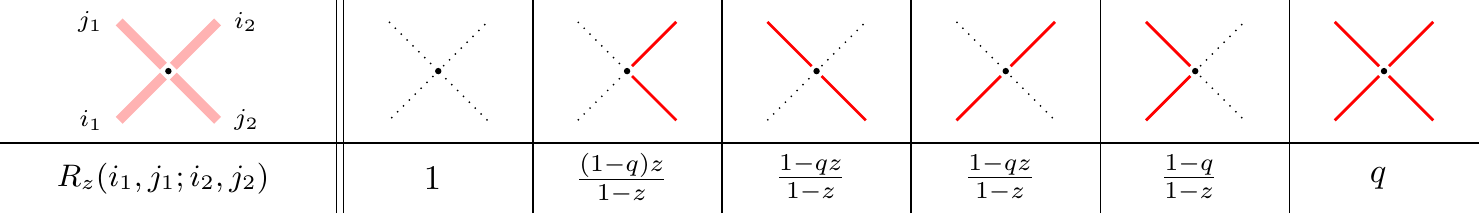}
  \caption{The cross vertex weights $R_z(i_1, j_1; i_2, j_2)$.}
	\label{fig:table_R}
\end{figure}

\begin{proposition}\label{prop:sHL_YBE}
Consider the vertex weights $w, w^*$ and $R$, defined respectively in \Cref{fig:table_w,fig:table_w_tilde},
and \Cref{fig:table_R}. 
Then we have 
\begin{equation} 
	\label{eq:sHL_YBE}
	\begin{split}
		&\sum_{k_1,k_2,k_3}R
		_{u v}(i_2, i_1; k_2, k_1)\,
		w^*_{v,s} (i_3, k_1; k_3, j_1)\,
		w_{u,s}(k_3,k_2; j_3,j_2) \\
		&\hspace{50pt}
		= 
		\sum_{k_1,k_2,k_3} w^*_{v,s} (k_3, i_1; j_3, k_1)\,
		w_{u,s}(i_3,i_2; k_3,k_2)\,
		R_{u v}(k_2, k_1; j_2, j_1)
	\end{split}
\end{equation}
for all 
$i_1,i_2,j_1,j_2\in \left\{ 0,1 \right\}$
and 
$i_3,j_3\in \mathbb{Z}_{\ge0}$.
\end{proposition}

\subsection{Fusion} \label{app:fusion}

Through a fusion procedure we generalize vertex weights $w_{u,s}$ and allow
configurations with multiple paths crossing a vertex in the horizontal
direction. 
This technique of generalizing solutions to the Yang-Baxter
equation 
was originally introduced in
\cite{KulishReshSkl1981yang} 
and consist in collapsing together a series of
vertically attached vertices with spectral parameters 
forming a
geometric
progression of ratio $q$. 
The fusion of vertex weights also admits a
probabilistic interpretation 
\cite{CorwinPetrov2015},
\cite{BorodinPetrov2016inhom},
\cite{BorodinWheelerSpinq}.  

Define the fused vertex weight
\begin{equation}\label{eq:w_fused_J}
\begin{split}
	w_{u,s}^{(J)} (i_1,j_1;i_2,j_2) &= \mathbf{1}_{i_1+j_1=i_2+j_2}\,
	\frac{(-1)^{i_1+j_2}q^{\frac{1}{2} i_1(i_1-1+2 j_1)} s^{j_2-i_1} u^{i_1} (u/s;q)_{j_1-i_2} (q;q)_{j_1} }{(q;q)_{i_1} (q;q)_{j_2} (s u;q)_{j_1+i_1}}\\
&
\qquad \qquad \times \setlength\arraycolsep{1pt}
{}_4 \overline{ \phi}_3\left(\begin{minipage}{4cm}
\center{$q^{-i_1}; q^{-i_2},  s u q^J,  qs/u$}\\\center{$s^2,q^{1+j_2-i_1}, q^{1-i_2-j_2+J}$}
\end{minipage}
\Big\vert\, q,q\right),
\end{split}
\end{equation}
where $\setlength\arraycolsep{1pt}{}_4 \overline{ \phi}_3$ is the regularized $q$-hypergeometric series
\eqref{eq:hypergeom_series}.
Here $J$ is originally a positive integer 
representing the number of vertices which were
fused together. 
However, it is easy to see that $w^{(J)}$ depends 
on $q^J$ in a rational way, thus $q^J$ can be regarded as the fourth independent
parameter in \eqref{eq:w_fused_J} (along with $u,s$, and $q$).
Since the regularized series 
$\setlength\arraycolsep{1pt}{}_4 \overline{ \phi}_3$
terminates, \eqref{eq:w_fused_J}
depends on all these parameters in a rational way.
Moreover, in case $i_1,i_2\to \infty$, the weight $w$ loses its dependence of $j_1$ and we have
\begin{equation} \label{eq:w_fused_infinity_normalized}
    \lim_{n \to \infty} w^{(J)}_{u,s} (n,j_1;n+j_1-j_2,j_2 )
		= 
		(-u q^J)^{j_2} \frac{(q^{-J} ; q)_{j_2} }{(q;q)_{j_2}} \frac{(suq^J;q)_\infty}{(su;q)_\infty}.
\end{equation}
Just as in the $J=1$ case, the fused boundary weight is obtained removing the normalization factor from \eqref{eq:w_fused_infinity_normalized}, and we define 
\begin{equation}\label{eq:w_fused_infinity}
    w^{(J)}_{u,s} \biggl(\begin{tikzpicture}[baseline=-2.5pt]
    	\draw[fill] (0,0) circle [radius=0.025];
			\node at (0,.3) {$\infty$};
			\node at (0,-.3) {$\infty$};
			\draw [red] (0.1,0) --++ (0.4, 0) node[right, black] {$k$};
			\draw [red] (0.1,0.05) --++ (0.4, 0);
			\draw [red] (0.1,-0.05) --++ (0.4, 0);
        \addvmargin{1mm}
        \addhmargin{1mm}
	  \end{tikzpicture} \biggr)\,= (-u q^J)^k \frac{(q^{-J} ; q)_k }{(q;q)_k}.
\end{equation}
This normalization is needed
to assign weight 1 to the empty configuration of paths in the grid. The fused analog of the dual weights $w$
is defined similarly to 
\eqref{eq:w_w_tilde_relation}:
\begin{equation} \label{eq:w_fused_I_dual}
	w^{*,(I)}_{v,s}(i_1,j_1;i_2,j_2) 
	= 
	\frac{(s^2;q)_{i_1} (q;q)_{i_2} }{ (q;q)_{i_1} (s^2;q)_{i_2} }
	\,
	w^{(I)}_{v,s}(i_2,j_1;i_1;j_2).
\end{equation}
These quantities also depend on $v,s,q$, and $q^I$ in a rational way.

What makes the fused weights remarkable 
is that they satisfy a general version of the Yang-Baxter equation
(previously in \Cref{sub:app_basic}
the horizontal occupation numbers had to be either 0 or 1). 
In
order to state this equation 
we need to consider the fusion of the cross weights~$R_z$
leading to
\begin{equation} \label{eq:R_fused_I_J}
\begin{split}
R_{z}^{(I,J)} \left( i_1,j_1;i_2,j_2 \right) &:= 
\mathbf{1}_{i_2 + j_1 = i_1 + j_2 }\,
\frac{ q^{ i_2 i_1 +\frac{1}{2}j_2(j_2 -1) + j_2 J } ( -z )^{j_2} (q;q)_{j_1}  }{ (z;q)_{j_1 + i_2} (q;q)_{j_2} (q;q)_{i_2} (q^{1-J}/z;q)_{i_1 -j_1} } \\
& \qquad \qquad \times 
{}_4 \overline{ \phi}_3\left(\begin{minipage}{4.2cm}
\center{$q^{-i_2}; q^{-i_1}, z q^I,  q^{1-J}/z$}\\\center{$q^{-J},q^{1+j_2-i_2}, q^{1-i_1-j_2+I}$}
\end{minipage} \Big\vert\, q,q\right).
\end{split}
\end{equation}

\begin{proposition}\label{prop:YBE_general_fused}
Consider the weights $w^{(J)}, w^{*,(I)}$ and $R^{(I,J)}$ defined in \eqref{eq:w_fused_J}, \eqref{eq:w_fused_I_dual}, \eqref{eq:R_fused_I_J}. 
Then we have 
\begin{equation} \label{YBE}
\begin{split}
	&\sum_{k_1,k_2,k_3}R_{uv}^{(I,J)}(i_2, i_1; k_2, k_1) \,
	w^{*,(I)}_{v,s} (i_3, k_1; k_3, j_1) \,
	w^{(J)}_{u,s}(k_3,k_2; j_3,j_2) \\
	&\hspace{50pt} = 
	\sum_{k_1,k_2,k_3} w^{*,(I)}_{v,s} (k_3, i_1; j_3, k_1) \,
	w^{(J)}_{u,s}(i_3,i_2; k_3,k_2)\, R_{uv}^{(I,J)}(k_2, k_1; j_2, j_1),
\end{split}
\end{equation}
for all admissible values of $i_1,i_2,j_1,j_2$
(that is, $i_1,j_1\in \left\{ 0,1,\ldots,I-1  \right\}$
for $I$ a positive integer,
or $i_1,j_1\in \mathbb{Z}_{\ge0}$
if $q^{I}$ is generic, and similarly for $i_2,j_2$),
and $i_3,j_3\in \mathbb{Z}_{\ge0}$.
See \Cref{fig:IJ_YBE_illustration_S4} for an illustration.
\end{proposition}
Note that in \eqref{YBE} (and in all other Yang-Baxter equations in this
Appendix) for fixed boundary occupation numbers $i_1,i_2,i_3,j_1,j_2,j_3$ 
the sums over $k_1,k_2,k_3$ in both sides are finite due to arrow preservation, 
so there are no convergence 
issues when $i_3$ and $j_3$ are finite.
For situations with infinitely many paths
one has to impose certain restrictions on parameters,
cf. \Cref{def:adm_rho} and 
\Cref{prop:U_well_defined}. 
\begin{remark} \label{rem:symmetry_R_fused}
    The fused cross weights $R^{(I,J)}$ inherit 
		symmetries of the unfused weight $R$ of Figure \ref{fig:table_R}. One of these is given by the identity
    \begin{equation} \label{eq:symmetry_R_fused}
        R^{(I,J)}_{z}(i_1,j_1;i_2,j_2)=R^{(J,I)}_{z}(j_1,i_1;j_2,i_2)
    \end{equation}
		for all $i_1,j_1,i_2,j_2\in \mathbb{Z}_{\ge0}$.
\end{remark}

\begin{proposition}
	\label{prop:emergence_of_the_cross_vertex_weight}
	Consider the vertex weight $R^{(I,J)}_{z}$ 
	defined in \eqref{eq:R_fused_I_J}.
	Then we have
\begin{equation} \label{eq:sum_cross_R}
    \sum_{k_1, k_2} R_z^{(I,J)}(a_2, a_1; k_1, k_2) 
		= 
		R_z^{(I,J)} (0, I; 0, I)
		=
		R_z^{(I,J)} (J,0 ; J,0)
		= 
		\frac{(z q^I;q)_{\infty} (z q^J ; q)_\infty}
		{(z;q)_\infty (zq^{I+J}; q)_\infty}
\end{equation}
for all $a_1,a_2 \in \mathbb{Z}_{\geq 0}$.
\end{proposition}
\begin{proof}
		The second and the third
		equalities in \eqref{eq:sum_cross_R} follow, after algebraic
		manipulations, from the definition of the fused cross weight $R^{(I,J)}_z$
		given in \eqref{eq:R_fused_I_J}.
    
		The first equality in \eqref{eq:sum_cross_R} is a trivial
		check in the case when $I=J=1$, using the definition of $R_z$ of
		\Cref{fig:table_R}. It lifts to more general $I,J$ as the fusion
		procedure does not affect the structure of the identity.
\end{proof}

\subsection{Spin \texorpdfstring{$q$}{q}-Whittaker specialization} \label{app:sqW_specialization}

The spin $q$-Whittaker specialization of the general fused weights \eqref{eq:w_fused_J},
\eqref{eq:w_fused_I_dual}
is obtained
by setting $u=s$ and $q^J=-\xi/s$ (recall that one can regard $q^J$ as a
generic parameter).
After this specialization the complicated
expression 
$w^{(J)}_{u,s}(i_1,j_1;i_2,j_2)$
\eqref{eq:w_fused_J} factorizes 
and becomes $W_{\xi,s}(i_1,j_1;i_2,j_2)$ given by 
\eqref{eq:Whit_W}.
Analogously, 
the dual fused weight $w^{*,(I)}_{v,s}(i_1,j_1;i_2,j_2)$
\eqref{eq:w_fused_I_dual}
turns into $W^*_{\theta,s}(i_1,j_1;i_2,j_2)$
\eqref{eq:W_W_tilde_relation}
after setting $v=s$ and $q^I=-\theta/s$.

The most general Yang Baxter equation \eqref{YBE}
specializes to Yang-Baxter equations involving
$W_{\xi,s}$ and $W^*_{\theta,s}$ as long as the corresponding
specializations are applied to the cross weight 
$R_{uv}^{(I,J)}$, too. Let us record the 
resulting identities:

\begin{proposition}
	\label{prop:sHL_sqW_YBE}
	We have the following Yang-Baxter equations:
	\begin{align}
		\label{eq:YBE_ws_W}
		\begin{split}
		&
		\sum_{k_1,k_2,k_3}
		\mathcal{R}_{\xi,v,s}(i_2, i_1; k_2, k_1)\,
		w^*_{v,s} (i_3, k_1; k_3, j_1)\,
		W_{\xi,s}(k_3,k_2; j_3,j_2) \\
		&\hspace{50pt}
		= 
		\sum_{k_1,k_2,k_3} \,
		w^*_{v,s} (k_3, i_1; j_3, k_1)\,
		W_{\xi,s}(i_3,i_2; k_3,k_2) \,
		\mathcal{R}_{\xi,v,s}(k_2, k_1; j_2, j_1);
		\end{split}
		\\
		\label{eq:YBE_W_w}
		\begin{split}
		&
		\sum_{k_1,k_2,k_3}
		\mathcal{R}^*_{\theta,u,s}(i_2, i_1; k_2, k_1)\,
		W^*_{\theta,s} (i_3, k_1; k_3, j_1)\,
		w_{u,s}(k_3,k_2; j_3,j_2) \\
		&\hspace{50pt} 
		= 
		\sum_{k_1,k_2,k_3}\,
		W^*_{\theta,s} (k_3, i_1; j_3, k_1)\,
		w_{u,s}(i_3,i_2; k_3,k_2)\,
		\mathcal{R}^*_{\theta,u,s}(k_2, k_1; j_2, j_1);
		\end{split}
		\\
		\label{eq:YBE_W_W}
		\begin{split}
		&
		\sum_{k_1,k_2,k_3}
		\mathbb{R}_{\xi,\theta,s}(i_2, i_1; k_2, k_1)\,
		W^*_{\theta,s} (i_3, k_1; k_3, j_1)\,
		W_{\xi,s}(k_3,k_2; j_3,j_2) \\
		&\hspace{50pt} 
		= 
		\sum_{k_1,k_2,k_3}
		W^*_{\theta,s} (k_3, i_1; j_3, k_1)\,
		W_{\xi,s}(i_3,i_2; k_3,k_2)\,
		\mathbb{R}_{\xi,\theta,s}(k_2, k_1; j_2, j_1).
		\end{split}
	\end{align}
	The cross vertex weights in \eqref{eq:YBE_ws_W} and \eqref{eq:YBE_W_w}
	are given in \Cref{fig:table_R_I,fig:table_R_J},
	respectively.
	Unlike with these two cases, in the third identity
	\eqref{eq:YBE_W_W}
	the cross vertex weights do not factorize
	(here $i_{1},i_{2},j_{1},j_{2}\in \mathbb{Z}_{\ge0}$):
	\begin{equation}\label{eq:Whittaker_cross_weight}
		\begin{split}
			\mathbb{R}_{\xi,\theta,s}(i_1,j_1;i_2,j_2)
			&=
			\mathbf{1}_{i_2 + j_1 = i_1 + j_2 }\,
			\frac{ q^{ i_2 i_1 +\frac{1}{2}j_2(j_2 -1) }(s \xi)^{j_2} (q;q)_{j_1}  }
			{ (s^2;q)_{j_1 + i_2} (q;q)_{j_2} (q;q)_{i_2} (-q/(s \xi );q)_{i_1 -j_1} } 
			\\
			& \qquad \qquad \times 
			{}_4 \overline{ \phi}_3
			\left(\begin{minipage}{5.2cm}
			\center{$q^{-i_2}; q^{-i_1}, -s \theta,  -q/(s \xi)$}
			\\
			\center{$-s/ \xi,q^{1+j_2-i_2}, -\theta q^{1-i_1-j_2}/s$}
			\end{minipage} \Big\vert\, q,q\right).
		\end{split}
	\end{equation}
\end{proposition}

\begin{figure}[htbp]
  \centering
	\includegraphics{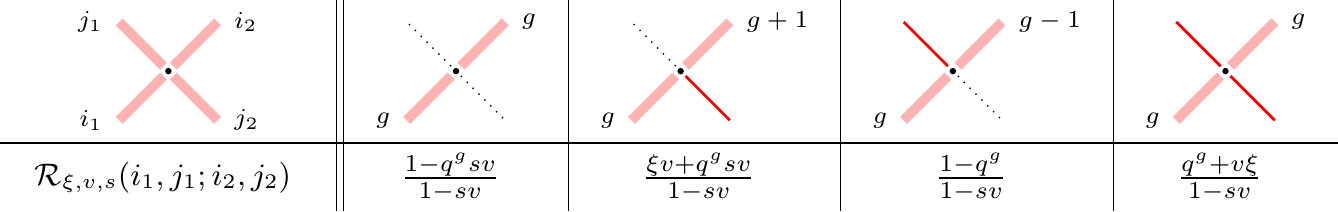}
	\caption{The cross vertex weights $\mathcal{R}_{\xi,v,s}(i_1, j_1; i_2, j_2)$, $j_1,j_2\in \left\{ 0,1 \right\}$, $i_1,i_2\in \mathbb{Z}_{\ge0}$.} 
	\label{fig:table_R_J}
\end{figure}

\begin{figure}[htbp]
  \centering
	\includegraphics{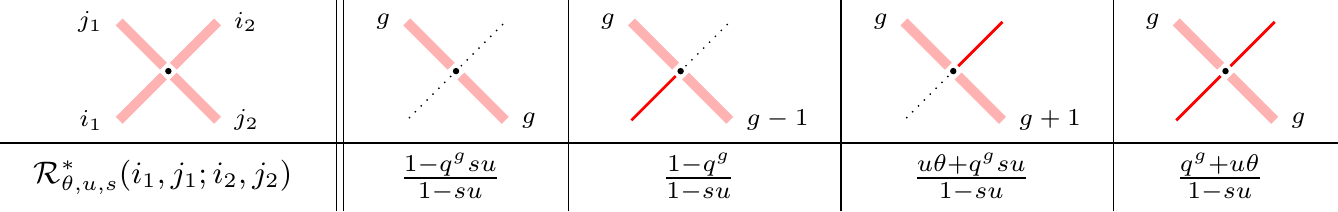}
  \caption{The cross 
	vertex weights $\mathcal{R}^*_{\theta,u,s}(i_1, j_1; i_2, j_2)$,
	$i_1,i_2\in \left\{ 0,1 \right\}$, $j_1,j_2\in \mathbb{Z}_{\ge0}$.}
	\label{fig:table_R_I}
\end{figure}

\subsection{Scaled geometric specialization}
\label{app:YBE_scaled_geometric}
The scaled geometric specialization of the general fused weight $w^{(J)}_{u,s}$ is given by setting $u=- \epsilon \alpha$, $q^J = 1/\epsilon$ and taking the limit $\epsilon \to 0$. Analogously we can specialize the dual weight $w^{*,(I)}_{v,s}$ taking $v=-\beta \epsilon$, $q^I = 1/\epsilon$ and again $\epsilon \to 0$. In this case the expessions \eqref{eq:w_fused_J}, \eqref{eq:w_fused_infinity} simplify:
\begin{equation} \label{eq:w_tilde}
    \widetilde{w}_{\alpha,s}(i_1,j_1;i_2,j_2) 
		=
		\mathbf{1}_{i_1 + j_1 = i_2 + j_2 }\, 
		\frac{(-\alpha/s)^{i_1} (-s)^{j_2} (q;q)_{j_1}}{(q;q)_{i_1}(q;q)_{j_2}} 
		\,
    {}_3\overline{\phi}_2 \left(\begin{minipage}{2.5cm}
			\center{$q^{-i_1}; q^{-i_2}, -s\alpha$}
			\\	\center{$s^2,q^{1+j_2-i_1}$}
	\end{minipage} \Big\vert\, q,-\frac{sq^{1+i_2+j_2}}{\alpha}\right),
\end{equation}
and
\begin{equation}\label{eq:w_tilde_infinity}
    \widetilde{w}_{\alpha,s} \biggl(\begin{tikzpicture}[baseline=-2.5pt]
    	\draw[fill] (0,0) circle [radius=0.025];
			\node at (0,.3) {$\infty$};
			\node at (0,-.3) {$\infty$};
			\draw [red] (0.1,0) --++ (0.4, 0) node[right, black] {$k$};
			\draw [red] (0.1,0.05) --++ (0.4, 0);
			\draw [red] (0.1,-0.05) --++ (0.4, 0);
        \addvmargin{1mm}
        \addhmargin{1mm}
	  \end{tikzpicture} \biggr)\,=  \frac{ \alpha^k }{(q;q)_k}.
\end{equation}
The dual weights $\widetilde{w}^*_{\beta,s}$
are defined in the usual way as in \eqref{eq:w_w_tilde_relation}.
\begin{figure}[ht]
    \centering
    \includegraphics{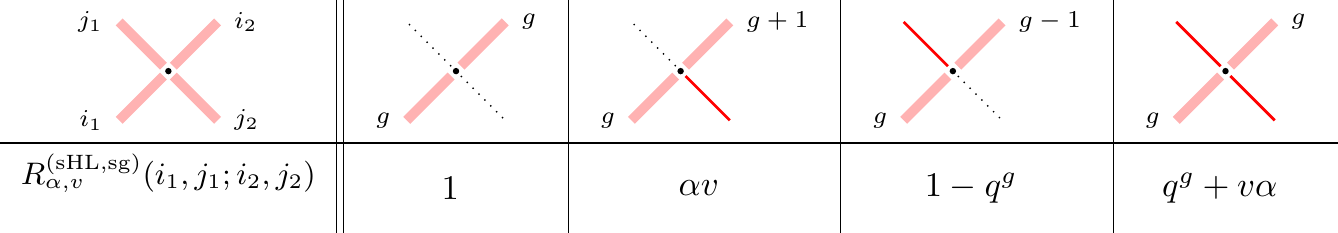}
		\caption{The cross vertex weight 
		$R^{(\mathrm{sHL,sg})}_{\alpha,v}(i_1,j_1;i_2,j_2)$, 
		$j_1,j_2\in\{0,1\}$, $i_1,i_2\in\mathbb{Z}_{\geq 0}$.}
    \label{fig:table_R_sHL_sg}
\end{figure}

We also consider the scaled geometric specialization of the fused cross weight  $R^{(I,J)}$, in this case in the parameters $v,q^I$, defining 
\begin{equation}\label{eq:R_tilde_J}
\begin{split}
    R_{u,\beta}^{(\mathrm{sg},J)}(i_1,j_1;i_2,j_2) 
		&=
		\mathbf{1}_{i_2 + j_1 = i_1 + j_2 }\,
		\frac{(-uq^J\beta)^{i_2} (q^{-J};q)_{i_2}}{(q;q)_{i_2}(q^{-J};q)_{i_1}} \,
    {}_3\overline{\phi}_2 \left(\begin{minipage}{2.5cm}
			\center{$q^{-i_1}; q^{-i_2}, -u \beta$}
			\\	\center{$q^{-J},q^{1+j_2-i_2}$}
			\end{minipage} \Big\vert\, q,-\frac{q^{1+i_1+j_2}}{u q^J \beta}\right).
\end{split}
\end{equation}
The scaled geometric specialization of $R^{(I,J)}$ in the parameters $u,q^J$ can be derived from \eqref{eq:R_tilde_J} using the symmetry \eqref{eq:symmetry_R_fused} and it is 
\begin{equation} \label{eq:R_tilde_dual}
    R^{(I,\mathrm{sg})}_{\alpha,v}(i_1,j_1;i_2,j_2) = R^{(\mathrm{sg},I)}_{v,\alpha}(j_1,i_1;j_2,i_2).
\end{equation}
Further degenerations of $R^{(I,\mathrm{sg})}_{\alpha, v}$ involve specializations of parameters $v,q^I$ in one of the three cases, $\mathrm{sHL}(v)$
(which is simply $I=1$), 
$\mathrm{sqW}(\theta)$, or $\textrm{sg}(\beta)$.
These cross vertex weights
are given, respectively, in \Cref{fig:table_R_sHL_sg} and below:
\begin{align}
		R^{\mathrm{(sqW, sg)}}_{\alpha,\theta}(i_1,j_1;i_2,j_2)
		&
		=
		\mathbf{1}_{i_2 + j_1 = i_1 + j_2 }\, 
		\frac{(\alpha \theta)^{j_2} (-s/\theta;q)_{j_2}}
		{(q;q)_{j_2}(-s/\theta;q)_{j_1}} 
		\,
    {}_3\overline{\phi}_2 \left(\begin{minipage}{2.5cm}
			\center{$q^{-j_1}; q^{-j_2}, -s \alpha$}
			\\	\center{$-s/\theta,q^{1+i_1-j_1}$}
			\end{minipage} \Big\vert\, q,\frac{q^{1+j_1+i_2}}{\alpha \theta}\right), \label{eq:R_sqW_sg}
    \\
		R^{\mathrm{(sg, sg)}}_{\alpha,\beta}(i_1,j_1;i_2,j_2)
		&=
		\mathbf{1}_{i_2 + j_1 = i_1 + j_2 }\,
		\frac{(\alpha \beta)^{j_2}}{(q;q)_{j_2}} \,
    {}_2\overline{\phi}_1 \left(\begin{minipage}{2cm}
			\center{$q^{-j_1}; q^{-j_2}$}
			\\	\center{$q^{1+i_1-j_1}$}
			\end{minipage} \Big\vert\, q,\frac{q^{1+j_1+i_2}}{\alpha \beta}\right). \label{eq:R_sg_sg}
\end{align}

These cross vertex weights enter a number of 
Yang-Baxter equations which are specializations
of the general fused one \eqref{YBE}:
\begin{proposition}
	\label{prop:YBE_scaled_geometric}
	We have the following Yang-Baxter equations:
	\begin{align}
		\label{eq:YBE_sHL_sg}
		\begin{split}
		&
		\sum_{k_1,k_2,k_3}
		R^{\mathrm{(sHL,sg)}}_{\alpha,v}(i_2, i_1; k_2, k_1)\,
		w^*_{v,s} (i_3, k_1; k_3, j_1)\,
		\widetilde{w}_{\alpha,s}(k_3,k_2; j_3,j_2) \\
		&\hspace{50pt}
		= 
		\sum_{k_1,k_2,k_3} \,
		w^*_{v,s} (k_3, i_1; j_3, k_1)\,
		\widetilde{w}_{\alpha,s}(i_3,i_2; k_3,k_2) \,
		R^{\mathrm{(sHL,sg)}}_{\alpha,v}(k_2, k_1; j_2, j_1);
		\end{split}
		\\
		\label{eq:YBE_sqW_sg}
		\begin{split}
		&
		\sum_{k_1,k_2,k_3}
		R^{\mathrm{(sqW,sg)}}_{\alpha,\theta}(i_2, i_1; k_2, k_1)\,
		W^*_{\theta,s} (i_3, k_1; k_3, j_1)\,
		\widetilde{w}_{\alpha,s}(k_3,k_2; j_3,j_2) \\
		&\hspace{50pt} 
		= 
		\sum_{k_1,k_2,k_3}\,
		W^*_{\theta,s} (k_3, i_1; j_3, k_1)\,
		\widetilde{w}_{\alpha,s}(i_3,i_2; k_3,k_2)\,
		R^{\mathrm{(sqW,sg)}}_{\alpha,\theta}(k_2, k_1; j_2, j_1);
		\end{split}
		\\
		\label{eq:YBE_sg_sg}
		\begin{split}
		&
		\sum_{k_1,k_2,k_3}
		R^{\mathrm{(sg,sg)}}_{\alpha,\beta}(i_2, i_1; k_2, k_1)\,
		\widetilde{w}^*_{\beta,s} (i_3, k_1; k_3, j_1)\,
		\widetilde{w}_{\alpha,s}(k_3,k_2; j_3,j_2) \\
		&\hspace{50pt} 
		= 
		\sum_{k_1,k_2,k_3}
		\widetilde{w}^*_{\beta,s} (k_3, i_1; j_3, k_1)\,
		\widetilde{w}_{\alpha,s}(i_3,i_2; k_3,k_2)\,
		R^{\mathrm{(sg,sg)}}_{\alpha,\beta}(k_2, k_1; j_2, j_1).
		\end{split}
	\end{align}
Dual cases of \eqref{eq:YBE_sHL_sg},\eqref{eq:YBE_sqW_sg},\eqref{eq:YBE_sg_sg} obtained swapping the specializations are easily derived making use of the symmetry of the cross weight \eqref{eq:R_tilde_dual}.
\end{proposition}

In Section \ref{sec:summary_sHL_sqW}, 
Cauchy Identities for spin Hall-Littlewood 
and spin $q$-Whittaker functions were stated 
as corollaries of the Yang-Baxter equations given in this appendix.
In particular, 
the emergence of the prefactors
in the right-hand sides
of all the skew Cauchy identities 
can be traced to Proposition 
\ref{prop:emergence_of_the_cross_vertex_weight}.

\subsection{Nonnegativity of terms in the Yang-Baxter equations}
\label{sub:YBE_nonnegativity}

Here we list conditions which are sufficient
for the nonnegativity of all terms in both sides of the Yang-Baxter equations
described in the previous parts of this Appendix.
We will not discuss which of these assumptions are necessary.
If the terms are nonnegative, then by
\Cref{prop:Bij_exists}
a stochastic bijectivization of the 
Yang-Baxter equation exists.
We assume that $s\in(-1,0)$ and $q\in(0,1)$ throughout the rest of
the subsection.

First, 
the weights $w_{u,s}$ and $w^*_{v,s}$ given in 
\Cref{fig:table_w}
and 
\Cref{fig:table_w_tilde} 
are nonnegative for
$u,v\in[0,1]$.
The cross vertex weights $r_{u/v}$ from \Cref{fig:table_r} are nonnegative when in addition $u<v$. 
Thus, 
\begin{equation*}
	\label{eq:nonneg_YBE_A1}
	\begin{minipage}{.8\textwidth}
		All summands in both sides of the Yang-Baxter
		equation \eqref{eq:YBE_rww}
		containing the weights $w_{u,s}, w_{v,s}$, and $r_{u/v}$
		are nonnegative if 
		$0\le u<v\le 1$.
	\end{minipage}
\end{equation*}
Next, the cross vertex weights 
$R_{uv}$ from \Cref{fig:table_R} are nonnegative 
when $0\le uv<1$. Therefore, 
\begin{equation*}
	\label{eq:nonneg_YBE_A2}
	\begin{minipage}{.8\textwidth}
		All summands in both sides of the Yang-Baxter
		equation \eqref{eq:sHL_YBE}
		containing the weights $w_{u,s}, w^*_{v,s}$, and $R_{uv}$
		are nonnegative if 
		$u,v\in[0,1)$. This in fact implies that $(u,v)\in \mathsf{Adm}$
		for the sHL/sHL skew Cauchy structure (\Cref{def:sHL_sHL_structure}).
	\end{minipage}
\end{equation*}

Let us now turn to the spin $q$-Whittaker weights. 
The weights $W_{\xi,s}$ and $W^*_{\theta,s}$ are nonnegative when 
$\xi , \theta \in[-s,-s^{-1}]$. The weights $\mathcal{R}_{\xi,v,s}$ and $\mathcal{R}^*_{\theta,u,s}$ 
from \Cref{fig:table_R_I,fig:table_R_J} are nonnegative when $u,v\in[0,1)$ and 
$\xi,\theta \in [-s,-s^{-1}]$. Thus, we have 
\begin{equation*}
	\label{eq:nonneg_YBE_A7_A8}
	\begin{minipage}{.8\textwidth}
		All summands in both sides of the Yang-Baxter
		equation \eqref{eq:YBE_ws_W}
		containing the weights $W_{\xi,s}, w^*_{v,s}$, and $\mathcal{R}_{\xi,v,s}$
		are nonnegative if 
		$v\in[0,1)$ and $\xi\in [-s,-s^{-1}]$.
		Similarly, the summands in \eqref{eq:YBE_W_w} are nonnegative
		for $u\in[0,1),\theta\in[-s,-s^{-1}]$.
	\end{minipage}
\end{equation*}

Further, let us consider 
\eqref{eq:YBE_W_W} containing $W_{\xi,s}$, $W_{\theta,s}^*$, 
and the non-factorized weights
$\mathbb{R}_{\xi,\theta,s}$~\eqref{eq:Whittaker_cross_weight}.
Their nonnegativity
is not as straightforward, and requires an additional restriction on the parameters $s$ and $q$:
\begin{proposition}
	\label{prop:mathbbR_nonnegative}
	For $\xi,\theta \in[-s,-s^{-1}]$, $q\in(0,1)$ and $s\in[-\sqrt q,0)$, 
	we have
	\begin{equation*}
		\mathbb{R}_{\xi,\theta,s}(i_1,j_1;i_2,j_2)\ge0\qquad \text{for all $i_1,j_1,i_2,j_2\in \mathbb{Z}_{\ge0}$}.
	\end{equation*}
\end{proposition}
The proof of this proposition is similar to 
\cite[Proposition 3.1]{CMP_qHahn_Push},
with an additional simplification 
in the second case
due to a 
symmetry of $\mathbb{R}_{\xi,\theta,s}$.
\begin{proof}[Proof of \Cref{prop:mathbbR_nonnegative}]
	Throughout the proof we will assume that $i_2+j_1=i_1+j_2$.
	We need to show that
	\begin{equation}
		\label{eq:mathbbR_nonneg_proof1}
		\frac{ s^{j_2} 
		(-\theta q^{1-i_1-j_2}/s;q)_{i_2}
		}
		{  (-q/(s \xi);q)_{i_1 -j_1} } 
		\,{}_4 { \phi}_3
		\left(\begin{minipage}{4.2cm}
				\center{$q^{-i_2}, -\frac{q}{s \xi},-s \theta,q^{-i_1}$}
		\\
		\center{$-\frac{\theta}{s}q^{1-i_1-j_2},-\frac{s}{\xi},q^{1+j_2-i_2}$}
	\end{minipage} \bigg\vert\, q,q\right)\ge0.
	\end{equation}
	Here we used \eqref{eq:hypergeom_series} 
	to get to the usual $q$-hypergeometric function, 
	and also the fact that the remaining 
	prefactor in $\mathbb{R}_{\xi,\theta,s}$ having the form
	\begin{equation*}
		\frac{
			q^{ i_2 i_1 +\frac{1}{2}j_2(j_2 -1) } \xi^{j_2}(q;q)_{j_1} 
			(-s/\xi;q)_{i_2}
			(q^{1+j_2-i_2};q)_{i_2}
		}
		{
			(s^2;q)_{j_1 + i_2} (q;q)_{j_2} (q;q)_{i_2}
		}
	\end{equation*}
	is nonnegative under our parameter restrictions 
	in a straightforward way.
	
	We will use
	Watson's transformation formula
	\cite[(III.19)]{GasperRahman}
	\begin{equation}
		\label{eq:Watson}
		{_{4}}\phi_{3} 
		\left(
			\begin{minipage}{2cm}
				\center{$q^{-n},a,b,c$}
			\\
			\center{$d,e,f$}
			\end{minipage}
			\bigg\vert\, q,q
		\right)
		=
		\frac{(d/b;q)_n (d/c; q)_{n}}{(d;q)_n(d/(bc); q)_{n}}\,
		{_{8}}\phi_{7} 
		\left(
			\begin{minipage}{5.2cm}
				\center{$q^{-n},\sigma,q\sigma^{1/2},-q\sigma^{1/2},\frac{f}{a},\frac{e}{a},b,c$}
			\\
			\center{$\sigma^{1/2},-\sigma^{1/2},e,f,\frac{ef}{ab},\frac{ef}{ac},\frac{efq^n}{a}$}
			\end{minipage}
			\bigg\vert\, q,\frac{efq^n}{bc}
		\right),
	\end{equation}
	where $d e f= a b c q^{1-n}$ and $\sigma = ef/aq$.

	\medskip\noindent
	\textbf{Case 1.} When $i_2\le j_2$, we apply \eqref{eq:Watson} to 
	\eqref{eq:mathbbR_nonneg_proof1} with $n=i_2$.
	The prefactor in \eqref{eq:Watson} combined with the one from 
	\eqref{eq:mathbbR_nonneg_proof1} becomes
	\begin{equation*}
		\frac{ s^{j_2} 
		}
		{  (-q/(s \xi);q)_{i_1 -j_1} } 
		\frac{(q^{1-i_1-j_2}/s^{2};q)_{i_2}(- \theta q^{1-j_2}/s;q)_{i_2}}{(q^{1-j_2}/s^2;q)_{i_2}}.
	\end{equation*}
	We have
	\begin{equation*}
		\frac{(q^{1-i_1-j_2}/s^{2};q)_{i_2}}{(q^{1-j_2}/s^2;q)_{i_2}}
		=\prod_{m=1}^{i_2}\frac{q^{m-j_2}q^{-i_1}-s^2}{q^{m-j_2}-s^2}\ge0,
	\end{equation*}
	since $m-j_2\le 0$ in the product.
	We also have
	\begin{equation*}
		\frac
		{s^{j_2}(- \theta q^{1-j_2}/s;q)_{i_2}}
		{(-q/(s \xi);q)_{i_1 -j_1}}
		=s^{j_2}\prod_{k=1}^{i_2}\left( 1+\frac{\theta}{s}\,q^{k-j_2} \right)\prod_{m=1}^{j_2-i_2}
		\left( 1+\frac{q^{1-m}}{ s \xi } \right)\ge0,
	\end{equation*}
	since all factors above (including $s^{j_2}$)
	are nonpositive, and there is a total of $2j_2$ of them.

	The $q$-hypergeometric function 
	after applying \eqref{eq:Watson} to 
	\eqref{eq:mathbbR_nonneg_proof1}
	takes the form
	\begin{equation*}
		{_{8}}\phi_{7} 
		\left(
			\begin{minipage}{8.2cm}
				\center{$q^{-i_2},\sigma,q\sigma^{1/2},-q\sigma^{1/2},-s \xi q^{j_2-i_2},\frac{s^2}{q},-s \theta,q^{-i_1}$}
			\\
			\center{$\sigma^{1/2},-\sigma^{1/2},-\frac{s}{\xi},q^{1+j_2-i_2},
				-\frac{s}{\theta}q^{j_2-i_2}
				,
			s^2q^{j_1},
			s^2q^{j_2}
			$}
			\end{minipage}
			\bigg\vert\, q,
			\frac{q^{i_1+j_2+1}}{\xi\theta}
		\right),
	\end{equation*}
	with $\sigma=s^2 q^{j_2-i_2-1}\in(0,1)$ because $s^2\le q$.
	One readily sees that each summand in this
	(terminating) $q$-hypergeometric series is nonnegative.
	Indeed, the only negative signs may come from
	$(q^{-i_2};q)_k$,
	$(q^{-i_1};q)_k$,
	and 
	$(s^2 q^{-1};q)_k$.
	However, the product of the former two factors is always nonnegative,
	and 
	$(s^2 q^{-1};q)_k\ge0$ also due to our additional parameter 
	restriction $s^2\le q$.
	This implies the nonnegativity of $\mathbb{R}_{\xi,\theta,s}(i_1,j_1;i_2,j_2)$
	when $i_2\le j_2$.

	\medskip\noindent
	\textbf{Case 2.}
	When $i_2>j_2$, the claim follows 
	due to the symmetry of $\mathbb{R}_{\xi,\theta,s}$. Namely, by means of Remark \ref{rem:symmetry_R_fused}, we have
	\begin{equation*}
		\mathbb{R}_{\xi,\theta,s}(i_1,j_1;i_2,j_2)
		=
		\mathbb{R}_{\theta,\xi,s}(j_1,i_1;j_2,i_2)
	\end{equation*}
	for all $i_1,j_1,i_2,j_2\in \mathbb{Z}_{\ge0}$.
	This completes the proof.
\end{proof}

\Cref{prop:mathbbR_nonnegative} implies that
\begin{equation*}
	\label{eq:nonneg_YBE_A9}
	\begin{minipage}{.8\textwidth}
		All summands in both sides of the Yang-Baxter
		equation \eqref{eq:YBE_W_W}
		containing the weights $W_{\xi,s}, W^*_{\theta,s}$, and $\mathbb{R}_{\xi,\theta,s}$
		are nonnegative if 
		$\xi,\theta\in [-s,-s^{-1}]$, $q\in(0,1)$, and $s\in[-\sqrt q,0)$.
	\end{minipage}
\end{equation*}

Finally, we address the nonnegativity
of terms of the Yang-Baxter equations 
involving scaled geometric 
specializations from \Cref{prop:YBE_scaled_geometric}. 

\begin{proposition} \label{prop:w_tilde_positivity}
For $\alpha \in[0, -s^{-1}]$, $q\in(0,1)$ and $s\in(-1,0)$ we have
\begin{equation*}
    \widetilde{w}_{\alpha,s}(i_1,j_1;i_2,j_2) \geq 0 \qquad \text{for all }i_1,j_1,i_2,j_2\in\mathbb{Z}_{\geq0}.
\end{equation*}
\end{proposition}
\begin{proof}
    Under our assumptions the prefactor
    \begin{equation*}
        \frac{(-\alpha/s)^{i_1} (-s)^{j_2} (q;q)_{j_1}}{(q;q)_{i_1}(q;q)_{j_2}}
    \end{equation*}
    is nonnegative. 
		To check the remaining term, we write down the generic summand of the terminating $q$-hypergeometric series as
		(cf. \eqref{eq:hypergeom_series}):
    \begin{equation*}
        \left( \frac{-s q^{1+j_2+i_2}}{\alpha} \right)^k 
				\frac{(q^{-i_1};q)_k}{(q;q)_k}
				(q^{-i_2};q)_k
				(-s\alpha;q)_k (s^2q^k;q)_{i_1-k} 
				(q^{1+j_2-i_1+k};q)_{i_1-k},
    \end{equation*}
    where $k=0, \dots, i_1$. 
		The leading monomial term, along with 
		$(s^2q^k;q)_{i_1-k}$ and $(q;q)_k$ are always nonnegative.
		The
		$q$-Pochhammer symbols of $q^{-i_1}$ and $q^{-i_2}$ either vanish, or they both carry a sign $(-1)^k$, so that their contribution is nonnegative too. 
		Next, $(q^{1+j_2-i_1+k};q)_{i_1-k}$ is either nonnegative if $1+j_2-i_1+k > 0$, 
		or vanishes if $1+j_2-i_1+k \le 0$ (in the latter case, the last term of the product has power $j_2\ge0$,
		which means that that product passes through $1-q^0=0$).
		Finally, $(-s\alpha;q)_k\ge0$ because $\alpha\le -s^{-1}$.
\end{proof}
\Cref{prop:w_tilde_positivity} 
and the explicit form of 
$R^{(\mathrm{sHL,sg})}_{\alpha,v}$
(\Cref{fig:table_R_sHL_sg})
implies that
\begin{equation*}
    \begin{minipage}{.8\textwidth}
		All summands in both sides of the Yang-Baxter
		equation \eqref{eq:YBE_sHL_sg}
		containing the weights $\widetilde{w}_{\alpha,s}, w^*_{v,s}$, and $R^{(\mathrm{sHL,sg})}_{\alpha,v}$
		are nonnegative if 
		$\alpha\in [0,-s^{-1}]$, $v\in[0,1)$.
	\end{minipage}
\end{equation*}
In order to demonstrate the nonnegativity of \eqref{eq:YBE_sqW_sg} we consider the corresponding
cross vertex weight:
\begin{proposition} \label{prop:R_sqw_sg_positivity}
	For $\alpha \in [0,-s^{-1}]$ and $\theta \in [-s,-s^{-1}]$, we have
\begin{equation*}
    R^{(\mathrm{sqW,sg})}_{\alpha,\theta}(i_1,j_1;i_2,j_2) \geq 0 \qquad \text{for all }i_1,j_1,i_2,j_2\in\mathbb{Z}_{\geq0}.
\end{equation*}
\end{proposition}
\begin{proof}
    Assume first that $\theta > -s$. 
		In \eqref{eq:R_sqW_sg}, the factors outside ${}_3\overline{\phi}_2$
		are nonnegative. In the expansion of
		${}_3\overline{\phi}_2$
		using \eqref{eq:hypergeom_series},
		one readily sees that all terms are nonnegative 
		similarly to the proof of \Cref{prop:w_tilde_positivity}
		above (here we use the fact that $-s\alpha$ and $-s/\theta$ are less than 1 because of our assumptions). 
    
		We can now take the limit $\theta \to -s$ and show that the weight
		$R^{(\mathrm{sqW,sg})}$ survives this transition.
		To do so, expand 
		${}_3\overline{\phi}_2$ using \eqref{eq:hypergeom_series}, and collect terms containing
		$-s/\theta$:
    \begin{equation*}
        \frac{(-s/\theta;q)_{j_2} (-q^k s/\theta;q)_{j_1-k}}{(-s/\theta;q)_{j_1}} = (-q^k s/\theta;q)_{j_2-k},
    \end{equation*}
		with $k=0,\dots, \min(j_1,j_2)$. The last
		expression is nonsingular at $\theta=-s$, and is nonnegative.
\end{proof}
Therefore, 
\begin{equation*}
    \begin{minipage}{.8\textwidth}
		All summands in both sides of the Yang-Baxter
		equation \eqref{eq:YBE_sqW_sg}
		containing $\widetilde{w}_{\alpha,s}, W^*_{\theta,s}$, and $R^{(\mathrm{sqW,sg})}_{\alpha,\theta}$
		are nonnegative if 
		$\alpha\in [0,-s^{-1}]$, $\theta\in[-s,-s^{-1}]$.
	\end{minipage}
\end{equation*}

We come now to the last Yang-Baxter equation we stated \eqref{eq:YBE_sg_sg},
in which one readily sees (similarly to  
\Cref{prop:w_tilde_positivity,prop:R_sqw_sg_positivity} above)
that $R^{(\mathrm{sg,sg})}_{\alpha,\beta}$ is nonnegative when $0\le \alpha,\beta\le-s^{-1}$.
Therefore, 
\begin{equation*}
    \begin{minipage}{.8\textwidth}
		All summands in both sides of the Yang-Baxter
		equation \eqref{eq:YBE_sg_sg}
		containing $\widetilde{w}_{\alpha,s}, \widetilde{w}^*_{\beta,s}$, and $R^{(\mathrm{sg,sg})}_{\alpha,\beta}$
		are nonnegative if 
		$\alpha ,\beta \in [0,-s^{-1}]$.
	\end{minipage}
\end{equation*}

\printbibliography

\end{document}